\documentclass[10pt, letterpaper, openany]{book}

\usepackage[dvips, hcentering,
includehead,width=14.2cm,
top=0.5cm,
height=21.6cm,
paperheight=23.8cm,paperwidth=17cm
]{geometry}

\usepackage{textcomp}
\usepackage{amsmath}
\usepackage{amsthm}
\usepackage{amsfonts}
\usepackage{amssymb}
\usepackage{mathtools}
\usepackage{enumitem}
\usepackage{svg}
\svgsetup{inkscapearea=page,inkscapelatex=false}
\usepackage{float}
\usepackage{gensymb}
\usepackage{longtable,booktabs,array,calc}
\usepackage{hyperref}
\usepackage{cleveref}

\usepackage{xcolor}
\hypersetup{
    colorlinks,
    linkcolor={red!50!black},
    citecolor={blue!50!black},
    urlcolor={blue!80!black}
}

\usepackage[
backend=biber,
style=alphabetic,
]{biblatex}
\theoremstyle{plain}
\newtheorem{theorem}{Theorem}[section]
\newtheorem{lemma}[theorem]{Lemma}
\newtheorem{corollary}[theorem]{Corollary}

\newtheorem{proposition}[theorem]{Proposition}
\theoremstyle{definition}
\newtheorem{definition}{Definition}[section]
\theoremstyle{remark}
\newtheorem{remark}{Remark}[section]

\def\tightlist{}
\setlist[enumerate,1]{label=\arabic*.}
\setlist[enumerate,2]{label=(\alph).}
\makeatletter
\def\maxwidth{\ifdim\Gin@nat@width>\linewidth\linewidth\else\Gin@nat@width\fi}
\def\maxheight{\ifdim\Gin@nat@height>\textheight\textheight\else\Gin@nat@height\fi}
\makeatother
\setkeys{Gin}{width=\maxwidth,height=\maxheight,keepaspectratio}

\title{Optimality of Gerver's Sofa}
\author{Jineon Baek\thanks{Department of Mathematics, Yonsei University, Seoul, Korea. \texttt{jineon@yonsei.ac.kr}}}
\date{\today}

\begin{document}

\frontmatter
\maketitle

\chapter*{Abstract}
\addcontentsline{toc}{chapter}{\protect\numberline{~}Abstract}

We resolve the \emph{moving sofa problem} by showing that Gerver's construction with 18 curve sections attains the maximum area $2.2195\cdots$.


\begin{center}
\includegraphics{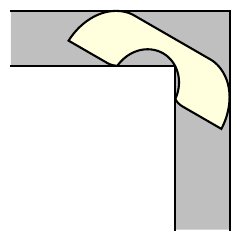}
\end{center}


\tableofcontents

\mainmatter
\chapter{Moving Sofa Problem}
\label{sec:moving-sofa-problem}

\section{Introduction}
\label{sec:introduction}
Moving a large couch through a narrow hallway requires a well-planned pivoting. The \emph{moving sofa problem} is asked in a two-dimensional idealization of such a situation:

\begin{quote}
What is the largest area \(\alpha_{\text{max}}\) of a connected planar shape that can move around the right-angled corner of a hallway with unit width?
\end{quote}

Such a movable shape is called a \emph{moving sofa} that we define precisely as below.

\begin{definition}

Define the \emph{hallway} \(L\) as the union \(L := H_L \cup V_L\) of its \emph{horizontal side} \(H_L := (-\infty, 1] \times [0, 1]\) and \emph{vertical side} \(V_L := [0, 1] \times (-\infty, 1]\).

\label{def:hallway}
\end{definition}

\begin{definition}

A \emph{moving sofa} \(S\) is any translation\footnote{We allow arbitary translation of a moving sofa \(S\) to locate it at any position we want, even outside the hallway \(L\). Only a translation of \(S\) needs to be inside the horizontal side \(H_L\), navigate its way inside \(L\), and end at the vertical side \(V_L\).} of a nonempty, connected, and closed\footnote{Taking the closure of \(S\) does not hurt the movability.} subset of \(H_L\) that can be moved inside \(L\) by a continuous rigid motion to a subset of \(V_L\).\footnote{Recall that the \emph{special Euclidean group} \(\mathrm{SE}(2)\) is the Lie group of all sign-preserving isometries of \(\mathbb{R}^2\). The movability can be stated formally as follows: there is a continuous curve \(\Phi_t \in \text{SE}(2)\) parametrized by \(t \in [0, 1]\), such that \(\Phi_0\) is a translation, \(\Phi_0(S) \subseteq H_L\), \(\Phi_t(S) \subseteq L\) for all \(t \in [0, 1]\), and \(\Phi_1(S) \subseteq V_L\).}

\label{def:moving-sofa}
\end{definition}

The moving sofa problem combines the two objectives of \emph{motion planning} and \emph{area maximization}. Despite numerous works on each subject, the problem has remained open since the initial publication by Leo Moser in 1966 \autocite{moser1966problem}.

\begin{definition}

Denote the area (Borel measure) of a Borel measurable \(X \subseteq \mathbb{R}^2\) as \(|X|\).

\label{def:area}
\end{definition}

The best bounds known so far on the maximum area \(\alpha_{\max}\) of a moving sofa are summarized as
\begin{equation}
\label{eqn:area-bounds}
|G| = 2.2195\cdots \leq \alpha_{\max} \leq 2.37.
\end{equation}
The lower bound comes from Gerver’s sofa \(G\) of area \(|G| = 2.2195\dots\) constructed in 1992 \autocite{gerverMovingSofaCorner1992} (see \Cref{fig:gerver}). The upper bound comes from a computer-assisted approach of Kallus and Romik in 2018 \autocite{kallusImprovedUpperBounds2018}.

\begin{figure}
\centering
\includegraphics{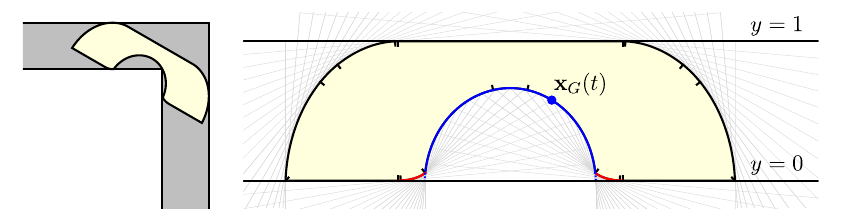}
\caption{Gerver’s sofa \(G\). The ticks denote the endpoints of 18 analytic curves and segments constituting the boundary of \(G\) \autocite{romikDifferentialEquationsExact2018}. The supporting hallways \(L_t\) containing \(G\) are depicted as grey in the right side.}
\label{fig:gerver}
\end{figure}

There were many evidences supporting that Gerver’s sofa \(G\) attains the maximum area \(\alpha_{\max} = |G|\). Gerver proved that a maximum-area moving sofa satisfies a certain local optimality condition (Theorem 1 of \autocite{gerverMovingSofaCorner1992}), and showed that his sofa \(G\) also satisfies the same condition (Theorem 2 of \autocite{gerverMovingSofaCorner1992}). Local optimality of \(G\) was further explored in \autocite{romikDifferentialEquationsExact2018} and \autocite{dengSolvingMovingSofa2024}, and many numerical experiments also supported \(\alpha_{\max} = |G|\) \autocite{gibbsComputationalStudySofas2014,batschNumericalApproachAnalysing2022,lengDeepLearningEvidence2024}.

We show that Gerver’s sofa \(G\) indeed attains the maximum area. The proof does not require computer assistance, except for numerical computations that can be done on a scientific calculator.

\begin{theorem}

Gerver’s sofa \(G\) attains the maximum area \(\alpha_{\max}\) of a moving sofa.

\label{thm:main}
\end{theorem}

The problem is difficult because there is no universal formula for the area that works for all possible moving sofas. To address this, we prove a property called the \emph{injectivity condition} for a maximum-area moving sofa \(S_{\max}\). For each moving sofa \(S\) satisfying the condition, we will define a larger shape \(R\) that resembles the shape of Gerver’s sofa (\Cref{fig:upper-bound}). The area \(\mathcal{Q}(S)\) of \(R\) is then an upper bound of the area of \(S\), and \(\mathcal{Q}(S)\) matches the exact area of \(S\) if it is Gerver’s sofa \(G\). Injectivity condition of \(S\) ensures that the boundary of region \(R\) forms a Jordan curve, allowing us to compute \(\mathcal{Q}(S)\) by using Green’s theorem.

\begin{figure}
\centering
\includegraphics{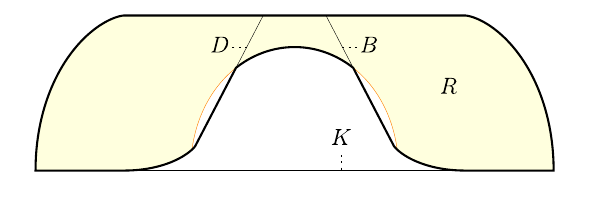}
\caption{A moving sofa \(S\) (light yellow) is enclosed by a slightly larger region \(R\) (bold lines) of area \(\mathcal{Q}(S)\) with a shape similar to Gerver’s sofa. Three convex bodies \(K, B\), and \(D\) represent different parts of \(R\) (bold and thin lines). \(K\) is a superset of \(R\), and \(B, D\) are subsets of \(R\).}
\label{fig:upper-bound}
\end{figure}

The upper bound \(\mathcal{Q}(S)\) of the area of a moving sofa \(S\) is then maximized with respect to \(S\) as follows. We use Brunn-Minkowski theory to express \(\mathcal{Q}\) as a quadratic functional on the space \(\mathcal{L}\) of tuples \((K, B, D)\) of convex bodies (\Cref{fig:upper-bound}). We use Mamikon’s theorem to establish the global concavity of \(\mathcal{Q}\) on \(\mathcal{L}\) (\Cref{fig:mamikon-thm-on-sofa}). We use the local optimality equations on Gerver’s sofa \(G\) by Romik \autocite{romikDifferentialEquationsExact2018} to show that \(S = G\) locally maximizes \(\mathcal{Q}(S)\). Because \(\mathcal{Q}\) is concave, \(G\) also maximizes \(\mathcal{Q}\) globally. As the upper bound \(\mathcal{Q}\) matches the area at \(G\), the sofa \(G\) also maximizes the area globally, estalishing \Cref{thm:main}.

The full proof of \Cref{thm:main} is divided into three main steps. Step 1 restricts the possible shapes of a maximum-area moving sofa \(S_{\max}\). Step 2 establishes the injectivity condition for \(S_{\max}\). Step 3 constructs the upper bound \(\mathcal{Q}(S)\) for the area of a moving sofa \(S\) satisfying the injectivity condition, and maximizes \(\mathcal{Q}(S)\) with respect to \(S\).

\begin{enumerate}
\def\labelenumi{\arabic{enumi}.}
\tightlist
\item
  Reduce the possible shapes of \(S_{\max}\).

  \begin{enumerate}
  \def\labelenumii{(\alph{enumii})}
  \tightlist
  \item
    \(S_{\max}\) is \emph{monotone} (\Cref{sec:_monotone-sofas-and-caps}, \Cref{sec:monotone-sofas-and-caps}).
  \item
    \(S_{\max}\) is \emph{balanced} (\Cref{sec:_balanced-maximum-sofas-and-caps}, \Cref{sec:balanced-maximum-sofas-and-caps}).
  \item
    \(S_{\max}\) have \emph{rotation angle} \(\pi/2\) (\Cref{sec:_rotation-angle-of-balanced-maximum-sofas}, \Cref{sec:rotation-angle-of-balanced-maximum-sofas}).
  \end{enumerate}
\item
  Show that \(S_{\max}\) satisfies the \emph{injectivity condition} (\Cref{sec:_injectivity-condition}, \Cref{sec:injectivity-condition}).
\item
  Establish the upper bound \(\mathcal{Q}\) of sofa area with injectivity condition (\Cref{sec:_optimality-of-gerver's-sofa}, \Cref{sec:optimality-of-gerver's-sofa}).

  \begin{enumerate}
  \def\labelenumii{(\alph{enumii})}
  \tightlist
  \item
    Define the convex domain \(\mathcal{L}\) of \(\mathcal{Q}\) (\Cref{sec:_definition-of-q}, \Cref{sec:domain-of-q}).
  \item
    Define a quadratic functional \(\mathcal{Q}\) on \(\mathcal{L}\) and show that it is an upper bound of sofa area (\Cref{sec:quadraticity-of-q}, \Cref{sec:definition-of-q}).
  \item
    Show that \(\mathcal{Q}\) is concave on \(\mathcal{L}\) (\Cref{sec:optimality-of-q-at-gerver's-sofa}, \Cref{sec:concavity-of-q}).
  \item
    Show that Gerver’s sofa is a local (and thus global) optimum of \(\mathcal{Q}\) (\Cref{sec:optimality-of-q-at-gerver's-sofa}, \Cref{sec:directional-derivative-of-q}).
  \end{enumerate}
\end{enumerate}

Step 1-(a) narrows down the possible shapes of \(S_{\max}\) to a \emph{monotone sofa}, a convex body with a dent carved out by the inner corner of the supporting hallways (\Cref{fig:monotone-sofa}). Step 1-(b) reprove an important local optimality condition by Gerver that the side lengths of \(S_{\max}\) should balance each other (\Cref{thm:gerver}). As the original proof by Gerver has a logical gap that does not address the connectedness of a moving sofa, we introduce new ideas and rework the proof carefully. Step 1-(c) uses previous steps and elementary geometry to show that \(S_{\max}\) rotates the full right angle in its movement.

Step 2 proves the \emph{injectivity condition} on \(S_{\max}\) which is the key for establishing the upper bound \(\mathcal{Q}\) later. It states that the trajectory of the inner corner \((0, 0)\) of \(L\) does not make self-loops in the perspective (frame of reference) of the moving sofa (\Cref{fig:injectivity-figure}). To prove this condition for \(S_{\max}\), we establish a new differential inequality on \(S_{\max}\) (\Cref{eqn:ineq-example}) heavily inspired by an ODE of Romik that balance the differential sides of Gerver’s sofa (\Cref{eqn:ode-example}).

Step 3-(a) extends the space of all moving sofas \(S\) with injectivity condition to a collection \(\mathcal{L}\) of tuples \((K, B, D)\) of convex bodies, so that each \(S\) maps to \((K, B, D) \in \mathcal{L}\) one-to-one (but not necessarily onto). The convex bodies describe different parts of the region \(R\) enclosing \(S\) (\Cref{fig:upper-bound}). Step 3-(b) defines the upper bound \(\mathcal{Q}\) on the extended domain \(\mathcal{L}\). We follow the boundary of \(R\) and express its area \(\mathcal{Q}\) using Green’s theorem and the quadratic area expressions on \(K, B\), and \(D\) from Brunn-Minkowski theory. We use injectivity condition and Jordan curve theorem to rigorously show that \(\mathcal{Q}(K, B, D)\) is an upper bound of the area of \(S\).

Step 3-(c) uses Mamikon’s theorem to establish the concavity of \(\mathcal{Q}\) on \(\mathcal{L}\) (\Cref{fig:mamikon-thm-on-sofa}). Step 3-(d) calculates the directional derivative of \(\mathcal{Q}\) at the convex bodies \((K, B, D) \in \mathcal{L}\) arising from Gerver’s sofa \(G\). The local optimality ODEs on \(G\) by Romik \autocite{romikDifferentialEquationsExact2018} are used to show that the directional derivative is always non-positive. This implies that \(G\) is a local optimum of \(\mathcal{Q}\) in \(\mathcal{L}\). The concavity of \(\mathcal{Q}\) on \(\mathcal{L}\) implies that \(G\) is also a global optimum of \(\mathcal{Q}\) in \(\mathcal{L}\). As the value of \(\mathcal{Q}\) at \(G\) matches the area, the sofa \(G\) also globally maximizes the area, completing the proof of \Cref{thm:main}.

\Cref{sec:monotone-sofas-and-caps} to \Cref{sec:optimality-of-gerver's-sofa} provide the full details of the proof of \Cref{thm:main}. Given the large volume, \Cref{sec:_monotone-sofas-and-caps} to \Cref{sec:_optimality-of-gerver's-sofa} overviews each chapter and explains its motivation. Readers are strongly encouraged to start with the overview sections to understand the core idea hidden in the details.

The notations and definitions used in the overviews will be often simpler that the ones used in the full proof. That is, the definitions made in this \Cref{sec:moving-sofa-problem} starting \Cref{sec:_monotone-sofas-and-caps} are specific to this chapter alone. Starting from \Cref{sec:monotone-sofas-and-caps}, all notations and definitions will be redefined for consistency in the detailed proofs. We always assume the plane with~\(x\)- and \(y\)-coordinates, and the variables~\(x\)~and~\(y\) are always associated with these coordinates.

\subsubsection{Acknowledgements}

The author thanks Dan Romik for his thorough support and encouragement that greatly helped the research process.~His feedback on the presentation significantly improved the clarity of this work. His package \texttt{MovingSofas.nb}\footnote{\url{https://www.math.ucdavis.edu/~romik/data/uploads/software/movingsofas-v1.3.nb}} helped making the intricate details of the problem much more accessible to the author. The package was also used to generate figures of Gerver’s sofa in this work.

Acknowledgment is extended to Joseph Gerver and Thomas Hales for their interest in this work and their in-depth discussions. The author also appreciates David Speyer’s efforts in understanding the details and help in refining the presentation.~The author thanks Michael Zieve and Joonkyung Lee for their mentorship and valuable advice.

Thanks are also due to Martin Strauss, Jeffrey Lagarias and Alexander Barvinok for their interest, help, and advice during the early stages of the research, as well as to Rolf Schneider for his suggestions on the proof of \Cref{thm:boundary-measure}. The author acknowledges Hyunuk Nam, Seewoo Lee, Changki Yun, Jaemin Choi, Yeonghyeon Kim, Joonhyung Shin, Yugeun Shim, and Seungwon Park for their interest, discussions, and encouragement.

A prior version of the proof of \Cref{thm:main} was computer-assisted. Although the software developed for this purpose\footnote{\url{https://github.com/jcpaik/sofa-designer}} does not appear in the final proof, it played an important role in shaping the intuition and strategy behind the full proof. The author thanks an anonymous mentor and Hyunuk Nam for their discussions that helped the development of the software.

This research was supported by the National Research Foundation of Korea (NRF) under grant MSIT NRF-2022R1C1C1010300.~The author also acknowledges support from the Korea Foundation for Advanced Studies during the completion of this research.

\section{Monotone Sofas and Caps}
\label{sec:_monotone-sofas-and-caps}
\begin{quote}
\textbf{Summary:} This section is an overview of \Cref{sec:monotone-sofas-and-caps}. We show that a moving sofa \(S\) of maximum area can be assumed to be a \emph{monotone sofa}, which is an intersection of the \emph{supporting hallways} \(L_t\) of \(S\) (\Cref{sec:_monotone-sofa}). A monotone sofa \(S\) is equal to its \emph{cap} \(K := \mathcal{C}(S)\), a convex body, subtracted by the \emph{niche} \(\mathcal{N}(K)\) determined by cap \(K\). Thus, the monotone sofa \(S\) can be identified with its cap \(K\), and the moving sofa problem becomes the maximization of the \emph{sofa area functional} \(\mathcal{A}(K) = |K| - |\mathcal{N}(K)|\) with respect to the cap \(K\) (\Cref{sec:_cap-and-niche}).
\end{quote}
\subsection{Monotone Sofa}
\label{sec:_monotone-sofa}
A fundamental idea of Gerver \autocite{gerverMovingSofaCorner1992} is to see a moving sofa \(S\) as the intersection of rotating hallways. Look at the movement of \(S\) inside the hallway \(L\) in perspective of \(S\). Then \(S\) is fixed in our frame of reference and \(L\) rotates and translates around \(S\) while containing \(S\) inside (bottom of \Cref{fig:movement-pov}). So \(S\) is a common subset of the rotating hallways (right side of \Cref{fig:gerver}).

\begin{figure}
\centering
\includegraphics{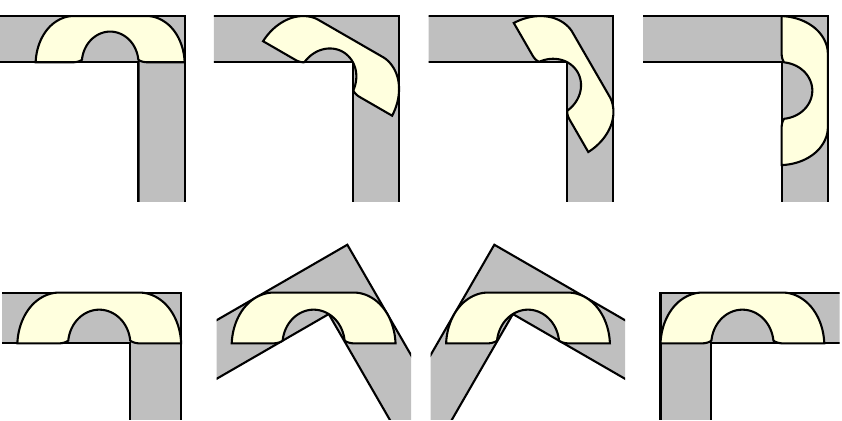}
\caption{The movement of a moving sofa in the perspective of hallway (top) and sofa (bottom).}
\label{fig:movement-pov}
\end{figure}

We will make the details of this idea precise. First, define the angle \(\omega\) that \(S\) rotates inside \(L\).

\begin{definition}

The \emph{rotation angle} \(\omega\) of a moving sofa \(S\) is the \emph{clockwise} angle that it rotates as it moves from \(H_L\) to \(V_L\) inside \(L\).\footnote{This is the angular difference between the two rigid motions \(\Phi_0\) and \(\Phi_1\) sending \(S\) to \(H_L\) and \(V_L\) respectively.}

\label{def:rotation-angle-intro}
\end{definition}

Define the unit-width strips \(H\) and \(V_\omega\).

\begin{definition}

Let \(R_t : \mathbb{R}^2 \to \mathbb{R}^2\) denote the rotation of \(\mathbb{R}^2\) around the origin by the counterclockwise angle of \(t \in \mathbb{R}\).

\label{def:rotation-map-intro}
\end{definition}

\begin{definition}

Define the horizontal strip \(H := \mathbb{R} \times [0, 1]\), vertical strip \(V := [0, 1] \times \mathbb{R}\), and its rotation \(V_\omega\) around the origin by a counterclockwise angle \(\omega \in \mathbb{R}\).

\label{def:strips-intro}
\end{definition}

Gerver showed that we can assume \(\omega \in (0, \pi/2]\) for the moving sofa problem (see \Cref{thm:rotation-angle-simple-bound} for details). Let \(S\) be any moving sofa with rotation angle \(\omega \in (0, \pi/2]\). Without loss of generality, we will always translate \(S\) and put it in the \emph{standard position} defined as below. Recall that a line \emph{supports} \(S\) if it contains a point of \(S\) but does not separate any two points of \(S\).

\begin{definition}

A moving sofa \(S\) with rotation angle \(\omega \in (0, \pi/2]\) is in \emph{standard position} if the upper sides \(y=1\) of \(H\) and \(x \cos \omega + y \sin \omega = 1\) of \(V_\omega\) support \(S\) from above.

\label{def:standard-position-intro}
\end{definition}

\begin{proposition}

For any moving sofa \(S\) with rotation angle \(\omega \in (0, \pi/2]\), there is a translation of \(S\) in standard position which is (i) unique if \(\omega < \pi/2\), or (ii) unique up to horizontal translations if \(\omega = \pi/2\).

\label{pro:standard-position-shape-intro}
\end{proposition}

\begin{proof}
Observe that the lines \(y=1\) and \(x \cos \omega + y \sin \omega = 1\) intersect properly if \(\omega < \pi/2\), and overlaps if \(\omega = \pi/2\).
\end{proof}

A moving sofa \(S\) put in standard position is a common subset of \(H\), \(V_\omega\) and rotating hallways \(L_t\) parametrized by its counterclockwise angle \(t \in [0, \omega]\).

\begin{proposition}

Fix an arbitrary moving sofa \(S\) with rotation angle \(\omega \in (0, \pi/2]\) in standard position. Then \(S\) is contained in each of the following sets.

\begin{enumerate}
\def\labelenumi{\arabic{enumi}.}
\tightlist
\item
  The horizontal strip \(H\).
\item
  For every angle \(t \in [0, \omega]\), the rotating hallway \(L_t\) which is a translation of \(R_t(L)\).
\item
  The rotated vertical strip \(V_\omega = R_\omega(V)\).
\end{enumerate}

\label{pro:moving-sofa-common-subset}
\end{proposition}

\begin{proof}
The initial position of \(S\) at \(L\) is contained in \(H_L \subset H\). So the width of \(S\) measured along the \(y\)-axis is at most one. Because \(S\) is in standard position, the line \(y=1\) supports \(S\) from above and we have \(S \subseteq H\).

The sofa \(S\) is rotated clockwise by \(\omega\) after its movement in \(L\). By the intermediate value theorem, for every \(t \in [0, \omega]\) there is a moment in the movement where a copy of \(S\) is rotated clockwise by \(t\) inside \(L\). See this in the frame of reference of \(S\) to conclude that \(S \subset L_t\) for some translation \(L_t\) of \(R_t(L)\).

The final position of \(S\) at \(L\) is contained in \(V_L \subset V\). Look at this in the frame of reference of \(S\). Then \(S\) is in a translation of \(V_\omega\), so the width of \(S\) measured along the direction \((\cos \omega, \sin \omega)\) is at most one. Because \(S\) is in standard position, the line \(x \cos \omega + y \sin \omega = 1\) is a supporting line above \(S\), and we have \(S \subseteq V_\omega\).
\end{proof}

By \Cref{pro:moving-sofa-common-subset}, any moving sofa \(S\) with rotation angle \(\omega \in (0, \pi/2]\) in standard position is contained in the intersection
\begin{equation}
\label{eqn:sofa-intersection}
\mathcal{I} := H \cap V_\omega \cap \bigcap_{t \in [0, \omega]} L_t.
\end{equation}
of two strips \(H\), \(V_\omega\) and the hallways \(L_t\) each rotated counterclockwise by \(t \in [0, \omega]\) and translated. So we have \(S \subseteq \mathcal{I}\), and it is natural to identify a maximum-area moving sofa \(S\) with the intersection \(\mathcal{I}\) and maximize \(\mathcal{I}\) by fixing \(H\), \(V_\omega\) and translating the hallways \(L_t\) for each \(t \in [0, \omega]\). All known derivations of Gerver’s sofa \(G\) \autocite{gerverMovingSofaCorner1992,romikDifferentialEquationsExact2018,dengSolvingMovingSofa2024} follow this approach.

However, recall that a moving sofa \(S\) is defined as a \emph{connected} set (e.g.~page 267 of \autocite{gerverMovingSofaCorner1992}). So the connectedness of \(\mathcal{I}\) in \Cref{eqn:sofa-intersection} is necessary to identify a maximum-area \(S\) with the intersection \(\mathcal{I}\). But it has not been rigorously established in the existing works that uses the idea \(S = \mathcal{I}\) \autocite{gerverMovingSofaCorner1992,romikDifferentialEquationsExact2018,kallusImprovedUpperBounds2018}.\footnote{Gerver requires a moving sofa \(S\) to be connected (Page 267 of \autocite{gerverMovingSofaCorner1992}). The proof of Theorem 1 in \autocite{gerverMovingSofaCorner1992} then defines a subcollection \(\mathcal{T}\) of intersections \(\mathcal{I}\) in \Cref{eqn:sofa-intersection} and uses compactness to find a set \(T \in \mathcal{T}\) of maximum area. However, Gerver does not show in his proof that the set \(T\) should be connected, which is a logical gap not trivial to fix. In \autocite{romikDifferentialEquationsExact2018}, Romik assumes the equality \(S = \mathcal{I}\) (Equation 8, p319) to give a streamlined derivation of Gerver’s sofa, but does not rigorously prove \(S = \mathcal{I}\) for a maximum-area \(S\). In \autocite{kallusImprovedUpperBounds2018}, Kallus and Romik require \(S\) to be connected and choose the largest-area connected component \(S\) of \(\mathcal{I}\), allowing the possibility of \(S \neq \mathcal{I}\).} Also, \Cref{pro:moving-sofa-common-subset} does not yet imply that the hallways \(L_t\) should move continuously with respect to \(t\).

(See the right side of \Cref{fig:monotone-sofa}) To resolve these issues, we let each rotated hallway \(L_t\) in the \Cref{eqn:sofa-intersection} be the \emph{supporting hallway} of angle \(t\) making contact with \(S\). We first give names to the different parts of \(L_t\) for further discussions.

\begin{definition}

(See \Cref{fig:hallway-detailed}) Let \(L_t\) be the hallway rotated counterclockwise by \(t \in [0, \omega]\) in \Cref{eqn:sofa-intersection}. Let \(\mathbf{x}(t)\) be the \emph{inner corner} of \(L_t\) corresponding to the point \((0, 0)\) of \(L\). Let \(\mathbf{y}(t)\) be the \emph{outer corner} of \(L_t\) corresponding to the point \((1, 1)\) of \(L\). Let \(a(t)\) and \(c(t)\) be the right and left \emph{outer walls} of \(L_t\) respectively, corresponding to the walls \(x=1\) and \(y=1\) of \(L\). Let \(b(t)\) and \(d(t)\) be the right and left \emph{inner walls} of \(L_t\) respectively, corresponding to the walls \(x=0\) and \(y=0\) of \(L\).

\label{def:rotating-hallway-parts-intro}
\end{definition}

Starting from any hallway \(L_t\) of counterclockwise angle \(t\) containing \(S\), the \emph{supporting hallway} is obtained by pushing \(L_t\) in the directions of \(-(\cos t, \sin t)\) and \(-(-\sin t, \cos t)\) continuously, until the two outer walls \(a(t)\) and \(c(t)\) of \(L_t\) makes contact with \(S\). As this move only pulls the inner walls \(b(t)\) and \(d(t)\) of \(L_t\) away from \(S\), the new supporting hallway \(L_t\) still contains \(S\) and now moves continuously with respect to \(t\).

After letting each \(L_t\) be the supporting hallways of \(S\), the intersection \(\mathcal{I}\) in \Cref{eqn:sofa-intersection} is now completely determined by \(S\), so that we will denote it as \(\mathcal{I}(S)\). We show that this \(\mathcal{I}(S)\) is always connected for any moving sofa \(S\) of rotation angle \(\omega \in (0, \pi/2]\) (\Cref{thm:monotonization-is-connected}). By looking at \(\mathcal{I}(S) \subseteq L_t\) in the frame of reference of \(L_t\), the intersection \(\mathcal{I}(S)\) also admits a continuous movement inside \(L\). So \(\mathcal{I}(S)\) is a moving sofa containing \(S\) (\Cref{thm:monotonization}).

Define a \emph{monotone sofa} as the intersection \(\mathcal{I}(S)\) of supporting hallways arising from some moving sofa \(S\). Then we can always assume that a maximum-area sofa \(S\) is monotone by taking the intersection \(\mathcal{I}(S)\) and making it larger. In particular, Gerver’s sofa \(G\) is a monotone sofa because \(G\) is the intersection of supporting hallways (\Cref{fig:gerver}). We also show that for any monotone sofa \(S\), taking the intersection again does not enlarge the set and \(S = \mathcal{I}(S)\) itself is the intersection of supporting hallways \(L_t\) of \(S\) (\Cref{thm:monotonization-idempotent}),

\begin{figure}
\centering
\includegraphics{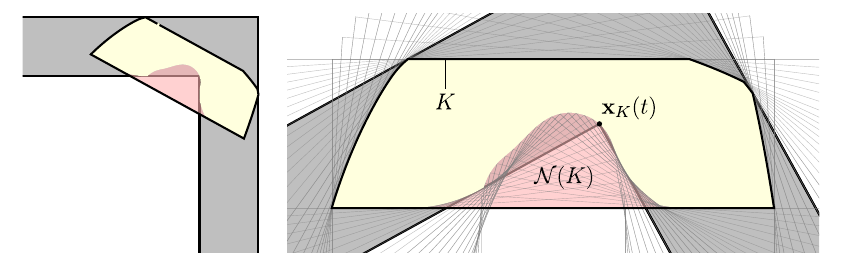}
\caption{The movement of a monotone sofa \(S\) with rotation angle \(\omega = \pi/2\) in perspective of the hallway (left) and the sofa (right).}
\label{fig:monotone-sofa}
\end{figure}

\subsection{Cap and Niche}
\label{sec:_cap-and-niche}
Let \(S\) be a monotone sofa with rotation angle \(\omega \in (0, \pi/2]\). The outer walls \(a(t)\) and \(c(t)\) of the supporting hallways \(L_t\) of \(S\) form the supporting lines of a convex body \(K := \mathcal{C}(S)\) that we call the \emph{cap} of \(S\). Define the \emph{parallelogram} \(P_\omega := H \cap V_\omega\). Then the cap \(K = \mathcal{C}(S)\) is
\begin{equation}
\label{eqn:sketch-cap}
\mathcal{C}(S) := P_\omega \cap \bigcap_{t \in [0, \omega]} Q_t^+
\end{equation}
where \(Q_t^+\) is the closed convex cone with vertex \(\mathbf{y}(t)\) bounded from above by the outer walls \(a(t)\), \(c(t)\) of \(L_t\). Because \(S\) was in standard position (\Cref{def:standard-position-intro}), the cap \(K\) is inscribed in the parallelogram \(P_\omega\) and makes contact with all four sides of \(P_\omega\) ((1) of \Cref{def:cap}).

The monotone sofa \(S\) is obtained from the cap \(K\) by subtracting the \emph{niche} \(\mathcal{N}(K)\) of cap \(K\), the union of all the triangular regions carved out by the inner walls \(b(t)\), \(d(t)\) of \(L_t\). Explicitly, define the \emph{fan}
\[
F_\omega := \left\{ (x, y) : y \geq 0, x \cos \omega + y \sin \omega \geq 0 \right\}
\]
bounded from below by the bottom sides of the parallelogram \(P_\omega\). Then the niche \(\mathcal{N}(K)\) is
\begin{equation}
\label{eqn:sketch-niche}
\mathcal{N}(K) = F_\omega \cap \bigcup_{t \in [0, \omega]} Q_t^-
\end{equation}
where \(Q_t^-\) is the open convex cone with vertex \(\mathbf{x}(t)\) bounded from above by the inner walls \(b(t)\) and \(d(t)\) of \(L_t\). We can derive \(S = K \setminus \mathcal{N}(K)\) from the equality \(L_t = Q_t^+ \setminus Q_t^-\) (\Cref{thm:monotonization-structure}). Note that \(L_t\) and \(Q_t^-\) can be recovered from the supporting lines of cap \(K\) (\Cref{lem:cap-same-support-function}), so the niche \(\mathcal{N}(K)\) is indeed determined by \(K\).

Because a monotone sofa \(S = K \setminus \mathcal{N}(K)\) is completely determined by its cap \(K := \mathcal{C}(S)\), we will identify \(S\) with its cap \(K\). We will prove \(\mathcal{N}(K) \subset K\) using elementary geometry (\Cref{thm:niche-in-cap}). Then the area \(|K| - |\mathcal{N}(K)|\) of \(S\) can be understood in terms of the cap and niche separately. We will define \(\mathcal{K}_\omega^\mathrm{c}\) as the space of all caps with rotation angle \(\omega \in (0, \pi/2]\). Now the moving sofa problem becomes the maximization of the \emph{sofa area functional} \(\mathcal{A}_\omega(K) := |K| - |\mathcal{N}(K)|\) on \(K \in \mathcal{K}_\omega^\mathrm{c}\).

\section{Balancing Argument of Gerver}
\label{sec:_balancing-argument-of-gerver}
\begin{quote}
\textbf{Summary:} This section reviews an important theorem of Gerver, stating that there is a maximum-area moving sofa which is a limit of polygons with opposite sides of the same length (\Cref{sec:balancing-argument}). We argue that the balancing argument of Gerver, while holds the essence of the proof, has a subtle logical gap that does not take account of the connectedness of a moving sofa (\Cref{sec:logical-gap}).
\end{quote}
\subsection{Balancing Argument}
\label{sec:balancing-argument}
Call a polygon \(P\) \emph{balanced} if, for any two parallel lines \(l^+\) and \(l^-\) of distance one on the plane, the total length of all edges of \(P\) in one line \(l^+\) is equal to that of the other line \(l^-\). Theorem 1 in \autocite{gerverMovingSofaCorner1992} by Gerver states that there exist a maximum-area moving sofa \(S_{\omega}\) that can be approximated sufficiently close by balanced polygons \(S_\Theta\).

We copy the full statement of the theorem as appears exactly in Gerver’s paper \autocite{gerverMovingSofaCorner1992} (footnote ours). We will rephrase the theorem in our words, so the reader may skim it for first read.

\begin{theorem}

(Theorem 1 in \autocite{gerverMovingSofaCorner1992}) There exists a real number \(\gamma\), \(\pi/3 \leq \gamma \leq \pi/2\), and a region \(S\), such that \(S\) can move around the corner of \(H\),\footnote{This \(H\) in \autocite{gerverMovingSofaCorner1992} is the hallway \(L\) in our paper.} rotating through an angle of \(-\gamma\) in the process,\footnote{This \(\gamma\) in \autocite{gerverMovingSofaCorner1992} is the rotation angle \(\omega\) in our paper. His proof of the bound \(\pi/3 \leq \gamma \leq \pi/2\) is factored out separately as \Cref{thm:rotation-angle-simple-bound}.} such that no region of greater area can move around the corner, and such that for arbitrarily large \(n\), \(S\) can be approximated arbitrarily closely by a polygonal region \(P_n\) with the following properties:\footnote{This \(P_n\) in \autocite{gerverMovingSofaCorner1992} is the polygon \(S_{\Theta_n}\) in our description (\Cref{eqn:polygon}).} The boundary of \(P_n\) is a balanced polygon. \(P_n\) is the intersection of \(n+1\) sets \(H_\alpha\) (where \(\alpha = k \gamma/n\) and \(0 \leq k \leq n\)). \(H_0\) is the half-strip\footnote{This \(H_0\) in \autocite{gerverMovingSofaCorner1992} is the horizontal side \(H_L\) of \(L\) in our paper.} \(x \leq 1\), \(0 \leq y \leq 1\). \(H_\gamma\) is a translation of the half strip\footnote{This \(H_{\gamma}\) in \autocite{gerverMovingSofaCorner1992} is the vertical side \(V_L\) of \(L\) in our paper rotated counterclockwise by \(\gamma\).} \(y \leq 1\), \(0 \leq x \leq 1\) rotated by angle \(\gamma\). For \(0 < \alpha < \gamma\), \(H_\alpha\) is a translation of \(H\) rotated by angle\footnote{This \(H_\alpha\) in \autocite{gerverMovingSofaCorner1992} is the rotating hallway \(L_\alpha\) containing \(S\) in our paper. The proof of Theorem 1 in \autocite{gerverMovingSofaCorner1992} actually takes each \(L_\alpha\) as the supporting hallway of angle \(\alpha\), using the support functions \(p(\alpha)\) and \(q(\alpha)\) of \(S\).} \(\gamma\).

\label{thm:gerver}
\end{theorem}

We now explain the \Cref{thm:gerver} and its proof by Gerver in our words. Fix the rotation angle \(\omega \in (0, \pi/2]\). As described in \Cref{sec:_monotone-sofas-and-caps}, a maximum-area moving sofa \(S\) is the connected intersection
\[
\mathcal{I} := H \cap V_\omega \cap \bigcap_{t \in [0, \omega]} L_t
\]
of two unit-width strips \(H, V_\omega\) and hallways \(L_t\) of counterclockwise angle \(t\). Discretize the problem by taking a finite nonempty subset \(\Theta\) of \((0, \omega)\) and the polygon intersection
\begin{equation}
\label{eqn:polygon}
S_\Theta := H \cap V_\omega \cap \bigcap_{t \in \Theta} L_t
\end{equation}
instead. The approximated problem now is to maximize the area of \(S_\Theta\) by translating the hallways \(L_t\) each rotated counterclockwise by \(t \in \Theta\).

(See \Cref{fig:balanced-sofa}) Gerver’s main idea in \autocite{gerverMovingSofaCorner1992} is that each maximum-area polygon \(S_\Theta\) in \Cref{eqn:polygon} should be balanced. \Cref{thm:gerver} states that, as \(n \to \infty\) and the angle set \(\Theta = \Theta_n := \left\{ i \omega / n : 1 \leq i < n \right\}\) gets denser in \([0, \omega]\), the balanced polygons \(S_\Theta\) should converge to some maximum-area moving sofa \(S_\omega\). For the proof of \Cref{thm:gerver}, Gerver uses the following \emph{balancing argument} to show that each \(S_\Theta\) is indeed balanced,\footnote{This balancing argument on \(S_{\Theta}\) (or \(P_n\) in \autocite{gerverMovingSofaCorner1992}) is done in the second paragraph of page 273 in \autocite{gerverMovingSofaCorner1992}.} and use compactness to show that such \(S_\Theta\)’s converge to some maximum-area sofa \(S_\omega\).

\begin{quote}
\textbf{Balancing Argument:} Assume for the sake of contradiction that a maximum-area polygon \(S_\Theta\) in \Cref{eqn:polygon} is not balanced. Take any pair of two parallel lines \(l^+\) and \(l^-\) of distance one, so that the total side lengths \(s^+\) and \(s^-\) of \(S_\Theta\) respectively on the lines \(l^+\) and \(l^-\) are not equal. Then all sides of \(S_\Theta\) on \(l^{\pm}\) are contributed by exactly one of \(X = H\), \(V_\omega\) or \(L_t\). Let \(\pm v\) be the normal unit vectors of parallel lines \(l^{\pm}\) respectively, directing outwards from each other. If \(s^+ > s^-\) (resp. \(s^- > s^+\)), translate \(X\) slightly by \(\epsilon v\) (resp. \(- \epsilon v\)) for sufficiently small \(\epsilon > 0\). If we pushed either \(X = H\) or \(V_\omega\), translate the whole \(S_\Theta\) with \(H, V_\omega, L_t\) together, to put \(H\) and \(V_\omega\) back to their initial positions. We just increased the area of \(S_\Theta\) by \(\epsilon |s^+ - s^-| + o(\epsilon) > 0\) by translating the hallways \(L_t\), contradicting the maximality of \(S_\Theta\).
\end{quote}

\begin{figure}
\centering
\includegraphics{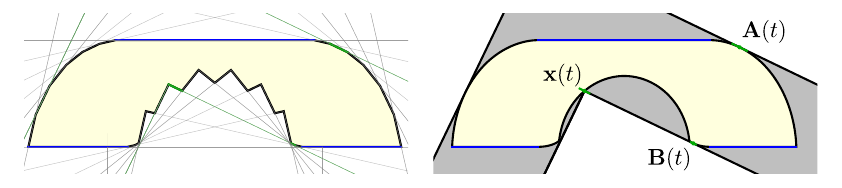}
\caption{A maximum-area polygon intersection \(S_\Theta\) should have balanced side lengths (left). By taking the angle set \(\Theta\) denser in \([0, \omega]\), the polygon \(S_\Theta\) converges to a maximum-area monotone sofa with balanced side lengths (right).}
\label{fig:balanced-sofa}
\end{figure}

\subsection{Logical Gap}
\label{sec:logical-gap}
The balancing argument of Gerver, while holds great importance and contains the gist of the proof of \Cref{thm:gerver}, has a subtle logical gap that does not address the connectedness of moving sofas.

In the first paragraph of \autocite{gerverMovingSofaCorner1992}, he defines a moving sofa as a \emph{connected} planar region. However, neither the connectedness of the polygons \(S_\Theta\), nor the limiting shape \(S_\omega\) of \(S_\Theta\) are established in the proof of Theorem 1 in \autocite{gerverMovingSofaCorner1992}.\footnote{In comparison, a lot of work in \Cref{sec:monotone-sofas-and-caps} and \Cref{sec:balanced-maximum-sofas-and-caps} are done to ensure the connectedness of the intersection \(\mathcal{I}\) or \(S_\Theta\) that we find.} To fill this gap in Gerver’s proof, it is natural to simply assume that each maximum-area polygon \(S_\Theta\) is taken \emph{among} connected intersections.\footnote{The other option is to allow each maximum-area polygon \(S_\Theta\) to be disconnected, but then proving that its limit \(S_\omega\) is connected would require completely new ideas.} However, this will not work because the balancing argument on \(S_\Theta\) may break the connectedness of \(S_\Theta\). See the following example.

(See \Cref{fig:balancing-progress}) Take the angle set \(\Theta = \left\{ \pi/6, \pi/3 \right\}\) and rotation angle \(\omega = \pi/2\). Define the unit vector \(\mathbf{u} := (\cos \pi/6, \sin \pi/6)\). Take a sufficiently small positive real number \(c > 0\). Take the hallways \(L_{\pi/6}\), \(L_{\pi/3}\) with angles in \(\Theta\) and inner corners \(\mathbf{x}(\pi/6) = (0, 1) - c \mathbf{u}\), \(\mathbf{x}(\pi/3) = (-0.9, 0.98)\) respectively. The intersection \(S_\Theta\) in \Cref{eqn:polygon} is not balanced, as the side of \(S_\Theta\) with normal angle \(\mathbf{u}\) is larger than the side with opposite normal angle \(-\mathbf{u}\) for all \(c \geq 0\) (depicted green). The balancing argument will now push \(L_{\pi/6}\) in the positive direction of \(\mathbf{u}\), decreasing \(c\) as long as \(c \geq 0\). But as \(c\) becomes negative, the intersection \(S_\Theta\) becomes disconnected.

Thus, while the balancing argument of Gerver can guarantee the balancedness of a maximum-area \(S_\Theta\), it cannot guarantee the connectedness of \(S_\Theta\). In the example above, it is actually possible to preserve the connectedness of \(S_\Theta\) by carefully choosing another pair of edges to balance. However, such an extra consideration is not also made in \autocite{gerverMovingSofaCorner1992}. The next \Cref{sec:_balanced-maximum-sofas-and-caps} provides a strategy that circumvents this issue.

\begin{figure}
\centering
\includegraphics{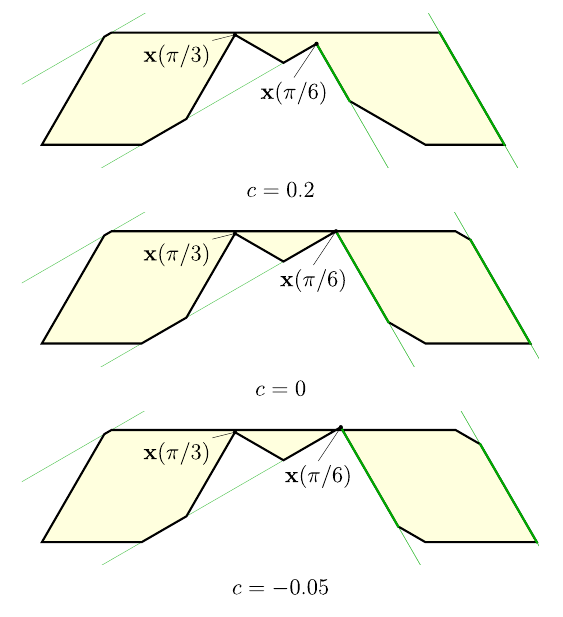}
\caption{Balancing argument breaks the connectivity of a polygon intersection \(S_\Theta\).}
\label{fig:balancing-progress}
\end{figure}

\section{Balanced Maximum Sofas and Caps}
\label{sec:_balanced-maximum-sofas-and-caps}
\begin{quote}
\textbf{Summary:} This section is an overview of \Cref{sec:balanced-maximum-sofas-and-caps}. We rework the proof of \Cref{thm:gerver} by Gerver, taking account of the connectedness of moving sofas. We show the existence of a \emph{balanced maximum sofa}, a monotone sofa of the maximum area that can be approximated sufficiently close by balanced polygons.
\end{quote}
\subsection{Limit of Maximum Polygon Caps}
\label{sec:limit-of-maximum-polygon-caps}
Our goal now is to bridge the gap discussed in \Cref{sec:logical-gap} and show that the \emph{connected} polygon intersection
\begin{equation}
\label{eqn:polygon-conn}
S_\Theta := H \cap V_\omega \cap \bigcap_{t \in \Theta} L_t
\end{equation}
of maximum area is balanced. Recall that the strips \(H\) and \(V_\omega\) are fixed, and each hallways \(L_t\) of counterclockwise angle \(t \in \Theta\) can translate freely.

We will first write the polygon \(S_\Theta = K \setminus \mathcal{N}_\Theta(K)\) as the difference of the \emph{polygon cap} \(K := \mathcal{C}_\Theta(K)\) and the \emph{polygon niche} \(\mathcal{N}_\Theta(K)\), analogous to the cap \(K\) and niche \(\mathcal{N}(K)\) of a monotone sofa \(S\) in \Cref{sec:_monotone-sofas-and-caps}. Explicitly, the polygon cap \(K\) is defined as
\[
\mathcal{C}_\Theta(K) := P_\omega \cap \bigcap_{t \in \Theta} Q_t^+
\]
following the \Cref{eqn:sketch-cap} of caps, and the polygon niche \(\mathcal{N}_\Theta(K)\) is defined as
\[
\mathcal{N}_\Theta(K) := F_\omega \cap \bigcup_{t \in \Theta} Q_t^-
\]
following the \Cref{eqn:sketch-niche} of niche. From \(L_t = Q_t^+ \setminus Q_t^-\), we can also obtain \(S_\Theta = K \setminus \mathcal{N}_\Theta(K)\) back.

Instead of maximizing \(S_\Theta\) directly, we will maximize the \emph{polygon area functional} \(\mathcal{A}_\Theta(K) := |K| - |\mathcal{N}_\Theta(K)|\) with respect to the polygon cap \(K\), where we allow the polygon sofa \(S_{\Theta} = K \setminus \mathcal{N}_\Theta(K)\) to be disconnected. For an example, we allow the case \(c = -0.05\) in \Cref{fig:balancing-progress} where \(\mathcal{N}_\Theta(K) \not\subset K\). Call such a maximizer \(K_\Theta\) of \(\mathcal{A}_\Theta(K)\) a \emph{maximum polygon cap}.

We will show in \Cref{sec:maximum-polygon-cap} that the side lengths of maximum polygon cap \(K_\Theta\) and niche \(\mathcal{N}_\Theta(K_\Theta)\) balance each other (\Cref{thm:balanced-polygon-sofa}) and that \(\mathcal{N}_\Theta(K_\Theta) \subset K_\Theta\) (\Cref{thm:balanced-polygon-sofa-connected}). This is the technical part of the proof that we outline in the next \Cref{sec:balancedness-of-maximum-polygon-cap}. By taking the angle set \(\Theta\) denser in \([0,\omega]\), the maximum polygon caps \(K_\Theta\) converge to some cap \(K_\omega\) with rotation angle \(\omega\) that we call as the \emph{balanced maximum cap} (\Cref{def:balanced-maximum-cap}).

As the maximum polygon caps \(K_\Theta\) converge to a balanced maximum cap \(K_\omega\), that \(\mathcal{N}_\Theta(K_\Theta) \subset K_\Theta\) implies \(\mathcal{N}(K_\omega) \subseteq K_\omega\) too, so that the set \(S_\omega := K_\omega \setminus \mathcal{N}(K_\omega)\) is connected and forms a monotone sofa.\footnote{\Cref{thm:niche-in-cap} shows that for any cap \(K\), we have \(\mathcal{N}(K) \subset K\) if and only if the set \(K \setminus \mathcal{N}(K)\) is connected.} We call such \(S_\omega\) a \emph{balanced maximum sofa} (\Cref{def:balanced-maximum-sofa}). As each \(K_\Theta\) is a maximizer of \(\mathcal{A}_\Theta\), the limit \(K_\omega\) is also a maximizer of \(\mathcal{A}_\omega\), and the area \(\mathcal{A}_\omega(K_\omega) = |K_\omega| - \left| \mathcal{N}(K_\omega) \right|\) of a balanced maximum sofa \(S_\omega\) achieves the maximum area among all monotone sofas of rotation angle \(\omega\).

\subsection{Balancedness of Maximum Polygon Cap}
\label{sec:balancedness-of-maximum-polygon-cap}
Now we overview the technical proof that the side lengths of a maximum polygon cap \(K_\Theta\) and its polygon niche \(\mathcal{N}_\Theta(K_\Theta)\) are balanced. That is, for any unit vector \(v\), the total length of all sides of \(K_\Theta\) and \(\mathcal{N}_\Theta(K_\Theta)\) with normal angle \(v\) is equal to that of normal angle \(-v\) (\Cref{def:polygon-cap-balanced}; see \Cref{fig:balanced-polygon-sofa-color}). We will also obtain \(\mathcal{N}_\Theta(K_\Theta) \subset K_\Theta\) as a consequence. We omit many details that can be found in the full \Cref{sec:balanced-maximum-sofas-and-caps}.

We extend the space \(\mathcal{K}_\Theta^\mathrm{c}\) of all polygon caps \(K\) with angle set \(\Theta\) using the \emph{support function} of \(K\).

\begin{definition}

Define \(u_t := (\cos t, \sin t)\) and \(v_t := (-\sin t, \cos t)\).

\label{def:unit-vectors-intro}
\end{definition}

\begin{definition}

For any planar convex body \(K\) (a compact, convex subset of \(\mathbb{R}^2\)), define the \emph{support function} \(h_K(t) := \sup \left\{ u_t \cdot p : p \in K \right\}\).

\label{def:support-function-intro}
\end{definition}

The support function \(h_K(t)\) of \(K\) is the signed distance from the origin \((0, 0)\) to the supporting line of \(K\) with normal vector \(u_t\) outwards from \(K\). Let \(\Theta^\diamond = \Theta \cup (\Theta + \pi/2) \cup \left\{ \omega, \pi/2 \right\}\). We embed the space \(\mathcal{K}_\Theta^\mathrm{c}\) of all polygon caps \(K\) with the angle set \(\Theta\) to the space \(\mathcal{H}_\Theta\) of all functions \(h : \Theta^\diamond \to \mathbb{R}\) by taking the support function \(h_K\) and restricting it to \(\Theta^\diamond\). This embedding allows to see \(\mathcal{H}_\Theta\) as an extension of \(\mathcal{K}_\Theta^\mathrm{c}\).

We will extend the polygon area functional \(\mathcal{A}_\Theta(K)\) on \(K \in \mathcal{K}_\Theta^\mathrm{c}\) to the larger space \(h \in \mathcal{H}_\Theta\). To do so, we write the cap \(K\) and niche \(\mathcal{N}_\Theta(K)\) as the \emph{Nef polygons} obtained from boolean set operations on half-planes. For any \(t \in S^1\) and \(h \in \mathbb{R}\), define the closed half-planes \(H_{\pm}(t, h)\) and the open half-planes \(H_{\pm}^\circ(t, h)\) with the boundary \(l(t, h)\) as the following.
\begin{gather*}
H_-(t, h) := \left\{ p \in \mathbb{R}^2 : p \cdot u_t \leq h \right\} \qquad H_-^{\circ}(t, h) := \left\{ p \in \mathbb{R}^2 : p \cdot u_t < h \right\} \\
H_+(t, h) := \left\{ p \in \mathbb{R}^2 : p \cdot u_t \geq h \right\} \qquad H_+^{\circ}(t, h) := \left\{ p \in \mathbb{R}^2 : p \cdot u_t > h \right\}
\end{gather*}

Let \(h := h_K\) be the support function of \(K\). Then we can then write the cap
\[
K = \bigcap_{t \in \left\{ \omega, \pi/2 \right\} } \big(H_-(t, h(t)) \cap H_+(t, h(t) - 1)\big) \cap \bigcap_{t \in \Theta \cup (\Theta + \pi/2)} H_-(t, h(t))
\]
and and the niche
\begin{align*}
\mathcal{N}_{\Theta}(K) & = \bigcap_{t \in \left\{ \omega, \pi/2 \right\} } H_+(t, h(t) - 1) \cap \phantom{{}X{}} \\
& \phantom{{}={}} \bigcup_{t \in \Theta} \left( H_-^\circ(t, h(t) - 1) \cap H_-^\circ(t + \pi/2, h(t + \pi/2) - 1) \right)
\end{align*}
purely as boolean operations on the half-planes \(H_{\pm}(t, h(t))\) and \(H_{\pm}(t, h(t) - 1)\) determined by \(h = h_K\) (\Cref{def:height-extensions}). This amounts to saying that \(K\) and \(\mathcal{N}_\Theta(K)\) are the \emph{Nef polygons} determined by such half-planes. The Nef polygon formulas \(\mathcal{C}_\Theta(h)\) and \(\mathcal{N}_\Theta(h)\) of \(K\) and \(\mathcal{N}_\Theta(K)\) respectively in \(h = h_K\) generalizes to all \(h \in \mathcal{H}_\Theta\). Now the polygon area \(\mathcal{A}_\Theta(K)\) extends to
\[
\mathcal{A}_\Theta(h) := |\mathcal{C}_\Theta(h)| - |\mathcal{N}_\Theta(h)|
\]
over all \(h \in \mathcal{H}_\Theta\).

To prove that a maximum polygon cap \(K_\Theta\) is balanced, we use the method of contradiction and assume that \(K_\Theta\) is not balanced. Let \(h := h_{K_\Theta}\) so that \(\mathcal{A}_\Theta(K_\Theta) = \mathcal{A}_\Theta(h)\). \Cref{lem:not-balanced-positive} carefully chooses the angle \(t \in \Theta^\diamond\) so that the balancing move on unbalanced sides of \(K_\Theta = \mathcal{C}_\Theta(h)\) and \(\mathcal{N}_\Theta(K_\Theta) = \mathcal{N}_\Theta(h)\) always moves a hallway in a \emph{positive} direction of \(u_t\). So the move increases the value of \(h(t)\) by a sufficiently small \(\epsilon > 0\) and makes \(\mathcal{A}_\Theta(h)\) slightly larger. Let \(h^+ \in \mathcal{H}_\Theta\) be the incremented function so that \(\mathcal{A}_\Theta(h^+) > \mathcal{A}_\Theta(h)\).

Our way of choosing the angle \(t\) guarantees that the convex polygon \(K^+ := \mathcal{C}_\Theta(h^+)\) obtained back from \(h^+\) is always a \emph{translation} of some polygon cap \(K^0 \in \mathcal{K}_\Theta^\mathrm{c}\) (\Cref{lem:height-positive-increment}). By translating \(K^+\) back to \(K^0\), we conclude \(\mathcal{A}_\Theta(h^+) = \mathcal{A}_\Theta(K^0)\) and thus
\[
\mathcal{A}_\Theta(K_\Theta) = \mathcal{A}_\Theta(h) < \mathcal{A}_\Theta(h^+) = \mathcal{A}_\Theta(K^0),
\]
reaching contradiction with the maximality of \(K_\Theta\). This is the sketch of the rigorous proof of \Cref{thm:balanced-polygon-sofa}.

We can then use the balancedness of \(K_\Theta\) and \(\mathcal{N}_\Theta(K_\Theta)\) to show that \(\mathcal{N}_\Theta(K_\Theta) \subset K_\Theta\). See \Cref{fig:balanced-polygon-sofa-color}. Balancedness essentially implies that the \emph{polyline} \(\mathbf{p}_{K_\Theta}\) obtained from the bottom sides of \(S_\Theta = K_\Theta \setminus \mathcal{N}_\Theta(K_\Theta)\) is a ‘permutation’ of the upper sides of polygon cap \(K_\Theta\). This implies that the polyline \(\mathbf{p}_{K_\Theta}\) should be contained inside \(K_\Theta\), so that \(\mathcal{N}_\Theta(K_\Theta) \subset K_\Theta\). This is the essential idea behind \Cref{thm:balanced-polygon-sofa-connected} that rigorously proves \(\mathcal{N}_\Theta(K_\Theta) \subset K_\Theta\).

\section{Rotation Angle of Balanced Maximum Sofas}
\label{sec:_rotation-angle-of-balanced-maximum-sofas}
\begin{quote}
\textbf{Summary:} This section is an overview of \Cref{sec:rotation-angle-of-balanced-maximum-sofas}. We show that the balanced maximum sofa found in the previous step admits a movement with rotation angle \(\pi/2\).
\end{quote}
\subsection{Statement}
\label{sec:_statement}
Recall that the \emph{rotation angle} \(\omega\) of a moving sofa \(S\) is the clockwise angle that \(S\) rotates during its movement inside \(L\) (\Cref{def:rotation-angle}). It can be strictly less than the angle \(\pi/2\) of the hallway \(L\). For example, the square \(S := [0, 1]^2\) have rotation angle \(\omega = 0\) as it can be moved inside \(L\) by only translation.

Gerver showed that there is a maximum-area moving sofa \(S_{\max}\) with rotation angle \(\pi/ 3 \leq \omega \leq \pi/2\) (\Cref{thm:gerver}). His argument, reproduced in the \Cref{thm:rotation-angle-simple-bound} below, actually proves a slightly improved lower bound \(\omega \geq \sec^{-1}(2.2) = 62.96\cdots\degree\).

\begin{theorem}

(Modification of page 271 of \autocite{gerverMovingSofaCorner1992}) Let \(S\) be any moving sofa of area \(\geq 2.2\). Then \(S\) admits a movement in \(L\) with rotation angle \(\omega \in [\sec^{-1}(2.2) , \pi/2]\).

\label{thm:rotation-angle-simple-bound}
\end{theorem}

\begin{proof}
Assume any movement of \(S\) inside \(L\) with rotation angle \(\omega \in \mathbb{R}\). Without loss of generality, we can assume that \(S\) is in its initial position at \(H_L \subseteq H\).

First assume \(\omega \leq -\pi/4\). By the intermediate value theorem, there is a moment where \(S\) is rotated clockwise by \(-\pi/4 \in [\omega, 0]\) (or, counterclockwise by \(\pi/4\)) inside \(L\) during its assumed movement. Looking at this in the perspective of \(S\), the sofa \(S\) is contained in a hallway \(L'\) rotated clockwise by \(\pi/4\) and translated. The intersection \(H \cap L'\) containing \(S\) have area \(\sqrt{2} = 1.4142\dots\) and we get contradiction as \(|S| \geq 2.2\).

Now assume \(|\omega| < \sec^{-1}(2.2)\). The sofa \(S\) is rotated clockwise by \(\omega\) in its final position at \(V_L \subseteq V\). Look at this in perspective of \(S\), then \(S\) is contained in a translation of \(V_\omega\). The intersection of \(H\) and a translation of \(V_\omega\) is a parallelogram of area \(\sec(\omega) < 2.2\). This contradicts \(|S| \geq 2.2\).

So we should have \(\sec^{-1}(2.2) \leq \omega\) because \(\sec^{-1}(2.2) = 62.96\cdots\degree > \pi/4\). We finish the proof by assuming \(\omega > \pi/2\) and finding another movement of \(S\) in \(L\) with rotation angle \(\pi/2\). By the intermediate value theorem, there is a moment in the movement of \(S\) with rotation angle \(\omega\), where \(S\) is rotated clockwise by \(\pi/2 \in [0, \omega]\) in \(L\). Call the position of \(S\) at this moment \(S_{\pi/2}\). Instead of following the rest of the movement of \(S\), translate \(S_{\pi/2}\) horizontally in the positive direction of the \(x\)-axis until it makes contact with the outer wall \(x=1\) of \(L\). Since \(S\) was initially in \(H_L\), the width of \(S_{\pi/2}\) measured along the \(x\)-axis is at most one. So after the horizontal translation, \(S_{\pi/2}\) will lie completely inside the destination \(V_L\), finishing a full moment of \(S\) with rotation angle \(\pi/2\).
\end{proof}

By \Cref{thm:rotation-angle-simple-bound} and Gerver’s sofa of area \(|G| = 2.2195\dots > 2.2\), we can assume the rotation angle \(\omega \in [\sec^{-1}(2.2), \pi/2]\) of a maximum-area moving sofa. \Cref{sec:rotation-angle-of-balanced-maximum-sofas} proves the equality \(\omega = \pi/2\) for balanced maximum sofas.

\begin{theorem}

Let \(S_\omega\) be an arbitrary balanced maximum sofa with area \(\geq 2.2\) and rotation angle \(\omega \in [\sec^{-1}(2.2) , \pi/2]\). Then a rotated copy of \(S_\omega\) admits a movement inside \(L\) with rotation angle \(\omega = \pi/2\).

\label{thm:angle}
\end{theorem}

To prove \Cref{thm:angle}, we will show that (a rotated copy of) \(S_\omega\) can rotate an extra angle of \(\pi/2 - \omega\) inside \(H_L\) before its movement with angle \(\omega\).

The main step is to show that a triangular region \(\Delta_\omega\) is disjoint from \(S_\omega\) (\Cref{thm:balanced-consumed}). Recall that the cap of the monotone sofa \(S_\omega\) is inscribed in the parallelogram \(P_\omega\) of width 1 with the lower-left corner \(O := (0, 0)\) and upper-right corner \(o_\omega := (\tan(\pi/4 - \omega/2), 1)\), making an angle of \(\omega + \pi/2\) at both corners (see \Cref{eqn:sketch-cap} and the left of \Cref{fig:triangle-full}). The region \(\Delta_\omega\) is then defined as the triangular region near \(O\) formed by three vertices \(O\), \(o_\omega - (0, 1)\), and \(o_\omega - (\cos \omega, \sin \omega)\).

Once we show the main step that \(S_\omega \subseteq P_\omega\) is disjoint from \(\Delta_\omega\), we obtain enough room to rotate \(S_\omega\) \emph{counterclockwise} by an angle of \(\pi/2 - \omega\) inside the horizontal side \(H_L\) (see the right of \Cref{fig:triangle-full}). Follow this rotation of \(S_\omega\) in reverse so that it rotates \emph{clockwise} by \(\pi/2 - \omega\). Then follow the original movement of \(S_\omega\) in \(L\) with rotation angle \(\omega\). We have just found the movement of a rotated copy of \(S_\omega\) with full rotation angle \(\pi/2\), proving \Cref{thm:angle}.

\begin{figure}
\centering
\includegraphics{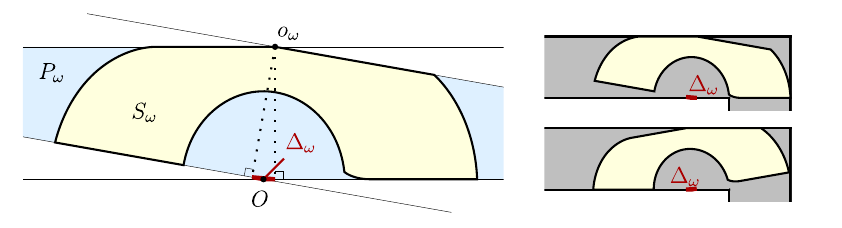}
\caption{The moving sofa \(S_\omega\) of maximum area with a fixed rotation angle \(\omega\) is inscribed in the parallelogram \(P_\omega\) and disjoint from the triangular region \(\Delta_\omega\) (left). So it can rotate counterclockwise by the angle of \(\pi/2-\omega\) inside the horizontal side \(H_L\) (right).}
\label{fig:triangle-full}
\end{figure}

\subsection{Proof Outline}
\label{sec:proof-outline}
We now outline the proof of the main step that \(S_\omega\) is disjoint from \(\Delta_\omega\) (\Cref{thm:balanced-consumed}). We use the balancedness of \(S_\omega\) established in \Cref{sec:_balanced-maximum-sofas-and-caps}. In particular, the horizontal sides of \(S_\omega\) should be equal in their length (the blue sides of \Cref{fig:balanced-sofa} and \Cref{fig:proof-outline}).

\begin{figure}
\centering
\includegraphics{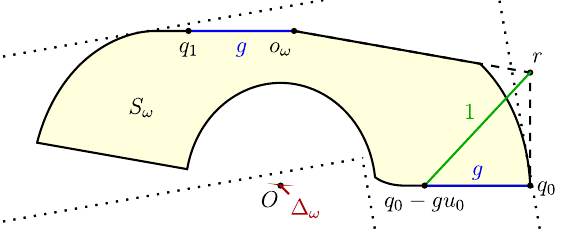}
\caption{Proof of \Cref{thm:angle}. We find two points \(q_0, q_1 \in S_\omega\) sufficiently away from the origin. Then we take a supporting hallway \(L_t\) containing \(S_\omega\) and thus the two points \(q_0, q_1 \in S_\omega\) (dashed), so that \(L_t\) is disjoint from \(\Delta_\omega\).}
\label{fig:proof-outline}
\end{figure}

(See \Cref{fig:proof-outline}) Let \(K\) be the cap of the monotone sofa \(S_\omega\). We will find two points \(q_0, q_1 \in S_\omega\) on the upper boundary of \(K\) sufficiently away from \(O\). Have the right endpoint \(q_0\) of \(K\) on the \(x\)-axis sufficiently away from the origin \(O\) by the distance \(d_{\omega, \min}\) (\Cref{def:d-min}), by using \(|K| > 2.2\) and reflecting \(K\) along the line passing through \(O\) and \(o_\omega\) if necessary. Take a right triangle with the right-angled vertex \(q_0\) and side 1 (green), to find a lower bound \(g\) (blue) of the horizontal side length of \(S_\omega\) on the line \(y = 0\). By the balancedness of \(S_\omega\), the side length of \(S_\omega\) on the line \(y=1\) is also bounded from below by \(g\). Define \(q_1 \in K\) as the point on the line \(y=1\) exactly \(g\) away from the endpoint \(o_\omega\) on this side.

Now that we found two points \(q_0, q_1 \in S_\omega\), we take the supporting hallway \(L_t\) of \(S_\omega\) angle \(t = \pi/2 - \omega \in (0, \omega)\). Using that \(L_t\) contains the two points \(q_0, q_1 \in S_\omega\) sufficiently away from \(O\), technical calculations show that the region \(\Delta_\omega\) must be enclosed by the inner walls of \(L_t\) as in \Cref{fig:proof-outline}. So \(\Delta_\omega\) must be disjoint with \(L_t\) and thus also with \(S_\omega\) as desired.

\section{Surface Area Measure}
\label{sec:_surface-area-measure}
\begin{quote}
\textbf{Summary:} This section is an overview of \Cref{sec:surface-area-measure}. A planar convex body \(K\) does not necessarily have a differentiable boundary. Using the Lebesgue–Stieltjes measure and Brunn-Minkowski theory, we prove an equality that allows us to use the \emph{surface area measure} \(\sigma_K\) of \(K\) as a weak derivative of the boundary of \(K\).
\end{quote}

(See \Cref{fig:convex-body}) A \emph{planar convex body} \(K\) is a nonempty, compact, and convex subset of \(\mathbb{R}^2\). For any planar convex body \(K\) and angle \(t\), define \(l_K(t)\) as the supporting line of \(K\) with normal vector \(u_t\) directing outwards from \(K\). Define the \emph{edge} \(e_K(t)\) of \(K\) as the intersection \(e_K(t) := K \cap l_K(t)\).

Let \(S^1\) be the circle taken as \(\mathbb{R}\) modulo \(2\pi\). The \emph{surface area measure} \(\sigma_K\) of \(K\) is a measure on \(S^1\) that describes the length of the edges \(e_K(t)\) of a planar convex body \(K\) in terms of the angle \(t\). For any Borel subset \(X\) of \(S^1\), the value \(\sigma_K(X)\) is equal to the one-dimensional length of the set \(\bigcup_{t \in X} e_K(t)\). We give two examples.

\begin{enumerate}
\def\labelenumi{\arabic{enumi}.}
\item
  If \(K\) is the rectangle \([-1, 1] \times [0, 1]\), then \(\sigma_K\) measures the side lengths of \(K\). That is, \(\sigma_K\left( \left\{ t \right\} \right)\) is equal to \(1\) if \(t = 0, \pi\), and equal to \(2\) if \(t = \pi/2, 3\pi/2\). The measure \(\sigma_K\) on \(S^1\) is zero outside the finite set \(\left\{ 0, \pi/2, \pi, 3\pi/2 \right\}\) of normal angles of \(K\).

  In general, if \(K\) is a polygon, then \(\sigma_K\) at the singleton \(\left\{ t \right\}\) of angle \(t\) will measure the side length of the edge of \(K\) with normal vector \(u_t\).
\item
  If \(K\) is the semicircle \(\left\{ (x, y) : x^2 + y^2 \leq 1, y \geq 0 \right\}\) of radius one above the \(x\)-axis, then \(\sigma_K\) measures the \emph{differential} side lengths of \(K\). That is, \(\sigma_K\) restricted to \([0, \pi]\) is the usual Borel measure on \([0, \pi]\). The value \(\sigma_K\left( \left\{ 3\pi/2 \right\} \right)\) is equal to 2. The measure \(\sigma_K\) is zero on \((\pi, 2\pi) \setminus \left\{ 3\pi/2 \right\}\).

  In general, if \(K\) has a smooth boundary and each edge \(e_K(t)\) is a single point on the boundary with curvature \(\kappa(t) > 0\), then the density of the measure \(\sigma_K\) at the point \(e_K(t)\) is the \emph{radius of curvature} \(R(t) := 1 / \kappa(t)\).
\end{enumerate}

The measure \(\sigma_K\) is very useful in our analysis of \(K\). Recall that the support function \(h_K(t) := \sup \left\{ p \cdot u_t : p \in K \right\}\) is the signed distance that the edge \(e_K(t)\) makes from the origin. The area \(|K|\) of \(K\) can be expressed using \(\sigma_K\) and \(h_K\) as
\[
|K| = \frac{1}{2} \int_{t \in S^1}h_K(t) \, \sigma_K(dt).
\]

The measure \(\sigma_K\) also acts as a ‘weak derivative’ of a possibly non-differentiable boundary of \(K\). Recall that \(v_t := (-\sin t, \cos t)\). For each angle \(t\), define the \emph{vertices} \(v_K^-(t)\) and \(v_K^+(t)\) of \(K\) as the endpoints of the edge \(e_K(t)\) that are furthermost in the direction of \(-v_t\) and \(v_t\) respectively. Note that \(v_K^+(t)\) depends on \(K\) but \(v_t\) does not. Then the equality
\begin{equation}
\label{eqn:to-reveal}
\mathrm{d} v_K^+(t) = v_t \, \sigma_K
\end{equation}
holds, whose meaning we elaborate as below.

Let \(I := [a, b]\) be a closed interval. For any right-continuous \(f : I \to \mathbb{R}\) of bounded variation, let \(\mathrm{d} f\) denote the \emph{Lebesgue–Stieltjes measure} of \(f\) which is the unique Borel measure on \(I\) such that \(\mathrm{d} f \left( \left\{ a \right\} \right) = 0\) and \(\mathrm{d} f \left( (c, d] \right) = f(d) - f(c)\) for any \((c, d] \subset I\).

The Lebesgue–Stieltjes measure \(\mathrm{d} f\) acts as a rigorous justification of the differential \(df\). We can state informal calculations of differentials like \(d(t^2) = 2t \, dt\) rigorously as the equality \(\mathrm{d}(t^2) = 2t \, \mathrm{d}t\) of measures, where the variable \(t\) parametrizes the interval \(I\). Note that the Lebesgue–Stieltjes measure \(\mathrm{d} t\) is from the function \(g(t) = t\) on \(t \in I\), so that \(\mathrm{d} t\) denotes the usual Borel measure of \(I\). Correspondingly, the measure \(2t \, \mathrm{d} t\) have density function \(2t\) on \(t \in I\). So the value of measure \(2t\, \mathrm{d} t\) on \((c, d]\) is \(\int_c^d 2t\,dt = d^2 - c^2\), which is equal to that of \(d\left( t^2 \right)\).

It turns out that the vertex \(v_K^+ : S^1 \to \mathbb{R}^2\) as a function of angle \(t \in S^1\) is right-continuous and of bounded variation. So using the notion above, the left-hand side \(\mathrm{d} v_K^+(t)\) of \Cref{eqn:to-reveal} makes sense as a pair of Lebesgue–Stieltjes measures of the \(x\) and \(y\)-coordinates of \(v_K^+(t)\). The right-hand side of \Cref{eqn:to-reveal} is the pair \((- \sin t \cdot \sigma_K, \cos t \cdot \sigma_K)\) of measures on \(S^1\), where \(-\sin t \cdot \sigma_K\) is the measure \(\sigma_K\) on \(S^1\) multiplied pointwise with the measurable function \(-\sin t\) on \(t \in S^1\).

Intuitively, \Cref{eqn:to-reveal} states that the differential of \(v_K^+\) is the vector with direction \(v_t\) and side length \(\sigma_K\) at \(t \in S^1\). The equality will be used frequently in later parts.

\section{Injectivity Condition}
\label{sec:_injectivity-condition}
\begin{quote}
\textbf{Summary:} This section is an overview of \Cref{sec:injectivity-condition}. We show the key property, called the \emph{injectivity condition}, on any balanced maximum sofa \(S\) of rotation angle \(\pi/2\). The main idea is to prove (\Cref{sec:a-differential-inequality}) and solve for (\Cref{sec:solving-the-differential-inequality}) a differential inequality on \(S\) that compares the differential side lengths of \(S\) (\Cref{eqn:ineq-example}). The inequality is inspired by an ODE of Romik \autocite{romikDifferentialEquationsExact2018} that balances the differential side lengths of moving sofas (\Cref{eqn:ode-example}).
\end{quote}
\subsection{Statement}
\label{sec:__statement}
Recall from \Cref{sec:_monotone-sofas-and-caps} and \Cref{eqn:sketch-cap} that a monotone sofa \(S\) with rotation angle \(\omega = \pi/2\) is the intersection
\[
S = H \cap \bigcap_{t \in [0, \pi/2]} L_t
\]
of the strip \(H\) and supporting hallways \(L_t\) of \(S\) (the vertical strip \(V_{\omega}\) overlaps with \(H\) as \(\omega = \pi/2\)). Recall that \(\mathbf{x}(t)\) is the inner corner of \(L_t\). The curve \(\mathbf{x} : [0, \pi/2] \to \mathbb{R}^2\), called the \emph{rotation path} of \(S\) by Romik \autocite{romikDifferentialEquationsExact2018}, determines \(L_t\), the monotone sofa \(S\), and its area \(\alpha(\mathbf{x})\) completely. Gerver’s sofa \(G\) is derived so that any local perturbation of the rotation path \(\mathbf{x} := \mathbf{x}_G\) of \(G\) does not increase the area \(\alpha(\mathbf{x})\) \autocite{gerverMovingSofaCorner1992,romikDifferentialEquationsExact2018,dengSolvingMovingSofa2024}.

A major obstacle in showing the global optimality of \(G\) is that there is no managable formula of the area \(\alpha(\mathbf{x})\) of the sofa in terms of the rotation path \(\mathbf{x} : [0, \pi/2] \to \mathbb{R}^2\). All known derivations of \(G\) assumes a specific shape of \(G\) to find a workable formula of \(\alpha(\mathbf{x})\) \autocite{gerverMovingSofaCorner1992,romikDifferentialEquationsExact2018,dengSolvingMovingSofa2024}. We prove the following condition to overcome this obstacle. Recall that \(u_t = (\cos t, \sin t)\) and \(v_t = (-\sin t, \cos t)\).

\begin{theorem}

(Injectivity condition; abridged) The rotation path \(\mathbf{x} : [0, \pi/2] \to \mathbb{R}^2\) of any balanced maximum sofa \(S\) is continuously differentiable, and \(\mathbf{x}'(t) \cdot u_t < 0\) and \(\mathbf{x}'(t) \cdot v_t > 0\) for all \(t \in (0, \pi/2)\).

\label{thm:injectivity-abridged}
\end{theorem}

\Cref{thm:injectivity-abridged} is an abridged version that captures the essense of the full statement (\Cref{thm:injectivity}). We call it the \emph{injectivity condition} as it implies that the rotation path \(\mathbf{x} : [0, \pi/2] \to \mathbb{R}^2\) does not self-intersect (\Cref{fig:injectivity-figure}). Assuming \Cref{thm:injectivity-abridged}, we have
\[
\mathbf{x}'(t) \cdot (1, 0) = \cos t \; (\mathbf{x}'(t) \cdot u_t) - \sin t \; (\mathbf{x}'(t) \cdot v_t) < 0
\]
for all \(t \in (0, \pi/2)\) so the \(x\)-coordinate of \(\mathbf{x}(t)\) strictly decreases as \(t\) increases. Thus the trajectory of \(\mathbf{x}(t)\) forms a Jordan arc, and the area enclosed by \(\mathbf{x}(t)\) can be expressed using Green’s theorem.

\begin{figure}
\centering
\includegraphics{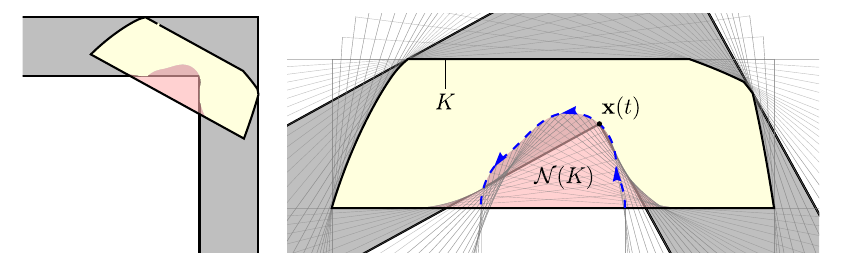}
\caption{Injectivity condition on a monotone sofa \(S = K \setminus \mathcal{N}(K)\) with cap \(K\) implies that the inner corner \(\mathbf{x}(t)\) of the supporting hallways \(L_t\), as a curve over \(t \in [0, \pi/2]\), does not self-intersect.}
\label{fig:injectivity-figure}
\end{figure}

\subsection{A Differential Inequality}
\label{sec:a-differential-inequality}
Assume an arbitrary balanced maximum sofa \(S\) with rotation angle \(\pi/2\) and supporting hallways \(L_t\) for \(t \in [0, \pi/2]\). In \autocite{romikDifferentialEquationsExact2018}, Romik introduced a set of ODEs that balance the \emph{differential} side lengths of \(S\).

\begin{definition}

(See the figures below \Cref{thm:gerver-odes}) For any angle \(t\) where the wall \(a(t)\) (resp. \(b(t)\), \(c(t)\), and \(d(t)\)) of the hallway \(L_t\) is tangent to \(S\), define \(\mathbf{A}(t)\) (resp. \(\mathbf{B}(t)\), \(\mathbf{C}(t)\), and \(\mathbf{D}(t)\)) as the corresponding point of tangency.

\label{def:sofa-boundaries-intro}
\end{definition}

The balancing ODEs by Romik are in terms of the four curves \(\mathbf{A}(t)\), \(\mathbf{B}(t)\), \(\mathbf{C}(t)\), \(\mathbf{D}(t)\) in \Cref{def:sofa-boundaries-intro} and the inner corner \(\mathbf{x}(t)\) of \(L_t\). In particular, he derives Gerver’s sofa \(S = G\) by parametrizing the bounday with the five curve segments (\Cref{fig:gerver-curves}), each on a different interval of \(t\), and solving for the ODEs on the curve segments.

Assume for now that all five curves are well-defined and continuously differentiable on their respective domain. This be alleviated in the full proof at \Cref{sec:injectivity-condition}. As \(S\) is a monotone sofa, the points \(\mathbf{A}(t)\) and \(\mathbf{C}(t)\) are well-defined over all \(t \in (0, \pi/2)\) and extends naturally to \(t \in [0, \pi/2]\) by taking limits. If the point \(\mathbf{B}(t)\) is well-defined (that is, if \(S\) makes contact with the wall \(b(t)\) of \(L_t\)), then since \(\mathbf{A}(t)\) and \(\mathbf{B}(t)\) are points of tangency of parallel lines \(a(t)\) and \(b(t)\) of distance one, we have \(\mathbf{B}(t) = \mathbf{A}(t) - u_t\) (Theorem 1 of \autocite{romikDifferentialEquationsExact2018}). Likewise, we have \(\mathbf{D}(t) = \mathbf{B}(t) - v_t\) if \(\mathbf{D}(t)\) is well-defined.

(See the right side of \Cref{fig:balanced-sofa}) Assume that for some angle \(t\) and its neighborhood, the supporting hallway \(L_t\) makes contact with \(S\) at three points \(\mathbf{A}(t)\), \(\mathbf{B}(t)\), and \(\mathbf{x}(t)\). Recall that \(S\) is the limit of balanced polygons \(S_\Theta\) (\Cref{sec:limit-of-maximum-polygon-caps}). As \(S_\Theta\) converges to \(S\), the three balanced sides (green) of \(S_\Theta\) on the walls \(a(t)\) and \(b(t)\) becomes the differential sides of \(G\) contributed by three points \(\mathbf{A}(t)\), \(\mathbf{B}(t)\), and \(\mathbf{x}(t)\). So their lengths balance each other as
\begin{equation}
\label{eqn:ode-example}
\left< \mathbf{A}'(t) , v_t \right> = \left< - \mathbf{B}'(t), v_t \right> + \left< \mathbf{x}'(t), v_t \right>
\end{equation}
which is one of the ODEs by Romik (Equation (20) of \autocite{romikDifferentialEquationsExact2018}). See the equations following \Cref{thm:gerver-odes} for many other examples.

\Cref{eqn:ode-example} is very useful but depends on the assumption that \(S\) makes contact with \(L_t\) at three points \(\mathbf{A}(t), \mathbf{B}(t), \mathbf{x}(t)\). So we (essentially) prove the following weaker \emph{inequality} that works for any \(S\) regardless of whether it makes contact with \(L_t\) at \(\mathbf{B}(t)\) or \(\mathbf{x}(t)\).
\begin{equation}
\label{eqn:ineq-example}
\left< \mathbf{A}'(t) , v_t \right> \leq \max\left( \left< - \mathbf{B}'(t), v_t \right>, 0 \right)  + \left| \left< \mathbf{x}'(t), v_t \right> \right|
\end{equation}
Even if \(S\) does not make contact with \(L_t\) on the line \(b(t)\), the point \(\mathbf{B}(t)\) extends naturally over all \(t \in [0, \pi/2]\) by letting \(\mathbf{B}(t) := \mathbf{A}(t) - u_t\).

We sketch the idea behind \Cref{eqn:ineq-example}. Our description here is only a rough sketch of the ideas and we hide the details of magnitude analysis. See the proof of \Cref{thm:balanced-discrete-ineq} for full details.

Take the maximum polygon sofa \(S_\Theta\) that approximates \(S\). Take three adjacent angles \(t - \delta, t, t + \delta\) from the finite angle set \(\Theta\). It turns out that the side of \(S_\Theta\) on the line \(a(t)\) is of magnitude \(\delta \left< \mathbf{A}'(t) , v_t \right> + O(\delta^2)\), so after dividing by \(\delta\), converges to \(\left< \mathbf{A}'(t) , v_t \right>\) as \(\delta \to 0\) and \(\Theta\) gets denser in \([0, \pi/2]\). This is the left-hand side of \Cref{eqn:ineq-example}.

Let \(\vec{b}(t)\) be the half-line on \(b(t)\) from \(\mathbf{x}(t)\) that represents the right inner wall of \(L_t\) starting with \(\mathbf{x}(t)\). We now overestimate all sides of \(S_\Theta\) on the half-line \(\vec{b}(t)\). Define \(R\) as the union of three closed half-planes \(H^\mathrm{d}(s)\), each of angle \(s = t - \delta, t, t + \delta\) bounded from below by the left inner wall \(d(s)\). We will overestimate the sides of \(S_\Theta\) on \(\vec{b}(t) \cap R\) and \(\vec{b}(t) \setminus R\) respectively. Adding two estimates below and sending \(\delta \to 0\) will give the right-hand side of \Cref{eqn:ineq-example}.

\begin{itemize}
\item
  The length of set \(\vec{b}(t) \cap R\) is of magnitude \(\leq \delta \left|\left< \mathbf{x}'(t) , v_t \right>\right| + O(\delta^2)\).

  To see this, observe that the point \(\mathbf{x}(t)\) is on the boundary \(d(t)\) of \(H^\mathrm{d}(t)\) and is away from the boundary \(d(t \pm \delta)\) of \(H^\mathrm{d}(t \pm \delta)\) by the signed distance \(\mp \delta \left< \mathbf{x}'(t) , v_t \right> + O(\delta^2)\) along the direction \(v_t\). Exact verification is done in \Cref{lem:leg-computation}.
\item
  The sides of \(S_\Theta\) on \(\vec{b}(t) \setminus R\) are of magnitude \(\leq \delta\max\left( \left< - \mathbf{B}'(t), v_t \right>, 0 \right) + O(\delta^2)\).

  To see this, observe that for each angle \(s = t - \delta, t, t + \delta\), the set \(S_\Theta\) is disjoint fron the inner quadrant \(Q_s^-\) of \(L_t\) bounded from above by \(b(s)\) and \(d(s)\). So for each \(s\), the set \(S_\Theta \setminus R\) is contained in the closed half-plane \(H^\mathrm{b}(s)\) bounded from below by the line \(b(s)\). Now the sides of \(S_\Theta\) on \(\vec{b}(t) \setminus R\) is contained in the segment of \(H^\mathrm{b}(t - \delta) \cap H^\mathrm{b}(t) \cap H^\mathrm{b}(t + \delta)\) on the line \(b(t)\). This segment is contributed by the lines \(b(t)\) and \(b(t \pm \delta)\) and is of the claimed magnitude.
\end{itemize}

In the actual proof prested in \Cref{sec:injectivity-condition}, \Cref{eqn:ineq-example} is not stated as-is and formulated quite differently as
\begin{equation}
\label{eqn:ineq-actual}
\sigma_K \leq k_0(g(t)) \, \mathrm{d} t
\end{equation}
in \Cref{thm:balanced-ineq-limit} where
\[
k_0(x) := \max\left( |x - 1|, (|x - 1| + 1) / 2 \right).
\]
This is to ensure that the inequality works for general \(S\) that may have the contact point \(\mathbf{A}(t)\) that is not differentiable in \(t\), or have more than one contact points with outer wall \(a(t)\). So the actual proof proceeds with \Cref{eqn:ineq-actual} that works for any \(K\), but it essentially follows the idea behind the proof of \Cref{eqn:ineq-example} sketched above.

We derive \Cref{eqn:ineq-actual} from \Cref{eqn:ineq-example} as below. This explains why the inequalities are more or less equivalent and why the function \(k_0\) is involved. First use \(\mathbf{B}(t) = \mathbf{A}(t) - u_t\) and \(u_t' = v_t\) and write
\[
\left< - \mathbf{B}'(t), v_t \right> = - \left< \mathbf{A}'(t), v_t \right> + 1.
\]
Then by letting \(\alpha := \left< \mathbf{x}'(t), v_t \right>\), \Cref{eqn:ineq-example} implies
\[
\left< \mathbf{A}'(t), v_t \right> \leq \max\left( \left| \alpha \right|, \left( |\alpha| + 1 \right) /2 \right) 
\]
in both cases \(\left< \mathbf{A}'(t), v_t \right> < 1\) and \(\left< \mathbf{A}'(t), v_t \right> \geq1\). Using \Cref{eqn:to-reveal}, the left-hand side is the differential side length at \(\mathbf{A}(t)\) equal to the density of \(\sigma_K\). It turns out that the value \(\alpha = \left< \mathbf{x}'(t), v_t \right>\) in right-hand side is equal to \(g(t) - 1\) where \(g(t)\) is the \emph{arm length} that will be defined soon. Substituting both sides, we get \Cref{eqn:ineq-actual}.

\subsection{Solving the Differential Inequality}
\label{sec:solving-the-differential-inequality}
We now sketch the argument that solves the \Cref{eqn:ineq-actual} and proves the injectivity condition.

Recall that \(\mathbf{y}(t)\) is the outer corner of \(L_t\) corresponding to \((1, 1)\) of \(L\) (\Cref{def:rotating-hallway-parts-intro}). For each \(t \in [0, \pi/2]\), the \emph{arm lengths} \(f(t)\) and \(g(t)\) measure the distance from outer corner \(\mathbf{y}(t)\) to \(\mathbf{A}(t)\) and \(\mathbf{C}(t)\) respectively.\footnote{The actual definition of arm lengths (\Cref{def:cap-tangent-arm-length}) have signs \(f_K^{\pm}(t)\) and \(g_K^{\pm}(t)\) in superscript as we cannot guarantee that the cap \(K\) meets the line \(a(t)\) at a single point \(\mathbf{A}(t)\). This sketch assumes that \(K\) meets \(a(t)\) at a single point, so that \(f(t) = f_K^{\pm}(t)\) (and the same for \(g\)).} A computation (\Cref{thm:inner-corner-deriv}) shows that
\[
\mathbf{x}'(t) = -(f(t) - 1) u_t + (g(t) - 1) v_t
\]
so that proving \(f(t), g(t) > 1\) on \(t \in (0, \pi/2)\) is sufficient for establishing the injectivity condition (\Cref{thm:injectivity-abridged}).

We will express \Cref{eqn:ineq-actual} purely in terms of arm lengths. The derivative of \(f(t)\) is
\[
f'(t) = g(t) - \left< \mathbf{A}'(t), v_t \right> 
\]
because \(\left< \mathbf{y}'(t), v_t \right> = g(t)\) and \(f(t) = \left< \mathbf{y}(t) - \mathbf{A}(t), v_t \right>\) (\Cref{thm:arm-length-differentiation}). As the side length \(\left< \mathbf{A}'(t), v_t \right>\) corresponds to \(\sigma_K\) at \(t\), \Cref{eqn:ineq-actual} is equivalent to
\begin{equation}
\label{eqn:iteration}
f'(t) \geq g(t) - k_0(g(t)) = m_0(g(t))
\end{equation}
where
\[
m_0(x) := x - k_0(x) = x - \max\left( |x - 1|, (|x - 1| + 1) / 2 \right)
\]
is monotonically increasing. This is done rigorously in \Cref{thm:leg-length-bounds}.

We now use \Cref{eqn:iteration} to iteratively obtain better lower bounds \(f_0(t), f_1(t), \dots\) of \(f(t)\) on \(t \in [0, \pi/2]\). Let \(f_0(t) := 0\) so that \(f_0(t)\) is a trivial lower bound of \(f(t)\). The same argument on \(S\) reflected along the \(y\)-axis shows that \(f_0(\pi/2 - t)\) is a lower bound of \(g(t)\). We have \(f(0) = 1\) because the point \(\mathbf{A}(0)\) should be on the \(x\)-axis. \Cref{eqn:iteration} implies that
\[
f(t) \geq 1 + \int_0^t m_0(g(u))\, du \geq 1 + \int_0^t m_0(f_0(\pi/2 - u))\, du.
\]
We just obtained a new lower bound of \(f(t)\) in the right-hand side. By letting
\begin{equation}
\label{eqn:lb-iter}
f_1(t) := \max\left(f_0(t), 1 + \int_0^t m_0(f_0(\pi/2 - u))\,du\right)
\end{equation}
we obtain a better lower bound \(f_1(t)\) of \(f(t)\). A symmetric argument also shows that \(g(t) \geq f_1(\pi/2 - t)\). Further iterations of \Cref{eqn:lb-iter} will give monotonically increasing lower bounds \(f_2, f_3, \dots\) of \(f\). Somewhat magically, eleven iterations of this improvement gives \(f(t) \geq f_{11}(t) > 1\), proving the injectivity condition (\Cref{fig:arm-length-lower-bounds}). Detailed computations are done in \Cref{sec:bounding-arm-lengths}.

\begin{figure}
\centering
\includegraphics{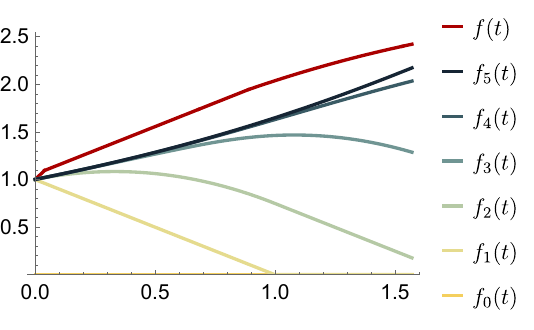}
\caption{The arm length \(f(t)\) of Gerver’s sofa \(G\) and the lower bounds \(f_0(t), f_1(t), \dots\) of \(f(t)\). Numerical computations show that three iterations are sufficient to give \(f(t) \geq f_3(t) > 1\). But to minimize computer assistance, we do more iterations and show \(f_i(t) \geq (i - 1) / 12\) for \(i \leq 10\) in \Cref{lem:lower-bound-j-iter}, which is sufficient to prove the injectivity hypothesis.}
\label{fig:arm-length-lower-bounds}
\end{figure}

\section{Optimality of Gerver's Sofa}
\label{sec:_optimality-of-gerver's-sofa}
\begin{quote}
\textbf{Summary:} This section is an overview of \Cref{sec:optimality-of-gerver's-sofa} that proves the main \Cref{thm:main}. The previous \Cref{sec:convex-domain-and-convex-curves} prepares a minimal theoretical framework needed to execute the ideas below in \Cref{sec:optimality-of-gerver's-sofa}.

We establish an upper bound \(\mathcal{Q}(S)\) of the area of any monotone sofa \(S\) satisfying the injectivity condition. To do so, we construct a region \(R\) enclosing \(S\) so that \(R = S\) if \(S\) is Gerver’s sofa \(G\). The upper bound \(\mathcal{Q}\) is then defined as the area of \(R\) (\Cref{sec:_definition-of-q}). We define a convex space \(\mathcal{L}\) of tuples \((K, B, D)\) of convex bodies, so that each sofa \(S\) embeds one-to-one to a tuple \((K, B, D) \in \mathcal{L}\) and \(\mathcal{Q}\) is a quadratic functional on \(\mathcal{L}\) via Brunn-Minkowski theory (\Cref{sec:quadraticity-of-q}). The concavity of \(\mathcal{Q}\) on \(\mathcal{L}\) is established using Mamikon’s theorem, and the local optimality of \(\mathcal{Q}\) at \(G\) is established using the local optimality ODEs on \(G\) by Romik (\Cref{sec:optimality-of-q-at-gerver's-sofa}). As \(G\) is a local optimum of a globally concave \(\mathcal{Q}\), it is also a global optimum of \(\mathcal{Q}\) and thus the area.
\end{quote}
\subsection{Definition of $\texorpdfstring{\mathcal{Q}}{Q}$}
\label{sec:_definition-of-q}
(See \Cref{fig:gerver}) The niche of Gerver’s sofa \(G\) has a characteristic shape made of one ‘core’ colored blue and the two ‘tails’ colored red. Assuming that a maximum-area sofa \(S_{\max}\) follows the same shape, the derivation \(S_{\max} = G\) is essentially done in the existing works establishing the local optimality of \(G\) \autocite{gerverMovingSofaCorner1992,romikDifferentialEquationsExact2018,dengSolvingMovingSofa2024}. So the difficulty of the moving sofa problem lies in showing that \(S_{\max}\) indeed follows the same shape as \(G\).

We circumvent the difficulty by defining a \emph{larger} region \(R\) contains \(S_{\max}\) and have the desired shape of one core and two tails. Then the upper bound \(\mathcal{Q}\) of the area of a moving sofa is simply defined as the area of \(R\). Take any monotone sofa \(S = K \setminus \mathcal{N}(K)\) of rotation angle \(\pi/2\) with cap \(K\), satisfying the injectivity condition. Recall that \(\mathbf{x}(t)\), \(b(t)\), and \(d(t)\) are respectively the inner corner, right inner wall, and left inner wall of supporting hallway \(L_t\) with angle \(t\). We construct \(R\) as follows.

(See \Cref{fig:sofa-ub-deco}) There is a specific angle \(\varphi \in [0.039, 0.040]\) such that the rotation path \(\mathbf{x}_G : [0, \pi/2] \to \mathbb{R}\) of Gerver’s sofa draws the core portion of the niche at the interval \([\varphi, \pi/2 - \varphi]\). Using the same angles, cut the cap \(K\) and niche \(\mathcal{N}(K)\) of arbitrary monotone sofa \(S\) into three parts, using the lines \(b(\varphi)\) and \(d(\pi/2 - \varphi)\) passing through \(\mathbf{x}(\varphi)\) and \(\mathbf{x}(\pi/2 - \varphi)\) of \(S\) respectively. Let \(H^\mathrm{R}\) (resp. \(H^\mathrm{L}\)) be the half-planes bounded from left by \(b(\varphi)\) (resp. right by \(d(\pi/2 - \varphi)\)). Let \(H^\mathrm{M}\) be the region \(\mathbb{R}^2 \setminus H^\mathrm{R} \setminus H^\mathrm{L}\). Then the sets \(H^\mathrm{R}, H^\mathrm{M}, H^{\mathrm{L}}\) partition the plane into three parts.\footnote{The half-planes \(H^\mathrm{R}\) and \(H^\mathrm{L}\) do overlap technically, but it does not matter as the region of overlap \(H^\mathrm{R} \cap H^\mathrm{L}\) is disjoint from the sets \(K\) and \(\mathcal{N}(K)\) that we divide (\Cref{lem:cap-left-right-disjoint}).}

The injectivity condition on \(S\) (\Cref{thm:injectivity-abridged}) implies that the rotation path \(\mathbf{x}(t)\) should be in \(H^\mathrm{R}\), \(H^\mathrm{M}\), and \(H^\mathrm{L}\) as \(t\) is in the interval \([0, \varphi]\), \([\varphi, \pi/2 - \varphi]\), and \([\pi/2 - \varphi, \pi/2]\) respectively (\Cref{lem:monotonicity-intervals}). Using this, we take a subset \(N'\) of the niche \(\mathcal{N}(K)\) as follows. In the region \(H^\mathrm{R}\), take only the region swept out by the right inner wall \(b(t)\) as \(t \in [\varphi, \pi/2]\), where \(\mathbf{x}(t)\) is outside \(H^\mathrm{R}\) (colored blue). In the region \(H^\mathrm{M}\), take only the region bounded by the lines \(b(\varphi), d(\pi/2 - \varphi)\), \(y=0\), and the inner corner \(\mathbf{x}(t)\) restricted to \([\varphi, \pi/2 - \varphi]\) (colored green). In the region \(H^\mathrm{L}\), take only the region swept out by the left inner wall \(d(t)\) as \(t \in [0, \pi/2 - \varphi]\), where \(\mathbf{x}(t)\) is outside \(H^\mathrm{L}\) (colored red). The injectivity condition guarantees that the final region \(N'\) is a subset of the niche \(\mathcal{N}(K)\).

\begin{figure}
\centering
\includegraphics{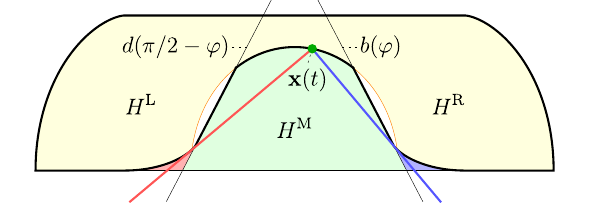}
\caption{The overestimated region \(R\) is obtained by taking the region \(N'\) (blue in \(H^\mathrm{R}\), green in \(H^\mathrm{M}\), red in \(H^\mathrm{L}\)) away from the cap \(K\) of monotone sofa \(S\).}
\label{fig:sofa-ub-deco}
\end{figure}

The region \(R\) is now simply defined as \(K \setminus N'\). While \(R\) may not be a moving sofa, the ‘niche’ \(N'\) of \(R\) consists of one core and two tails like that of \(G\) does. For a general monotone sofa \(S\), the endpoints of the core and two tails of \(N'\) does not match each other, so we simply connect them by the line segments each on \(b(\varphi)\) and \(d(\pi/2 - \varphi)\). As Gerver’s sofa \(G\) is constructed by design to have the matching endpoints of core and tails, the region \(R\) is equal to \(S\) if \(S = G\) (\Cref{thm:upper-bound-q-gerver-match}).

\subsection{Quadraticity of $\texorpdfstring{\mathcal{Q}}{Q}$}
\label{sec:quadraticity-of-q}
The collection \(\mathcal{K}\) of all planar convex bodies form a \emph{convex domain} with the barycentric operation \(c_\lambda(K_1, K_2) := (1 - \lambda) K_1 + \lambda K_2\). Here,
\[
aK := \left\{ ap : p \in K \right\}
\]
is the \emph{dilation} of a convex body \(K\) by \(a \geq 0\), and
\[
K_1 + K_2 := \left\{ x_1 + x_2 \in \mathbb{R}^2 : x_1 \in K_1, x_2 \in K_2 \right\}
\]
is the \emph{Minkowski sum} of convex bodies \(K_1, K_2\). The operations satisfy necessary properties (commutative, associative, and distributive) that makes \(\mathcal{K}\) an abstract convex cone.

Many values on convex body \(K\), including the support function \(h_K(t) := v_K^+(t) \cdot u_t\), vertex \(v_K^+(t)\), and surface area measure \(\sigma_K\) are convex-linear in \(K\). Correspondingly, the area
\[
|K| = \frac{1}{2} \int_{t \in S^1} h_K(t) \, \sigma_K(t)
\]
of \(K \in \mathcal{K}\) is a \emph{quadratic} functional on the convex domain \(\mathcal{K}\). This notion of a quadratic functional on a (abstract) convex domain is established rigorously in \Cref{sec:convex-domain}.

In the previous \Cref{sec:_definition-of-q}, we defined an overestimation \(R\) of a monotone sofa \(S\). We will now define the three convex bodies \(K\), \(B\), and \(D\) from \(S\) that represents different parts of the region \(R\). This step is very important. While the area of \(R\) does not have a quadratic expression involving \(K\) only, it does have a quadratic expression involving \(K, B\), and \(D\) (\Cref{def:upper-bound-q}) that is amenable to further analysis (see also \Cref{rem:cap-tail-extension}).

(Compare \Cref{fig:sofa-ub-deco} with \Cref{fig:upper-bound}) Again, let \(S\) be any monotone sofa of rotation angle \(\pi/2\) satisfying the injectivity condition. The convex body \(K\) is the (usual) cap of \(S\). Convex bodies \(B\) and \(D\) are the portions of the region \(R\) in the half-planes \(H^\mathrm{R}\) and \(H^\mathrm{L}\) respectively. Put precisely, \(B\) (resp. \(D\)) is the cap \(K\) intersected with the closed half-planes bounded from below by \(b(t)\) for all \(t \in [\varphi, \pi/2]\) (resp. \(d(t)\) for all \(t \in [0, \pi/2 - \varphi]\)). This defines the convex bodies \(K, B\), and \(D\) from \(S\).

The three convex bodies satisfy certain linear constraints (\Cref{lem:right-left-body}). An example is
\begin{equation}
\label{eqn:constraint-example}
h_K(t) + h_B(\pi + t) \leq 1
\end{equation}
for every \(t \in [\varphi^\mathrm{R}, \pi/2]\), which holds because \(B\) is bounded from below by the line \(b(t)\) which is distance one away from the the supporting line \(a(t)\) of \(K\). With this, define \(\mathcal{L}\) as the collection of all tuples \((K, B, D)\) of convex bodies satisfying such constraints (\Cref{def:cap-tail-space}). The collection of all monotone sofas \(S\) with rotation angle \(\pi/2\) now embeds to a subset of \(\mathcal{L}\) by constructing the convex bodies \((K, B, D)\) from \(S\) as above. The space \(\mathcal{L}\) is a convex domain with pairwise barycentric operation
\[
c_\lambda((K_1, B_1, D_1), (K_2, B_2, D_2)) := \left( c_\lambda(K_1, K_2), c_\lambda(B_1, B_2), c_\lambda(D_1, D_2) \right).
\]

We now define the upper bound \(\mathcal{Q}(K, B, D)\) of the area of a monotone sofa \(S\) as a quadratic functional on \((K, B, D) \in \mathcal{L}\) (\Cref{def:upper-bound-q}). Recall that \(\mathcal{Q}\) is equal to \(|K| - |N'|\) where \(|N'|\) is the area of the underestimated niche in \Cref{fig:sofa-ub-deco}. Using injectivity condition, we essentially\footnote{The actual proof takes three Jordan curves, two bounding the red and blue regions of \Cref{fig:sofa-ub-deco} (\Cref{lem:cap-left-right-tail}) and one bounding the green region of \Cref{fig:sofa-ub-deco} (\Cref{lem:cap-middle-lower-estimate}).} show that the boundary \(\gamma\) of \(N'\) is a Jordan curve. Take \(\gamma\) counterclockwise, then by Green’s theorem the region \(N'\) have area
\[
\mathcal{J}(\gamma) := \frac{1}{2} \int_a^b \gamma(t) \times \gamma'(t)\, dt
\]
which we call the \emph{curve area functional} on \(\gamma : [a, b] \to \mathbb{R}^2\). So we have \(\mathcal{Q} = |K| - \mathcal{J}(\gamma)\) in particular.

For a convex body \(X = B\) or \(D\), define the segment \(\mathbf{u}_X^{a, b}\) of the boundary of \(X\) as the union of all edges of \(X\) with normal vectors \(u_t\) of angle \(t \in (a, b)\). We further express \(\mathcal{J}(\gamma)\) as a quadratic term on \(K, B\), and \(D\) by breaking the boundary \(\gamma\) of \(N'\) into the following five segments.

\begin{enumerate}
\def\labelenumi{\arabic{enumi}.}
\tightlist
\item
  The segment \(\textbf{d}_D := \mathbf{u}_D^{3\pi/2, 3\pi/2 + \varphi^\mathrm{L}}\) of the boundary of \(D\), representing the left tail of \(N'\).
\item
  The line segment connecting the right end \(Y_D\) of the left tail \(\mathbf{d}_D\), to the left end \(\mathbf{x}_K^\mathrm{L} := \mathbf{x}_K(\pi/2 - \varphi)\) of the core \(\mathbf{x}_K\).
\item
  The rotation path \(\mathbf{x}_K : [\varphi, \pi/2 - \varphi] \to \mathbb{R}^2\) of cap \(K\) reversed in direction, representing the core of \(N'\).
\item
  The line segment connecting the right end \(\mathbf{x}_K^\mathrm{R} := \mathbf{x}_K(\varphi)\) of the core \(\mathbf{x}_K\), to the left end \(X_B\) of the right tail \(\mathbf{b}_B\).
\item
  The segment \(\mathbf{b}_B := \mathbf{u}_{B}^{\pi + \varphi^\mathrm{R}, 3\pi/2}\) of the boundary of \(B\), representing the right tail of \(N'\).
\end{enumerate}

Each segment corresponds to each term in the appearing order of the \Cref{def:upper-bound-q} of \(\mathcal{Q}(K, B, D) = |K| - \mathcal{J}(\gamma) =\)
\[
|K| + \mathcal{J}\left( \mathbf{d}_D \right) + \mathcal{J} \left( Y_D,   \mathbf{x}_K^\mathrm{L} \right) - \mathcal{J}\left( \mathbf{x}_K|_{[\varphi, \pi/2 - \varphi]} \right) + \mathcal{J} \left( \mathbf{x}_K^\mathrm{R}, X_B \right)  + \mathcal{J}\left( \mathbf{b}_B \right)
\]
where \(\mathcal{J}(p, q) := (x_p y_q - x_q y_p)/2\) is the curve area functional of the segment from \(p = (x_p, y_p)\) to \(q = (x_q, y_q)\).

We now argue that \(\mathcal{Q}\) is quadratic in \(K\), \(B\), and \(D\). The area \(|K|\) is quadratic in \(K\) as seen above. The quadraticity of the core term \(\mathcal{J}\left( \mathbf{x}_K|_{[\varphi, \pi/2 - \varphi]} \right)\) comes from linearity of \(\mathbf{x}_K\) in \(K\). \Cref{thm:convex-curve-area-functional} computes
\[
\mathcal{J}\left( \mathbf{u}_X^{a, b} \right) = \frac{1}{2} \int_{t \in (a, b)}h_K(t)\, \sigma_K(dt)
\]
which is quadratic in \(K\). This establishes the quadraticity of two tail terms \(\mathcal{J}(\mathbf{d}_D)\) and \(\mathcal{J}(\mathbf{b}_B)\). The terms on two line segments come from bilinearity of \(\mathcal{J}(p, q) := (x_p y_q - x_q y_p)/2\) in \(p, q \in \mathbb{R}^2\).

\subsection{Optimality of $\texorpdfstring{\mathcal{Q}}{Q}$ at Gerver's Sofa}
\label{sec:optimality-of-q-at-gerver's-sofa}
(See \Cref{fig:mamikon-simple}) Let \(K\) be any convex body. Take an interval \([a, b]\) of length \(\leq 2\pi\). For each angle \(t \in [a, b]\), assume a tangent segment \(s_t\) of \(K\) with length \(\alpha (t)\) and one endpoint \(v_K^+(t)\) on \(K\), making an angle of \(t\) from the \(y\)-axis. \emph{Mamikon’s theorem} states that the region swept out by the segments \(s_t\) over all \(t \in [a, b]\) is exactly \(\frac{1}{2}\int_a^b \alpha(t)^2\,dt\).

\begin{figure}
\centering
\includegraphics{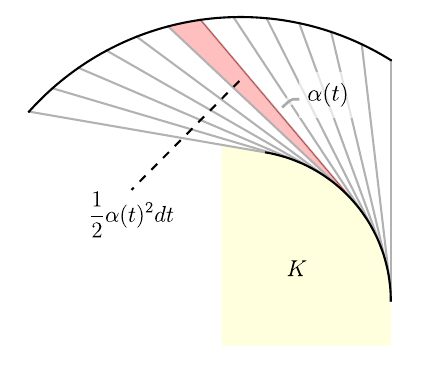}
\caption{Mamikon’s theorem.}
\label{fig:mamikon-simple}
\end{figure}

(See \Cref{fig:mamikon-thm-on-sofa}) Mamikon’s theorem is used to show that \(\mathcal{Q}\) is globally concave on \(\mathcal{L}\). The idea is to attach multiple ‘Mamikon regions’ (grey) to the region \(R\) of area \(\mathcal{Q}\). The length \(\alpha(t)\) of each tangent segment in grey with angle \(t\) turns out to be linear in \(\mathcal{L}\). So the area \(\frac{1}{2}\int_a^b \alpha(t)^2\,dt\) of each Mamikon region is convex and quadratic in \(\mathcal{L}\). We show in \Cref{lem:mamikon-middle-eq} that the total area of \(R\) and all Mamikon regions (bounded by bold lines) is linear in \(K\). So the area \(\mathcal{Q}\) of \(R\) is a linear functional (bold lines) subtracted by convex quadratic functionals (grey regions), which is concave. \Cref{sec:concavity-of-q} rigorously checks the full details.

\begin{figure}
\centering
\includegraphics{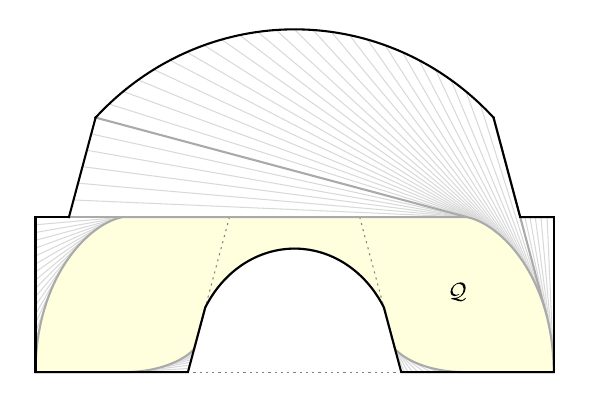}
\caption{Mamikon’s theorem applied to the upper bound \(\mathcal{Q}\) of sofa area. As Mamikon regions in grey are added to the region \(R\) with area \(\mathcal{Q}\), the resulting shape bounded by bold lines have an area linear in \(K\).}
\label{fig:mamikon-thm-on-sofa}
\end{figure}

To establish the main \Cref{thm:main}, it suffices to show that the tuple \((K, B, D) \in \mathcal{L}\) arising from Gerver’s sofa \(G\) is a maximizer of \(\mathcal{Q}\). Assuming this, recall that a balanced maximum sofa \(S^*\) attaining the maximum area also satisfies the injectivity condition (\Cref{thm:injectivity-abridged}). So the maximum-area \(S^*\) also corresponds to another tuple \((K^*, B^*, D^*) \in \mathcal{L}\) of convex bodies as described in \Cref{sec:quadraticity-of-q}. Because the region \(R\) of Gerver’s sofa \(G\) matches with \(G\), we have \(\mathcal{Q}(K, B, D) = |R| = |G|\). By the optimality of \(\mathcal{Q}\) at \((K, B, D)\), and that \(\mathcal{Q}(K^*, B^*, D^*)\) is an upper bound of the area \(|S^*|\), we have
\[
\mathcal{Q}(K, B, D) \geq \mathcal{Q}(K^*, B^*, D^*) \geq |S^*|.
\]
So we have \(|G| \geq |S^*|\), proving the main \Cref{thm:main}.

We now show that \(\mathcal{Q}\) is maximized at the point \((K, B, D) \in \mathcal{L}\) from Gerver’s sofa \(G\). Choose an arbitrary \((K^*, B^*, D^*) \in \mathcal{L}\). The directional derivative of \(\mathcal{Q}\) at \((K, B, D)\) in the direction towards \((K^*, B^*, D^*)\) is defined as
\begin{equation}
\label{eqn:dir-deriv-sample}
\begin{split}
& \phantom{{}={}} D \mathcal{Q}(K, B, D; K^*, B^*, D^*) \\
& := \left. \frac{d}{d\lambda} \right|_{\lambda = 0} \mathcal{Q}(c_{\lambda}((K, B, D), (K^*, B^*, D^*))
\end{split}
\end{equation}
where \(\lambda \in [0, 1]\) interpotates between \((K, B, D)\) and \((K^*, B^*, D^*)\). If the value is \(\leq 0\) regardless of the choice of \((K^*, B^*, D^*) \in \mathcal{L}\), then \((K, B, D)\) indeed achieves the maximum value of concave and quadratic \(\mathcal{Q}\) as desired; this can be shown by quadraticity of \(\mathcal{Q}\) (\Cref{thm:quadratic-variation}).

So it remains to compute \Cref{eqn:dir-deriv-sample} and show that it is non-positive. Assuming that \((K^*, B^*, D^*)\) is close enough to \((K, B, D)\), the value of \(D \mathcal{Q}\) is approximately the rate of change of \(\mathcal{Q}\) along the interval \(\lambda \in [0, 1]\), so
\[
D \mathcal{Q}(K, B, D; K^*, B^*, D^*)  \simeq \mathcal{Q}(K^*, B^*, D^*) - \mathcal{Q}(K, B, D)
\]
and we use them interchangeably here for ease of explanation (we do not use this in full calculation). Instead of computing the full \(D\mathcal{Q}\) for general \((K^*, B^*, D^*) \in \mathcal{L}\) as in \Cref{sec:directional-derivative-of-q}, we take two representative cases of \((K^*, B^*, D^*)\) and illustrate how \(D\mathcal{Q}\) is computed.

Recall that the boundary of \(G\) is parametrized by the contact points \(\mathbf{A}(t)\), \(\mathbf{B}(t)\), \(\mathbf{C}(t)\), \(\mathbf{D}(t)\), and \(\mathbf{x}(t)\) that \(G\) makes with supporting hallways \(L_t\) (\Cref{fig:gerver-curves}). In both representative cases of \((K^*, B^*, D^*)\), fix a particular angle \(t \in (0, \pi/2)\) so that \(G\) meets \(L_t\) at three points \(\mathbf{A}(t)\), \(\mathbf{B}(t)\), and \(\mathbf{x}(t)\) as in the right side of \Cref{fig:balanced-sofa}. Also, fix a sufficient small \(\delta > 0\) and let \(I := [t, t + \delta]\). Take an arbitrary \(\epsilon > 0\) that is sufficiently small relative to \(\delta\). We now assume the first case of \((K^*, B^*, D^*) \in \mathcal{L}\).

\begin{quote}
\textbf{Case 1:} Recall that \(G = H \cap \bigcap_{s \in [0, \pi/2]} L_s\) is the intersection of the horizontal strip \(H\) and supporting hallways \(L_s\). Translate each \(L_s\) to \(L_s^* := L_s + \epsilon u_t\) by \(\epsilon u_t\) for any \(s \in I = [t, t + \delta]\) and fix \(L_s^* := L_s\) for any other \(s \not\in I\). Take the new sofa \(G^* := H \cap \bigcap_{t \in [0, \pi/2]} L_s^*\) which is a slight perturbation of \(G\). Assume the case where the convex bodies \((K^*, B^*, D^*) \in \mathcal{L}\) come from \(G^*\) as described in \Cref{sec:quadraticity-of-q}.
\end{quote}

In this Case 1, the new sofa \(G^*\) is obtained from \(G\) by the following changes in region. We ignore second-order or smaller terms of \(\delta\) and \(\epsilon\) in length.

\begin{enumerate}
\def\labelenumi{\arabic{enumi}.}
\tightlist
\item
  Adding a rectangle of approximate base \(\delta \left< \mathbf{A}'(t), u_t \right>\) and height \(\epsilon\) near the point \(\mathbf{A}(t)\).
\item
  Removing a rectangle of approximate base \(\delta \left< -\mathbf{B}'(t), u_t \right>\) and height \(\epsilon\) near the point \(\mathbf{B}(t)\).
\item
  Removing a parallelogram with approximate sides of vector \(\delta \mathbf{x}'(t)\) and \(\epsilon u_t\) near the point \(\mathbf{x}(t)\).
\end{enumerate}

So the area change for Case 1 is approximately
\begin{equation}
\label{eqn:approx}
\begin{split}
\phantom{{}={}} & \phantom{{}\simeq{}} D \mathcal{Q}(K, B, D; K^*, B^*, D^*) \simeq |G^*| - |G| \\
\phantom{{}={}} & \simeq \delta \left( \epsilon \left< \mathbf{A}'(t), u_t \right> - \epsilon \left< -\mathbf{B}'(t), u_t \right> - \epsilon \left< \mathbf{x}'(t), u_t \right> \right).
\end{split}
\end{equation}
Romik’s ODE \Cref{eqn:ode-example} on \(G\) now balance the differential side lengths and make the value of \(D\mathcal{Q}\) from \((K, B, D)\) to \((K^*, B^*, D^*)\) equal to zero. In fact, this is essentially how Romik derived his ODE for \(G\) in \autocite{romikDifferentialEquationsExact2018}.

Now we assume a slightly more general Case 2 where the edges of \(K\) and \(B\) along the angle \(t\) can move in the direction of \(\pm u_t\) independently.

\begin{quote}
\textbf{Case 2.} Take \(\delta' > 0\) sufficiently smaller than \(\delta\) so that \(U := (t - \delta', t + \delta + \delta')\) is a sufficiently close neighborhood of \(I = [t, t + \delta]\). Take \(\epsilon_K\) and \(\epsilon_B > 0\) sufficiently smaller than \(\delta\) and \(\delta'\). Take any \((K^*, B^*, D^*) \in \mathcal{L}\) so that the followings are true.

\begin{itemize}
\tightlist
\item
  \(h_{K^*}(s) = h_K(s) + \epsilon_K\) for \(s \in I\) and \(h_{K^*}(s) = h_K(s)\) for \(s \not\in U\)
\item
  \(h_{B^*}(s) = h_B(s) - \epsilon_B\) for \(s \in I + \pi = [\pi + t, \pi + t + \delta]\) and \(h_{B^*}(s) = h_B(s)\) for \(s \not\in U + \pi\).
\end{itemize}
\end{quote}

Observe that this Case 2 reduces to the previous Case 1 if we let \(\epsilon_K = \epsilon_B = \epsilon\). Recall that for a general \((K^*, B^*, D^*) \in \mathcal{L}\), the value \(\mathcal{Q}(K^*, B^*, D^*)\) is equal to \(|K^*| - |N'|\) where \(N'\) is the region bounded by five curves and lines from \(K^*, B^*\), and \(D^*\) as in \Cref{sec:quadraticity-of-q}. In this case, the rectangle in (1) near \(\mathbf{A}(t)\) now have height \(\epsilon_K\), and the rectangle near \(\mathbf{B}(t)\) in (2) and parallelogram near \(\mathbf{x}(t)\) in (3) have sides of length \(\epsilon_B\) and \(\epsilon_K\) respectively. Now the change in the value of \(\mathcal{Q}\) is approximately
\[
\begin{split}
\phantom{{}={}} & \phantom{{}\simeq{}} D \mathcal{Q}(K, B, D; K^*, B^*, D^*) \\
\phantom{{}={}} & \simeq \delta \left( \epsilon_K \left< \mathbf{A}'(t), u_t \right> - \epsilon_B \left< -\mathbf{B}'(t), u_t \right> - \epsilon_K \left< \mathbf{x}'(t), u_t \right> \right).
\end{split}
\]

Now it takes a bit more than just \Cref{eqn:ode-example} to show \(D \mathcal{Q} \leq 1\). Because the hallway \(L_t\) meets \(G\) at the points \(\mathbf{A}(t)\) and \(\mathbf{B}(t)\), we have \(h_K(t) + h_B(\pi + t) = 1\). The condition in \Cref{eqn:constraint-example} to \((K^*, B^*, D^*) \in \mathcal{L}\) implies \(h_{K^*}(t) + h_{B^*}(\pi + t) \leq 1\). So we should have \(\epsilon_K \leq \epsilon_B\). Now \Cref{eqn:ode-example} with \(\left< -\mathbf{B}'(t), u_t \right> \geq 0\), that the differential side at \(\mathbf{B}(t)\) is nondegenerate, imply
\[
\begin{split}
\phantom{{}={}} & \phantom{{}\simeq{}} D \mathcal{Q}(K, B, D; K^*, B^*, D^*) \simeq |G^*| - |G| \\
\phantom{{}={}} & \simeq \delta \left( \epsilon_K \left< \mathbf{A}'(t), u_t \right> - \epsilon_B \left< -\mathbf{B}'(t), u_t \right> - \epsilon_K \left< \mathbf{x}'(t), u_t \right> \right) \\
& \leq \delta \left( \epsilon_K \left< \mathbf{A}'(t), u_t \right> - \epsilon_K \left< -\mathbf{B}'(t), u_t \right> - \epsilon_K \left< \mathbf{x}'(t), u_t \right> \right) = 0.
\end{split}
\]

In the full calculation of \(D\mathcal{Q}\), all the differential sides of \(K, B,\) and \(D\) are simultaneously taken into account. The interval \(t \in [0, \pi/2]\) is divided into five intervals \(I_1, \dots, I_5\), where the set of contact points between \(G\) and \(L_t\) is fixed for each \(t \in I_i\) (\Cref{def:gerver-intervals}; see also the second column of \Cref{thm:gerver-odes}). For each \(t \in I_i\), the differential sides with normal vectors \(\pm u_t\) (or \(\pm v_t\)) contribute to \(D \mathcal{Q}\) in a way analogous to the Case 2 above. In particular, Case 2 corresponds to the case \(t \in I_4\) and normal vectors \(\pm u_t\).

In the full proof of \(D\mathcal{Q} \leq 0\) at Gerver’s sofa \(G\), we first import a total of 10 ODEs in \(\mathbf{A}, \mathbf{B}, \mathbf{C}, \mathbf{D},\) and \(\mathbf{x}\) that balance all the differential sides of \(G\) from \autocite{romikDifferentialEquationsExact2018} (\Cref{thm:gerver-odes}). Each ODE is effective on an interval \(t \in I_i\) and a specific direction \(\pm u_t\) or \(\pm v_t\). Then the side lengths of \(G\) on \(\mathbf{A}(t), \mathbf{B}(t), \mathbf{C}(t), \mathbf{D}(t),\) and \(\mathbf{x}(t)\) are represented as the surface area measures of \(K, B,\) and \(D\) (\Cref{pro:measure-translation}) and a measure \(\iota_K\) (\Cref{def:i-cap}). The 10 ODEs of Romik are translated to equalities of measures on each \(I_i\) in \Cref{thm:upper-bound-q-gerver}. They are used in the full computation at \Cref{thm:variation-a2-gerver} where the contribution by each angle \(t \in I_i\) is shown to be nonnegative. The analysis of Case 2 above, in particular, corresponds to the third case in the proof of \Cref{thm:variation-a2-gerver}.

\chapter{Monotone Sofas and Caps}
\label{sec:monotone-sofas-and-caps}
This chapter follows the overview in \Cref{sec:_monotone-sofas-and-caps} and shows that a maximum-area sofa is a \emph{monotone sofa} \(S\) which is the \emph{cap} \(K\) subtracted by \emph{niche} \(\mathcal{N}(K)\).

\begin{itemize}
\tightlist
\item
  \Cref{sec:planar-convex-body} prepares standard notions on a planar convex body \(K\).
\item
  \Cref{sec:supporting-hallway} defines the \emph{supporting hallway} \(L_S(t)\) of any moving sofa \(S\) making contact with \(S\) in the outer walls.
\item
  \Cref{sec:monotone-sofa} defines the intersection \(\mathcal{I}(S)\) of unit-width stripes \(H\), \(V_\omega\), and the supporting hallways \(L_S(t)\) containing \(S\). We show that \(\mathcal{I}(S)\) is connected (\Cref{thm:monotonization-is-connected}) so that it is a larger moving sofa containing \(S\) (\Cref{thm:monotonization}). With this, we define a \emph{monotone sofa} as the intersection \(\mathcal{I}(S)\) coming from some moving sofa \(S\) (\Cref{def:monotone-sofa}).
\item
  \Cref{sec:cap-and-niche} defines the notion of cap \(K\) (\Cref{def:cap}) and niche \(\mathcal{N}(K)\) (\Cref{def:niche}), and shows that a monotone sofa \(S\) is equal to the cap \(K := \mathcal{C}(S)\) subtracted by niche \(\mathcal{N}(K)\) (\Cref{thm:monotone-sofa-structure}).
\item
  \Cref{sec:cap-contains-niche} shows that for a monotone sofa \(S\), its cap \(K\) contains the niche \(\mathcal{N}(K)\) as a subset (\Cref{thm:niche-in-cap}). Along the way, many notions on the cap and niche are defined. We define the space of all caps \(\mathcal{K}_\omega^\mathrm{c}\) with rotation angle \(\omega\) and turn the problem into the maximization of the sofa area functional \(\mathcal{A}_\omega(K) := |K| - |\mathcal{N}(K)|\) on caps \(K \in \mathcal{K}_\omega^\mathrm{c}\).
\end{itemize}
\section{Planar Convex Body}
\label{sec:planar-convex-body}
\begin{definition}

A \emph{planar convex body} \(K\) is a nonempty, compact, and convex subset of \(\mathbb{R}^2\). Define \(\mathcal{K}\) as the collection of all planar convex bodies \(K\).

\label{def:convex-body}
\end{definition}

All convex bodies appearing in this work will be planar, so we will omit the word ‘planar’. Many authors also require \(K^\circ\) to be nonempty, but we follow \autocite{schneider_2013} and allow \(K^\circ\) to be empty. That is, a closed line segment or a point is also a convex body.

We define standard notions on a convex body \(K\). Note that the notions generalize naturally to any nonempty and compact \(S \subseteq \mathbb{R}^2\). Here, \(S^1\) is taken as \(\mathbb{R} / 2\pi \mathbb{Z}\). We will denote an element or interval of \(S^1\) by its representation in \(\mathbb{R}\).

\begin{figure}
\centering
\includegraphics{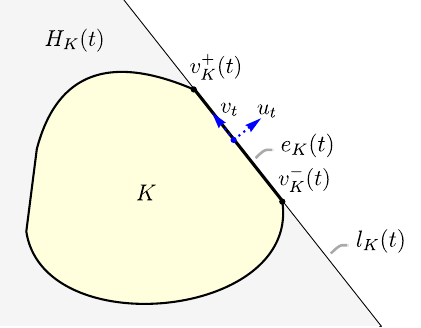}
\caption{A planar convex body \(K\) with its edge, vertices, supporting line, and half-plane.}
\label{fig:convex-body}
\end{figure}

\begin{definition}

For any subset \(X\) of \(\mathbb{R}^2\), denote the topological closure, boundary, and interior as \(\overline{X}\), \(\partial X\), and \(X^\circ\) respectively.

\label{def:topological-operations}
\end{definition}

\begin{definition}

For any angle \(t\) in \(S^1\) or \(\mathbb{R}\), define the unit vectors \(u_t = \left( \cos t, \sin t \right)\) and \(v_t = \left( -\sin t,\cos t \right)\).

\label{def:frame}
\end{definition}

\begin{definition}

For any \(t\in S^1\) and \(h \in \mathbb{R}\), define the line \(l(t, h)\) with the \emph{normal angle} \(t\) and the signed distance \(h\) from the origin as
\[
l(t, h) = \left\{ p \in \mathbb{R}^2 : p \cdot u_t = h \right\}.
\]

\label{def:line}
\end{definition}

\begin{definition}

For any \(t \in S^1\) and \(h \in \mathbb{R}\), define the \emph{closed half-planes} \(H_{\pm}(t, h)\) and the \emph{open half-planes} \(H_{\pm}^\circ(t, h)\) with the boundary \(l(t, h)\) as the following.
\begin{gather*}
H_-(t, h) := \left\{ p \in \mathbb{R}^2 : p \cdot u_t \leq h \right\} \qquad H_-^{\circ}(t, h) := \left\{ p \in \mathbb{R}^2 : p \cdot u_t < h \right\} \\
H_+(t, h) := \left\{ p \in \mathbb{R}^2 : p \cdot u_t \geq h \right\} \qquad H_+^{\circ}(t, h) := \left\{ p \in \mathbb{R}^2 : p \cdot u_t > h \right\}
\end{gather*}
We say that \(H_{-}(t, h)\) and \(H_{-}^{\circ}(t, h)\) have the \emph{normal angle} \(t\) and \emph{normal vector} \(u_t\).

\label{def:half-plane}
\end{definition}

Consequently, \(H_+(t, h)\) and \(H_+^{\circ}(t, h)\) have the normal angle \(t + \pi\).

\begin{definition}

For any nonempty and compact \(S \subseteq \mathbb{R}^2\), define its \emph{support function} \(h_S : S^1 \to \mathbb{R}\) as the value \(h_S(t) := \sup \left\{ s \cdot u_t : s \in S \right\}\).

\label{def:support-function}
\end{definition}

\begin{definition}

For any nonempty and compact \(S \subseteq \mathbb{R}^2\) and angle \(t \in S^1\), define the \emph{supporting line} \(l_S(t)\) of \(S\) with \emph{normal angle} \(t\) as \(l_S(t) := l(t, h_S(t))\). Define the \emph{supporting half-plane} \(H_S(t)\) of \(S\) with \emph{normal angle} \(t\) as \(H_S(t) := H_-(t, h_S(t))\).

\label{def:supporting-line-half-plane}
\end{definition}

\begin{definition}

For any nonempty and compact \(S\) and angle \(t\) in \(S^1\) or \(\mathbb{R}\), the \emph{width} of \(S\) along the direction of angle \(t\) (or unit vector \(u_t\)) is the distance between the parallel supporting lines \(l_S(t)\) and \(l_S(t + \pi)\) of \(S\) defined as \(h_S(t) + h_S(t + \pi)\).

\label{def:width}
\end{definition}

In this work only, we use the following notions of \emph{vertices} and \emph{edges} of a convex body \(K\).

\begin{definition}

For any convex body \(K\) and \(t \in S^1\), define the \emph{edge} \(e_K(t)\) of \(K\) as the intersection of \(K\) with the supporting line \(l_K(t)\).

\label{def:convex-body-edge}
\end{definition}

\begin{definition}

For any convex body \(K\) and \(t \in S^1\), let \(v_K^+(t)\) and \(v_K^-(t)\) be the endpoints of the edge \(e_K(t)\) such that \(v_K^+(t)\) is positioned farthest in the direction of \(v_t\) and \(v_K^-(t)\) is positioned farthest in the opposite direction \(-v_t\). We call \(v_K^{\pm}(t)\) the \emph{vertices} of \(K\).

\label{def:convex-body-vertex}
\end{definition}

It is possible that the edge \(e_K(t)\) can be a single point. In such case, the supporting line \(l_K(t)\) makes contact with \(K\) at the single point \(v_K^+(t) = v_K^-(t)\).

\begin{definition}

Define \(\mathcal{H}^1\) as the \emph{Hausdorff measure of dimension one} on \(\mathbb{R}^2\).

\label{def:hausdorff-measure}
\end{definition}

That is, if \(X\) is a disjoint union of finite line segments in \(\mathbb{R}^2\), then \(\mathcal{H}^1(X)\) is the sum of all lengths of the line segments.

\begin{definition}

Denote the \emph{Hausdorff distance} between convex bodies \(K_1\) and \(K_2\) as \(d_{\mathrm{H}}\). This is the supremum norm between \(h_{K_1}\) and \(h_{K_2}\) (Lemma 1.8.14, page 66 of \autocite{schneider_2013}).

\label{def:hausdorff-distance}
\end{definition}

\begin{definition}

Denote the \emph{surface area measure} of a convex body \(K\) as \(\sigma_K\). For simplicity, denote \(\sigma_K(\left\{ t \right\})\) as \(\sigma_K(t)\).

\label{def:surface-area-measure}
\end{definition}

We refer to page 214 of \cite{schneider_2013} for a full construction of \(\sigma_K\). It measures the side lengths of \(K\), which is the defining property of \(\sigma_K\).

\begin{theorem}

(Theorem 4.2.3 of \cite{schneider_2013}) For any Borel subset \(X\) of \(S^1\), the value \(\sigma_K(X)\) is the Hausdorff measure \(\mathcal{H}^1\) of the union \(\bigcup_{t \in X} e_K(t)\) of edges of \(K\).

\label{thm:surface-area-measure}
\end{theorem}

So if \(K\) is a convex polygon, then \(\sigma_K\) is a discrete measure such that the measure \(\sigma_K\left( t \right)\) at point \(t\) is the length of the edge \(e_K(t)\). For another example, assume that \(K\) is a smooth convex polygon where for every \(t \in S^1\), the tangent line \(l_K(t)\) always meets \(K\) at a single point \(v(t) = v_K^{\pm}(t)\) which is smooth in \(t \in S^1\). Then the distribution function \(R : S^1 \to \mathbb{R}_{\geq 0}\) of \(\sigma_K\) is the radius of curvature \(R(t) = \left\lVert v'(t) \right\rVert\) of \(\partial K\) at \(v(t)\).

\begin{proposition}

For any convex body \(K\) and \(t \in S^1\), \(\sigma_K\left( t \right)\) is the length of the edge \(e_K(t)\). Consequently, \(v_K^+(t) = v_K^-(t) + \sigma_K\left( t \right) v_t\).

\label{pro:surface-area-measure-side-length}
\end{proposition}

\begin{proof}
Let \(X = \left\{ t \right\}\) in \Cref{thm:surface-area-measure}.
\end{proof}

We add the following definition.

\begin{definition}

Let \(K\) be any convex body. For every \(a, b \in S^1\) such that \(b \neq a, a + \pi\), define \(v_K(a, b)\) as the intersection \(l_K(a) \cap l_K(b)\).

\label{def:convex-body-tangent-lines-intersection}
\end{definition}

We prove a technical lemma on the limit of vertices of \(K\).

\begin{theorem}

Let \(K\) be any convex body and \(t \in S^1\) be an arbitrary angle. We have the following right limits converging to \(v_K^+(t)\). In particular, the vertex \(v_K^+(t)\) is right-continuous on \(t \in S^1\).
\[
\lim_{ s \to t^+ } v_K^+(t) = \lim_{ s \to t^+ } v_K^-(u) = \lim_{ s \to t^+ } v_K(t, s) = v_K^+(t)
\]
Similarly, we have the following left limits.
\[
\lim_{ s \to t^- } v_K^+(s) = \lim_{ s \to t^- } v_K^-(s) = \lim_{ s \to t^- } v_K(s, t) = v_K^-(t)
\]

\label{thm:limits-converging-to-vertex}
\end{theorem}

\begin{proof}
We only compute the right limits. Left limits can be shown using a symmetric argument.

Let \(\epsilon > 0\) be arbitrary. Let \(p = v_K^+(t) + \epsilon v_t\). By the definition of \(v_K^+(t)\), the point \(p\) is not in \(K\). As \(\mathbb{R}^2 \setminus K\) is open, we can take some positive \(\epsilon' < \epsilon\) such that the closed line segment \(s\) connecting \(p\) and \(q = p - \epsilon' u_t\) is disjoint from \(K\) as well. Define the closed right-angled triangle \(T\) with vertices \(v_K^+(t)\), \(p\), and \(q\). Take the line \(l\) that passes through both \(q\) and \(v_K^+(t)\). Call the two closed half-planes divided by the line \(l\) as \(H_T\) and \(H'\), where \(H_T\) contains \(p\) and \(H'\) does not contain \(p\). By definition, \(H'\) has normal angle \(s \in (t, t + \pi/2)\).

We show that \(K \cap H_T \subseteq T\). Observe that the intersection \(X := H_K(t) \cap H_T\) is a cone centered at the point \(v_K^+(t)\), with the line segment \(s\) dividing \(X\) into triangle \(T\) and an unbounded convex set \(X \setminus T\). Now take any point \(r \in K \cap H_T\). If \(r \not\in T\), then since \(r \in X \setminus T\) and \(v_K^+(t) \in T\) the line connecting \(r \in K\) and \(v_K^+(t) \in K\) should pass through a point in \(s\). This, however, contradicts that \(s\) is disjoint from convex \(K\). So we should have \(r \in T\) and thus \(K \cap H_T \subseteq T\).

Now take arbitrary \(t_0 \in (t, s)\). We show that the edge \(e_K(t_0)\) should lie inside \(T\). It suffices to show that any point \(z\) in \(K\) that attains the maximum value of \(z \cdot u_{t_0}\) is in \(T\). Define the fan \(F := H_K(t) \cap H'\), so that \(F\) is bounded by lines \(l_K(t)\) and \(l\) with the vertex \(v_K^+(t)\). If \(z \in F\), it should be that \(z = v_K^+(t) \in T\), because \(v_K^+(t) \in K\) and \(v_K^+(t) \cdot u_{t_0} > z \cdot u_{t_0}\) for every point \(z\) in \(F\) other than \(z = v_K^+(t)\). If \(z \in K \setminus F\) on the other hand, we have \(K \setminus F = K \setminus H' \subseteq K \cap H_T \subseteq T\) so \(z \in T\). This completes the proof of \(e_K(t_0) \subseteq T\).

Observe that the triangle \(T\) contains \(v_K^+(t)\) and has diameter \(< 2\epsilon\) because the two perpendicular sides of \(T\) containing \(p\) have length \(\leq \epsilon\). So the endpoints \(v_K^+(u)\) and \(v_K^-(u)\) of the edge \(e_K(t_0) \subseteq T\) are distance at most \(2\epsilon\) away from \(v_K^+(t)\). This completes the epsilon-delta argument for \(\lim_{ s \to t^+ } v_K^+(s) = \lim_{ s \to t^+ } v_K^-(s) = v_K^+(t)\).

From \(e_K(t_0) \subseteq T\) and that the vertex \(p\) of \(T\) maximizes the value of \(z \cdot u_{t_0}\) over all \(z \in T\), we get that \(p\) is either on \(l_K(t_0)\) or outside the half-plane \(g_K(t_0)\). On the other hand we have \(v_K^+(t) \in g_K(t_0)\). So the line \(l_K(t_0)\) passes through the segment connecting \(p\) and \(v_K^+(t)\), and the intersection \(v_K(t, t_0) = l_K(t) \cap l_K(t_0)\) is inside \(T\). This with that the diameter of \(T\) is less than \(2 \epsilon\) proves \(\lim_{ s \to t^+ } v_K(t, s) = v_K^+(t)\).
\end{proof}

\section{Supporting Hallway}
\label{sec:supporting-hallway}
In this \Cref{sec:supporting-hallway}, we define the \emph{supporting hallways} \(L_S(t)\) of a moving sofa \(S\). We first name the parts of the hallway \(L\).

\begin{figure}
\centering
\includegraphics{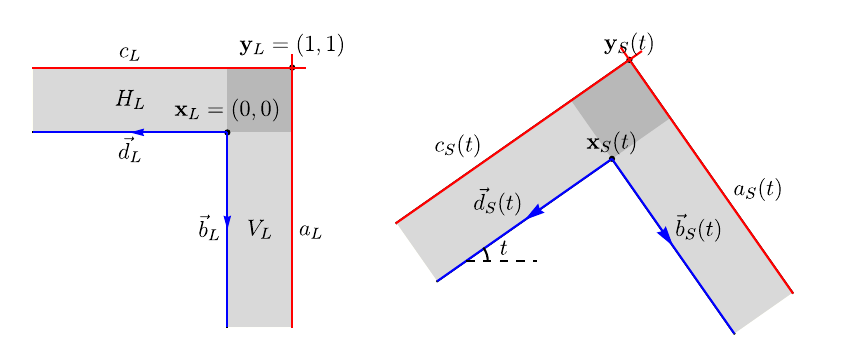}
\caption{The hallway \(L\) (\Cref{def:hallway}, on left), supporting hallway \(L_S(t)\) (\Cref{def:tangent-hallway}, on right), and their corresponding parts (\Cref{def:hallway-parts}, \Cref{def:rotating-hallway-parts}).}
\label{fig:hallway-detailed}
\end{figure}

\begin{definition}

Let \(\mathbf{x}_L = (0, 0)\) and \(\mathbf{y}_L = (1, 1)\) be the inner and outer corner of \(L\) respectively.

Let \(a_L\) and \(c_L\) be the lines \(x=1\) and \(y=1\) representing the \emph{outer walls} of \(L\) passing through \(\mathbf{y}_L\). Let \(\vec{b}_L\) and \(\vec{d}_L\) be the half-lines \(\left\{ 0 \right\} \times (-\infty, 0]\) and \((-\infty, 0] \times \left\{ 0 \right\}\) from the inner corner \(\mathbf{x}_L\) representing the \emph{inner walls} of \(L\). Let \(b_L\) and \(d_L\) be the lines \(x=0\) and \(y=0\) extending \(\vec{b}_L\) and \(\vec{d}_L\) respectively.

Let \(Q_L^+ = (-\infty, 1]^2\) be the closed quarter-plane bounded by outer walls \(a_L\) and \(c_L\). Let \(Q_L^- = (-\infty, 0)^2\) be the open quarter-plane bounded by inner walls \(\vec{b}_L\) and \(\vec{d}_L\), so that \(L = Q_L^+ \setminus Q_L^-\).

\label{def:hallway-parts}
\end{definition}

Now we define the supporting hallways on any nonempty and compact subset of \(\mathbb{R}^2\).

\begin{definition}

For any nonempty and compact \(S \subset \mathbb{R}^2\) and angle \(t \in S^1\), define the rigid transformation \(f_{S, t}\) of \(\mathbb{R}^2\) as
\[
f_{S, t}(p) := R_t(p) + (h_S(t) - 1)  u_t + (h_S(t + \pi/2) - 1) v_t
\]
and define the \emph{supporting hallway} \(L_S(t) := f_{S, t}(L)\) of \(S\) with angle \(t\).

\label{def:tangent-hallway}
\end{definition}

The following \Cref{pro:tangent-hallway} is the defining property of \(L_S(t)\).

\begin{proposition}

For any shape \(S\) and angle \(t \in S^1\), the supporting hallway \(L_S(t)\) is the unique translation of \(R_t(L)\) such that the outer walls of \(L_S(t)\) corresponding to the outer walls \(a\) and \(c\) of \(L\) are the tangent lines \(l_S(t)\) and \(l_S(t + \pi/2)\) of \(S\) respectively.

\label{pro:tangent-hallway}
\end{proposition}

\begin{proof}
Let \(c_1\) and \(c_2\) be arbitrary real values. Then \(L' = R_t(L) + c_1 u_t + c_2 v_t\) is an arbitrary rigid transformation of \(L\) rotated counterclockwise by \(t\). The outer walls of \(L'\) corresponding to the outer walls \(a\) and \(c\) of \(L\) (\Cref{def:hallway-parts}) are \(l(t, c_1 + 1)\) and \(l(t + \pi/2, c_2 + 1)\) respectively. They match with the tangent lines \(l_S(t) = l(t, h_S(t))\) and \(l_S(t + \pi/2) = l(t + \pi/2, h_S(t + \pi/2))\) of \(S\) if and only if \(c_1 = h_S(t) - 1\) and \(c_2 = h_S(t + \pi/2) - 1\). That is, if and only if \(L' = L_S(t)\).
\end{proof}

Name the parts of supporting hallway \(L_S(t)\) corresponding to the parts of \(L\) (\Cref{def:hallway-parts}) under the transformation \(f_{S, t}\) for future use.

\begin{definition}

For any shape \(S\) and angle \(t \in S^1\), let \(\mathbf{x}_S(t), \mathbf{y}_S(t), a_S(t)\), \(b_S(t), c_S(t)\), \(d_S(t)\), \(\vec{b}_S(t)\), \(\vec{d}_S(t)\), \(Q^+_S(t)\), \(Q^-_S(t)\) be the parts of \(L_S(t)\) corresponding to the parts \(\mathbf{x}_L\), \(\mathbf{y}_L\), \(a_L\), \(b_L\), \(c_L\), \(d_L\), \(\vec{b}_L\), \(\vec{d}_L\), \(Q^+_L\), \(Q^-_L\) of \(L\) respectively with \(f_{S, t}\). That is, for any \(? = \mathbf{x}\), \(\mathbf{y}\), \(a\), \(b\), \(c\), \(d\), \(\vec{b}\), \(\vec{d}\), \(Q^+\), \(Q^-\), let \(?_S(t) := f_{S, t}(?_L)\).

\label{def:rotating-hallway-parts}
\end{definition}

\begin{proposition}

We have \(L_S(t) = Q_S^+(t) \setminus Q_S^-(t)\). Also we can express the parts of \(L_S(t)\) purely in terms of the supporting lines, half-planes, and support function \(h_S\) of \(S\) as below.
\begin{gather*}
\mathbf{x}_S(t) = (h_S(t) - 1) u_t + (h_S(t + \pi/2) - 1) v_t \\
\mathbf{y}_S(t) = h_S(t) u_t + h_S(t + \pi/2) v_t \\
a_S(t) = l_S(t) = l(t, h_S(t)) \qquad
b_S(t) = l(t, h_S(t) - 1) \\
c_S(t) = l_S(t + \pi/2) = l(t + \pi/2, h_S(t + \pi/2)) \\ 
d_S(t) = l(t + \pi/2, h_S(t + \pi/2) - 1) \\
Q_S^+(t) = H_S(t) \cap H_S(t + \pi/2) = H_-(t, h_S(t)) \cap H_-(t + \pi/2, h_S(t + \pi/2)) \\
Q_S^-(t) = H_-^{\circ}(t, h_S(t) - 1) \cap H_-^{\circ}(t + \pi/2, h_S(t + \pi/2))
\end{gather*}

\label{pro:rotating-hallway-parts}
\end{proposition}

\begin{proof}
Follows from \Cref{def:tangent-hallway} and \Cref{pro:tangent-hallway}.
\end{proof}

With the following \Cref{pro:tangent-hallway-contains}, we can assume that the rotating hallways \(L_t\) in \Cref{pro:moving-sofa-common-subset} containing a moving sofa \(S\) are the supporting hallways \(L_t = L_S(t)\) of \(S\).

\begin{proposition}

Let \(S\) be any nonempty compact set contained in a translation of \(R_t(L)\) with angle \(t \in S^1\). Then the supporting hallway \(L_S(t)\) with angle \(t\) also contains \(S\).

\label{pro:tangent-hallway-contains}
\end{proposition}

\begin{proof}
Assume that a translation \(L'\) of \(R_t(L)\) contains \(S\). Then \(L' = f(L)\) for some rigid transformation \(f\) with counterclockwise rotation \(t\). Define \(Q'^+ = f(Q^+_L)\) and \(Q'^- = f(Q^-_L)\) as the quarter-planes of \(L'\) corresponding to that of \(L\). Then \(L' = Q'^+ \setminus Q'^-\) and \(Q'^+\) is a convex cone containing \(S\) with boundaries of normal angles \(t\) and \(t + \pi/2\). By \Cref{pro:rotating-hallway-parts}, \(Q_S^+(t)\) is the intersection of two supporting half-planes of \(S\) with normal angles \(t\) and \(t + \pi/2\). So we should have \(Q_S^+(t) \subseteq Q'^+\). Shifting this by \(-u_t - v_t\), we get \(Q_S^-(t) \subseteq Q'^-\). Now \(S \subset L'\) is disjoint from \(Q_S^-(t) \subseteq Q'^-\), and we have \(S \subseteq Q_S^+(t) \setminus Q_S^-(t) = L_S(t)\).
\end{proof}

\section{Monotone Sofa}
\label{sec:monotone-sofa}
We now define the notion of monotone sofas. We first prepare basic definitions.

\begin{definition}

Let \(R_t : \mathbb{R}^2 \to \mathbb{R}^2\) denote the rotation of \(\mathbb{R}^2\) around the origin by the counterclockwise angle of \(t \in \mathbb{R}\).

\label{def:rotation-map}
\end{definition}

\begin{definition}

Define the horizontal strip \(H := \mathbb{R} \times [0, 1]\), vertical strip \(V := [0, 1] \times \mathbb{R}\), and its rotation \(V_\omega := R_\omega(V)\) around the origin by a counterclockwise angle \(\omega \in \mathbb{R}\).

\label{def:strips}
\end{definition}

\begin{definition}

The \emph{rotation angle} \(\omega\) of a moving sofa \(S\) is the clockwise angle that it rotates as it moves from \(H_L\) to \(V_L\) inside \(L\).

\label{def:rotation-angle}
\end{definition}

\begin{definition}

A moving sofa \(S\) with rotation angle \(\omega \in (0, \pi/2]\) is in \emph{standard position} if \(h_S(\omega) = h_S(\pi/2) = 1\).

\label{def:standard-position}
\end{definition}

We will assume without loss of generality that a moving sofa \(S\) can be put in standard position after some translation.

\begin{definition}

For any \(\omega \in (0, \pi/2]\), define the \emph{parallelogram} \(P_\omega := H \cap V_\omega\) with rotation angle \(\omega\). Let \(O := (0, 0)\) and \(o_\omega := (\tan(\pi/4 - \omega/2), 1)\) represent the lower left and upper right vertices of \(P_\omega\) respectively.

\label{def:parallelogram}
\end{definition}

\begin{proposition}

For any moving sofa \(S\) with rotation angle \(\omega \in (0, \pi/2]\), there is a translation of \(S\) in standard position which is (i) unique if \(\omega < \pi/2\), or (ii) unique up to horizontal translations if \(\omega = \pi/2\). After such a translation, we have \(S \subseteq P_\omega\).

\label{pro:standard-position-shape}
\end{proposition}

\begin{proof}
If \(\omega < \pi/2\), there is a unique translation of \(S\) making contact with the two supporting lines \(l(\omega, 1)\) and \(l(\pi/2, 1)\) from below. If \(\omega = \pi/2\), there is a unique translation of \(S\) making contact with the supporting line \(l(\pi/2, 1)\) from below, up to horizontal translations. We have \(S \subseteq H\) and \(S \subseteq V_\omega\) by the proof of \Cref{pro:moving-sofa-common-subset}.
\end{proof}

Define the intersection \(\mathcal{I}\) in \Cref{eqn:sofa-intersection} of \Cref{sec:_monotone-sofas-and-caps} with the supporting hallways \(L_t := L_S(t)\).

\begin{definition}

Let \(S\) be any moving sofa with rotation angle \(\omega \in (0, \pi/2]\) in standard position. Define the intersection
\[
\mathcal{I}(S) = P_\omega \cap \bigcap_{t \in [0, \omega]} L_S(t).
\]

\label{def:monotonization}
\end{definition}

We will establish that \(\mathcal{I}(S)\) is a moving sofa containing \(S\).

\begin{theorem}

For any moving sofa \(S\) with rotation angle \(\omega \in (0, \pi/2]\) in standard position, \(\mathcal{I}(S)\) is a moving sofa with the same rotation angle \(\omega\) in standard position containing \(S\).

\label{thm:monotonization}
\end{theorem}

With \Cref{thm:monotonization}, we will call any sofa of form \(\mathcal{I}(S)\) a \emph{monotone sofa}, and it suffices to consider monotone sofas for the moving sofa problem.

\begin{definition}

Define a \emph{monotone sofa} as the intersection \(\mathcal{I}(S)\) of some moving sofa \(S\) in standard position.

\label{def:monotone-sofa}
\end{definition}

\subsubsection{\texorpdfstring{Proof of \Cref{thm:monotonization}}{Proof of }}

\begin{proposition}

For any moving sofa \(S\) in standard position, \(S \subseteq \mathcal{I}(S)\).

\label{pro:monotonization-contains-sofa}
\end{proposition}

\begin{proof}
Assume rotation angle \(\omega \in (0, \pi/2]\). We have \(S \subseteq P_\omega\) by \Cref{pro:standard-position-shape}. For all \(t \in [0, \omega]\), we have \(S \subset L_t\) for some hallway \(L_t\) rotated counterclockwise by \(t\) from (2) of \Cref{pro:moving-sofa-common-subset}, and then \(S \subseteq L_S(t)\) by \Cref{pro:tangent-hallway-contains}.
\end{proof}

We prepare the following terminologies.

\begin{definition}

Say that a set \(X \subseteq \mathbb{R}^2\) is \emph{closed in the direction of} vector \(v \in \mathbb{R}^2\) if, for any \(x \in X\) and \(\lambda \geq 0\), we have \(x + \lambda v \in X\).

\label{def:closed-in-direction}
\end{definition}

\begin{definition}

Any line \(l\) of \(\mathbb{R}^2\) divides the plane into two half-planes. Assuming \(l\) is not parallel to the \(y\)-axis, call the \emph{left side} (resp. \emph{right side}) of \(l\) as the closed half-plane with boundary \(l\) containing the point \(- Nu_0\) (resp. \(Nu_0\)) for sufficiently large \(N\).

\label{def:line-half-plane-directions}
\end{definition}

\begin{definition}

Let \(S\) be any moving sofa with rotation angle \(\omega \in (0, \pi/2]\) in standard position. Define the convex set
\[
\mathcal{C}(S) := P_\omega \cap \bigcap_{0 \leq t \leq \omega} Q^+_S(t).
\]

\label{def:cap-sofa}
\end{definition}

We prepare useful lemmas on \(\mathcal{C}(S)\).

\begin{proposition}

For any moving sofa \(S\) in standard position, \(S \subseteq \mathcal{I}(S) \subseteq \mathcal{C}(S)\).

\label{pro:cap-contains-sofa}
\end{proposition}

\begin{proof}
By \Cref{pro:monotonization-contains-sofa} and \(L_S(t) \subset Q_S^+(t)\) for all \(t \in [0, \omega]\).
\end{proof}

\begin{definition}

For any \(\omega \in [0, \pi/2]\), define the set \(J_\omega := [0, \omega] \cup [\pi/2, \omega + \pi/2]\).

\label{def:set-j}
\end{definition}

\begin{lemma}

Let \(S\) be any moving sofa with rotation angle \(\omega \in [0, \pi/2]\) in standard position. Then the support functions \(h_X\) of the sets \(X = S, \mathcal{I}(S), \mathcal{C}(S)\) are the same on the set \(J_\omega\). Consequently, for any \(t \in [0, \omega]\), the supporting hallways \(L_X(t)\) on the sets \(X = S, \mathcal{I}(S), \mathcal{C}(S)\) are the same.

\label{lem:cap-same-support-function}
\end{lemma}

\begin{proof}
We have \(S \subseteq \mathcal{I}(S) \subseteq \mathcal{C}(S)\) by \Cref{pro:cap-contains-sofa}. So it remains to show \(h_{\mathcal{C}(S)}(t) \leq h_S(t)\) for every \(t\) in \(J_\omega\) to show that \(h_X\)s on \(J_\omega\) are the same. This follows from \Cref{def:cap-sofa} as \(\mathcal{C}(S) \subseteq H_S(t)\) for any \(t \in J_\omega\). To show that the supporting hallways \(L_X(t)\) are the same, observe that \(L_X(t)\) depends solely on the values \(t, h_X(t)\), and \(h_X(t + \pi/2)\) by its \Cref{def:tangent-hallway}.
\end{proof}

We establish the connectedness of \(\mathcal{I}(S)\) which is the hardest part.

\begin{theorem}

For any moving sofa \(S\) with rotation angle \(\omega \in (0, \pi/2]\) in standard position, \(\mathcal{I}(S)\) is connected.

\label{thm:monotonization-is-connected}
\end{theorem}

\begin{proof}
Fix an arbitrary point \(p\) in \(\mathcal{I}(S)\). It suffices to show that \(\mathcal{I}(S)\) is connected by finding a line segment \(s_\theta\) inside \(\mathcal{I}(S)\) that connects \(p\) to the connected subset \(S\) of \(\mathcal{I}(S)\). Here \(\theta \in [\omega, \pi/2]\) is a value that will be fixed later. Letting \(\theta \in [\omega, \pi/2]\) arbitrary as of now, define the line \(l_\theta\) passing through \(p\) in the direction of \(u_\theta\) and let \(s_\theta := l_\theta \cap \mathcal{I}(S)\). Then \(s_\theta\) is a subset of \(\mathcal{I}(S)\) containing \(p\).

Our goal now is to show that \(s_\theta\) is a line segment that overlaps with \(S\) for some \(\theta \in [\omega, \pi/2]\). We first show that \(s_\theta\) is a nonempty line segment. Define the set \(X := \bigcup_{0 \leq t \leq \omega} Q^-_S(t)\). By plugging \(L_S(t) = Q_S^+(t) \setminus Q_S^-(t)\) in \Cref{pro:rotating-hallway-parts} to \Cref{def:monotonization}, we have \(\mathcal{I}(S) = \mathcal{C}(S) \setminus X\). The set \(\mathcal{C}(S)\) is a convex body containing \(S\) by \Cref{pro:standard-position-shape}, and the set \(X\) is closed in the direction of \(-u_\theta\) (\Cref{def:closed-in-direction}) since each \(Q_S^-(t)\) is. Now \(s_\theta\) is a line segment because it is the line segment \(l_\theta \cap \mathcal{C}(S)\) subtracted by the half-line \(l_\theta \setminus X\). Our goal now is to find some \(\theta \in [\omega, \pi/2]\) such that \(l_\theta\) meets \(S\), so that \(s_\theta\) connects \(p\) to \(S\) inside \(\mathcal{I}(S)\).

Assume by contradiction that for every \(\theta \in [\omega, \pi/2]\) the line \(l_\theta\) is disjoint from \(S\). Because the line \(l_\theta\) is disjoint from \(S\) for any \(\theta \in [\omega, \pi/2]\), the set \(S\) is inside the set \(Y = \mathbb{R}^2 \setminus \bigcup_{\theta \in [\omega, \pi/2]} l_\theta\). Note that \(Y\) has exactly two connected components \(Y_L\) and \(Y_R\) on the left and the right side of the lines \(l_\theta\) respectively. We will find a point at each \(S \cap Y_L\) and \(S \cap Y_R\), reaching the contradiction as \(S\) is connected.

By \Cref{lem:cap-same-support-function}, we have \(l_{\mathcal{I}(S)}(t) = l_S(t)\) for every \(t \in J_\omega = [0, \omega] \cup [\pi/2, \omega + \pi/2]\). Because \(p \in \mathcal{I}(S)\), the line \(l_{\pi/2}\) passing through \(p\) is on the left side of \(l_{S}(0)\). So any point of \(e_S(0)\) is on the right side of \(l_{\pi/2}\), and should be in \(S \cap Y_R\). Likewise, as \(p \in \mathcal{I}(S)\), the line \(l_{\omega}\) passing through \(p\) is on the right side of \(l_S(\omega + \pi/2)\). So any point of \(e_S(\omega + \pi/2)\) is on the left side of \(l_\omega\), and should be in \(S \cap Y_L\). This establishes the contradiction we wanted, and we finally prove that \(\mathcal{I}(S)\) is connected.
\end{proof}

We are now ready to prove \Cref{thm:monotonization}.

\begin{proof}[Proof of \Cref{thm:monotonization}]
We first show that \(S' := \mathcal{I}(S)\) is a moving sofa. As \(S'\) is connected by \Cref{thm:monotonization-is-connected}, it suffices to show that \(S'\) can move continuously inside \(L\) from \(H_L\) to \(V_L\) with rotation angle \(\omega\). See the movement of \(L_S(t) = L_{S'}(t)\) containing \(S'\) for \(t \in [0, \omega]\) (\Cref{lem:cap-same-support-function}) in perspective of the hallway, to find a movement of \(S'\) inside \(L\). In particular, \(S' \subseteq H\) so the horizontal side of \(L_{S'}(0)\) (corresponding to \(H_L\) of \(L\)) contains \(S'\). Also, \(S' \subseteq V_\omega\) so the vertical side of \(L_{S'}(\omega)\) (corresponding to \(V_L\) of \(L\)) contains \(S'\). Since \(h_{S'}\) is continuous, this movement of \(S'\) is also continuous.

Because \(L_{S'}(\omega)\) is rotated counterclockwise by \(\omega\), the sofa \(S'\) have rotation angle \(\omega\). By \Cref{pro:monotonization-contains-sofa}, the sofa \(S'\) contains \(S\). The sofa \(S'\) is in standard position because \(S \subseteq S' \subseteq H \cap V_\omega\) and \(S\) is in standard position.
\end{proof}

\section{Cap and Niche}
\label{sec:cap-and-niche}
We show that any intersection \(\mathcal{I}(S)\) of a moving sofa \(S\) is equal to the \emph{cap} \(K\) of \(S\) minus the \emph{niche} \(\mathcal{N}(K)\) of \(K\) (\Cref{thm:monotonization-structure}). We first define the notion of \emph{cap} as a kind of convex body.

\begin{definition}

A \emph{cap} \(K\) with \emph{rotation angle} \(\omega \in (0, \pi/2]\) is a convex body such that the followings hold.

\begin{enumerate}
\def\labelenumi{\arabic{enumi}.}
\tightlist
\item
  \(h_K(\omega) = h_K(\pi/2) = 1\) and \(h_K(\omega + \pi) = h_K(3\pi/2) = 0\).
\item
  \(K\) is an intersection of closed half-planes with normal angles in \(J_\omega \cup \{\omega + \pi, 3\pi/2\}\).
\end{enumerate}

\label{def:cap}
\end{definition}

\begin{definition}

Define the \emph{space of caps} \(\mathcal{K}_\omega^\mathrm{c}\) with the \emph{rotation angle} \(\omega \in (0, \pi/2]\) as the collection of all caps \(K\) with rotation angle \(\omega\).

\label{def:cap-space}
\end{definition}

We now show that the convex set \(\mathcal{C}(S)\) from a moving sofa \(S\) (\Cref{def:cap-sofa}) is a cap, justifying calling \(\mathcal{C}(S)\) \emph{the cap of} \(S\).

\begin{theorem}

For any moving sofa \(S\) with rotation angle \(\omega \in (0, \pi/2]\) in standard position, the set \(\mathcal{C}(S)\) in \Cref{def:cap-sofa} is a cap with rotation angle \(\omega\) as in \Cref{def:cap}.

\label{thm:cap-hallway-intersection}
\end{theorem}

\begin{proof}
The second condition of \Cref{def:cap} on \(\mathcal{C}(S)\) is satisfied by \Cref{def:cap-sofa}. Since \(S\) is in standard position, it suffices to check \(h_{\mathcal{C}(S)}(\omega + \pi) = h_{\mathcal{C}(S)}(3\pi/2) = 0\) in the first condition of \Cref{def:cap}.

We first prove the case \(\omega = \pi/2\). Since \(h_S(\pi/2) = 1\) and \(S \subseteq \mathcal{C}(S)\) by \Cref{pro:cap-contains-sofa}, we can take a point \(q \in S \cap l(\pi/2, 1) \subseteq \mathcal{C}(S)\). For any \(t \in [0, \pi/2]\), each \(Q_S^+(t)\) is closed in the direction of \(-v_0\) (\Cref{def:closed-in-direction}). This with \(q \in \mathcal{C}(S)\) implies that \(q - v_0 \in \mathcal{C}(S)\). Note that \(q - v_0 \in l(\pi/2, 0)\), so \(q - v_0 \in \mathcal{C}(S) \subseteq P_{\pi/2}\) implies \(h_{\mathcal{C}(S)}(3\pi/2) = 0\) as we desired.

We now prove the case \(\omega < \pi/2\). Since \(h_S(\omega) = h_S(\pi/2) = 1\), we can take two points \(q_\omega \in S \cap l(\omega, 1)\) and \(q_{\pi/2} \in S \cap l(\pi/2, 1)\). Observe that the three points \(q_\omega, o_\omega, q_{\pi/2}\) are in monotonically decreasing order of \(x\)-coordinates and form an angle of \(\omega + \pi/2\). Take any supporting hallway \(L_S(t)\) with angle \(t \in [0, \omega]\). Then \(Q_S^+(t)\) should contain both \(q_\omega, q_{\pi/2} \in S\), so we also have \(Q_S^+(t) \ni o_\omega\). This implies \(o_\omega \in \mathcal{C}(S)\). For any \(t \in [0, \omega]\), each \(Q_S^+(t)\) is closed in the directions \(v_0\) and \(u_\omega\) (\Cref{def:closed-in-direction}). So \(o_\omega \in \mathcal{C}(S)\) implies \(o_\omega - v_0, o_\omega - u_\omega \in \mathcal{C}(S)\). This with \(\mathcal{C}(S) \subseteq H \cap V_\omega\) implies \(h_{\mathcal{C}(S)}(\omega + \pi) = h_{\mathcal{C}(S)}(3\pi/2) = 0\) as we desired.
\end{proof}

\begin{definition}

With \Cref{thm:cap-hallway-intersection}, call \(\mathcal{C}(S)\) the \emph{cap of the moving sofa} \(S\).

\label{def:cap-real-definition}
\end{definition}

We now define the niche of a cap. Note that the following fan \(F_\omega\) contains \(H \cap V_\omega\) in particular.

\begin{definition}

For any angle \(\omega \in [0, \pi/2]\), define the \emph{fan} \(F_\omega := H_+(\omega, 0) \cap H_+(\pi/2, 0)\).

\label{def:fan}
\end{definition}

\begin{definition}

Let \(K \in \mathcal{K}_\omega^\mathrm{c}\) be arbitrary. Define the \emph{niche} of \(K\) as
\[
\mathcal{N}(K) := F_{\omega} \cap \bigcup_{t \in (0, \omega)} Q^-_K(t).
\]

\label{def:niche}
\end{definition}

\begin{theorem}

Let \(S\) be a moving sofa with rotation angle \(\omega \in (0, \pi/2]\) in standard position. The monotone sofa \(\mathcal{I}(S)\) from \(S\) is equal to \(K \setminus \mathcal{N}(K)\) where \(K := \mathcal{C}(S)\) is the cap of \(S\).

\label{thm:monotonization-structure}
\end{theorem}

\begin{proof}
By writing each \(L_S(t)\) as \(Q_S^+(t) \setminus Q_S^-(t)\), the set \(\mathcal{I}(S)\) can be represented as follows.
\begin{equation}
\label{eqn:monotonization}
\begin{split}
\mathcal{I}(S) & = P_\omega \cap \bigcap_{t \in [0, \omega]} L_S(t) \\
& = \left( P_\omega \cap \bigcap_{t \in [0, \omega]} Q^+_S(t) \right) \setminus \left( F_\omega \cap \bigcup_{t \in [0, \omega]} Q^-_S(t) \right)
\end{split}
\end{equation}
By \Cref{lem:cap-same-support-function} we have \(Q_S^-(t) = Q_K^-(t)\). So we have \(\mathcal{I}(S) = K \setminus \mathcal{N}(K)\) by the definitions of \(K\) and \(\mathcal{N}(K)\).
\end{proof}

A monotone sofa is its cap subtracted by its niche. Note that unlike \Cref{thm:monotonization-structure} above, the following \Cref{thm:monotone-sofa-structure} does not depend on another moving sofa.

\begin{theorem}

For any monotone sofa \(S\) with cap \(K := \mathcal{C}(S)\), we have \(S = K \setminus \mathcal{N}(K)\).

\label{thm:monotone-sofa-structure}
\end{theorem}

\begin{proof}
Let \(S = \mathcal{I}(S')\) for some moving sofa \(S'\) in standard position. By \Cref{lem:cap-same-support-function} and that \Cref{def:cap-sofa} depends solely on the values of support function on \(J_\omega\), the cap \(K = \mathcal{C}(S)\) of \(S\) is also the cap \(\mathcal{C}(S')\) of \(S'\). So by \Cref{thm:monotonization-structure} we have \(S = \mathcal{I}(S') = K \setminus \mathcal{N}(K)\).
\end{proof}

Taking the intersection \(\mathcal{I}(-)\) enlarges any moving sofa to monotone sofas and fixes monotone sofas.

\begin{theorem}

For any moving sofa \(S'\) in standard position, we have \(\mathcal{I}(\mathcal{I}(S')) = \mathcal{I}(S')\). Consequently, the equality \(S = \mathcal{I}(S)\) holds if and only if \(S\) is a monotone sofa.

\label{thm:monotonization-idempotent}
\end{theorem}

\begin{proof}
Let \(S := \mathcal{I}(S')\) so that \(S\) is a montone sofa with cap \(K\). Then \(\mathcal{I}(S) = K \setminus \mathcal{N}(K) = S\) by \Cref{thm:monotonization-structure} and \Cref{thm:monotone-sofa-structure}. So \(\mathcal{I}(\mathcal{I}(S')) = \mathcal{I}(S')\) and the equality \(S = \mathcal{I}(S)\) holds for monotone sofa \(S\). On the other hand, any moving sofa \(S\) with equality \(S = \mathcal{I}(S)\) is immediately a monotone sofa.
\end{proof}

\section{Cap Contains Niche}
\label{sec:cap-contains-niche}
We now define the parts of an arbitrary cap \(K\).

\begin{definition}

Let \(K \in \mathcal{K}_\omega^\mathrm{c}\) be arbitrary. For any \(t \in [0, \omega]\), define the \emph{vertices} \(A^+_K(t) = v^+_K(t)\), \(A^-_K(t) = v^-_K(t)\), \(C^+_K(t) = v^+_K(t + \pi/2)\), and \(C^-_K(t) = v^-_K(t + \pi/2)\) of \(K\).

\label{def:cap-vertices}
\end{definition}

Note that the outer wall \(a_K(t)\) (resp. \(c_K(t)\)) of \(L_K(t)\) is in contact with the cap \(K\) at the vertices \(A_K^+(t)\) and \(A_K^-(t)\) (resp. \(C_K^+(t)\) and \(C_K^-(t)\)) respectively. We also define the \emph{upper boundary} of a cap \(K\).

\begin{definition}

Let \(K \in \mathcal{K}_\omega^\mathrm{c}\) be arbitrary. Define the \emph{upper boundary} \(\delta K\) of \(K\) as the set \(\delta K = \bigcup_{t \in [0, \omega + \pi/2]} e_K(t)\).

\label{def:upper-boundary-of-cap}
\end{definition}

For any cap \(K\) with rotation angle \(\omega\), the upper boundary \(\delta K\) is exactly the points of \(K\) making contact with the outer walls \(a_K(t)\) and \(c_K(t)\) of supporting hallways \(L_K(t)\) for every \(t \in [0, \omega]\). In particular, upper boundary \(\delta K\) is a curve from the right endpoint \(A_K^-(0)\) to the left endpoint \(C_K^+(\omega)\). We collect some observations on \(\delta K\).

\begin{proposition}

Let \(K \in \mathcal{K}_\omega^\mathrm{c}\) be arbitrary. The set \(\delta K\) is the boundary of \(K\) in the subset topology of \(F_\omega\).

\label{pro:upper-boundary-interior}
\end{proposition}

\begin{proof}
Let \(X\) be the boundary of \(K\) in the subset topology of \(F_\omega\). We first show \(\delta K \subseteq X\). Take any point \(z\) of \(\delta K\). Then \(z \in e_K(t)\) for some \(t \in [0, \omega + \pi/2]\). Since \(K\) is a planar convex body, for any \(\epsilon > 0\) the point \(z' = z + \epsilon u_t\) is not in \(K\). This with \(z \in K\) implies \(z \in X\). We now show \(X \subseteq \delta K\). Assume by contrary that there is a point \(z \in X \setminus \delta K\). Then for every \(t \in [0, \omega + \pi/2]\) we have \(z \not \in e_K(t)\) so that \(z\) is in the interior of \(H_K(t)\). So by compactness of \([0, \omega + \pi/2]\), an open ball \(U\) of radius \(\epsilon\) centered at \(z\) is contained in the half-space \(H_K(t)\) for all \(t \in [0, \omega + \pi/2]\). Now \(U \cap F_\omega \subseteq K\) and so \(z \not\in X\), leading to contradiction.
\end{proof}

\begin{proposition}

For any cap \(K\), its upper boundary \(\delta K\) is connected.

\label{pro:upper-boundary-connected}
\end{proposition}

\begin{proof}
Let \(\omega \in (0, \pi/2]\) be the rotation angle of \(K\). Let \(I := [0, \omega + \pi/2]\). Assume contradictory that \(\delta K = \bigcup_{t \in I} e_K(t)\) is disconnected. That is, there are open subsets \(U, V\) of \(\mathbb{R}^2\) such that \(\delta K\) is the disjoint union of nonempty subsets \(\delta K \cap U\) and \(\delta K \cap V\). For any \(t \in I\), each \(e_K(t) \subset \delta K\) is connected so it should be contained in exactly one of \(U\) or \(V\). Now define \(I_U := \left\{ t \in I : e_K(t) \subseteq U \right\}\) and \(I_V := \left\{ t \in I : e_K(t) \subseteq V \right\}\). Then \(I\) is a disjoint union of \(I_U\) and \(I_V\). By \Cref{thm:limits-converging-to-vertex} and that each \(e_K(t)\) is the line segment connecting \(v_K^-(t)\) to \(v_K^+(t)\), the sets \(I_U\) and \(I_V\) are open in the subspace topology of \(I\). Since \(I\) is connected, either \(I_U = I\) or \(I_V = I\), and they imply either \(\delta K \subset U\) or \(\delta K \subset V\), leading to contradiction.
\end{proof}

We also define some notions on the niche \(\mathcal{N}(K)\) of a cap \(K\).

\begin{definition}

For any \(K \in \mathcal{K}_\omega^\mathrm{c}\) and \(t \in (0, \omega)\), define the \emph{wedge} \(T_K(t) := F_\omega \cap Q^-_K(t)\) of \(K\) with angle \(t\).

\label{def:wedge}
\end{definition}

\begin{proposition}

For any \(K \in \mathcal{K}_\omega^\mathrm{c}\), we have \(\mathcal{N}(K) = \cup_{t \in (0, \omega)} T_K(t)\).

\label{pro:wedge}
\end{proposition}

\begin{proof}
Immediate from \Cref{def:niche}.
\end{proof}

The following \Cref{def:wedge-endpoints} defines the left and right endpoints of the wedge \(T_K(t)\).

\begin{definition}

For any \(K \in \mathcal{K}_\omega^\mathrm{c}\) and \(t \in (0, \omega)\), define \(W_K(t)\) as the intersection of lines \(b_K(t)\) and \(l(\pi/2, 0)\), and define \(Z_K(t)\) as the intersection of lines \(d_K(t)\) and \(l(\omega, 0)\).

\label{def:wedge-endpoints}
\end{definition}

Note that if the wedge \(T_K(t)\) contains the origin \(O\), then \(T_K(t)\) is a quadrilateral with vertices \(O, W_K(t), Z_K(t)\), and \(\mathbf{x}_K(t)\), and the points \(W_K(t)\) and \(Z_K(t)\) are the leftmost and rightmost point of \(\overline{T_K(t)}\) respectively.

\begin{definition}

For any cap \(K\) with rotation angle \(\omega\) and \(t \in (0, \omega)\), define the \emph{right wedge gap} \(w_K(t) := (A_K^-(0) - W_K(t)) \cdot u_0\) with angle \(t\), which is the signed distance from \(W_K(t)\) to \(A_K^-(0)\) along the line \(l(\pi/2, 0)\) in the direction of \(u_0\). Likewise, define the \emph{left wedge gap} \(z_K(t) = (C_K^+(\omega) - Z_K(t)) \cdot v_\omega\) with angle \(t\), which is the signed length from \(Z_K(t)\) to \(C_K^+(\omega)\) along the line \(l(\omega, 0)\) in the direction of \(v_\omega\).

\label{def:wedge-side-lengths}
\end{definition}

We introduce the notion of reflecting a cap \(K\) that will reduce symmetric arguments without loss of generality.

\begin{definition}

Let \(\omega \in (0, \pi/2]\) be arbitrary. Define \(M_\omega : \mathbb{R}^2 \to \mathbb{R}^2\) as the reflection along the line passing through \(O\) and \(o_\omega\). For any cap \(K\) with rotation angle \(\omega\), define the \emph{mirror reflection} \(K^\mathrm{m} := M_\omega(K)\) of \(K\).

\label{def:mirror-reflection}
\end{definition}

Many definitions on \(K\) are symmetric along the reflection \(M_\omega\).

\begin{proposition}

Let \(K \in \mathcal{K}_\omega^\mathrm{c}\) be arbitrary. The parts of supporting hallway \(L_K(t)\), cap \(K\), and niche \(\mathcal{N}(K)\) are equivariant under \(M_\omega\). That is, for any \(t \in [0, \omega]\):

\begin{itemize}
\tightlist
\item
  \(?_{K^{\mathrm{m}}}(t) = M_\omega(?_K(\omega - t))\) for \(? = L, \mathbf{x}, \mathbf{y}, a, b, c, d, W, Z\).
\item
  \(A^{\pm}_{K^{\mathrm{m}}}(t) = M_\omega(C^{\mp}_{K^{\mathrm{m}}}(\omega - t))\) and \(C^{\pm}_{K^{\mathrm{m}}}(t) = M_\omega(A^{\mp}_{K^{\mathrm{m}}}(\omega - t))\).
\item
  \(w_{K^{\mathrm{m}}}(t) = z_K(\omega - t)\) and \(z_{K^{\mathrm{m}}}(t) = w_K(\omega - t)\).
\item
  \(\delta K^{\mathrm{m}} = M_\omega(\delta K)\), \(T_{K^{\mathrm{m}}}(t) = M_\omega(T_{K}(\omega - t))\), and \(\mathcal{N}(K^{\mathrm{m}}) = M_\omega(\mathcal{N}(K))\).
\item
  \(\sigma_{K^{\mathrm{m}}}(E) = \sigma_{K}(\omega + \pi/2 - E)\) for any \(E \subseteq S^1\).
\end{itemize}

\label{pro:mirror-reflection}
\end{proposition}

\begin{proof}
Check the symmetry of the definition of each. For the surface area measure \(\sigma_{K^{\mathrm{m}}}\), use Equation (4.14), page 215 of \autocite{schneider_2013}.
\end{proof}

\begin{remark}

We use the mirror reflection of a cap in \Cref{def:mirror-reflection} and \Cref{pro:mirror-reflection} extensively to exploit the symmetry without loss of generality. For example, say we want to show both \(w_K(t) > 0\) and \(z_K(t) > 0\) for arbitrary \(K \in \mathcal{K}_\omega^\mathrm{c}\) and \(t \in (0, \omega)\). Then it suffices to show the case \(w_K(t) > 0\), as the symmetric case \(z_K(t) = w_{K^{\mathrm{m}}}(\omega - t) > 0\) follows by \Cref{pro:mirror-reflection}.

\label{rem:mirror-symmetry}
\end{remark}

We now show that for any monotone sofa \(S\), its cap \(K\) contains the niche \(\mathcal{N}(K)\) (\Cref{thm:niche-in-cap}). The positivity of \(w_K(t)\) and \(z_K(t)\) in \Cref{thm:wedge-ends-in-cap} is important in establishing \(\mathcal{N}(K) \subset K\).

\begin{definition}

Say that a point \(p_1\) is \emph{further than} (resp. \emph{strictly further than}) the point \(p_2\) \emph{in the direction} of nonzero vector \(v \in \mathbb{R}^2\) if \(p_1 \cdot v \geq p_2 \cdot v\) (resp. \(p_1 \cdot v > p_2 \cdot v\)).

\label{def:further-in-direction}
\end{definition}

\begin{theorem}

Let \(K \in \mathcal{K}_\omega^\mathrm{c}\) be arbitrary. For any angle \(t \in (0, \omega)\), we have \(w_K(t), z_K(t) > 0\).

\label{thm:wedge-ends-in-cap}
\end{theorem}

\begin{proof}
By mirror symmetry (\Cref{rem:mirror-symmetry}), it suffices to show \(w_K(t) > 0\). We need to show that the point \(A_K^-(0)\) is strictly further than the point \(W_K(t)\) in the direction of \(u_0\) (\Cref{def:further-in-direction}). The point \(q := a_K(t) \cap l(\pi/2, 1)\) is strictly further than \(W_K(t) = b_K(t) \cap l(\pi/2, 0)\) in the direction of \(u_0\), because the lines \(a_K(t)\) and \(b_K(t)\) form the boundary of a unit-width vertical strip rotated counterclockwise by \(t\). The points \(q = l_K(t) \cap l_K(\pi/2)\), \(A_K^-(t), A_K^-(0)\) are consecutively further in the direction of \(u_0\) because \(K\) is a convex body, completing the proof.
\end{proof}

\begin{lemma}

Fix an arbitrary \(K \in \mathcal{K}_\omega^\mathrm{c}\) and an angle \(t \in (0, \omega)\). If the inner corner \(\mathbf{x}_K(t)\) is in \(K\), then the wedge \(T_K(t)\) is a subset of \(K\).

\label{lem:niche-in-cap}
\end{lemma}

\begin{proof}
Assume \(\mathbf{x}_K(t) \in K\). If \(\omega = \pi/2\), then by \(\mathbf{x}_K(t) \in K\) the wedge \(T_K(t)\) is the triangle with vertices \(W_K(t)\), \(\mathbf{x}_K(t)\), and \(Z_K(t)\) in counterclockwise order. Also \(W_K(t)\) is further than \(Z_K(t)\) in the direction of \(u_0\) (\Cref{def:further-in-direction}). This with \(w_K(t), z_K(t) > 0\) (\Cref{thm:wedge-ends-in-cap}) implies that all the three vertices of \(T_K(t)\) are in \(K\).

If \(\omega < \pi/2\), we divide the proof into four cases on whether the origin \(O = (0, 0)\) lies strictly below the lines \(b_K(t)\) and \(d_K(t)\) or not respectively.

\begin{itemize}
\tightlist
\item
  If \((0, 0)\) lies on or above both \(b_K(t)\) and \(d_K(t)\), then the corner \(\mathbf{x}_K(t)\) should be outside \(F_\omega^\circ\) and this contradicts \(\mathbf{x}_K(t) \in K\).
\item
  If \((0, 0)\) lies on or above \(b_K(t)\) but lies strictly below \(d_K(t)\), then \(T_K(t)\) is a triangle with vertices \(\mathbf{x}_K(t)\), \(Z_K(t)\) and the intersection \(p := l(\omega, 0) \cap b_K(t)\) in clockwise order, with the point \(p\) on the line segment connecting \(Z_K(t)\) and \((0, 0)\). As \(z_K(t) > 0\) (\Cref{thm:wedge-ends-in-cap}) the point \(Z_K(t)\) lies in the segment connecting \(C^+_K(\omega)\) and the origin \((0, 0)\). So the vertices \(\mathbf{x}_K(t), Z_K(t), p\) of \(T_K(t)\) are in \(K\), showing \(T_K(t) \subseteq K\).
\item
  If \((0, 0)\) lies strictly below \(b_K(t)\) but lies on or above \(d_K(t)\), apply the previous case to the mirror reflection \(K^{\mathrm{m}}\) and reflect back (\Cref{rem:mirror-symmetry}).
\item
  If \((0, 0)\) lies strictly below both \(b_K(t)\) and \(d_K(t)\), then \(T_K(t)\) is a quadrilateral with vertices \(\mathbf{x}_K(t)\), \(Z_K(t)\), \(W_K(t)\) and \((0, 0)\). As \(w_K(t) > 0\) (resp. \(z_K(t) > 0\)) by \Cref{thm:wedge-ends-in-cap}, the point \(W_K(t)\) (resp. \(Z_K(t)\)) is in the line segment connecting \((0, 0)\) and \(A^-_K(0)\) (resp. \(C^+_K(\omega)\)). So all vertices of \(T_K(t)\) are in \(K\) and \(T_K(t) \subseteq K\).
\end{itemize}

\end{proof}

\begin{lemma}

For any \(K \in \mathcal{K}_\omega^\mathrm{c}\), we have \(A^-_K(0), C^+_K(\omega) \in K \setminus \mathcal{N}(K)\).

\label{lem:cap-ends-not-in-niche}
\end{lemma}

\begin{proof}
We only need to show that \(A^-_K(0), C^+_K(\omega)\) are not in \(\mathcal{N}(K)\). That is, for any \(t \in (0, \omega)\), neither points are in \(T_K(t)\). Since \(w_K(t) > 0\) by \Cref{thm:wedge-ends-in-cap}, the point \(A_K^-(0)\) is on the right side of the boundary \(b_K(t)\) of \(T_K(t)\). So \(A_K^-(0) \not\in T_K(t)\). Similarly, \(z_K(t) > 0\) implies \(C_K^+(\omega) \not\in T_K(t)\).
\end{proof}

We identify the exact condition where \(\mathcal{N}(K) \subseteq K\) for a general cap \(K\).

\begin{theorem}

For any \(K \in \mathcal{K}_\omega^\mathrm{c}\), the followings are all equivalent.

\begin{enumerate}
\def\labelenumi{\arabic{enumi}.}
\tightlist
\item
  \(\mathcal{N}(K) \subseteq K\)
\item
  \(\mathcal{N}(K) \subseteq K \setminus \delta K\)
\item
  For every \(t \in (0, \omega)\), either \(\mathbf{x}_K(t) \not\in F_\omega^\circ\) or \(\mathbf{x}_K(t) \in K\).
\item
  The set \(S = K \setminus \mathcal{N}(K)\) is connected.
\end{enumerate}

\label{thm:monotonization-connected-iff}
\end{theorem}

\begin{proof}
(1) and (2) are equivalent because the niche \(\mathcal{N}(K)\) is open in the subset topology of \(F_\omega\) by \Cref{def:niche}, and the set \(K \setminus \delta K\) is the interior of \(K\) in the subset topology of \(F_\omega\) by \Cref{pro:upper-boundary-interior}.

(1 \(\Rightarrow\) 3) We will prove the contraposition and assume \(\mathbf{x}_K(t) \in F_\omega^\circ \setminus K\). Then a neighborhood of \(\mathbf{x}_K(t)\) is inside \(F_\omega\) and disjoint from \(K\), so a subset of \(T_K(t)\) is outside \(K\) and \(\mathcal{N}(K) \not\subseteq K \setminus \delta K\). (3 \(\Rightarrow\) 1) If \(\mathbf{x}_K(t) \not \in F_\omega^\circ\) then \(T_K(t)\) is an empty set. If \(\mathbf{x}_K(t) \in K\) then by \Cref{lem:niche-in-cap} we have \(T_K(t) \subseteq K\).

(2 \(\Rightarrow\) 4) As \(\delta K\) is disjoint from \(\mathcal{N}(K)\), we have \(\delta K \subseteq S\). We show that \(S\) is connected. The set \(\delta K\) is connected by \Cref{pro:upper-boundary-connected}. Take any point \(p \in S\). Take the half-line \(r\) starting from \(p\) in the upward direction \(v_0\). Then \(r\) touches a point in \(\delta K\) as \(p \in K\). Moreover, \(r\) is disjoint from \(\mathcal{N}(K)\) as the set \(\mathcal{N}(K) \cup (\mathbb{R}^2 \setminus F_\omega)\) is closed in the direction \(-v_0\) (\Cref{def:closed-in-direction}). Now \(r \cap K\) is a line segment inside \(S\) connecting arbitrary \(p \in S\) to a point in \(\delta K\). So \(S\) is connected.

(4 \(\Rightarrow\) 3) Assume by contradiction that \(\mathbf{x}_K(t) \in F_\omega^\circ \setminus K\) for some \(t \in (0, \omega)\). The ray with initial point \(\mathbf{x}_K(t)\) and direction \(v_0\) is disjoint from \(K\), as \(F_\omega^\circ \setminus K\) is closed in the direction \(v_0\). The ray with initial point \(\mathbf{x}_K(t)\) and the opposite direction \(-v_0\) is not in \(S\), as \(\mathbf{x}_K(t)\) is the corner of \(Q_K^-(t)\), and \(Q_K^-(t)\) is closed in the direction of \(-v_0\). So the vertical line \(l\) passing through \(\mathbf{x}_K(t)\) in the direction of \(v_0\) is disjoint from \(S\).

By \(z_K(t) > 0\) and \(\mathbf{x}_K(t) \in F_\omega^{\circ}\), the points \(C_K^+(\omega)\), \(Z_K(t)\), \(\mathbf{x}_K(t)\) are consecutively strictly further in the direction of \(u_0\). Likewise, by \(w_K(t) > 0\) and \(\mathbf{x}_K(t) \in F_\omega^{\circ}\), the points \(\mathbf{x}_K(t), W_K(t), A_K^-(0)\) are consecutively strictly further in the direction of \(u_0\). So the point \(C_K^+(\omega)\) lies strictly left to \(l\), and \(A_K^-(0)\) lies strictly right to \(l\). By \Cref{lem:cap-ends-not-in-niche} we have \(C_K^+(\omega), A_K^-(0) \in S\), and this with \(S\) disjoint from \(l\) contradicts that that \(S\) is connected.
\end{proof}

\begin{theorem}

Let \(K \in \mathcal{K}_\omega^\mathrm{c}\) be arbitrary. Then \(K\) is the cap of a monotone sofa if and only if \(K\) contains \(\mathcal{N}(K)\).

\label{thm:niche-in-cap}
\end{theorem}

\begin{proof}
Assume first that \(K\) is the cap of a monotone sofa \(S\). We have \(S = K \setminus \mathcal{N}(K)\) by \Cref{thm:monotone-sofa-structure}. In particular, \(K \setminus \mathcal{N}(K)\) is a moving sofa so it is connected. Since (4) implies (1) in \Cref{thm:monotonization-connected-iff} we have \(\mathcal{N}(K) \subseteq K\).

Now assume that \(\mathcal{N}(K)\) contains \(K\). Then the set \(S = K \setminus \mathcal{N}(K)\) is connected by (4) of \Cref{thm:monotonization-connected-iff}. We have \(S = K \setminus \mathcal{N}(K) = P_\omega \cap \bigcap_{t \in [0, \omega]} L_K(t)\) as in \Cref{eqn:monotonization}. Now see the movement of \(L_K(t)\) containing \(S\) for \(t \in [0, \omega]\) in persepective of \(S\) to show that \(S\) is a moving sofa (\(S \subseteq P_\omega\) and \(h_K(\omega) = h_K(\pi/2) = 1\) implies that \(S\) moves from \(H_L\) to \(V_L\)). By (2) of \Cref{thm:monotonization-connected-iff} we have \(\delta K \subseteq S \subseteq K\). So \(h_S\) and \(h_K\) agree on \(J_\omega\), and we have \(\mathcal{C}(S) = K\). Now \(S = K \setminus \mathcal{N}(K) = \mathcal{I}(S)\) by \Cref{thm:monotone-sofa-structure} and \(S\) is monotone by \Cref{thm:monotonization-idempotent}.
\end{proof}

\begin{remark}

Not all cap \(K\) contains its niche \(\mathcal{N}(K)\). For example, take a long cap \(K = [0, 100] \times [0, 1]\) with rotation angle \(\omega = \pi/2\). Then \(K\) is very wide and the inner quadrant \(Q_K^-(\pi/4)\) of \(L_K(\pi/4)\) is outside \(K\), so that \(\mathcal{N}(K) \not\subseteq K\). By \Cref{thm:niche-in-cap}, this \(K\) can never be the cap of a monotone sofa with rotation angle \(\pi/2\).

\label{rem:niche-not-in-cap}
\end{remark}

From now on, we will \emph{always} understand a monotone sofa \(S\) with rotation angle \(\omega\) by its cap \(K := \mathcal{C}(S)\) in \(\mathcal{K}_\omega^\mathrm{c}\). By \Cref{thm:monotone-sofa-structure}, the monotone sofa \(S = K \setminus \mathcal{N}(K)\) can be recovered from its cap \(K\). In other words, the collection of all monotone sofas \(S\) with the rotation angle \(\omega \in (0, \pi/2]\) embeds into \(\mathcal{K}_\omega^\mathrm{c}\) by taking the cap \(S \mapsto \mathcal{C}(S)\). That is, the space of caps \(\mathcal{K}_\omega^\mathrm{c}\) extends the space of all montone sofas with rotation angle \(\omega\).

Define the \emph{sofa area functional} \(\mathcal{A}_\omega\) on \(\mathcal{K}_\omega^\mathrm{c}\) extending the area \(|S|\) of a monotone sofa \(S\).

\begin{definition}

For any angle \(\omega \in (0, \pi/2]\), define the \emph{sofa area functional} \(\mathcal{A}_\omega : \mathcal{K}_\omega^\mathrm{c} \to \mathbb{R}\) on the space of caps \(\mathcal{K}_\omega^\mathrm{c}\) as \(\mathcal{A}_\omega(K) = |K| - |\mathcal{N}(K)|\).

\label{def:sofa-area-functional}
\end{definition}

\begin{theorem}

For any cap \(K := \mathcal{C}(S)\) of a monotone sofa \(S\) with rotation angle \(\omega\), we have \(\mathcal{A}_\omega(K) = |S|\).

\label{thm:sofa-area-functional}
\end{theorem}

\begin{proof}
An immediate consequence of \Cref{thm:monotone-sofa-structure} and \Cref{thm:niche-in-cap}.
\end{proof}

We will change the moving sofa problem to the maximization of the sofa area functional \(\mathcal{A}_\omega\) on cap space \(\mathcal{K}_\omega^\mathrm{c}\). But recall that as in \Cref{rem:niche-not-in-cap}, not all cap \(K \in \mathcal{K}_\omega^\mathrm{c}\) is the cap \(\mathcal{C}(S)\) of a monotone sofa \(S\). So it is not obvious yet that the maximizer \(K_\omega\) of \(\mathcal{A}_\omega\) will correspond to a monotone sofa \(S_\omega\) of maximum area with a fixed rotation angle \(\omega\). This will be established later in \Cref{thm:limiting-maximum-sofa}.

\chapter{Balanced Maximum Sofas and Caps}
\label{sec:balanced-maximum-sofas-and-caps}
This chapter follows the overview in \Cref{sec:_balanced-maximum-sofas-and-caps} and shows the existence of a \emph{balanced maximum sofa}, a monotone sofa of the maximum area that can be approximated sufficiently close by balanced polygons.

\begin{itemize}
\tightlist
\item
  \Cref{sec:simple-nef-polygon} rigorously define the notion of \emph{simple Nef polygons} and prove \Cref{thm:simple-nef-polygon} that measures the area difference in balancing moves on a simple Nef polygon.
\item
  \Cref{sec:polygon-cap-and-niche} builds the notion of \emph{polygon cap} \(K\) with finite angle set \(\Theta\) and \emph{polygon niche} \(\mathcal{N}_\Theta(K)\).
\item
  \Cref{sec:extensions-of-polygon-cap-space} extends the space \(\mathcal{K}_\Theta^\mathrm{c}\) of all polygon caps with angle set \(\Theta\) to a larger space \(\mathcal{H}_\Theta\) of functions \(h : \Theta^\mathrm{\diamond} \to \mathbb{R}\).
\item
  \Cref{sec:maximum-polygon-cap} defines the \emph{polyline} \(\mathbf{p}_K\) of polygon cap \(K\) that represents the bottom sides of polygon sofa \(S_\Theta := K \setminus \mathcal{N}_\Theta(K)\). \Cref{thm:balanced-polygon-sofa} shows that for a \emph{maximum polygon cap} \(K\) that maximizes \(|K| - |\mathcal{N}_\Theta(K)|\), the upper sides of \(K\) should balance with the sides of polyline \(\mathbf{p}_K\), following the outline in \Cref{sec:balancedness-of-maximum-polygon-cap}. \Cref{thm:balanced-polygon-sofa-connected} shows that \(\mathcal{N}_\Theta(K) \subset K\) for a maximum polygon cap \(K\).
\item
  \Cref{sec:balanced-maximum-sofa} defines the \emph{balanced maximum cap} \(K_\omega\) as the limit of maximum polygon caps \(K = K_\Theta\) with angle set \(\Theta\) and rotation angle \(\omega\), as \(\Theta\) gets denser in \([0, \omega]\). The set \(S_\omega := K_\omega \setminus \mathcal{N}(K_\omega)\) is then a \emph{balanced maximum sofa} which attains the maximum area among all monotone sofas of rotation angle \(\omega \in (0, \pi/2]\).
\end{itemize}
\section{Simple Nef Polygon}
\label{sec:simple-nef-polygon}
A \emph{Nef polygon} is a subset of \(\mathbb{R}^2\) that can be obtained by applying a finite number of boolean operations (e.g.~complement \(\mathbb{R}^2 \setminus X\) of \(X\) or intersection \(X_1 \cap X_2\) of \(X_1\) and \(X_2\)) to open or closed half-planes \autocite{bieriNefPolyhedra1995}. For our purpose, we introduce the notion of \emph{simple Nef polygon}. The polygon cap \(K\) and niche \(\mathcal{N}_\Theta(K)\) that will be defined in the next \Cref{sec:maximum-polygon-cap} are simple Nef polygons. We will establish \Cref{thm:simple-nef-polygon} to balance the sides of \(K\) and \(\mathcal{N}_\Theta(K)\).

\begin{definition}

Let \(\textsf{true}\) and \(\textsf{false}\) be the \emph{boolean values} denoting the truth value of a predicate (e.g., \(1+1=2\) is \(\textsf{true}\), but \(1 + 2 = 4\) is \(\textsf{false}\)). A (\(n\)-ary) \emph{boolean function} \(\mathcal{E}\) is a function from \(\left\{ \textsf{true}, \textsf{false} \right\}^n\) to \(\left\{ \textsf{true}, \textsf{false} \right\}\).

\label{def:boolean-function}
\end{definition}

\begin{definition}

For two boolean values \(P\) and \(Q\), write \(P \Rightarrow Q\) if and only if either \(P\) is \(\textsf{false}\) or both \(P\) and \(Q\) are \(\textsf{true}\). An \(n\)-ary boolean function \(\mathcal{E}\) is \emph{monotone} if for any boolean values \(P_1, P_2, \dots, P_n, Q_1, Q_2, \dots, Q_n\), with \(P_i \Rightarrow Q_i\) for all \(1 \leq i \leq n\), we have
\[
\mathcal{E}(P_1, \dots, P_n) \Rightarrow \mathcal{E}(Q_1, \dots, Q_n).
\]

\label{def:monotone-boolean-function}
\end{definition}

\begin{proposition}

Any \(n\)-ary boolean function \(\mathcal{E}(P_1, \dots, P_n)\) is monotone if it is obtained from variables \(P_1, \dots, P_n\) by applying logical conjunctions \(\land\) (AND) and disjunctions \(\lor\) (OR).

\label{pro:monotone-boolean-function}
\end{proposition}

\begin{proof}
Check that each \(P_i\) is itself monotone as an \(n\)-ary boolean function on the variables \(P_1, \dots, P_n\). Then observe that any logical conjunction or disjunction of monotone functions are monotone.
\end{proof}

\begin{definition}

For any \(n\)-ary boolean function \(\mathcal{E}\) and \(n\) closed or open half-planes \(H_1, \dots, H_n\) of \(\mathbb{R}^2\), define the set
\[
\mathcal{E}(H_1, \dots, H_n) := \left\{ p \in \mathbb{R}^2 : \mathcal{E}(p \in H_1, \dots,p \in H_n) \text{ is } \textsf{true} \right\}.
\]
Call any such set \(X := \mathcal{E}(H_1, \dots, H_n)\) a \emph{Nef polygon}.

\label{def:nef-polygon}
\end{definition}

\begin{definition}

Call \(X\) a \emph{simple Nef polygon} with \emph{defining half-planes} \(H_1, H_2, \dots, H_n\) if \(X = \mathcal{E}(H_1, \dots, H_n)\) for a monotone boolean function \(\mathcal{E}\) and the half-planes \(H_1, \dots, H_n\) have different boundaries \(l_1, \dots, l_n\).

\label{def:simple-nef-polygon}
\end{definition}

A line \(l\) in \(\mathbb{R}^2\) is a Nef polygon since \(l\) is the intersection of two closed half-planes \(H_l^+\) and \(H_l^-\) with boundary \(l\). However, \(l\) is not simple because the half-planes \(H_l^+\) and \(H_l^-\) shares the same boundary. The idea of \Cref{def:simple-nef-polygon} is that a point \(p \in \mathbb{R}^2\) is more likely to be contained in \(X\) if \(p\) is contained in more defining half-planes \(H_1, \dots, H_n\). Recall that \(\mathcal{H}^1\) is the Hausdorff measure of dimension one on \(\mathbb{R}^2\) measuring the length of finite segments.

\begin{definition}

For real-valued expressions \(f\) and \(g\) that may depend on other parameters, write \(f = O(g)\) if and only if there is an absolute constant \(C > 0\) that does \emph{not} depend on any parameters such that \(|f| \leq C g\).

If the inequality \(|f| \leq C g\) holds for a constant \(C > 0\) that depends \emph{only} on certain parameters (e.g.~\(a\) and \(b\)), write these as the subscripts of \(O\) (e.g.~\(f = O_{a, b}(g)\)).

\label{def:o-notation-subscript}
\end{definition}

\begin{theorem}

Let \(X = \mathcal{E}(H_1, \dots, H_n)\) be a simple Nef polygon with monotone boolean function \(\mathcal{E}\) and defining half-planes \(H_i = H_-(t_i, h_i)\) or \(H_-^{\circ}(t_i, h_i)\) of different boundaries \(l_i = l(t_i, h_i)\) for \(1 \leq i \leq n\). Fix an arbitrary index \(1 \leq i \leq n\). There exists a constant \(\epsilon = \epsilon(X, i) > 0\) such that the following holds.

Let \(\delta\) be any real value with \(|\delta| \leq \epsilon\). Define
\[
X'_\delta = \mathcal{E}(H_1, \dots, H_{i-1}, H_{\delta}', H_{i+1}, \dots, H_n)
\]
where \(H_{\delta}' := H_-(t_i, h_i + \delta)\) replaces \(H_i = H_-(t_i, h_i)\) or, \(H_{\delta}' := H_-^\circ(t_i, h_i + \delta)\) replaces \(H_i = H_-^\circ(t_i, h_i)\). Then we have
\[
\left| X_\delta' \right| = |X| + \mathcal{H}^1(\partial X \cap l_i) \cdot \delta + O_{X, i}(\delta^2).
\]

\label{thm:simple-nef-polygon}
\end{theorem}

\begin{proof}
The lines \(l_j\) for all index \(j \neq i\) divide the plane into open polygons \(R_1, R_2, \dots, R_N\) with \(N \leq 2^{n-1}\). Fix an arbitrary region \(R_k\) with \(1 \leq k \leq N\). We will show that there exists some \(\epsilon_{k} = \epsilon_{k}(X, i, k) > 0\) such that
\begin{equation}
\label{eqn:rk}
\left| X'_\delta \cap R_k \right| = |X \cap R_k| + \mathcal{H}^1( \partial X \cap l_i \cap R_k) \cdot \delta + O(\delta^2)
\end{equation}
for any \(\delta\) with \(|\delta| \leq \epsilon_k\). Once \Cref{eqn:rk} is shown, we can take \(\epsilon > 0\) to be the minimum of \(\epsilon_{k}\) over all \(k\) and sum \Cref{eqn:rk} over all \(k\) to complete the proof. Note that the boundary of any region \(R_k\) and the line \(l_i\) intersect in a finite number of points, so adding \(\mathcal{H}^1( \partial X \cap l_i \cap R_k)\) over all \(k\) gives \(\mathcal{H}^1(\partial X \cap l_i)\).

Take an arbitrary point \(p \in R_k\) in an open region \(R_k\). By the definition of \(R_k\), for any index \(j \neq i\), the predicate \(p \in H_j\) is a constant \(\textsf{true}\) or \(\textsf{false}\) no matter which \(p \in R_k\) we take. Define the restriction
\[
\mathcal{F}(Q) :=  \mathcal{E}(p \in H_1, \dots, p \in H_{i-1}, Q, p \in H_{i+1}, \dots, p \in H_n)
\]
of \(\mathcal{E}\) to a single boolean variable \(Q\). Given that \(p \in R_k\), the predicate \(\mathcal{F}(p \in H_i)\) is equivalent to \(p \in X\) and the predicate \(\mathcal{F}(p \in H_{i, \delta}')\) is equivalent to \(p \in X_{\delta}'\). Since \(\mathcal{E}\) is monotone, \(\mathcal{F}\) is also monotone and \(\mathcal{F}(Q)\) cannot be the negation of \(Q\). We now have three cases.

Case 1 (resp. Case 2): \(\mathcal{F}(Q)\) is the constant \(\textsf{false}\) (resp. \(\textsf{true}\)). In this case, for any \(p \in R_k\) the predicates \(\mathcal{F}(p \in H_i)\) and \(\mathcal{F}(p \in H_i')\) are \(\textsf{false}\) (resp. \(\textsf{true}\)) so \(R_k\) is disjoint from (resp. contained in) both \(X\) and \(X'_\delta\). So \(R_k\) is disjoint from \(\partial X\) and \Cref{eqn:rk} holds.

Case 3: \(\mathcal{F}(Q)\) is \(Q\). In this case, for any \(p \in R_k\) the predicate \(p \in H_i\) (resp. \(p \in H_{i, \delta}'\)) is equivalent to \(p \in X\) (resp. \(p \in X_\delta'\)). Consequently, we have \(X \cap R_k = H_i \cap R_k\) and \(X_\delta' \cap R_k = H'_\delta \cap R_k\). In particular, \(\partial X \cap R_k = l_i \cap R_k\). Now \Cref{eqn:rk} becomes
\begin{equation}
\label{eqn:rk-half-plane}
\left| H_\delta' \cap R_k \right| = |H_i \cap R_k| + \mathcal{H}^1( l_i \cap R_k) \cdot \delta + O(\delta^2).
\end{equation}
Define \(g(x) := |H_-(t, x) \cap R_k|\) so that \(|H_\delta' \cap R_k| = g(h_i + \delta)\) and \(|H_i \cap R_k| = g(h_i)\). Let \(f(x) := \mathcal{H}^1(l(t_i, x) \cap R_k)\) then by Cavalieri’s principle we have \(g(x) = \int_{-\infty}^x f(u)\,du\). Note that \(R_k\) is a convex polygon with edges different from \(l_i = l(t_i, h_i)\). So \(f(x)\) is Lipschitz near \(x = h_i\). Now by approximating \(g(x)\) near \(x = h_i\), we get the linear estimate \(g(h + \delta) = g(h) + f(h) \delta + O(\delta^2)\) which is exactly \Cref{eqn:rk-half-plane}. In all three cases, we establish \Cref{eqn:rk} and complete the proof.
\end{proof}

\section{Polygon Cap and Niche}
\label{sec:polygon-cap-and-niche}
Following the outline in \Cref{sec:_balancing-argument-of-gerver}, we define a polygon cap \(K\) and its polygon niche \(\mathcal{N}_\Theta(K)\).

\begin{definition}

Define an \emph{angle set} \(\Theta\) with \emph{rotation angle} \(\omega \in (0, \pi/2]\) as the pair \((\omega, X)\) of \(\omega\) and a nonempty finite subset \(X\) of \((0, \omega)\).

\label{def:angle-set}
\end{definition}

\begin{remark}

The angle set \(\Theta\) is essentially a finite subset \(X\) of \((0, \omega)\) that does not forget the value of \(\omega\). With an abuse of notation, we will treat the angle set \(\Theta\) like the subset \(X \subseteq (0, \omega)\) by, for example, saying \(t \in \Theta\) instead of saying \(t \in X\). Unless specified otherwise, the value \(\omega\) will always denote the rotation angle of \(\Theta\).

\label{rem:angle-set}
\end{remark}

\begin{definition}

For any angle set \(\Theta\) with rotation angle \(\omega\), define \(\Theta^\diamond\) as the set \(\Theta \cup (\Theta + \pi/2) \cup \left\{ \omega, \pi/2 \right\}\).

\label{def:angle-domain}
\end{definition}

\begin{definition}

Define the space \(\mathcal{K}_\Theta^\mathrm{c}\) of \emph{polygon caps with angle set} \(\Theta\) as the set of caps \(K \in \mathcal{K}_\omega^\mathrm{c}\) which is the intersection of closed half-planes with normal angles in \(\Theta^{\diamond} \cup \left\{\omega + \pi, 3\pi/2 \right\}\).

\label{def:angled-cap-space}
\end{definition}

Any cap \(K \in \mathcal{K}_\omega^\mathrm{c}\) can be approximated by the polygon cap \(\mathcal{C}_\Theta(K) \in \mathcal{K}_\Theta^\mathrm{c}\) with angle set \(\Theta\) circumscribing \(K\) with the edges of normal angles in \(\Theta^{\diamond} \cup \left\{\omega + \pi, 3\pi/2 \right\}\). The denser \(\Theta\) is in \((0, \omega)\), the better the approximation \(\mathcal{C}_\Theta(K)\) of \(K\).

\begin{definition}

For any \(K \in \mathcal{K}_\omega^\mathrm{c}\) and angle set \(\Theta\) of rotation angle \(\omega\), define
\[
\mathcal{C}_\Theta(K) = P_\omega \cap \bigcap_{t \in \Theta} Q_K^+(t).
\]

\label{def:angled-cap}
\end{definition}

\begin{proposition}

For any cap \(K \in \mathcal{K}_\omega^\mathrm{c}\), the set \(\mathcal{C}_\Theta(K) \in \mathcal{K}_\Theta^\mathrm{c}\) contains \(K\) as a subset. With this, call \(\mathcal{C}_\Theta(K)\) the \emph{polygon cap} with \emph{angle set} \(\Theta\) \emph{approximating} \(K\). The map \(\mathcal{C}_\Theta : \mathcal{K}_\omega^\mathrm{c} \to \mathcal{K}_\Theta^\mathrm{c}\) is a surjective map fixing the elements of \(\mathcal{K}_\Theta^\mathrm{c} \subset \mathcal{K}_\omega^\mathrm{c}\).

\label{pro:angled-cap}
\end{proposition}

\begin{proof}
That \(K \subseteq \mathcal{C}_\Theta(K)\) comes from \(K \subseteq P_\omega\) and \(K \subseteq Q_K^+(t) = H_K(t) \cap H_K(t + \pi/2)\). We have \(\mathcal{C}_\Theta(K) \in \mathcal{K}_\omega^\mathrm{c}\) as \(K \subseteq \mathcal{C}_\Theta(K) \subseteq P_\omega\). We have \(\mathcal{C}_\Theta(K) \in \mathcal{K}_\Theta^\mathrm{c}\) as the formula in \Cref{def:angled-cap} is the intersection of closed half-planes with normal angles in \(\Theta^{\diamond} \cup \left\{\omega + \pi, 3\pi/2 \right\}\).

Now it remains to show that for any \(K \in \mathcal{K}_\Theta^\mathrm{c}\), we have \(\mathcal{C}_\Theta(K) = K\). Since \(K \in \mathcal{K}_\Theta^\mathrm{c}\) is an intersection of closed half-planes with normal angles in \(\Pi := \Theta^{\diamond} \cup \left\{\omega + \pi, 3\pi/2 \right\}\), we have \(K = \bigcap_{t \in \Pi} H_K(t)\). As \(K \in \mathcal{K}_\Theta^\mathrm{c} \subset \mathcal{K}_\omega^\mathrm{c}\), we have \(h_K(\omega) = h_K(\pi/2) = 1\) and \(h_K(\omega + \pi) = h_K(3\pi/2) = 0\) so
\[
K = \bigcap_{t \in \Pi} H_K(t) = P_\omega \cap \bigcap_{t \in \Theta} Q_K^+(t) = \mathcal{C}_\Theta(K)
\]
as desired.
\end{proof}

\begin{definition}

For every cap \(K \in \mathcal{K}_\omega^\mathrm{c}\) and finite nonempty \(\Theta \subset (0, \omega)\), define its \emph{polygon niche}
\[
\mathcal{N}_\Theta(K) = P_\omega \cap \bigcup_{t \in \Theta} Q_K^-(t)
\]
with angle set \(\Theta\).

\label{def:angled-niche}
\end{definition}

\begin{proposition}

For any \(K \in \mathcal{K}_\omega^\mathrm{c}\), we have \(\mathcal{N}_\Theta(K) = \mathcal{N}_\Theta(\mathcal{C}_\Theta(K)) \subseteq \mathcal{N}(K)\).

\label{pro:angled-niche-polygon-cap}
\end{proposition}

\begin{proof}
Let \(K' := \mathcal{C}_\Theta(K)\). The support functions \(h_X\) of \(X = K, K'\) agree on the set \(\Theta \cup (\Theta + \pi/2)\) by \Cref{def:angled-cap}. Observe that the \Cref{def:angled-niche} of \(\mathcal{N}_\Theta(X)\) only depends on \(\omega\) and the values of \(h_X\) on \(\Theta \cup (\Theta + \pi/2)\). So we have \(\mathcal{N}_\Theta(K) = \mathcal{N}_\Theta(K')\). That \(\mathcal{N}_\Theta(K) \subseteq \mathcal{N}(K)\) follows from their respective definitions.
\end{proof}

\begin{definition}

Let \(\Theta\) be an angle set with rotation angle \(\omega\). Define the \emph{polygon sofa area functional} \(\mathcal{A}_\Theta : \mathcal{K}_\omega^\mathrm{c} \to \mathbb{R}\) as \(\mathcal{A}_\Theta(K) := |\mathcal{C}_\Theta(K)| - |\mathcal{N}_\Theta(K)|\).

\label{def:polygon-upper-bound}
\end{definition}

\begin{theorem}

For any polygon cap \(K \in \mathcal{K}^\mathrm{c}_\Theta\) we have \(\mathcal{A}_\Theta(K) = |K| - |\mathcal{N}_\Theta(K)|\). For every cap \(K \in \mathcal{K}_\omega^\mathrm{c}\), we have \(\mathcal{A}_{\omega}(K) \leq \mathcal{A}_\Theta(K)\).

\label{thm:polygon-upper-bound}
\end{theorem}

\begin{proof}
Consequence of \Cref{pro:angled-cap} and \Cref{pro:angled-niche-polygon-cap}.
\end{proof}

\begin{remark}

The upper bound \(\alpha_{\max} \leq 2.37\) of \autocite{kallusImprovedUpperBounds2018} was essentially established by computing the upper bound \(\mathcal{A}_\Theta\) of \(\mathcal{A}_{\pi/2}\) for a specific set \(\Theta\) of five angles.

\label{rem:polygon-upper-bound}
\end{remark}

\section{Extensions of Polygon Cap Space}
\label{sec:extensions-of-polygon-cap-space}
We extend the space \(\mathcal{K}_\Theta^\mathrm{c}\) of polygon caps to the space \(\mathcal{K}_\Theta^\mathrm{t}\) of all \emph{translates} of polygon caps.

\begin{definition}

Define the space of \emph{polygon cap translates} \(\mathcal{K}_\Theta^\mathrm{t}\) as the collection of all translations \(K'\) of every cap \(K \in \mathcal{K}_\Theta^\mathrm{c}\) with angle set \(\Theta\).

\label{def:cap-trans}
\end{definition}

\begin{proposition}

A convex polygon \(K'\) is in \(\mathcal{K}_\Theta^\mathrm{t}\) if and only if the followings are all true.

\begin{enumerate}
\def\labelenumi{\arabic{enumi}.}
\tightlist
\item
  The widths \(h_{K'}(\omega) + h_{K'}(\omega + \pi)\) and \(h_{K'}(\pi/2) + h_{K'}(3\pi/2)\) of \(K'\) along the angles \(\omega\) and \(\pi/2\) are exactly one.
\item
  \(K'\) is a convex polygon with normal angles in the set \(\Theta^\diamond\).
\end{enumerate}

\label{pro:cap-trans-space}
\end{proposition}

\begin{proof}
Any translate \(K'\) of an arbitrary \(K \in \mathcal{K}_\Theta^\mathrm{c}\) satisfies (1) and (2) immediately. On the other hand, for any \(K'\) satisfying conditions (1) and (2), find a translation \(K\) of \(K'\) so that \(h_K(\omega) = h_K(\pi/2) = 1\) and \(K \in \mathcal{K}_\Theta^\mathrm{c}\). Since the width of \(K\) measured along the angles \(\omega\) and \(\pi/2\) are one, we get \(h_K(\omega + \pi) = h_K(3\pi/2) = 0\).
\end{proof}

Now extend the set of all support functions of \(K \in \mathcal{K}_\Theta^\mathrm{t}\) to the space \(\mathcal{H}_\Theta\) of all functions \(h : \Theta^{\diamond} \to \mathbb{R}\).

\begin{definition}

Define \(\mathcal{H}_\Theta\) as the space of all functions \(h : \Theta^{\diamond} \to \mathbb{R}\).

\label{def:height-space}
\end{definition}

\begin{proposition}

The map \(\mathcal{K}_\Theta^\mathrm{t} \to \mathcal{H}_\Theta\) mapping each \(K' \in \mathcal{K}_\Theta^\mathrm{t}\) to its support function \(h_{K'}\) restricted to the domain \(\Theta^{\diamond}\) is an injection.

\label{pro:height-space-embedding}
\end{proposition}

\begin{proof}
By (2) of \Cref{pro:cap-trans-space}, the values of \(h_{K'}\) on \(\Theta^{\diamond}\) determine
\[
K' = \bigcap_{t \in \Theta^{\diamond} \cup \left\{ \omega + \pi, 3\pi/2 \right\}} H_{K'}(t) \in \mathcal{K}_\Theta^\mathrm{t}
\]
uniquely.
\end{proof}

We now have a series of extensions \(\mathcal{K}_\Theta^\mathrm{c} \subseteq \mathcal{K}_{\Theta}^\mathrm{t} \hookrightarrow \mathcal{H}_\Theta\). We will extend the notion of cap, niche, and the polygon sofa area functional \(\mathcal{A}_\Theta\) from \(\mathcal{K}_\Theta^\mathrm{c}\) to \(\mathcal{H}_\Theta\).

\begin{definition}

For any \(h \in \mathcal{H}_\Theta\), define its \emph{parallelogram}
\[
P_h := \bigcap_{t \in \left\{ \omega, \pi/2 \right\} } H_-(t, h(t)) \cap H_+(t, h(t) - 1)
\]
and \emph{cap}
\[
\mathcal{C}_\Theta(h) := P_h \cap \bigcap_{t \in \Theta \cup (\Theta + \pi/2)} H_-(t, h(t))
\]
and \emph{fan}
\[
F_h := \bigcap_{t \in \left\{ \omega, \pi/2 \right\} } H_+(t, h(t) - 1)
\]
and \emph{niche}
\[
\mathcal{N}_{\Theta}(h) := F_h \cap \bigcup_{t \in \Theta} \left( H_-^\circ(t, h(t) - 1) \cap H_-^\circ(t + \pi/2, h(t + \pi/2) - 1) \right)
\]
and the \emph{polygon sofa area functional}
\[
\mathcal{A}_\Theta(h) := |\mathcal{C}(h)| - |\mathcal{N}_\Theta(h)|.
\]

\label{def:height-extensions}
\end{definition}

The polygons appearing in \Cref{def:height-extensions} are simple Nef polygons (\Cref{def:simple-nef-polygon}).

\begin{proposition}

Take arbitrary \(h \in \mathcal{H}_\Theta\).

\begin{enumerate}
\def\labelenumi{\arabic{enumi}.}
\tightlist
\item
  The sets \(\mathcal{C}_\Theta(h)\) is a simple Nef polygon with the defining half-planes \(H_-(t, h(t))\) for each \(t \in \Theta^{\diamond}\), and \(H_{+}(t, h(t) - 1)\) for each \(t \in \left\{ \omega, \pi/2 \right\}\).
\item
  The set \(\mathcal{N}_\Theta(h)\) is a simple Nef polygon with the defining half-planes \(H_-^{\circ}(t, h(t) - 1)\) for each \(t \in \Theta \cup (\Theta + \pi/2)\), and \(H_{+}(t, h(t) - 1)\) for each \(t \in \left\{ \omega, \pi/2 \right\}\).
\end{enumerate}

\label{pro:cap-niche-nef-polygons}
\end{proposition}

\begin{proof}
Check that \Cref{def:simple-nef-polygon} holds for \(\mathcal{C}_\Theta(h)\) and \(\mathcal{N}_\Theta(h)\). The boundaries of the defining half-planes differ as they differ in its normal angle or distance from the origin. The formulas of \(\mathcal{C}_\Theta(h)\) and \(\mathcal{N}_\Theta(h)\) in \Cref{def:height-extensions} are defined from the mentioned half-planes by applying only union and intersection, so by \Cref{pro:monotone-boolean-function} the defining boolean function is monotone.
\end{proof}

We show that the extended notion of cap, niche, and polygon sofa area functional on \(\mathcal{H}_\Theta\) in \Cref{def:height-extensions} are compatible with that of \(\mathcal{K}_\Theta^\mathrm{c}\).

\begin{proposition}

For any polygon cap translate \(K' \in \mathcal{K}_\Theta^\mathrm{t}\) with \(h := h_K \in \mathcal{H}_\Theta\) we have \(\mathcal{C}_{\Theta}(h) = K'\).

\label{pro:cap-extension-compatible}
\end{proposition}

\begin{proof}
By (2) of \Cref{pro:cap-trans-space}, \(K' = \bigcap_{t \in \Theta^{\diamond} \cup \left\{ \omega + \pi, 3\pi/2 \right\}} H_{K'}(t)\). By (1) of \Cref{pro:cap-trans-space}, the intersection matches the \Cref{def:height-extensions} of \(\mathcal{C}_\Theta(h_K)\).
\end{proof}

\begin{proposition}

For any polygon cap \(K \in \mathcal{K}_\Theta^\mathrm{c}\) with \(h := h_K \in \mathcal{H}_\Theta\) we have \(\mathcal{N}_\Theta(h) = \mathcal{N}_\Theta(K)\). Consequently, we have \(\mathcal{A}_\Theta(h) = \mathcal{A}_\Theta(K)\).

\label{pro:niche-extension-compatible}
\end{proposition}

\begin{proof}
That \(\mathcal{N}_\Theta(h) = \mathcal{N}_\Theta(K)\) comes from the last equation of \Cref{pro:rotating-hallway-parts}. Then \(\mathcal{A}_\Theta(h) = \mathcal{A}_\Theta(K)\) comes from \Cref{pro:cap-extension-compatible}.
\end{proof}

Now inherit \Cref{def:height-extensions} from \(\mathcal{H}_\Theta\) to \(\mathcal{K}_\Theta^\mathrm{t}\) under the extension \(\mathcal{K}_\Theta^\mathrm{c} \subseteq \mathcal{K}_{\Theta}^\mathrm{t} \hookrightarrow \mathcal{H}_\Theta\) (\Cref{def:cap-translate-extensions}) and observe that they act equivariantly under translation (\Cref{thm:height-extensions}).

\begin{definition}

For any cap translate \(K' \in \mathcal{K}_\Theta^\mathrm{t}\) with support function \(h := h_{K'} \in \mathcal{H}_\Theta\), define \(\mathcal{N}_\Theta(K') := \mathcal{N}_\Theta(h)\) and \(\mathcal{A}_\Theta(K') := \mathcal{A}_\Theta(h')\).

\label{def:cap-translate-extensions}
\end{definition}

\begin{theorem}

For any translation \(K' := K + \mathbf{v} \in \mathcal{K}_{\Theta}'\) of a polygon cap \(K \in \mathcal{K}_\Theta^\mathrm{c}\) by \(\mathbf{v} \in \mathbb{R}^2\), we have \(\mathcal{N}_\Theta(K') = \mathcal{N}_\Theta(K) + \mathbf{v}\) and \(\mathcal{A}_\Theta(K') = \mathcal{A}_\Theta(K)\). So with \(\mathbf{v} := (0, 0)\), \Cref{def:cap-translate-extensions} is a proper extension of \(\mathcal{N}_\Theta\) and \(\mathcal{A}_\Theta\) from \(\mathcal{K}_\Theta^\mathrm{c}\) to \(\mathcal{K}_\Theta^\mathrm{t}\).

\label{thm:height-extensions}
\end{theorem}

\begin{proof}
By \Cref{pro:cap-niche-nef-polygons} and \Cref{def:cap-translate-extensions}, the polygon \(\mathcal{N}_{\Theta}(K')\) is a simple Nef polygon with defining half-planes \(H_{K'}(t)\) and \(H_{K'}(t) - u_t\) (or their complement and/or interior) for \(t \in \Theta^{\diamond}\). For any convex body \(K\) and fixed \(t \in S^1\), the supporting half-plane \(H_K(t)\) is equivariant under the translation of \(K\) so that \(H_{K + \mathbf{v}}(t) = H_K(t) + \mathbf{v}\). So we have \(\mathcal{N}_\Theta(K') = \mathcal{N}_\Theta(K) + \mathbf{v}\). We then have
\begin{align*}
\mathcal{A}_\Theta(K') & = |\mathcal{C}_\Theta(h_{K'})| - |\mathcal{N}_\Theta(h_{K'})| = |K'| - |\mathcal{N}_\Theta(K')| \\
& = |K| - |\mathcal{N}_\Theta(K)| = \mathcal{A}_\Theta(K)
\end{align*}
by \Cref{pro:cap-extension-compatible} and \Cref{pro:niche-extension-compatible}.
\end{proof}

\begin{proposition}

Assume that \(h^+ \in \mathcal{H}_\Theta\) is taken so that \(K^+ := \mathcal{C}_\Theta(h^+)\) is a polygon cap translate in \(\mathcal{K}_\Theta^\mathrm{t}\).\footnote{Note that we do not require \(h^+ = h_{K^+}\) here. It is possible for the resulting cap \(K^+\) to have support function strictly less than \(h^+\) after intersection, as the lines obtained from \(h^+\) are not guaranteed to be in convex position.} Then \(\mathcal{A}_\Theta(h^+) \leq \mathcal{A}_\Theta(K^+)\).

\label{pro:cap-translate-reduction}
\end{proposition}

\begin{proof}
By the definition of \(K^+ := \mathcal{C}_\Theta(h^+)\), we always have \(h_{K^+}(t) \leq h^+(t)\) on \(t \in \Theta^\diamond\) and \(K^+ \subseteq P_{h^+}\). Since \(K^+\) is a polygon cap translate, it has width 1 in the directions of \(u_\omega\) and \(u_{\pi/2}\), so is circumscribed in \(P_{h^+}\). So we have equality \(h_{K^+}(t) = h^+(t)\) for \(t = \omega, \pi/2\), and \(F_{h_{K^+}} = F_{h^+}\) in particular. From this and \(h_{K^+}(t) \leq h^+(t)\) on \(t \in \Theta^\diamond\), we have \(\mathcal{N}_{\Theta}(h_{K^+}) \subseteq \mathcal{N}_\Theta(h^+)\). By \Cref{pro:cap-extension-compatible} on \(K^+\) we have \(\mathcal{C}_\Theta(h_{K^+}) = K^+ = \mathcal{C}_\Theta(h^+)\). Now
\begin{align*}
\mathcal{A}_\Theta(h^+) & = |\mathcal{C}_\Theta(h^+)| - |\mathcal{N}_\Theta(h^+)| \\
& \leq |\mathcal{C}_\Theta(h_{K^+})| - |\mathcal{N}_\Theta(h_{K^+})| =\mathcal{A}_\Theta(h_{K^+}) = \mathcal{A}_\Theta(K^+)
\end{align*}
where we use \Cref{def:cap-translate-extensions} in the last equality.
\end{proof}

\section{Maximum Polygon Cap}
\label{sec:maximum-polygon-cap}
We now define the polygon cap \(K_\Theta\) attaining the maximum value of \(\mathcal{A}_\Theta\).

\begin{definition}

Call a polygon cap \(K_\Theta \in \mathcal{K}_\Theta^\mathrm{c}\) a \emph{maximum polygon cap} with \emph{angle set} \(\Theta\) if \(o_\omega \in K_\Theta\) and \(K := K_\Theta\) attains the maximum value of \(\mathcal{A}_\Theta : \mathcal{K}_\Theta^\mathrm{c} \to \mathbb{R}\).

\label{def:maximum-polygon-cap}
\end{definition}

\begin{remark}

For any rotation angle \(\omega < \pi/2\), any cap \(K \in K_\Theta\) must contain the point \(o_\omega\). So the condition \(o_\omega \in K_\Theta\) in \Cref{def:maximum-polygon-cap} is only relevant when \(\omega = \pi/2\), and its purpose is to prevent the cap from sliding horizontally far away.

\label{rem:maximum-polygon-cap}
\end{remark}

\begin{lemma}

The mirror reflection of any maximum polygon cap \(K_\Theta \in \mathcal{K}_\Theta^\mathrm{c}\) is a maximum polygon cap with the angle set \(\omega - \Theta\).

\label{lem:maximum-polygon-cap-mirror}
\end{lemma}

\begin{proof}
Observe that \(\mathcal{A}_\Theta\) and \(o_\omega\) are preserved under the mirror reflection \(M_\omega\).
\end{proof}

\begin{lemma}

Let \(\omega \in (0, \pi/2]\) and \(t \in (0, \omega)\) be arbitrary. There exists a constant \(c_{\omega, t} > 0\) such that the following holds. Let \(\Theta\) be any angle set of rotation angle \(\omega\) containing \(t\). Assume that \(K \in \mathcal{K}_\Theta^\mathrm{c}\) satisfies \(\mathcal{A}_\Theta(K) > 0\). Then the width of \(K\) along the direction \(u_0\) is at most \(c_{\omega, t}\).

\label{lem:polygon-cap-bounded}
\end{lemma}

\begin{proof}
If \(\omega < \pi/2\), then as \(K \subseteq P_\omega\) it suffices to take \(c_{\omega, t}\) as the width of \(P_\omega\) along the direction \(u_0\). Now assume \(\omega = \pi/2\) and fix \(t\). Take any \(K \in \mathcal{K}_\Theta^\mathrm{c}\) with the angle set \(\Theta\) containing \(t\), and let \(d\) be the width of \(K\) along \(u_0\). Then \(|K| \leq d\). The wedge \(T_K(t)\) is a right triangle of side \(\geq d - \sec t - \csc t\) and the acute angle \(t\). So \(|\mathcal{N}_\Theta(K)| \geq |T_K(t)| \geq Q_t(d)\), where \(Q_t(d)\) is a quadratic polynomial of \(d\) completely determined by \(t\) with positive leading coefficient. Now there exists a constant \(c_{\pi/2, t} > 0\) so that for any \(K \in \mathcal{K}_\Theta^\mathrm{c}\) with width \(d > c_{\pi/2, t}\) we have \(\mathcal{A}_\Theta(K) = |K| - |\mathcal{N}_\Theta(K)| \leq d - Q_t(d) \leq 0\). Take the contraposition to finish the proof.
\end{proof}

\begin{theorem}

For any angle set \(\Theta\), a maximum polygon cap \(K_\Theta\) exists.

\label{thm:maximum-polygon-cap}
\end{theorem}

\begin{proof}
We first show that \(\mathcal{A}_\Theta(K) = |K| - |\mathcal{N}_\Theta(K)|\) on \(K \in \mathcal{K}_\Theta^\mathrm{c}\) is continuous with respect to the Hausdorff distance \(d_\mathrm{H}\) on \(K \in \mathcal{K}_\Theta^\mathrm{c}\). By Theorem 1.8.20, page 68 of \autocite{schneider_2013}, \(|K|\) is continuous with respect to \(d_\mathrm{H}\). Fix an \(K \in \mathcal{K}_\Theta^\mathrm{c}\) and take any \(K' \in \mathcal{K}_\Theta^\mathrm{c}\) sufficiently close to \(K\) in \(d_\mathrm{H}\). By \Cref{pro:wedge}, the absolute value of \(|\mathcal{N}_\Theta(K')| - |\mathcal{N}_\Theta(K)|\) is at most the sum of the areas of the symmetric difference \(\Delta(t)\) between wedges \(T_K(t)\) and \(T_{K'}(t)\) over all angles \(t \in \Theta\). As \(K' \to K\) in \(d_\mathrm{H}\), we have \(h_{K'}(t) \to h_K(t)\) and \(h_{K'}(t + \pi/2) \to h_K(t)\) so \(|\Delta(t)| \to 0\). This shows that \(|\mathcal{N}_\Theta(K')| \to |\mathcal{N}_\Theta(K)|\) as \(K' \to K\) in \(d_\mathrm{H}\). So \(\mathcal{A}_\Theta(K)\) is continuous in \(K\) with respect to \(d_\mathrm{H}\).

Let \(\mathcal{B}_\Theta\) be the collection of all \(K \in \mathcal{K}_\Theta^\mathrm{c}\) such that \(\mathcal{A}_\Theta(K) \geq 0\) and \(o_\omega \in K\). Let \(K_1 \in \mathcal{K}_\Theta^\mathrm{c}\) be the polygon cap with the support function \(h_{K_1}(t) = 1\) for every \(t \in \Theta^{\diamond}\) (this \(K_1\) is equal to \(\mathcal{C}_\Theta(K)\) where \(K\) is the semicircle of radius 1 and angle \(\omega + \pi/2\) centered at \(O\)). Then \(\mathcal{N}_\Theta(K_1)\) is empty and \(o_\omega \in K_1\) so we have \(K_1 \in \mathcal{B}_\Theta\) and \(\mathcal{B}_\Theta\) is nonempty. By \Cref{lem:polygon-cap-bounded}, the width of any member of \(\mathcal{B}_\Theta\) is bounded by the constant \(c_{\omega, t}\). So by the Blasckhe selection theorem, the domain \(\mathcal{B}_\Theta\) is compact in \(\mathcal{K}_\Theta^\mathrm{c}\). So a maximizer \(K_\Theta\) of \(\mathcal{A}_\Theta(K_\Theta)\) on the compact and nonempty domain \(\mathcal{B}_\Theta\) exists.

We finally show that \(K := K_\Theta\) maximizes \(\mathcal{A}_\Theta(K)\) over the domain \(K \in \mathcal{K}_\Theta^\mathrm{c}\) larger than \(\mathcal{B}_\Theta\). Take any other \(K' \in \mathcal{K}_\Theta^\mathrm{c}\). Our goal is to show that \(\mathcal{A}_\Theta(K') \leq \mathcal{A}_\Theta(K_\Theta)\). If \(\mathcal{A}_\Theta(K') < 0\), then we have \(\mathcal{A}_\Theta(K') < 0 \leq \mathcal{A}_\Theta(K_\Theta)\) since \(K_\Theta \in \mathcal{B}_\Theta\) so the proof is done. So assume \(\mathcal{A}_\Theta(K') \geq 0\). If \(\omega < \pi/2\), then we have \(o_\omega \in K'\) so \(K' \in \mathcal{B}_\Theta\) and again the proof is done. So assume also \(\omega = \pi/2\). Now any horizontal translation of \(K'\) is also a cap \(K'' \in \mathcal{K}_\Theta^\mathrm{c}\) with the same \(\mathcal{A}_\Theta(K'') = \mathcal{A}_\Theta(K')\). Find one translate \(K''\) that contains the point \(o_\omega = (0, 1)\), then \(K'' \in \mathcal{B}_\Theta\) and we have \(\mathcal{A}_\Theta(K') = \mathcal{A}_\Theta(K'') \leq \mathcal{A}_\Theta(K_\Theta)\) also completing the proof.
\end{proof}

\begin{definition}

For any \(n \geq 1\) and a finite sequence \(p_1, p_2, \dots, p_n \in \mathbb{R}^2\) of points on a plane with strictly increasing \(x\)-coordinates, call the union of all closed segments connecting adjacent points \(p_i\) to \(p_{i+1}\) (\(1 \leq i < n\)) an \emph{\(x\)-monotone polyline}.

\label{def:polyline}
\end{definition}

\begin{figure}
\centering
\includegraphics{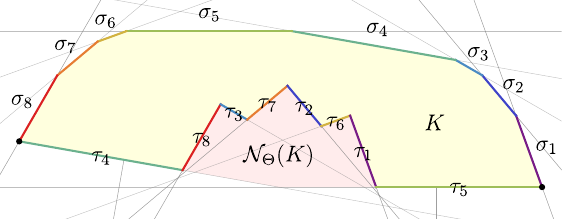}
\caption{A polygon cap \(K \in \mathcal{K}_\Theta^\mathrm{c}\) and polygon niche \(\mathcal{N}_\Theta(K)\) with balanced side lengths. Two black points \(A_K^-(0)\) and \(C_K^+(\omega)\) separate the upper boundary \(\delta K\) of \(K\) and the polyline \(\mathbf{p}_K\). Let \(\theta_1, \dots, \theta_8\) be the elements of \(\Theta\) in increasing order. Each side length \(\sigma_i := \sigma(\theta_i)\) of \(\delta K\) then balance exactly with a corresponding side length \(\tau_i := \tau(\theta_i)\) of \(\mathbf{p}_K\).}
\label{fig:balanced-polygon-sofa-color}
\end{figure}

\begin{theorem}

For any polygon cap \(K \in \mathcal{K}_\Theta^\mathrm{c}\), the boundary of the closed set \(F_\omega \setminus \mathcal{N}_\Theta(K)\) is a disjoint union of the following subsets, from left to right in \(\mathbb{R}^2\).

\begin{enumerate}
\def\labelenumi{\arabic{enumi}.}
\tightlist
\item
  The open half-line \(\vec{l}_K\) from \(C_K^+(\omega)\) extending in the direction \(v_\omega\) but not containing \(C_K^+(\omega)\).
\item
  An \(x\)-monotone polyline \(\mathbf{p}_K\) from left to right connecting \(C_K^+(\omega)\) to \(A_K^-(0)\), with each segment of normal angle \(t \in \Theta^\diamond\).
\item
  The open half-line \(\vec{r}_K\) from \(A_K^-(0)\) extending in the direction \(u_0\) but not containing \(A_K^-(0)\).
\end{enumerate}

\label{thm:polyline}
\end{theorem}

\begin{proof}
Let \(X := \bigcup_{t \in \Theta} Q_K^-(t)\). Observe that \(X\) is a Nef polygon closed in the direction of \(-v_0\) (\Cref{def:closed-in-direction}) with edges of normal angles in \(\Theta \cup \left( \Theta + \pi/2 \right)\). So since \(F_\omega \setminus \mathcal{N}_\Theta(K) = F_\omega \setminus X\), it suffices to show that \(\vec{l}_K\) and \(\vec{r}_K\) are disjoint from \(X\). Take any \(t \in \Theta\). Since \(w_K(t) > 0\) by \Cref{thm:wedge-ends-in-cap}, the point \(A_K^-(0)\) and thus the half-line \(\mathbf{r}_K\) is on the right side of the line \(b_K(t)\). So \(\mathbf{r}_K\) is disjoint from \(Q_K^-(t)\). Likewise, from \(z_K(t) > 0\) of \Cref{thm:wedge-ends-in-cap}, \(Q_K(t)\) is disjoint from \(Q_K^-(t)\) as well.
\end{proof}

\begin{definition}

Define the polyline \(\mathbf{p}_K\) in \Cref{thm:polyline} from \(C_K^+(\omega)\) to \(A_K^-(0)\) as the \emph{polyline of cap} \(K \in \mathcal{K}_\Theta^\mathrm{c}\).

\label{def:polyline-of-cap}
\end{definition}

\begin{definition}

For any cap \(K \in \mathcal{K}_\Theta^\mathrm{c}\) and angle \(t \in \Theta^\diamond\), define \(\tau_K(t)\) as the sum of the lengths of all edges in the polyline \(\mathbf{p}_K\) with normal angle \(t\).

\label{def:polyline-length}
\end{definition}

\begin{lemma}

Let \(K \in \mathcal{K}_\Theta^\mathrm{c}\) be arbitrary.

\begin{enumerate}
\def\labelenumi{\arabic{enumi}.}
\tightlist
\item
  For any \(t \in \Theta\), we have
  \[
  \mathcal{H}^1 \left( \partial \mathcal{N}_\Theta(K) \cap b_K(t) \right) = \mathcal{H}^1 \left( \partial \mathcal{N}_\Theta(K) \cap \vec{b}_K(t) \right) = \tau_K(t)
  \]
  and
  \[
  \mathcal{H}^1 \left( \partial \mathcal{N}_\Theta(K) \cap d_K(t) \right) = \mathcal{H}^1 \left( \partial \mathcal{N}_\Theta(K) \cap \vec{d}_K(t) \right) = \tau_K(t + \pi/2).
  \]
\item
  For any \(t \in \left\{ \omega, \pi/2 \right\}\), we have
  \[
  \mathcal{H}^1\left( \partial \mathcal{N}_\Theta(K) \cap l(t, 0) \right) = \mathcal{H}^1\left(\mathcal{N}_\Theta(K) \cap l(t, 0) \right) = \sigma_K(t + \pi) - \tau_K(t).
  \]
\end{enumerate}

\label{lem:polyline-length}
\end{lemma}

\begin{proof}
Define \(X := \bigcup_{t \in \Theta} Q_K^-(t)\), a Nef polygon with edges of normal angles in \(\Theta \cup (\Theta + \pi/2)\), and recall that \(\mathcal{N}_\Theta(K) = F_\omega \cap X\). Define the followings.

\begin{itemize}
\tightlist
\item
  The intersection \(M := X \cap \partial F_\omega\) which is open in the subspace topology of \(\partial F_\omega\).
\item
  The boundary \(N \subseteq F_\omega^\circ\) of \(\mathcal{N}_\Theta(K)\) in the subspace topology of \(F_\omega^\circ\).
\item
  The intersection \(P := \partial X \cap \partial F_\omega\) which is a finite set of points in \(\partial F_\omega\).
\end{itemize}

We check that the boundary \(\partial \mathcal{N}_\Theta(K)\) is a disjoint union of \(M, N, P\). Since \(F_\omega^\circ\) is open, we have \(N = \partial \mathcal{N}_\Theta(K) \cap F_\omega^\circ\). Since \(M, P \subseteq \partial F_\omega\), and they are disjoint because \(X\) is open, it suffices to show that the set \(X' := \partial \mathcal{N}_\Theta(K) \setminus N = \partial \mathcal{N}_\Theta(K) \cap \partial F_\omega\) is a union of \(M\) and \(P\). The union of \(M\) and \(P\) is \(\overline{X} \cap \partial F_\omega\) so it contains \(X'\), completing the check for \(\partial \mathcal{N}_\Theta(K)\).

Let \(Y := \partial(F_\omega \setminus \mathcal{N}_\Theta(K)) = \partial (F_\omega \setminus X)\). We check that \(Y = N \cup (\partial F_\omega \setminus M)\). We have \(Y \cap F_\omega^\circ = N\) as \(N\) is also the boundary of \(F_\omega \setminus \mathcal{N}_\Theta(K)\) in the subspace topology of \(F_\omega^\circ\). It remains to show \(Y \cap \partial F_\omega = \partial F_\omega \setminus M\). This is equivalent to showing that \(\partial F_\omega\) is a disjoint union of \(M = X \cap \partial F_\omega\) and \(\partial(F_\omega \setminus X) \cap \partial F_\omega\). They are disjoint because \(X\) is open. For any \(p \in \partial F_\omega \setminus M = \partial F_\omega \setminus X\), because \(F_\omega \setminus X\) is closed in the direction \(v_0\), there is a point \(p + \epsilon v_0 \in F_\omega \setminus X\) sufficiently close to \(p\), so \(p \in \partial (F_\omega \setminus X)\). This completes the check for \(Y\).

Define \(D := K \cap \partial F_\omega\). Then \(D = e_K(\omega + \pi) \cup e_K(3\pi/2)\) since \(K\) is a cap. By \Cref{thm:polyline}, the set \(\mathbf{p}_K\) is equal to \(Y \setminus \vec{l}_K \setminus \vec{r}_K\). Now, since \(Y = N \cup (\partial F_\omega \setminus M)\) by the check above, and \(D = \partial F_\omega \setminus \mathbf{l}_K \setminus \mathbf{r}_K\), we have \(\mathbf{p}_K = N \cup (D \setminus M)\). The portion \(N\) of \(\mathbf{p}_K\) is contributed by the lines \(b_K(t)\) and \(d_K(t)\) for angles \(t \in \Theta\). The portion \(D \setminus M\) of \(\mathbf{p}_K\) is contributed by the bottom sides \(\partial F_\omega\) of fan \(F_\omega\).

We prove (1). Take any \(t \in \Theta\).
\[
\mathcal{H}^1 \left( \partial \mathcal{N}_\Theta(K) \cap b_K(t) \right) = \mathcal{H}^1 \left( N \cap b_K(t) \right) =  \mathcal{H}^1 \left( \mathbf{p}_K \cap b_K(t) \right) = \tau_K(t)
\]
because \(M, P \subseteq \partial \mathcal{N}_\Theta(K)\) are on \(\partial F_\omega\) and so only intersect \(b_K(t)\) at a finite number of points. The value is also equal to \(\mathcal{H}^1 \left( \partial \mathcal{N}_\Theta(K) \cap \vec{b}_K(t) \right)\) because the sides of \(\mathcal{N}_\Theta(K)\) with normal angle \(t\) is only contributed by the half-line \(\vec{b}_K(t)\). Use mirror symmetry to show the corresponding equality for \(\tau_K(t + \pi/2)\).

We prove (2). Take any \(t \in \left\{ \omega, \pi/2 \right\}\).
\[
\mathcal{H}^1\left( \partial \mathcal{N}_\Theta(K) \cap l(t, 0) \right) = \mathcal{H}^1\left(M \cap l(t, 0) \right) = \mathcal{H}^1\left((D \setminus \mathbf{p}_K) \cap l(t, 0) \right) = \sigma_K(t + \pi) - \tau_K(t)
\]
because \(N\subseteq \partial \mathcal{N}_\Theta(K)\) is disjoint from \(\partial F_\omega\) and \(P\subseteq \partial \mathcal{N}_\Theta(K)\) is a finite number of points.
\end{proof}

Now we define the notion of balancedness on \(K \in \mathcal{K}_\Theta^\mathrm{c}\).

\begin{definition}

(See \Cref{fig:balanced-polygon-sofa-color}) Say that a polygonal cap \(K \in \mathcal{K}_\Theta^\mathrm{c}\) is \emph{balanced} if and only if for any \(t \in \Theta^{\diamond}\), we have \(\sigma_K(t) = \tau_K(t)\).

\label{def:polygon-cap-balanced}
\end{definition}

We now show that any maximum polygon cap \(K_\Theta \in \mathcal{K}_\Theta^\mathrm{c}\) is balanced. The following lemma chooses the right angle \(t \in \Theta^\diamond\) to balance \(K \in \mathcal{K}_\Theta^\mathrm{c}\).

\begin{lemma}

Assume that a polygon cap \(K \in \mathcal{K}_\Theta^\mathrm{c}\) is \emph{not} balanced. Then there exists an angle \(t \in \Theta^\diamond\) such that \(\sigma_K(t) > \tau_K(t)\).

\label{lem:not-balanced-positive}
\end{lemma}

\begin{proof}
By following the polyline \(\mathbf{p}_K\) from right to left, we have the identity
\[
C_K^+(\omega) - A_K^-(0) = \sum_{t \in \Theta^{\diamond}} \tau_K(t) v_t.
\]
Also, by following the upper boundary \(\delta K\) of \(K\) from right to left, we also have the identity
\[
C_K^+(\omega) - A_K^-(0) = \sum_{t \in \Theta^{\diamond}} \sigma_K(t) v_t.
\]
So we have \(\sum_{t \in \Theta^{\diamond}} (\tau_K(t) - \sigma_K(t)) (v_t \cdot u_0) = 0\) where \(v_t \cdot u_0 < 0\) for all \(t \in \Theta^{\diamond}\). If \(\sigma_K(t) \leq \tau_K(t)\) for all \(t \in \Theta^{\diamond}\), then the equality \(\sigma_K(t) = \tau_K(t)\) should hold and \(K\) should be balanced. Taking the contraposition concludes the proof.
\end{proof}

Execute the balancing step by pusing the edge \(e_K(t)\) in the positive direction of \(u_t\).

\begin{lemma}

Let \(h \in \mathcal{H}_\Theta\) be the support function of some \(K \in \mathcal{K}_\Theta^\mathrm{c}\). Let \(t \in \Theta^{\diamond}\) be arbitrary. Take sufficiently small \(\epsilon > 0\) relative to \(K\) and \(t\). Define \(h^+ \in \mathcal{H}_\Theta\) as \(h^+(t) := h(t) + \epsilon\) and \(h^+(s) := h(s)\) for all \(s \neq t\). Then we have
\[
\mathcal{A}_\Theta(h^+) = \mathcal{A}_\Theta(h) + (\sigma_K(t) - \tau_K(t)) \epsilon + O(\epsilon^2).
\]

\label{lem:balancing}
\end{lemma}

\begin{proof}
First consider the case \(t \not\in \left\{ \omega, \pi/2 \right\}\). We have
\[
|\mathcal{C}_\Theta(h^+)| = |\mathcal{C}_\Theta(h)| + \sigma_K(t) \epsilon + O(\epsilon^2)
\]
by applying \Cref{thm:simple-nef-polygon} to the simple Nef polygon \(\mathcal{C}_\Theta(h)\) ((1) of \Cref{pro:cap-niche-nef-polygons}). We also have
\[
|\mathcal{N}_\Theta(h^+)| = |\mathcal{N}_\Theta(h)| + \tau_K(t) \epsilon + O(\epsilon^2)
\]
by applying \Cref{thm:simple-nef-polygon} to the simple Nef polygon \(\mathcal{N}_\Theta(h) = \mathcal{N}_\Theta(K)\) ((2) of \Cref{pro:cap-niche-nef-polygons}) and using (1) of \Cref{lem:polyline-length}. Subtract the two equations above to conclude the proof.

Now consider the case \(t \in \left\{ \omega, \pi/2 \right\}\). Recall that by \Cref{pro:cap-niche-nef-polygons}, the simple Nef polygon \(\mathcal{C}_\Theta(h)\) have defining half-planes including \(H_-(t, h(t))\) and \(H_{+}(t, h(t) - 1)\). We will apply \Cref{thm:simple-nef-polygon} twice to \(C_\Theta(h)\) to get the following.
\[
|\mathcal{C}_\Theta(h^+)| = |\mathcal{C}_\Theta(h)| + (\sigma_K(t) - \sigma_K(t + \pi) ) \epsilon + O(\epsilon^2)
\]
Here, we first push \(H_-(t, h(t)\) to \(H_-(t, h^+(t))\) by \(\epsilon\) while keeping \(H_{+}(t, h(t) - 1)\) intact. Then, we push \(H_{+}(t, h(t) - 1)\) towards \(H_{+}(t, h^+(t) - 1)\) by \(\epsilon\) while keeping \(H_-(t, h^+(t))\) intact. Since the two half-planes have parallel boundaries of distance 1, the two applications of \Cref{thm:simple-nef-polygon} do not intefere with each other. We also have
\[
|\mathcal{N}_\Theta(h^+)| = |\mathcal{N}_\Theta(h)| + (\sigma_K(t + \pi) - \tau_K(t)) \epsilon + O(\epsilon^2)
\]
by applying \Cref{thm:simple-nef-polygon} to \(\mathcal{N}_\Theta(h)\) and using (2) of \Cref{lem:polyline-length}. Subtract the two equations above to conclude the proof.
\end{proof}

Choosing the right angle \(t \in \Theta^\diamond\) in \Cref{lem:not-balanced-positive} guarantees that the new function \(h^+ \in \mathcal{H}_\Theta\) with \(h^+(t) := h(t) + \epsilon\) corresponds to a cap translate \(K^+\).

\begin{lemma}

Let \(K \in \mathcal{K}_\Theta^\mathrm{c}\) be arbitrary with the support function \(h := h_K \in \mathcal{H}_\Theta\), so that \(K = \mathcal{C}_\Theta(h)\) by \Cref{pro:cap-extension-compatible}. Let \(t \in \Theta^\diamond\) be arbitrary such that \(\sigma_K(t) > 0\).

Take any \(\epsilon > 0\) sufficiently small relative to \(K\) and \(t\), and define \(h^+ \in \mathcal{H}_\Theta\) as \(h^+(t) := h(t) + \epsilon\) and \(h^+(s) := h(s)\) on \(s \neq t\). Then the intersection \(K^+ := \mathcal{C}_\Theta(h^+)\) is a polygon cap translate in \(\mathcal{K}_\Theta^\mathrm{t}\).

\label{lem:height-positive-increment}
\end{lemma}

\begin{proof}
We first prove the case \(t \not\in \left\{ \omega, \pi/2 \right\}\). We have \(K = \mathcal{C}_\Theta(h) \subseteq \mathcal{C}_\Theta(h^+) = K^+\) by the definition of \(\mathcal{C}_\Theta\) and \(K^+\). So \(K \subseteq K^+ \subseteq P_\omega\) and \(K^+ \in \mathcal{K}_\Theta^\mathrm{c}\) as we want.

Now assume \(t \in \left\{ \omega, \pi/2 \right\}\). We first show that \(K^+\) is a polygon cap translate. By \Cref{pro:cap-trans-space} it suffices to show that the width of \(K^+\) is one in the angles \(\omega\) and \(\pi/2\).

Since \(\sigma_K\left( t \right) > 0\), we have \(h_{K^+}(t) = h_K(t) + \epsilon\) for sufficiently small \(\epsilon > 0\).\footnote{That \(\sigma_K(t) > 0\) is extremely crucial here. Otherwise, \(K^+\) might have width less than zero in the angle of either \(t = \omega\) or \(t = \pi/2\), which is necessary for \(K^+\) to be a polygon cap translate. This is why we choose \(t\) according to \Cref{lem:not-balanced-positive}.} As \(\epsilon > 0\), we have \(h_{K^+}(t + \pi) = h_K(t + \pi) - \epsilon\) as well. So the width of \(K^+\) along \(u_t\) is one. If \(\omega = \pi/2\) then we are done. If \(\omega < \pi/2\), let \(t'\) be the other value than \(t\) in \(\left\{ \omega, \pi/2 \right\}\). We have \(h_{K^+}(t') = h_K(t')\) as \(l_K(t)\) moves upwards. Also, as \(\sigma_K(t' + \pi) \geq o_\omega \cdot u_0 > 0\), we have \(h_{K^+}(t' + \pi) = h_K(t' + \pi)\) as \(l_K(t + \pi)\) moves upwards. So the width of \(K^+\) along \(u_{t'}\) is also one, completing the proof.
\end{proof}

\begin{remark}

In \Cref{lem:height-positive-increment}, the adjusted function \(h^+\) might not be the supporting function \(h_{K^+}\) of \(K^+\).

\label{rem:height-positive-increment}
\end{remark}

Combine the steps above to show the balancedness of a maximum polygon cap.

\begin{theorem}

Any maximum polygon cap \(K_\Theta \in \mathcal{K}_\Theta^\mathrm{c}\) is balanced.

\label{thm:balanced-polygon-sofa}
\end{theorem}

\begin{proof}
Assume by contradiction that \(K := K_\Theta\) is not balanced. Let \(h := h_K \in \mathcal{H}_\Theta\) so that \(\mathcal{A}_\Theta(K) = \mathcal{A}_\Theta(h)\) by \Cref{pro:niche-extension-compatible}. Take \(t \in \Theta^{\diamond}\) with \(\sigma_K(t) > \tau_K(t) \geq 0\) as in \Cref{lem:not-balanced-positive}. Now take sufficiently small \(\epsilon > 0\) and define \(h^+ \in \mathcal{H}_\Theta\) and the polygon cap translate \(K^+ := \mathcal{C}_\Theta(h^+) \in \mathcal{K}_\Theta^\mathrm{c}\) as in \Cref{lem:height-positive-increment}. Then we have \(\mathcal{A}_\Theta(h) < \mathcal{A}_\Theta(h^+)\) by \Cref{lem:balancing}. We also have \(\mathcal{A}_\Theta(h^+) \leq \mathcal{A}_\Theta(K^+)\) by \Cref{pro:cap-translate-reduction}. Since \(K^+\) is a translation of some polygon cap \(K_0 \in \mathcal{K}_\Theta^\mathrm{c}\), we have \(\mathcal{A}_\Theta(K^+) = \mathcal{A}_\Theta(K_0)\). Summing up, we have
\[
\mathcal{A}_\Theta(K) = \mathcal{A}_\Theta(h) < \mathcal{A}_\Theta(h^+) \leq \mathcal{A}_\Theta(K^+) = \mathcal{A}_\Theta(K_0)
\]
where \(K_0 \in \mathcal{K}_\Theta^\mathrm{c}\) is a polygon cap, so \(K\in \mathcal{K}_\Theta^\mathrm{c}\) cannot attain the maximum value of \(\mathcal{A}_\Theta\), leading to contradiction.
\end{proof}

\begin{theorem}

Any maximum polygon cap \(K\) with angle set \(\Theta\) contains its polygon niche \(\mathcal{N}_\Theta(K)\).

\label{thm:balanced-polygon-sofa-connected}
\end{theorem}

\begin{proof}
Denote \(K_\Theta\) simply as \(K\). By \Cref{thm:polyline}, it suffices to show that any vertex \(p\) of the polyline \(\mathbf{p}_K\) is contained in \(K\). It suffices to show \(p \in H_K(s)\) for any \(s \in \Theta \cup \left\{ \omega \right\}\). Once this is done, apply a corresponding argument to mirror reflection \(K^\mathrm{m}\) of \(K\) (\Cref{lem:maximum-polygon-cap-mirror}) and reflect back to conclude \(p \in H_K(\omega + \pi/2 - t)\) for any \(t \in (\omega - \Theta) \cup \left\{ \omega \right\}\) as well. Then \(p \in F_\omega \cap \bigcap_{t \in \Theta^{\diamond}} H_K(t) = K\), completing the proof.

\(K\) is balanced by \Cref{thm:balanced-polygon-sofa}. Follow \(\mathbf{p}_K\) from right to left, from the right endpoint \(A_K^-(0)\) to the point \(p\) and stop. By summing up the contribution of edges between \(A_K^-(0)\) and \(p\) in \(\mathbf{p}_K\), we have
\[
A_K^-(0) - p = \sum_{t \in \Theta^{\diamond}} c_t v_t
\]
with coefficients \(c_t \in [0, \tau_K(t)]\) for every \(t \in \Theta^{\diamond}\). Under this constraint on \(c_t\), the value
\[
(p - A_K^-(0)) \cdot u_{s} = \sum_{t \in \Theta^{\diamond}} c_t (v_t \cdot u_{s})
\]
is maximized when \(c_t = \tau_K(t)\) for all \(t \leq s\) and \(c_t = 0\) otherwise. For such \(c_t\) we have the equality
\[
\sum_{t \in \Theta^{\diamond}} c_t v_t = A_K^+(s) - A_K^-(0)
\]
by using the balancedness \(\tau_K(t) = \sigma_K(t)\) and following the upper boundary \(\delta K\) from \(A_K^-(0)\) to \(A_K^+(s)\). So we have
\[
(p - A_K^-(0)) \cdot u_{s} \leq (A_K^+(s) - A_K^-(0)) \cdot u_{s}
\]
which implies \(p \cdot u_{s} \leq A_K^+(s) \cdot u_{s}\). This in turn implies \(p \in H_K(s)\) as desired.
\end{proof}

\section{Balanced Maximum Sofa}
\label{sec:balanced-maximum-sofa}
We now take the limit of the balanced polygon sofas in \Cref{thm:balanced-polygon-sofa} to find a monotone sofa \(S_\omega\).

\begin{definition}

Define the \emph{uniform angle set} \(\Theta_{\omega, n}\) of \(n\) intervals with rotation angle \(\omega \in (0, \pi/2]\) as \(\Theta_{\omega, n} := \left\{ i / n\omega : 1 \leq i < n \right\}\).

\label{def:uniform-angle-set}
\end{definition}

\begin{definition}

For every \(\omega \in (0, \pi/2]\), call a cap \(K_\omega \in \mathcal{K}_\omega^\mathrm{c}\) with the following additional data a \emph{balanced maximum cap} with the rotation angle \(\omega\).

\begin{enumerate}
\def\labelenumi{\arabic{enumi}.}
\tightlist
\item
  There exists a strictly increasing sequence \(1 < n_1 < n_2 < \dots\) of powers of two.
\item
  For each \(i \geq 1\), there exists a maximum polygon cap \(K_i\) with uniform angle set \(\Theta_i := \Theta_{\omega, n_i}\).
\item
  As \(i \to \infty\), the polygon cap \(K_i\) converges to \(K_\omega\) in Hausdorff distance \(d_\mathrm{H}\).
\end{enumerate}

\label{def:balanced-maximum-cap}
\end{definition}

\begin{proposition}

The mirror reflection of any balanced maximum cap \(K_\omega \in \mathcal{K}_\omega^\mathrm{c}\) is also a balanced maximum cap.

\label{pro:balanced-maximum-cap-mirror}
\end{proposition}

\begin{proof}
Apply \Cref{lem:maximum-polygon-cap-mirror} to the maximum polygon caps \(K_i\) converging to \(K_\omega\).
\end{proof}

\begin{theorem}

For every \(\omega \in (0, \pi/2]\), there exists a balanced maximum cap \(\mathcal{K}_\omega^\mathrm{c}\) with rotation angle \(\omega\).

\label{thm:balanced-maximum-cap}
\end{theorem}

\begin{proof}
For each \(i \geq 1\), take some maximum polygon cap \(K_i\) with the uniform angle set \(\Theta_{\omega,2^i}\) of \(2^i\) intervals (\Cref{thm:maximum-polygon-cap}). Since \(o_\omega \in K_i\) and every \(K_i\) is uniformly bounded in diameter \(\sqrt{1 + c_{\omega, \omega/2}}\) by \Cref{lem:polygon-cap-bounded}, we can use the Blaschke convergence theorem to find a convex body \(K\) that a subsequence of \(K_1, K_2, \dots\) converges to in \(d_\mathrm{H}\). Checking \(K \in \mathcal{K}_\omega^\mathrm{c}\) is easy.
\end{proof}

\begin{lemma}

Let \(X, Y\) and \(X_i, Y_i\) for all \(i \geq 1\) be bounded nonempty subsets of \(\mathbb{R}^2\). Assume that \(X_i \to X\) and \(Y_i \to Y\) in Hausdorff distance \(d_\mathrm{H}\) as \(i \to \infty\). Assume also that \(Y\) is compact. Then \(X_i \subseteq Y_i\) for all \(i \geq 1\) implies that \(X \subseteq Y\).

\label{lem:hausdorff-distance-containment}
\end{lemma}

\begin{proof}
For two points \(p, q \in \mathbb{R}^2\), let \(d(p, q)\) denote the Euclidean distance between them. For any point \(p \in \mathbb{R}^2\) and \(Z \subseteq \mathbb{R}^2\), define \(d(p, Z) := \inf_{q \in Z} d(p, q)\). Recall that \(d_\mathrm{H}(Z_1, Z_2)\) is defined as the maximum of \(\sup_{p \in Z_1} d(p, Z_2)\) and \(\sup_{q \in Z_2} d(Z_1, q)\). Take any \(p \in X\). Then \(X_i \to X\) in \(d_\mathrm{H}\) implies that \(d(p, X_i) \to 0\). So we can take \(p_i \in X_i \subseteq Y_i\) such that \(d(p, p_i) \to 0\) as well. That is, \(p_i \to p\). Because \(Y_i \to Y\) in \(d_\mathrm{H}\), we have \(d(p_i, Y) \to 0\). As \(p_i \to p\) and \(Y\) is compact, we have \(p \in Y\).
\end{proof}

\begin{theorem}

For any balanced maximum cap \(K_\omega \in \mathcal{K}_\omega^\mathrm{c}\) we have \(\mathcal{N}(K_\omega) \subset K_\omega\).

\label{thm:limiting-maximum-cap-connected}
\end{theorem}

\begin{proof}
Write \(K_\omega\) as \(K\) to avoid clutterintg. Take the maximum polygon cap \(K_i\) with the uniform angle set \(\Theta_i := \Theta_{\omega, n_i}\) of \(n_i\) intervals, where \(n_i\) is an increasing powers of two, so that \(K_i \to K\) in \(d_\mathrm{H}\). Let \(\Theta = \cup_i \Theta_i\) so that \(\Theta\) is the set of dyadic angles. Note that \(\mathcal{N}_{\Theta_1}(K) \subseteq \mathcal{N}_{\Theta_2}(K) \subseteq \dots\) by definition. We also have \(\bigcup_{j} \mathcal{N}_{\Theta_j}(K) = \mathcal{N}(K)\) as the open set \(Q_K^-(t)\) changes continuously with respect to \(t\). Fix an arbitrary \(j \geq 1\). For any \(i \geq j\), we have \(\mathcal{N}_{\Theta_j}(K_i) \subseteq \mathcal{N}_{\Theta_i}(K_i) \subseteq K_i\) by \Cref{thm:balanced-polygon-sofa-connected}. As \(K_i \to K\) in \(d_\mathrm{H}\), we have \(h_{K_i}\) converging to \(h_K\) uniformly so as \(i \to \infty\) and \(j\) is fixed, either \(\mathcal{N}_{\Theta_j}(K)\) is empty or \(\mathcal{N}_{\Theta_j}(K_i) \to \mathcal{N}_{\Theta_j}(K)\) in \(d_\mathrm{H}\). By \Cref{lem:hausdorff-distance-containment} on \(\mathcal{N}_{\Theta_j}(K_i) \subseteq K_i\), we get \(\mathcal{N}_{\Theta_j}(K) \subseteq K\). Now take \(j \to \infty\) to conclude \(\mathcal{N}(K) \subseteq K\).
\end{proof}

\begin{theorem}

A balanced maximum cap \(K_\omega \in \mathcal{K}_\omega^\mathrm{c}\) attains the maximum value of the sofa area functional \(\mathcal{A}_\omega : \mathcal{K}_\omega^\mathrm{c} \to \mathbb{R}\).

\label{thm:limiting-maximum-cap-max}
\end{theorem}

\begin{proof}
Write \(K := K_\omega\) to avoid cluttering, and take \(K_i\) and \(\Theta_i\) as in the proof of \Cref{thm:limiting-maximum-cap-connected} above. We will show \(\mathcal{A}_{\Theta_i}(K_i) \to \mathcal{A}_\omega(K)\) as \(n \to \infty\). We have \(|K_i| \to |K|\) as \(K_i \to K\) in Hausdorff distance and they are convex bodies (Theorem 1.8.20 of \autocite{schneider_2013}). So it remains to show \(|\mathcal{N}_{\Theta_i}(K_i)| \to |\mathcal{N}(K)|\) as \(n \to \infty\). We always have \(\mathcal{A}_{\Theta_i}(K_i) \geq \mathcal{A}_{\Theta_i}(K) \geq \mathcal{A}_\omega(K)\) by \Cref{def:maximum-polygon-cap} and \Cref{thm:polygon-upper-bound}. This with \(|K_i| \to |K|\) establishes \(\lim \sup_n{ |\mathcal{N}_{\Theta_i}(K_i)| } \leq |\mathcal{N}(K)|\). On the other hand, fix any \(m \geq 1\), then
\[
\liminf_{n} |\mathcal{N}_{\Theta_i}(K_i)| \geq \lim_{n \to \infty} |\mathcal{N}_{\Theta_m}(K_i)| =  |\mathcal{N}_{\Theta_m}(K)|
\]
because \(h_{K_i} \to h_K\) uniformly. Now taking \(m \to \infty\) shows \(\lim \inf_{n} |\mathcal{N}_{\Theta_i}(K_i)| \geq |\mathcal{N}(K)|\), completing the proof.

Let \(m_\omega\) be the supremum of \(\mathcal{A}_\omega : \mathcal{K}_\omega^\mathrm{c} \to \mathbb{R}\). Take any \(\epsilon > 0\). There exists some \(K' \in \mathcal{K}_\omega^\mathrm{c}\) such that \(\mathcal{A}_\omega(K') > m_\omega - \epsilon\). Now
\[
m_\omega - \epsilon < \mathcal{A}_\omega(K') \leq \mathcal{A}_{\Theta_i}(K') \leq \mathcal{A}_{\Theta_i}(K_i)
\]
by maximality of \(K_i\) on \(\mathcal{A}_{\Theta_i}\). Let \(n \to \infty\) to have \(m_\omega - \epsilon \leq \mathcal{A}_\omega(K)\). Take \(\epsilon \to 0\) to have \(m_\omega \leq \mathcal{A}_\omega(K)\), so that \(K\) attains the supremum of \(\mathcal{A}_\omega\) as desired.
\end{proof}

\begin{definition}

Call a monotone sofa \(S_\omega\) with the balanced maximum cap \(K_\omega := \mathcal{C}(S_\omega)\) a \emph{balanced maximum sofa}.

\label{def:balanced-maximum-sofa}
\end{definition}

\begin{theorem}

A balanced maximum sofa \(S_\omega := K_\omega \setminus \mathcal{N}(K_\omega)\) with the balanced maximum cap \(K_\omega\) exists, and attains the maximum area among all moving sofas of rotation angle \(\omega \in (0, \pi/2]\).

\label{thm:limiting-maximum-sofa}
\end{theorem}

\begin{proof}
By \Cref{thm:balanced-maximum-cap}, a balanced maximum cap \(K_\omega\) exists. Such \(K_\omega\) is the cap of a monotone sofa \(S_\omega\) by \Cref{thm:limiting-maximum-cap-connected} and \Cref{thm:niche-in-cap}. This \(S_\omega\) attains the maximum area by \Cref{thm:sofa-area-functional} and \Cref{thm:limiting-maximum-cap-max}.
\end{proof}

\chapter{Rotation Angle of Balanced Maximum Sofas}
\label{sec:rotation-angle-of-balanced-maximum-sofas}
This chapter proves the \Cref{thm:angle} that any maximum-area moving sofa admits a movement with rotation angle \(\omega = \pi/2\). Take any \(\omega \in [\sec^{-1}(2.2) , \pi/2)\). \Cref{sec:horizontal-side-lengths} compares the horizontal sides of a balanced maximum sofa \(S_\omega\). \Cref{sec:right-rotation-angle} proves the \Cref{thm:angle} by following the outline in \Cref{sec:proof-outline}.
\section{Horizontal Side Lengths}
\label{sec:horizontal-side-lengths}
We first establish the horizontal side length comparison in \Cref{thm:balanced-maximum-sofa-ineq}, which is a crucial step in the proof of main \Cref{thm:angle}.

\begin{definition}

For any cap \(K \in \mathcal{K}_\omega^\mathrm{c}\), define \(w_K^\circ := \inf_{t \in (0, \omega)} w_K(t)\) and \(z_K^\circ := \inf_{t \in (0, \omega)} z_K(t)\).

\label{def:wedge-gap-infimum}
\end{definition}

\begin{lemma}

Assume that \(\omega < \pi/2\) and the caps \(K, K' \in \mathcal{K}_\omega^\mathrm{c}\) have Hausdorff distance \(\epsilon := d_\mathrm{H}(K, K')\). Then \(|w_K^{\circ} - w_{K'}^{\circ}| \leq (1 + \sec \omega) \epsilon\)

\label{lem:wedge-gap-limit}
\end{lemma}

\begin{proof}
Since \(d_\mathrm{H}\) is the supremum norm between \(h_K\) and \(h_{K'}\), the distance between \(l_K(t)\) and \(l_{K'}(t)\) is at most \(\epsilon\). So the distance between \(W_K(t) = l_K(t) \cap l(\pi/2, 0)\) and \(W_{K'}(t)\) is at most \(\epsilon \sec \omega\). As \(|h_K(0) - h_{K'}(0)|\leq \epsilon\) the distance between \(A_K^-(0)\) and \(A_{K'}^-(0)\) is at most \(\epsilon\). So \(w_K(t) = (A_K^-(0) - W_K(t)) \cdot u_0\) and \(w_{K'}(t)\) differ by at most \((1 + \sec \omega) \epsilon\).
\end{proof}

\begin{theorem}

For any maximum polygon cap \(K \in \mathcal{K}_\Theta^\mathrm{c}\) with angle set \(\Theta\) of rotation angle \(\omega < \pi/2\), we have \(w_K^{\circ} \leq \sigma_K(\pi/2)\) and \(z_K^{\circ} \leq \sigma_K(\omega)\).

\label{thm:balanced-polygon-sofa-ineq}
\end{theorem}

\begin{proof}
By mirror symmetry (\Cref{lem:maximum-polygon-cap-mirror}), we only need to show \(w_K^{\circ} \leq \sigma_K(\pi/2)\).

Take any \(t \in \Theta\). Define the closed segment \(s_t\) of length \(w_K(t)\) connecting \(W_K(t)\) to \(A_K^-(0)\) from left to right. By \Cref{thm:wedge-ends-in-cap}, \(s_t\) is on the right side of the boundary \(b_K(t)\) of the wedge \(T_K(t)\). So \(s_t\) is disjoint from \(T_K(t)\).

Let \(s\) be the intersection of the edge \(e_K(3\pi/2)\) from \(O\) to \(A_K^-(0)\) and the segment \(s_t\) over all \(t \in \Theta\). Since each \(s_t\) is disjoint from \(T_K(t)\), the segment \(s\) is disjoint from \(\mathcal{N}_\Theta(K) = \bigcup_{t \in \Theta} T_K(t)\).

We will also check that \(s\) have length \(\geq w_K^\circ\). Since both \(e_K(3\pi/2)\) and \(s_t\) have right endpoint \(A_K^-(0)\), it suffices to check that both \(e_K(3\pi/2)\) and \(s_t\) have length \(\geq w_K^\circ\). For any \(t \in (0, \omega)\), the point \(W_K(t)\) is the intersection of lines \(l(t, h_K(t) - 1)\) and \(l(\pi/2, 0)\), we have
\[
\lim_{ t \to \omega^- } W_K(t) = l(\omega, h_K(\omega) - 1) \cap l(\pi/2, 0) = O.
\]
So
\[
\lim_{ t \to \omega^- } w_K(t) = \lim_{ t \to \omega^- } (A_K^-(0) - W_K(t)) \cdot u_0 = h_K(0)
\]
which is the length of \(e_K(3\pi/2)\). So \(e_K(3\pi/2)\) have length \(\geq w_K^\circ\). For any \(t \in \Theta\), the segment \(s_t\) have length \(w_K(t)\) so it has length \(\geq w_K^\circ\). Thus \(s\) have length \(\geq w_K^\circ\).

Then \(s\) is a segment of length \(\geq w_K^{\circ}\) in the edge \(e_K(3\pi/2)\) disjoint from \(\mathcal{N}_\Theta(K)\). So we have
\[
\mathcal{H}^1(s) + \mathcal{H}^1(\mathcal{N}_\Theta(K) \cap e_K(3\pi/2)) \leq \mathcal{H}^1(e_K(3\pi/2))
\]
which becomes
\[
w_K^{\circ} + \sigma_K(3\pi/2) - \tau_K(\pi/2) \leq \sigma_K(3\pi/2)
\]
by (2) of \Cref{lem:polyline-length}. So we have \(w_K^{\circ} \leq \tau_K(\pi/2)\). By the balancedness of \(K\) (\Cref{thm:balanced-polygon-sofa}), we have \(\tau_K(\pi/2) = \sigma_K(\pi/2)\) so the result follows.
\end{proof}

\begin{theorem}

(Theorem 4.2.1, page 212 of \cite{schneider_2013}) As \(n \to \infty\) and the convex bodies \(K_n\) converge to \(K\) in the Hausdorff distance \(d_\mathrm{H}\), the surface area measure \(\sigma_{K_n}\) on \(S^1\) converges weakly to \(\sigma_K\).

\label{thm:surface-area-weak-convergence}
\end{theorem}

\begin{theorem}

For \(\omega < \pi/2\), any balanced maximum cap \(K := K_\omega\) with rotation angle \(\omega\) satisfies \(\sigma_K(\pi/2) \geq w_K^\circ\) and \(\sigma_K(\omega) \geq z_K^\circ\).\Cref{thm:surface-area-weak-convergence}

\label{thm:balanced-maximum-sofa-ineq}
\end{theorem}

\begin{proof}
By mirror symmetry (\Cref{rem:mirror-symmetry}), it suffices to show \(\sigma_K(\pi/2) \geq w_K^\circ\). Take the maximum polygon caps \(K_n\) converging to \(K\). As \(K_n \to K\) in Hausdorff distance, \(w_{K_n} ^{\circ} \to w_K^{\circ}\) by \Cref{lem:wedge-gap-limit}. Because \(\sigma_{K_n} \to \sigma_K\) in weak convergence (\Cref{thm:surface-area-weak-convergence}), we have \(\limsup_{ n } \sigma_{K_n}(\pi/2) \leq \sigma_K(\pi/2)\). This combined with \(w_{K_n}^{\circ} \leq \sigma_{K_n}(\pi/2)\) (\Cref{thm:balanced-polygon-sofa-ineq}) prove \(w_K^\circ \leq \sigma_K(\pi/2)\) as we want.
\end{proof}

\section{Right Rotation Angle}
\label{sec:right-rotation-angle}
We will prove \Cref{thm:balanced-consumed} which is the main step for establishing \Cref{thm:angle} that \(\omega = \pi/2\). We factor out technical calculations in lemmas, and encourage the reader to jump to \Cref{thm:balanced-consumed} right away and refer to them only when needed. The outline in \Cref{sec:proof-outline} and \Cref{fig:proof-outline} should help navigating the proof.

\begin{definition}

Say that a right triangle \(T\) have \emph{base} \(b\) and \emph{angle} \(\theta\), if an edge \(e\) of \(T\) connecting the right-angled vertex to some other vertex \(B\) have length \(b\), and the angle at \(B\) is equal to \(\theta\). We call the length of the edge orthogonal to \(e\) the \emph{height} of \(T\).

\label{def:right-triangle}
\end{definition}

\begin{proposition}

For any \(\omega \in [0, \pi/2)\), define \(c_\omega = \tan((\pi/2 - \omega) / 2)\). Then \(o_\omega - v_0 = c_\omega u_0\) and \(o_\omega - u_\omega = c_\omega v_\omega\). Also, the parallelogram \(P_\omega\) has base \(e_{P_\omega}(3\pi/2)\) connecting \(O = (0, 0)\) to \(l(\pi/2, 0) \cap l(\omega, 1) = (\sec \omega, 0)\), and consequently \(c_\omega = \sec \omega - \tan \omega\).

\label{pro:omega-gap}
\end{proposition}

\begin{proof}
The points \(O, o_\omega - v_0, o_\omega\) (resp. \(O, o_\omega - v_\omega, o_\omega\)) forms a right triangle with base one and angle \((\pi/2 - \omega) / 2\). Let \(P := l(\pi/2, 0) \cap l(\omega, 1)\), then \(P = (\sec \omega, 0)\) by computation, and the points \(o_\omega, o_\omega - v_0, P\) forms a right triangle with base one and angle \(\omega\). This implies \(c_\omega = \sec \omega - \tan \omega\).
\end{proof}

\begin{definition}

Let \(\omega \in [\sec^{-1}(2.2), \pi/2)\) be arbitrary. Define \(d_{\omega, \min}\) as \(1.25\) if \(\omega < \tan^{-1}(2.2)\) and \(1.1\) otherwise.

\label{def:d-min}
\end{definition}

\begin{definition}

Let \(\omega \in [\sec^{-1}(2.2), \pi/2)\) and \(d \in [0, \tan \omega]\) be arbitrary. Define the region
\[
R_{\omega, d} := P_\omega \cap H_-(0, d + c_\omega) \cap H_-(\omega + \pi/2, d + c_\omega).
\]

\label{def:cap-clipped}
\end{definition}

\begin{lemma}

Let \(\omega \in [\sec^{-1}(2.2), \pi/2)\) be arbitrary. Then for \(d = d_{\omega, \min}\), the set \(R_{\omega, d}\) have area \(< 2.2\).

\label{lem:cap-support-elementary-bound}
\end{lemma}

\begin{proof}
We first prove the case \(\omega \geq \tan^{-1}(2.2)\). It suffices to show that the region \(R_{\omega, d}\) with \(d=1.1\) in \Cref{lem:cap-support-elementary-bound} have area \(\leq 2.2\). Define \(Q_1\) as the convex quadrilateral with vertices \(O, o_\omega - v_0, o_\omega, o_\omega - v_\omega\). Define \(Q_2\) as the rectangle with vertices \(A := o_\omega - v_0 + 1.1 u_0\), \(B := o_\omega + 1.1 u_0\), \(o_\omega, o_\omega - v_0\). Then \(|R_{\omega, 1.1}| = |Q_1| + 2|Q_2 \cap R_\omega|\).

Check that \(Q_1\) is contained in the right triangle with vertices \(o_\omega\), \(o_\omega - v_0\), \(o_\omega - v_0 - \cot \omega \cdot u_0\) of base one, height \(\cot \omega\) and angle \(\pi/2 - \omega\). So \(|Q_1| \leq (\cot \omega) / 2\). Check \(1.1 \leq \tan \omega\) by calculating \(\tan(\sec^{-1}(2.2)) > 1.959 > 1.1\). So \(|Q_2 \setminus R_\omega|\) is a right-angled triangle with base \(1.1\) and angle \(\pi/2 - \omega\) of area \((1.21 \cot \omega) / 2\). Now we have
\begin{align*}
|R_{\omega, 1.1}| & = |Q_1| + 2|Q_2 \cap R_\omega| \\
& \leq |Q_1| + 2 (|Q_2| - |Q_2 \setminus R_\omega|) \\
& \leq (\cot \omega)/2 + 2 (1.1 - (1.21 \cot \omega) / 2) < 2.2
\end{align*}
proving the goal.

Now we prove the case \(\omega < \tan^{-1}(2.2)\). We calculate the area of \(R_{\omega, d}\) explicitly. The region \(R_{\omega, d}\) is the parallelogram \(P_\omega\) of base \(\sec \omega\) and height one, subtracted by two right-angled triangles of base \(b := \tan \omega - d\) and height \(b \cot \omega\). So the area of \(R_{\omega, d}\) is \(\sec \omega - (\tan \omega - d)^2 \cot \omega\). Now it suffices to show that the area is \(< 2.2\) for \(d = d_{\omega, \min} = 1.25\). Since \(\omega \in [\sec^{-1}(2.2), \tan^{-1}(2.2)]\), the following estimates hold.
\begin{gather*}
2.2 \leq \sec \omega \leq \sec(\tan^{-1}(2.2)) = \sqrt{146}/5 \\
2.2 \geq \tan \omega \geq \tan(\sec^{-1}(2.2)) = \sqrt{96}/5
\end{gather*}
Now
\begin{align*}
|R_{\omega, 1.25}| & = \sec \omega - (\tan \omega - 1.25)^2 \cot \omega \\
& \leq \sqrt{146}/5 - \left(\sqrt{96}/5 - 1.25\right)^2 / 2.2 \\
& = 2.187736\dots < 2.2
\end{align*}
completing the proof.
\end{proof}

\begin{lemma}

For any constant \(d \geq 1\), the functions \((1 - d \cot \omega)^2\) and \(\cos^2 \omega\) are convex on \(\omega \in [\pi/4, \pi/2]\).

\label{lem:calculation-convex}
\end{lemma}

\begin{proof}
To see the convexity of \((1 - d \cot \omega)^2\), compute the second derivative
\begin{align*}
\frac{d^2}{d \omega^2} (1 - d \cot \omega)^2 & = 2 d \csc^2 \omega  \left( d \csc^2 \omega  + 2 d \cot^2 \omega  - 2 \cot  \omega  \right) \\
& = 2 d \csc^4 \omega  \left( d + 2 d \cos^2 \omega  - 2 \sin  \omega  \cos  \omega  \right) \\
& = 2 d \csc^4 \omega \left( 2 d + d \cos (2\omega) - \sin(2 \omega) \right)
\end{align*}
which is nonnegative. It is easy to check that \(\cos^2 \omega = (1 + \cos(2\omega))/2\) is convex.
\end{proof}

\begin{definition}

Let \(\omega \in [\sec^{-1}(2.2), \pi/2)\) and \(d \in [d_{\omega, \min}, \tan \omega]\) be arbitary. Define the values \(r_y, g \in \mathbb{R}\) and \(q_0, q_1 \in \mathbb{R}^2\) depending solely on \(\omega\) and \(d\) as follows.
\begin{gather*}
r_y := 1 - d \cot \omega \\
g := \sqrt{1 - r_y^2} \\
q_0 := o_\omega - v_0 + du_0 \\
q_1 := o_\omega - g u_0 \\
\end{gather*}

\label{def:calculation-variables}
\end{definition}

\begin{lemma}

Let \(\omega \in [\arctan(2.2), \pi/2)\) and \(d \in [d_{\omega, \min}, \tan \omega]\) be arbitary such that \(r_y \geq 0\). Then the following inequalities are true.

\begin{enumerate}
\def\labelenumi{\arabic{enumi}.}
\tightlist
\item
  \((q_0 - (o_\omega - v_0)) \cdot u_{\pi / 2 - \omega} > 1\)
\item
  \((q_1 - (o_\omega - u_\omega)) \cdot v_{\pi/2 - \omega} > 1\)
\end{enumerate}

\label{lem:calculation-inequalities}
\end{lemma}

\begin{proof}
(1) By the definition of \(q_0\), the inequality to show is \(d u_0 \cdot u_{\pi/2-\omega} > 1\). As \(u_0 \cdot u_{\pi/2 - \omega} = \sin \omega\), we need to check \(d_{\omega, \min} \sin \omega > 1\). Depending on whether \(\omega < \tan^{-1}(2.2)\) or not, computations
\[
1.25 \sin (\sec^{-1}(2.2)) = 1.1134\dots > 1 \qquad  1.1 \sin(\tan^{-1}(2.2)) = 1.0014\dots > 1
\]
prove the result.

(2) By the definition of \(q_1\), the inequality to show is \((u_\omega - g u_0) \cdot v_{\pi/2 - \omega} > 1\). Computing the left-hand side gives \(- \cos(2\omega) + g \cos \omega > 1\). Rearranging the terms, we need to prove \(g > 2 \cos \omega\) as \((1 + \cos(2\omega)) / \cos \omega = 2 \cos \omega\). Both \(g\) and \(2 \cos \omega\) are nonnegative, so we only need to compare their squares and prove \(g^2 = 1 - r_y^2 > 4 \cos^2 \omega\). Since \(0 \leq r_y = 1 - d \cot \omega \leq 1 - d_{\omega, \min} \cot \omega\), it suffices to prove
\begin{equation}
\label{eqn:omega-calc}
1 > (1 - d_{\omega, \min} \cot \omega)^2 + 4 \cos^2 \omega
\end{equation}
on \(\omega \in [\arctan(2.2), \pi/2)\).

As \(d_{\omega, \min}\) is the constant \(1.25\) or \(1.1\) on each interval \(\omega \in (\sec^{-1} (2.2), \tan^{-1} (2.2)]\) or \(\omega \in (\tan^{-1} (2.2), \pi/2]\), both \((1 - d_{\omega, \min} \cot \omega)^2\) and \(4 \cos^2 \omega\) are convex functions of \(\omega\) on each interval by \Cref{lem:calculation-convex}. So it suffices to check \Cref{eqn:omega-calc} at the four endpoints, which are true by following calculations.
\begin{gather*}
(1 - 1.25 \cot(\sec^{-1}(2.2)))^2 + 4 \cos^2 (\sec^{-1}(2.2)) < 0.9576 < 1 \\
(1 - 1.25 \cot(\tan^{-1}(2.2)))^2 + 4 \cos^2 (\tan^{-1}(2.2)) < 0.8714 < 1 \\
(1 - 1.1 \cot(\tan^{-1}(2.2)))^2 + 4 \cos^2 (\tan^{-1}(2.2)) < 0.9350 < 1 \\
(1 - 1.1 \cot(\pi/2))^2 + 4 \cos^2 (\pi/2) = 1
\end{gather*}

\end{proof}

\begin{theorem}

Let \(\omega \in [\sec^{-1}(2.2), \pi/2)\) be arbitrary. Let \(K := K_\omega\) be any balanced maximum cap with rotation angle \(\omega\) in \Cref{thm:balanced-maximum-cap}. Assume also that \(\mathcal{A}_\omega(K) \geq 2.2\). Then there exists some \(t \in (0, \omega)\) such that the three points \(O, o_\omega - v_0, o_\omega - u_\omega\) are in the closure of the quadrant \(Q_K^-(t)\).

\label{thm:balanced-consumed}
\end{theorem}

\begin{proof}
We first show that at least one of \(h_K(0)\) or \(h_K(\omega + \pi/2)\) should be large enough. That is, \(h_K(0) \geq d_{\omega, \min} + c_\omega\) or \(h_K(\omega + \pi/2) \geq d_{\omega, \min} + c_\omega\) should be true (see \Cref{pro:omega-gap} and \Cref{def:d-min}). Assume the contrary. Then the cap \(K\) is contained in the set \(R_{\omega, d}\) with \(d = d_{\omega, \min}\) in \Cref{def:cap-clipped}. But by \Cref{lem:cap-support-elementary-bound} the area of \(R_{\omega, d}\) is less than \(2.2\), contradicting the assumption \(|K| \geq \mathcal{A}_\omega(K) \geq 2.2\).

So we have either \(h_K(0) \geq d_{\omega, \min} + c_\omega\) or \(h_K(\omega + \pi/2) \geq d_{\omega, \min} + c_\omega\). Appeal to the mirror reflection\footnote{This argument depends on the fact that the statement of \Cref{thm:balanced-consumed} on \(K\) and its mirror image \(K^\mathrm{m}\) is equivalent. In particular, the angle \(t \in (0, \omega)\) that we take for \(K\) would correspond to the angle \(\omega - t\) for \(K^\mathrm{m}\).} of \(K\) () and assume without loss of generality that \(h_K(0) \geq d_{\omega, \min} + c_\omega\). Now define \(d := h_K(0) - c_\omega\) so that \(d \geq d_{\omega, \min}\) in particular. We now have the point \(A_K^-(0) = q_0\) in \(K\), where \(q_0 := o_\omega - v_0 + d u_0\) is in \Cref{def:calculation-variables}. In particular, \(q_0\) is in \(K\).

Define the intersection \(r := l_K(0) \cap l_K(\omega)\). Then \(r, o_\omega, o_\omega + d u_0\) forms a right-angled triangle. Solving for the coordinates of \(r\), we get \(r = (d + c_\omega, r_y)\) where \(r_y := 1 - d \cot \omega\) is defined in \Cref{def:calculation-variables}. Since \(K\) is a convex subset of \(P_\omega\) we have \(0 \leq r_y \leq 1\). Define \(g := \sqrt{1 - r_y^2}\) as in \Cref{def:calculation-variables} and the point \(s := q_0 - g u_0\), so that the triangle with vertices \(s, q_0, r\) is right-angled at \(q_0\) with base \(g\), height \(r_y\) and side of length 1.

The main idea is that the point \(s\) is not contained in \(Q_K^-(t)\) for any \(t \in (0, \omega)\), so that \(g \leq w_K^{\circ}\). We prove this rigorously. Take any \(t \in (0, \omega)\). Since \(r := l_K(0) \cap l_K(\omega)\) and \(K\) is a convex body, we have \(h_K(t) \leq r \cdot u_t\). So
\[
h_K(t) - 1 \leq r \cdot u_t - 1 \leq r \cdot u_t - (r - s) \cdot u_t = s \cdot u_t
\]
as \(r-s\) is a unit vector. Thus \(s \not\in H_-^{\circ}(h_K(t) - 1, t)\) and \(s\) is further than \(W_K(t) = b_K(t) \cap l(\pi/2, 0)\) in the direction of \(u_0\). This with \(s = q_0 - g u_0\) and \(W_K(t) = q_0 - w_K(t)u_0\) implies \(g \leq w_K(t)\). Thus we get \(g \leq w_K^{\circ}\).

By \Cref{thm:balanced-maximum-sofa-ineq} we have \(g \leq w_K^{\circ} \leq \sigma_K(\pi/2)\). Define \(q_1 := o_\omega - g u_0\) as in \Cref{def:calculation-variables}, then as \(o_\omega\) is the rightmost vertex of \(K\) on the edge \(e_K(\pi/2)\) we also have \(q_1 \in K\).

We now have two points \(q_0, q_1\) in \(K\), and the set \(X = \left\{ q_0, q_1 \right\}\) is a subset of \(K\). From now on, fix the angle \(t = \pi/2 - \omega \in (0, \pi/4) \subset (0, \omega)\) and take the quadrant \(Q_X^-(t)\). By \(X \subset K\) we have \(Q_X^-(t) \subseteq Q_K^-(t)\). So the proof is done if we show that the three points \(O, o_\omega - v_0, o_\omega - u_\omega\) are in the closure of \(Q_X^-(t)\).

We will show that \(q_1 - q_0 = - \alpha u_t + \beta v_t\) for real coefficients \(\alpha, \beta \geq 0\). By definition of \(q_0\) and \(q_1\) we have \(q_1 - q_0 = v_0 - (d + g) u_0\) where \(d + g \geq d \geq 1.1\). So the vector \(q_1 - q_0\) is in the convex cone generated by \(-u_0\) and \(v_0 - u_0\). Because the angle \(t = \pi/2 - \omega\) is in between \(0\) and \(\pi/4\), both \(-u_0\) and \(v_0 - u_0\) are in the convex cone generated by \(-u_t\) and \(v_t\). So the vector \(q_1 - q_0\) is in the convex cone generated by \(-u_t\) and \(v_t\), and we have \(\alpha, \beta \geq 0\).

Now by \(q_1 - q_0 = - \alpha u_t + \beta v_t\) for \(\alpha, \beta \geq 0\), and the definition of \(L_X(t)\), the point \(q_0\) lies in the outer wall \(a_X(t)\) of \(L_X(t)\) and \(q_1\) lies in the outer wall \(c_X(t)\) of \(L_X(t)\). Thus we can write
\begin{equation}
\label{eqn:quadrant}
Q_X^-(t) = H_-^{\circ}(t, q_0 \cdot u_t - 1) \cap H_-^{\circ}(t + \pi/2, q_1 \cdot v_t - 1).
\end{equation}

We now show that the three points \(O, o_\omega - v_0 = c_\omega u_0, o_\omega - u_\omega = c_\omega v_\omega\) (\Cref{pro:omega-gap}) are contained in \(Q_X^-(t)\). By \Cref{eqn:quadrant} and \(c_\omega > 0\), it suffices to show that \(c_\omega u_0 \in H_-^{\circ}(t, q_0 \cdot u_t - 1)\) and \(c_\omega v_\omega \in H_-^{\circ}(t + \pi/2, q_1 \cdot v_t - 1)\), both are true by \Cref{lem:calculation-inequalities}.
\end{proof}

\Cref{thm:angle} is now a consequence of \Cref{thm:balanced-consumed} as described in \Cref{sec:_statement}.

\begin{proof}[Proof of \Cref{thm:angle}]
By \Cref{thm:balanced-consumed}, the triangle \(\Delta_\omega\) formed by \(O, o_\omega - v_0, o_\omega - u_\omega\) is contained in \(\mathcal{N}(K_\omega)\), so is disjoint from \(S_\omega\) (left of \Cref{fig:triangle-full}). Now \(S_\omega \subseteq P_\omega \setminus \Delta_{\omega}\) and observe that the set \(P_\omega \setminus \Delta_\omega\) have width \(\leq 1\) for every direction \(u_t\) with angle \(t \in [\omega, \pi/2]\). So \(S_\omega\) can rotate counterclockwise by \(\pi/2 - \omega\) inside \(H\) (right of \Cref{fig:triangle-full}). Take \(S'\) as a copy of \(R_{\pi/2-\omega}(S_\omega)\) translated horizontally to the left inside \(H_L\). First rotate \(S'\) clockwise by \(\pi/2 - \omega\) inside \(H_L\). Then translate it to the right until it hits the wall \(x=1\) of \(L\), to put it in the initial position of the monotone sofa \(S_\omega\). Then follow the original movement of \(S_\omega\) with the rotation angle \(\omega\). We have found a movement of \(S'\) with rotation angle \(\pi/2\), so \(|S_\omega| = |S'| \leq |S_{\pi/2}|\) and this completes the proof.
\end{proof}

\chapter{Surface Area Measure}
\label{sec:surface-area-measure}
This chapter proves the equality \(\mathrm{d} v_K^+(t) = v_t \, \sigma_K\) described in \Cref{sec:_surface-area-measure}. \Cref{sec:lebesgue--stieltjes-measure} defines the Lebesgue–Stieltjes measure, and \Cref{sec:differential-gauss--minkowski-theorem} proves the equality in \Cref{thm:boundary-measure}.
\section{Lebesgue--Stieltjes Measure}
\label{sec:lebesgue--stieltjes-measure}
All real-valued measurable functions used in this and upcoming chapters will be bounded and defined on some finite interval \(I\) of \(\mathbb{R}\). All measures \(\mu\) used in this and upcoming chapters will be finite signed Borel measures on some finite interval \(I\) of \(\mathbb{R}\). That is, for any Borel subset \(X\) of \(I\), the value \(\mu(X)\) will be real and not \(\pm \infty\).

Here, we will define the \emph{Lebesgue–Stieltjes measure} \(\mathrm{d} f\) of a right-continuous \(f : I \to \mathbb{R}\) of bounded variation (\Cref{def:lebesgue-stieltjes}). We use the following notations.

\begin{definition}

For any bounded measurable function \(f\) and finite signed Borel measure \(\mu\) on a finite interval \(I\) of \(\mathbb{R}\), define the scalar multiplication \(f \, \mu\) of \(f\) and \(\mu\) as the measure on \(I\) defined as \(f \mu(X) := \int_{t \in X} f(t) \, \mu(dt)\). Note that \(f \, \mu\) is also a finite signed Borel measure.

\begin{itemize}
\tightlist
\item
  If \(f\) is a \emph{pair} \((f_1, f_2)\) of bounded measurable function and \(\mu\) is a finite signed Borel measure, the notion \(f \, \mu\) denotes the pair \((f_1 \, \mu, f_2 \, \mu)\).
\item
  If \(f\) is a bounded measurable function and \(\mu = (\mu_1, \mu_2)\) is a pair of finite signed Borel measure, then the notion \(f \, \mu\) denotes \((f \, \mu_1, f \, \mu_2)\).
\item
  If both \(f = (f_1, f_2)\) and \(\mu = (\mu_1, \mu_2)\) are pairs of bounded measurable functions and finite signed Borel measures respectively, then \(f \cdot \mu\) denotes \(f_1 \, \mu_1 + f_2 \, \mu_2\).
\end{itemize}

\label{def:function-measure-mult}
\end{definition}

Recall the following standard real analysis definition.

\begin{definition}

A function \(f : [a, b] \to \mathbb{R}\) is of \emph{bounded variation} if there is an absolute constant \(C\) such that, for any partition \(a = t_0 < t_1 < \dots < t_n = b\) of \(I\), the sum \(\sum_{i=1}^n \left| f(t_i) - f(t_{i-1}) \right|\) is bounded from above by \(C\).

\label{def:bounded-variation}
\end{definition}

We use the notation \(\mathrm{d} f\) to denote the \emph{Lebesgue–Stieltjes measure} of \(f\).

\begin{definition}

(Theorem 4.3, page 5 of \autocite{revuz2013continuous}) For any right-continuous \(f : I \to \mathbb{R}\) on interval \(I := [a, b]\) of bounded variation, define the \emph{Lebesgue–Stieltjes measure} \(\mathrm{d} f\) on \(I\) as the unique finite signed Borel measure such that \(\mathrm{d}f\left( \left\{ a \right\} \right) = 0\) and \(\mathrm{d} f((a, t]) = f(t) - f(a)\) for all \(t \in I\).

If \(f(t) = (f_1(t), f_2(t))\) is a \emph{pair} of such functions, then \(\mathrm{d} f\) denotes the pair \((\mathrm{d} f_1, \mathrm{d} f_2)\) of measures on \(I\).

\label{def:lebesgue-stieltjes}
\end{definition}

\begin{definition}

Fix an interval \(I := [a, b]\) parametrized by \(t\). Let \(f : I \to \mathbb{R}\) be arbitrary right-continuous function of bounded variation. Let \(g : I \to \mathbb{R}\) be bounded and measurable. Let \(X\) be any Borel subset of \(I\). Define the \emph{Lebesgue–Stieltjes integral} of \(g\) on \(X\) with respect to \(f\) as
\[
\int_{t \in X} g(t) \, \mathrm{d} f(dt)
\]
which is the integral of \(g\) on \(X\) with respect to the Lebesgue–Stieltjes measure \(\mathrm{d} f\) of \(f\). We also denote the integral as
\[
\int_{t\in X} g(t) \, df(t) \qquad \text{ or} \qquad \int_{X} g \, df.
\]

\label{def:lebesgue-stieltjes-integral}
\end{definition}

The map \(f \mapsto \mathrm{d} f\) is linear like differentials would do.

\begin{proposition}

For any right-continuous \(f, g : [a, b] \to \mathbb{R}\) of bounded variation and real values \(r, s \in \mathbb{R}\), we have \(\mathrm{d} (r f + s g) = r \, \mathrm{d} f + s \, \mathrm{d} g\) as signed measures on \(I\).

\label{pro:lebesgue-stieltjes-sum}
\end{proposition}

The product rule \(\mathrm{d}(fg) = g \, \mathrm{d}f + f \, \mathrm{d}g\) is more subtle. We need one of \(f\) or \(g\) to be continuous.

\begin{lemma}

(Proposition 4.5, page 6 of \autocite{revuz2013continuous}) For any right-continuous \(f, g : [a, b] \to \mathbb{R}\) of bounded variation, we have
\[
\int_{t \in (a, b]} g(t)\, df(t) + \int_{t \in (a, b]} f(t-) \, dg(t) = f(b) g(b) - f(a) g(a)
\]
where \(f(t-)\) denotes the left limit of \(f\) at \(t\).

\label{lem:integration-by-parts}
\end{lemma}

\begin{lemma}

For any right-continuous \(f, g : [a, b] \to \mathbb{R}\) of bounded variation, if one of \(f\) or \(g\) is continuous then the equality \(\mathrm{d}(fg) = g \, \mathrm{d}f + f \, \mathrm{d}g\) holds.

\label{lem:lebesgue-stieltjes-product}
\end{lemma}

\begin{proof}
Assume without loss of generality that \(f\) is continuous. It suffices to show that both sides, as measures on \([a, b]\), agree on the subset \((a, x]\) for all \(x \in (a, b]\). This is true by \Cref{lem:integration-by-parts}.
\end{proof}

Finally, we note the following characterization of absolutely continuous functions \(f\).

\begin{proposition}

For any right-continuous \(f : [a, b] \to \mathbb{R}\) of bounded variation, the followings are equivalent.

\begin{enumerate}
\def\labelenumi{\arabic{enumi}.}
\tightlist
\item
  The function \(f : [a, b] \to \mathbb{R}\) is absolutely continuous.
\item
  We have \(\mathrm{d} f = r \, \mathrm{d} t\) for some measurable and bounded \(r : [a, b] \to \mathbb{R}\).
\end{enumerate}

In such a case, we have \(f'(t)\) equal to \(r(t)\) for almost every \(t \in [a, b]\).

\label{pro:lebesgue-stieltjes-abs-cont}
\end{proposition}

\begin{proof}
(2 \(\Rightarrow\) 1) By \Cref{def:lebesgue-stieltjes}, we have \(f(t) = f(a) + \int_a^t r(s)\,ds\), so \(f\) should be absolutely continuous. (1 \(\Rightarrow\) 2) Let \(r : [a, b] \to \mathbb{R}\) be the derivative of \(f(t)\) that exists on almost every \(t \in [a, b]\). Then \(f(t) = f(a) + \int_a^t r(s)\,ds\) by absolute continuity of \(f\), so by the uniqueness of \(\mathrm{d} f\) we should have \(\mathrm{d} f = r \, \mathrm{d} t\).
\end{proof}

\section{Differential Gauss--Minkowski Theorem}
\label{sec:differential-gauss--minkowski-theorem}
Recall that for any \(K \in \mathcal{K}\), the vertex \(v_K^+(t)\) is right-continuous with respect to \(t \in S^1\) (\Cref{thm:limits-converging-to-vertex}). It is also of bounded variation, so the pair \(\mathrm{d} v_K^+(t)\) of Lebesgue–Stieltjes measures of the \(x\)- and \(y\)-coordinates of \(v_K^+(t)\) exists.

\begin{lemma}

For any \(K \in \mathcal{K}\) and any interval \([a, b]\) of \(S^1\), the function \(v_K^+ : [a, b] \to \mathbb{R}^2\) is of bounded variation.

\label{lem:vertex-bounded-variation}
\end{lemma}

\begin{proof}
For the interval \(t \in [0, \pi/4]\), observe that the \(x\)-coordinate (resp. \(y\)-coordinate) of \(v_K^+(t)\) monotonically decreases (resp. increases) with respect to \(t\), so \(v_K^+\) is of bounded variation on \([0, \pi/4]\). A similar logic can be used to angles \([\pi/4, \pi/2]\), \([\pi/2, 3\pi/4]\), and \([3\pi/4, 2\pi]\). Any larger domain \([a, b]\) of \(t\) can be divided into such intervals or their subintervals.
\end{proof}

\Cref{thm:boundary-measure} evaluates \(\mathrm{d} v_K^+(t)\) in terms of the surface area measure \(\sigma_K\) of \(K\). This will be used frequently.

\begin{theorem}

Let \(a, b \in \mathbb{R}\) be arbitrary such that \(a < b \leq a + 2\pi\). Let \(K\) be any planar convex body. Then the equality
\[
\mathrm{d} v_K^{+}(t) = v_t \, \sigma_K
\]
of pairs of measures on the \emph{half-open interval} \(I := (a, b]\) holds,\footnote{The equality does \emph{not} hold in general on the left endpoint \(\left\{ a \right\}\), as it is (somewhat artificially) defined in \Cref{def:lebesgue-stieltjes} that the Lebesgue-Stieltjes measure is zero on the left endpoint.} where \(v_K^+(t) : \overline{I} \to \mathbb{R}^2\) and \(v_t : \overline{I} \to \mathbb{R}^2\) are taken as functions of \(t \in \overline{I} = [a, b]\).

\label{thm:boundary-measure}
\end{theorem}

Note that the notations \(v_K^+(t)\) and \(v_t\) denote different things: \(v_K^+(t)\) is a vertex of the convex body \(K\), while \(v_t\) is the direction \((- \sin t, \cos t)\) independent of \(K\). Integrating \Cref{thm:boundary-measure} on any bounded measurable function \(p : I \to \mathbb{R}^2\), we get
\[
\int_{t \in I} p(t) \cdot d v_K^+(t) = \int_{t \in I} (p(t) \cdot v_t) \,\sigma(dt).
\]

\begin{proof}[Proof of \Cref{thm:boundary-measure}]
It suffices to check that the pairs of measures \(\mathrm{d} v_K^{+}(t)\) and \(u_t \, \sigma\) agree on the subset \((a, x]\) of \(I\) for any \(x \in I\). That is, we only need to check
\begin{equation}
\label{eqn:boundary-to-check}
v_K^+(x) - v_K^+(a) = \int_{t \in (a, x]} v_t\,\sigma_K(dt).
\end{equation}

We first show \Cref{eqn:boundary-to-check} for polygon \(K\). For polygon \(K\), the measure \(\sigma_K\) is a discrete measure where each proper edge \(e_K(t)\) of \(K\) with normal angle \(t\) corresponds to a point mass of \(\sigma_K\) concentrated at \(t\) with the weight \(\sigma_K\left( \left\{ t \right\} \right)\) which is the length of \(e_K(t)\). So the right-hand side of \Cref{eqn:boundary-to-check} is the sum of all vectors \(v_t \sigma_K\left( \left\{ t \right\} \right) = v_K^+(t) - v_K^-(t)\) over all normal angles \(t \in (a, x]\) of proper edges of \(K\). The telescopic sum is the left-hand side \(v_K^+(x) - v_K^+(a)\) as we want.

Now we prove \Cref{eqn:boundary-to-check} for general convex body \(K\). As in the proof of Theorem 8.3.3, page 466 of \cite{schneider_2013}, we can take a series \(K_1, K_2, \dots\) of polygons converging to \(K\) in the Hausdorff distance \(d_\mathrm{H}\) such that \(e_{K_n}(a) = e_{K}(a)\) and \(e_{K_n}(x) = e_{K}(x)\) for all \(n \geq 1\). By \Cref{thm:surface-area-weak-convergence}, the measure \(\sigma_{K_n}\) on \(S^1\) converges to \(\sigma_K\) weakly as \(n \to \infty\).

For any measure \(\sigma\) on \(S^1\), define the \emph{restriction} \(\sigma|_A\) of \(\sigma\) to a Borel subset \(A \subseteq S^1\) as the measure on \(S^1\) such that \(\sigma|_A(X) = \sigma(A \cap X)\) for all Borel subset \(X \subseteq S^1\). Define \(U\) as the open set \(S^1 \setminus \left\{ a, x \right\}\) of \(S^1\), and \(V\) as the open interval \((a, x)\) of \(S^1\). Define \(u_n\) and \(u\) as the restriction of \(\sigma_{K_n}\) and \(\sigma_K\) to \(U\), then \(u_n\) converges to \(u\) weakly as \(n \to \infty\) because \(\sigma_{K_n}(\{a\}) = \sigma_{K}(\{a\})\) and \(\sigma_{K_n}(\{x\}) = \sigma_{K}(\{x\})\).

Define \(\lambda_n\) and \(\lambda\) as the restriction of \(\sigma_{K_n}\) and \(\sigma_K\) to \(V\). We will prove that \(\lambda_n \to \lambda\) weakly as \(n \to \infty\). Take any continuity set \(X \subseteq S^1\) of \(\lambda\) so that \(\lambda(\partial X) = 0\). By the Portmanteau theorem on finite measures, it suffices to show \(\lambda_n(X) \to \lambda(X)\). Because \(\partial(X \cap V) \subseteq (\partial X \cap V) \cup \partial V\), and both \(u(\partial X \cap V) = \lambda(\partial X)\) and \(u(\partial V)\) are zero, the set \(X \cap V\) is a continuity set of \(u\). So \(u_n(X \cap V) \to u(X \cap V)\) and thus \(\lambda_n(X) \to \lambda(X)\) as \(n \to \infty\). This completes the proof that \(\lambda_n \to \lambda\) weakly as \(n \to \infty\).

Now take the limit \(n \to \infty\) to the \Cref{eqn:boundary-to-check} for polygons \(K_n\):
\[
v_{K_n}^+(x) - v_{K_n}^+(a) = \int_{t \in (a, x]} v_t \, \sigma_{K_n}(dt).
\]
The left-hand side is equal to \(v_K^+(x) - v_K^+(a)\) by \(e_{K_n}(a) = e_{K}(a)\) and \(e_{K_n}(x) = e_{K}(x)\). The right-hand side is equal to
\[
(v_{K_n}^+(x) - v_{K_n}^-(x)) + \int_{t \in S^1} v_t \, \lambda_n(dt)
\]
and by \(e_{K_n}(x) = e_{K}(x)\) and the weak convergence \(\lambda_n \to \lambda\), the expression converges to
\[
(v_{K}^+(x) - v_{K}^-(x)) + \int_{t \in S^1} v_t \, \lambda(dt) = \int_{t \in (a, x]} v_t\, \sigma_{K}(dt)
\]
thus completing the proof of \Cref{eqn:boundary-to-check} for general convex body \(K\).
\end{proof}

\begin{remark}

The \emph{Gauss–Minkowski correspondence} maps any planar convex body \(K\), up to translation, bijectively to the Borel measure \(\sigma := \sigma_K\) on \(S^1\) satisfying \(\int_{t \in S^1} v_t\,\sigma(dt) = 0\) (\autocite{marckert2014compact} or Theorem 8.3.1 of \autocite{schneider_2013}). \Cref{thm:boundary-measure} can be seen as a differential version of this correspondence, as we can recover it from integrating \Cref{thm:boundary-measure}.

\begin{itemize}
\tightlist
\item
  By integrating both sides of \Cref{thm:boundary-measure} over all \(t \in S^1\), we immediately the equality \(\int_{t \in S^1} v_t\,\sigma_K(dt) = 0\) which is one direction of the correspondence.
\item
  For any Borel measure \(\sigma\) on \(S^1\) such that \(\int_{t \in S^1} v_t\,\sigma(dt) = 0\), we can recover \(K\) with \(\sigma_K = \sigma\) by taking \(K\) as the convex hull of the partial integrals \(v^-(s) := \int_{(0, s)} v_{t} \, \sigma(dt)\) and \(v^+(s) := \int_{(0, s]} v_t\,\sigma(dt)\) of \Cref{thm:boundary-measure} for all \(s \in (0, 2\pi]\). While we omit the details, it is easy to see that such points are in convex position and \(v^{\pm}(s) = v_K^{\pm}(s)\), so that \(\sigma_K = \sigma\).
\end{itemize}

\label{rem:gauss-minkowski}
\end{remark}

\chapter{Injectivity Condition}
\label{sec:injectivity-condition}
This chapter follows the sketch in \Cref{sec:_injectivity-condition} and proves the injectivity condition on any balanced maximum sofa \(S\). The proof is centered around the differential inequality \Cref{eqn:ineq-example}.

\begin{itemize}
\tightlist
\item
  \Cref{sec:statement} states the full injectivity condition.
\item
  \Cref{sec:arm-lengths} defines the \emph{arm lengths} \(f_K\), \(g_K\) of a cap \(K\) and makes calculations related to it.
\item
  \Cref{sec:inequality-on-maximum-polygon-caps} establishes a discrete version of \Cref{eqn:ineq-example} on maximum polygon sofas (\Cref{thm:balanced-discrete-ineq}). The proof follows the sketch in \Cref{sec:a-differential-inequality}.
\item
  \Cref{sec:inequality-on-balanced-maximum-caps} takes limit on the inequality on maximum polygon sofas, and establish the continuous version of \Cref{eqn:ineq-example} (\Cref{thm:balanced-ineq-limit}) by taking the limit in the discrete version.
\end{itemize}
\section{Statement}
\label{sec:statement}
We showed in the previous \Cref{sec:rotation-angle-of-balanced-maximum-sofas} that we can assume the rotation angle \(\omega = \pi/2\) for the moving sofa problem (\Cref{thm:angle}). So we will omit \(\omega\) in the subscript to denote \(\omega = \pi/2\).

\begin{definition}

Define \(\mathcal{K}^\mathrm{c}\) as the space of caps \(\mathcal{K}_{\pi/2}^\mathrm{c}\) with rotation angle \(\pi/2\). Define \(\mathcal{A} : \mathcal{K}^\mathrm{c} \to \mathbb{R}\) as the sofa area functional \(\mathcal{A}_{\pi/2}\) with rotation angle \(\pi/2\).

\label{def:cap-space-right-angle}
\end{definition}

We fully state the main \Cref{thm:injectivity} of this paper. Recall that a \emph{balanced maximum cap} \(K\) is the cap of a \emph{balanced maximum sofa} \(S\) with rotation angle \(\pi/2\) attaining the maximum area \(\alpha_{\max}\) (\Cref{thm:limiting-maximum-sofa}). Note that the following \Cref{def:injectivity-condition} also defines the notions \(r_K\) and \(s_K\) for \(K\) satisfying the injectivity condition.

\begin{definition}

Say that a cap \(K \in \mathcal{K}^\mathrm{c}\) satisfies the \emph{injectivity condition} if the followings are true.

\begin{enumerate}
\def\labelenumi{\arabic{enumi}.}
\tightlist
\item
  There exists measurable functions \(r_K, s_K : [0, \pi/2] \to \mathbb{R}_{\geq 0}\), unique up to measure zero, such that \(\sigma_K = r_K(t) \mathrm{d} t\) on the interval \(t \in [0, \pi/2)\) and \(\sigma_K = s_K(t - \pi/2) \mathrm{d} t\) on the interval \(t \in (\pi/2, \pi]\).
\item
  The inner corner \(\mathbf{x}_K : [0, \pi/2] \to \mathbb{R}^2\) is continuously differentiable.
\item
  For all \(t \in (0, \pi/2)\), we have \(\mathbf{x}_K'(t) \cdot u_t < 0\) and \(\mathbf{x}_K'(t) \cdot v_t > 0\).
\end{enumerate}

\label{def:injectivity-condition}
\end{definition}

\begin{theorem}

Any balanced maximum cap \(K \in \mathcal{K}^\mathrm{c}\) satisfies the injectivity condition.

\label{thm:injectivity}
\end{theorem}

We will establish \Cref{thm:injectivity} in the last \Cref{sec:bounding-arm-lengths} of this chapter. Once this is done, we can also say the following.

\begin{theorem}

The cap \(K := \mathcal{C}(G)\) of Gerver’s sofa satisfies the injectivity condition.

\label{thm:injectivity-gerver}
\end{theorem}

\begin{proof}
Theorem 2 of \autocite{gerverMovingSofaCorner1992} explicitly constructs a sequence of maximum polygon sofas converging to Gerver’s sofa \(G\). So \(G\) is a balanced maximum sofa, and \Cref{thm:injectivity} proves the claim.
\end{proof}

\begin{remark}

Note that for the cap \(K := \mathcal{C}(G)\) of Gerver’s sofa, a slightly weaker version \(\mathbf{x}_K'(t) \cdot u_t \leq 0\) and \(\mathbf{x}_K'(t) \cdot v_t \geq 0\) of (3) of \Cref{def:injectivity-condition} is already assumed by Romik \autocite{romikDifferentialEquationsExact2018} in order to derive Gerver’s sofa \(G\).

Although Romik does not explicitly put the derived \(G\) back and verify this starting assumption, the equations determining \(G\) as provided in \autocite{romikDifferentialEquationsExact2018} should be sufficient to verify \Cref{thm:injectivity-gerver} independently of Theorem 2 by \autocite{gerverMovingSofaCorner1992}. In particular, the red graph of \Cref{fig:arm-length-lower-bounds} depicts the numerical values of \(f_K(t) = 1 - \mathbf{x}_K'(t) \cdot u_t\) (\Cref{pro:cap-nondegenerate-continuity}), verifying one inequality of (3) of \Cref{def:injectivity-condition} numerically.

\label{rem:injectivity-gerver}
\end{remark}

\section{Arm Lengths}
\label{sec:arm-lengths}
Define the \emph{arm lengths} \(f_K^{\pm}(t)\) and \(g_K^{\pm}(t)\) of supporting hallways of a cap \(K\) as the following.

\begin{definition}

Let \(K \in \mathcal{K}^\mathrm{c}\) and \(t \in [0, \pi/2]\) be arbitrary. Define
\[
f_K^+(t) = \left( \mathbf{y}_K(t) - A_K^+(t) \right) \cdot v_t \qquad f_K^-(t) = \left( \mathbf{y}_K - A_K^-(t) \right) \cdot v_t
\]
and
\[
g_K^+(t) = \left( \mathbf{y}_K(t) - C_K^+(t) \right)  \cdot u_t \qquad g_K^-(t) = \left( \mathbf{y}_K(t) - C_K^-(t) \right)  \cdot u_t.
\]

\label{def:cap-tangent-arm-length}
\end{definition}

\begin{proposition}

Let \(K \in \mathcal{K}^\mathrm{c}\) and \(t \in [0, \pi/2]\) be arbitrary. Then we have \(\mathbf{y}_K(t) = A^{\pm}_K(t) + f_K^{\pm}(t) v_t\) and \(\mathbf{y}_K(t) = C^{\pm}_K(t) + g_K^{\pm}(t) u_t\).

\label{pro:cap-tangent-arm-length}
\end{proposition}

\begin{proof}
The point \(\mathbf{y}_K(t)\) and the vertices \(A_K^{\pm}(t)\) are on the tangent line \(a_K(t)\) in the direction of \(v_t\). Likewise \(\mathbf{y}_K(t)\) and \(C_K^{\pm}(t)\) are on the tangent line \(c_K(t)\) in the direction of \(u_t\).
\end{proof}

\begin{proposition}

For any \(K \in \mathcal{K}^\mathrm{c}\) and \(t \in [0, \pi/2]\), we have \(f_{K^\mathrm{m}}^\pm (t) = g_K^{\mp}(t)\) and \(g_{K^\mathrm{m}}^\pm (t) = f_K^{\mp}(t)\).

\label{pro:cap-tangent-arm-mirror}
\end{proposition}

\begin{proof}
By \Cref{pro:mirror-reflection} and \Cref{def:cap-tangent-arm-length}.
\end{proof}

It turns out that the condition (3) of \Cref{def:injectivity-condition} is equivalent to stating that \(f_K^{\pm}(t) > 1\) and \(g_K^{\pm}(t) > 1\), because the derivative of \(\mathbf{x}_K\) can be expressed in terms of the arm lengths of \(K\) (\Cref{thm:inner-corner-deriv}).

\begin{definition}

For any function \(f\) from interval \(I \subseteq \mathbb{R}\) to \(\mathbb{R}\), denote its left (resp. right) derivative at \(t \in I\) as \(\partial^+ f\) (resp. \(\partial^- f\)).

\label{def:left-right-derivative}
\end{definition}

\begin{theorem}

For any \(K \in \mathcal{K}^\mathrm{c}\) and \(t \in [0, \pi/2)\), the right derivatives of the outer corner \(\mathbf{y}_K(t)\) and inner corner \(\mathbf{x}_K(t)\) exists for all \(0 \leq t < \pi/2\) and is equal to the following.
\begin{align*}
\partial^+ \mathbf{y}_K(t) = -f_K^+(t) u_t + g_K^+(t) v_t \qquad \partial^+ \mathbf{x}_K(t) = -(f_K^+(t) - 1) u_t + (g_K^+(t) - 1) v_t
\end{align*}
Likewise, the left derivatives of \(\mathbf{y}_K(t)\) and \(\mathbf{x}_K(t)\) exists for all \(0 < t \leq \pi/2\) and is equal to the following.
\begin{align*}
\partial^- \mathbf{y}_K(t) = -f_K^-(t) u_t + g_K^-(t) v_t \qquad \partial^+ \mathbf{x}_K(t) = -(f_K^-(t) - 1) u_t + (g_K^-(t) - 1) v_t
\end{align*}

\label{thm:inner-corner-deriv}
\end{theorem}

\begin{proof}
Fix an arbitrary cap \(K\) and omit the subscript \(K\) in the arm lengths \(f_K^{\pm}(t), g_K^{\pm}(t)\), vertices \(\mathbf{y}_K(t)\), \(\mathbf{x}_K(t)\) and the tangent line \(a_K(t)\). Take any \(0 \leq t < \pi/2\) and set \(s = t + \delta\) for sufficiently small and arbitrary \(\delta > 0\). We evaluate \(\partial^+ \mathbf{y}(t) = \lim_{\delta \rightarrow 0^+}(\mathbf{y}(s) - \mathbf{y}(t)) / \delta\). Define \(A_{t, s} = a(t) \cap a(s)\). Since \(A_{t, s}\) is on the lines \(a(t)\) and \(a(s)\), it satisfies both \(A_{t, s} \cdot u_t = \mathbf{y}(t) \cdot u_t\) and \(A_{t, s} \cdot u_s = \mathbf{y}(s) \cdot u_s\). Rewrite \(u_s = (\cos \delta) u_t + (\sin \delta) v_t\) on the second equation and we have
\begin{align*}
	(\cos \delta) A_{t, s} \cdot u_t + (\sin \delta) A_{t, s} \cdot v_t =  	\cos \delta (\mathbf{y}(s) \cdot u_t) + \sin \delta (\mathbf{y}(s) \cdot v_t).
\end{align*}
Group by \(\cos \delta\) and \(\sin \delta\) and substitute \(A_{t, s} \cdot u_t\) with \(\mathbf{y}(t) \cdot u_t\), then
\[ \cos \delta (\mathbf{y}(s) \cdot u_t - \mathbf{y}(t) \cdot u_t)
	= \sin \delta (A_{t, s}  (s) \cdot v_t - \mathbf{y}(s) \cdot v_t) .
	\]
Divide by \(\delta\) and send \(\delta \to 0^+\). We get the following limit as \(A_{t, s} \to A^+(t)\) (\Cref{thm:limits-converging-to-vertex}).
\[ \partial^+ (\mathbf{y}(t) \cdot u_t)  = (A^+(t) - \mathbf{y}(t)) \cdot v_t = - f^+(t)\]
A similar argument can be applied to show \(\partial^+ (\mathbf{y}(t) \cdot v_t) = g^+(t)\) and thus the first equation of the theorem. The right derivative of \(\mathbf{x}_K(t)\) comes from \(\mathbf{x}_K(t) = \mathbf{y}_K(t) - u_t - v_t\). A mirror-symmetric argument calculates the left derivative of \(\mathbf{y}_K\) and \(\mathbf{x}_K\).
\end{proof}

\begin{remark}

The resulting equation \(\partial^+ \mathbf{x}_K(t) = -(f_K^+(t) - 1) u_t + (g_K^+(t) - 1) v_t\) in \Cref{thm:inner-corner-deriv} can be interpreted intuitively as the following. Imagine moving the hallway \(L_K(t)\) slightly by incrementing \(t\) by small \(\epsilon > 0\). The wall \(c_K(t)\) rotates with the pivot \(C_K^+(t)\) as center. So the \(v_t\) component \(g_K^+(t) - 1\) of the derivative \(\partial^+ \mathbf{x}_K(t)\) is the distance from the pivot \(C_K^+(t)\) to \(\mathbf{x}_K(t)\) measured in the direction of \(u_t\). The \(u_t\) component \(-(f_K^+(t) - 1)\) of \(\partial^+ \mathbf{x}_K(t)\) can be interpreted similarly as the distance from the pivot \(A^+_K(t)\) to \(\mathbf{x}_K(t)\) along the direction \(v_t\).

\label{rem:inner-corner-deriv}
\end{remark}

We prove some lemmas that compute the arm lengths.

\begin{lemma}

For any \(K \in \mathcal{K}^\mathrm{c}\) and \(t \in [0, \pi/2]\), we have
\[
g_K^+(t) = \int_{u \in (t, t + \pi/2]} \sin(u - t) \, \sigma_K(du).
\]

\label{lem:arm-length-convolution}
\end{lemma}

\begin{proof}
Since \(\mathbf{y}_K(t), A_K^+(t) \in l_K(t)\), we have
\[
g_K^+(t) = (A_K^+(t) - C_K^+(t)) \cdot u_t = - u_t \cdot (v_K^+(t + \pi/2) - v_K^+(t)).
\]
Parametrize the interval \((t, t + \pi/2]\) by \(s\), then we have \(\mathrm{d} v_K^+(s) = v_{s} \sigma_K\) by \Cref{thm:boundary-measure} so
\begin{align*}
g_K^+(t) & = - u_t \cdot (v_K^+(t + \pi/2) - v_K^+(t)) \\
& = - u_t \cdot \int_{s \in (t, t + \pi/2]} d v_K^+(t') = - u_t \cdot \int_{s \in (t, t + \pi/2]} v_{s} \, \sigma_K(ds) \\
& = \int_{s \in (t, t + \pi/2]} (- u_t \cdot v_{s}) \, \sigma_K(ds) \\
& = \int_{s \in (t, t + \pi/2]} \sin(s - t) \, \sigma_K(d s)
\end{align*}
is proved.
\end{proof}

We calculate the differentiation of arm length \(f_K^+(t)\).

\begin{theorem}

For any \(K \in \mathcal{K}\), the function \(f^+_K(t)\) on \(t \in [0, \pi/2]\) is right-continuous and of bounded variation. Moreover, its Lebesgue–Stieltjes measure on the half-open interval \(t \in (0, \pi/2]\) is
\[
\mathrm{d} f_K^+(t) = g_K^+(t)\, \mathrm{d} t - \sigma_K
\]
where \(f_K^+(t), g_K^+(t), t\) are functions of \(t \in [0, \pi/2]\), so that \(\mathrm{d} t\) denotes the usual Borel measure of \((0, \pi/2]\).

\label{thm:arm-length-differentiation}
\end{theorem}

\begin{proof}
First evaluate
\begin{align*}
\mathrm{d} h_K(t) & = \mathrm{d} (v_K^+(t) \cdot u_t) = u_t \cdot \mathrm{d} v_K^+(t) + v_K^+(t) \cdot \mathrm{d} u_t \\
& = u_t \cdot (v_t \sigma_K) + (v_K^+(t) \cdot v_t) \mathrm{d} t = (v_K^+(t) \cdot v_t) \mathrm{d} t
\end{align*}
for all \(t \in S^1\) using \Cref{lem:lebesgue-stieltjes-product} and \Cref{thm:boundary-measure}. Now take the Lebesgue–Stieltjes measure of
\[
f_K^+(t) = \left( \mathbf{y}_K(t) - A_K^+(t) \right) \cdot v_t = h_K(\pi/2 + t) - v_K^+(t) \cdot v_t
\]
for all \(t \in (0, \pi/2]\) using \Cref{lem:lebesgue-stieltjes-product} and \Cref{thm:boundary-measure} to get
\begin{align*}
\mathrm{d} f_K^+(t) & = \mathrm{d} h_K(\pi/2 + t) - v_t \cdot \mathrm{d} v_K^+(t) - v_K^+(t) \cdot \mathrm{d} v_t \\
& = (v_K^+(\pi/2 + t) \cdot v_{\pi/2 + t}) \mathrm{d} t - v_t \cdot (v_t \sigma_K) + v_K^+(t) \cdot u_t \mathrm{d} t \\
& = - \sigma_K + (- v_K^+(\pi/2 + t) \cdot u_t + h_K(t)) \mathrm{d} t
= g_K^+(t) \mathrm{d} t - \sigma_K
\end{align*}
which completes the proof.
\end{proof}

\section{Inequality on Maximum Polygon Caps}
\label{sec:inequality-on-maximum-polygon-caps}
In this \Cref{sec:inequality-on-maximum-polygon-caps}, we prove \Cref{thm:balanced-discrete-ineq} that bounds the side lengths of a maximum polygon cap \(K\).

Recall that a balanced maximum cap \(K_\infty\) defined as the limit of \emph{maximum polygon caps} \(K_1, K_2, \dots\) converging to \(K\) in Hausdorff distance \(d_\text{H}\) (\Cref{def:balanced-maximum-cap}). Each maximum polygon cap \(K_i\) have angle set \(\Theta_{n_i} := \left\{ (\pi/2) j / n_i : 1 \leq j < n_i \right\}\) where \(1 < n_1 < n_2 < \dots\) are increasing powers of 2. To refer each polygon cap \(K_i\) more easily, we introduce the following notion.

\begin{definition}

For any \(n \geq 1\), define the angle set \(\Theta_n := \Theta_{\pi/2, n}\) with rotation angle \(\pi/2\) so that \(\Theta_n = \left\{ (\pi/2) i / n : 1 \leq i < n \right\}\).

\label{def:right-angle-set}
\end{definition}

\begin{definition}

Say that \(K \in \mathcal{K}\) is a \emph{maximum polygon cap} with \(n\) \emph{steps} of \emph{step size} \(\delta := (\pi/2) / n\), if \(n > 1\) is a power of two, and \(K\) is a maximum polygon cap with the angle set \(\Theta_n\).

\label{def:simple-notation-balanced-cap}
\end{definition}

The main \Cref{thm:balanced-polygon-sofa} of \Cref{sec:balanced-maximum-sofa} is that each maximum polygon cap \(K = K_i\) is \emph{balanced} in the side lengths of \(K_i\) and its polygon niche. Using this balancedness condition, we will establish an upper bound of the surface area measure \(\sigma_{K}\) of a maximum polygon cap \(K\) with \(n\) steps of step size \(\delta = (\pi/2) / n\) in \Cref{thm:balanced-discrete-ineq}. Then we will take the limit \(n \to \infty\) and accordingly \(\delta \to 0\), so that the polygon cap \(K\) converges to the balanced maximum cap \(K_{\infty}\) with its surface area measure \(\sigma_{K_\infty}\) bounded from above.

\begin{lemma}

Any maximum polygon cap \(K\) with \(n\) steps have the diameter at most \(5\). Consequently, the functions \(f_K^{\pm}\) and \(g_K^{\pm}\) are bounded from above by 5.

\label{lem:leg-bounded}
\end{lemma}

\begin{proof}
Since \(n > 1\) is a power of two by \Cref{def:simple-notation-balanced-cap}, the angle set \(\Theta_n\) of \(K\) contains the angle \(\pi/4\). By \Cref{thm:limiting-maximum-cap-connected}, the cap \(K\) satisfies the condition (1) of \Cref{thm:monotonization-connected-iff}. So the cap \(K\) also satisfies the condition (3) of \Cref{thm:monotonization-connected-iff}, and the \(y\)-coordinate of \(\mathbf{x}_K(\pi/4)\) is at most one. Now \(K \subseteq H \cap Q_K^+(\pi/4)\) and the polygon \(H \cap Q_K^+(\pi/4)\) have diameter at most \(2 + 2\sqrt{2} < 5\). So \(K\) have the diameter at most 5.

By \Cref{def:cap-tangent-arm-length}, for any \(t \in [0, \pi/2]\) we have \(A_K^{\pm}(t) - C_K^{\pm}(t) = - f_K^{\pm}(t) v_t + g_K^{\pm}(t) u_t\) and since \(A_K^{\pm}(t), C_K^{\pm}(t)\) are in \(K\), the values \(f_K^{\pm}(t)\) and \(g_K^{\pm}(t)\) are also bounded by the diameter of \(K\) which is at most 5.
\end{proof}

\begin{definition}

For any cap \(K \in \mathcal{K}^\mathrm{c}\) and angle \(t \in [0, \pi/2]\), define
\[
H_K^\mathrm{b}(t) := H_+(t, h_K(t) - 1)
\]
which is the closed half-plane with normal angle \(t + \pi\) bounded from below by \(b_K(t)\). Likewise, define
\[
H_K^\mathrm{d}(t) := H_+(t + \pi/2, h_K(t + \pi/2) - 1)
\]
which is the closed half-plane with normal angle \(t + 3\pi/2\) bounded from below by \(d_K(t)\).

\label{def:upper-half-planes}
\end{definition}

\begin{figure}
\centering
\includegraphics{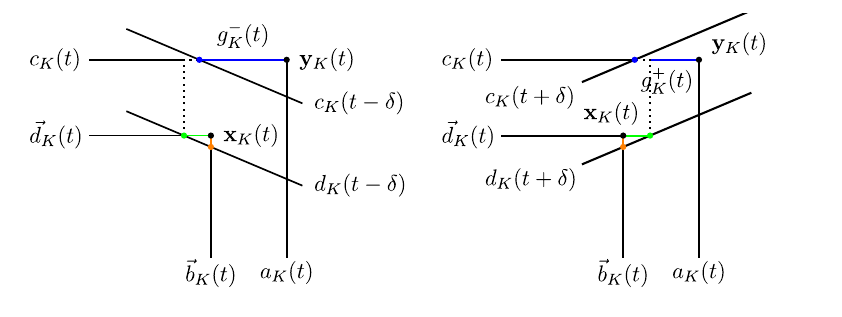}
\caption{The hallway \(L_K(t)\) depicted in upright position with sides \(c_K(t \pm \delta)\) and \(d_K(t \pm \delta)\) for the proof of \Cref{lem:leg-computation}. The points \(p, q, r\) are colored as orange, green, blue respectively.}
\label{fig:leg-computation}
\end{figure}

\begin{lemma}

For any maximum polygon cap \(K\) with the angle set \(\Theta := \Theta_n\) of step size \(\delta := (\pi/2)/n\), and any angle \(t \in \Theta\), we have the following calculations.

\begin{enumerate}
\def\labelenumi{\arabic{enumi}.}
\tightlist
\item
  \(\mathcal{H}^1\left( \vec{b}_K(t) \cap H_K^\mathrm{d}(t - \delta) \right) = \tan \delta \cdot \max(0, g_K^-(t) - 1 + \tan(\delta / 2))\)
\item
  \(\mathcal{H}^1 \left( \vec{b}_K(t) \cap H_K^\mathrm{d}(t + \delta) \right) = \tan \delta \cdot \max(0, 1 - g_K^+(t) + \tan(\delta / 2))\)
\end{enumerate}

\label{lem:leg-computation}
\end{lemma}

\begin{proof}
(1) See the left side of \Cref{fig:leg-computation}. Let \(p\) be the intersection (orange) of \(d_K(t - \delta)\) and \(b_K(t)\). Since \(p \in b_K(t)\), there is a unique real value \(\alpha\) such that \(p + \alpha v_t = \mathbf{x}_K(t)\). As \(\vec{b}_K(t)\) is the half-line from \(\mathbf{x}_K(t)\) extending in the direction of \(-v_t\), we have \(\mathcal{H}^1\left( \vec{b}_K(t) \cap H_K^\mathrm{d}(t - \delta) \right) = \max(\alpha, 0)\) and it remains to evaluate \(\alpha\). Define \(q\) as the intersection of \(d_K(t + \delta)\) and \(d_K(t)\), depicted green in the figure, and \(\beta\) as the unique real value such that \(q + \beta u_t = \mathbf{x}_K(t)\). Since the points \(p\), \(q\), \(\mathbf{x}_K(t)\) form a right-angled triangle of angle \(\delta\) at \(q\), we have \(\alpha = \tan \delta \cdot \beta\). It remains to compute \(\beta\).

Let \(r\) be the intersection of \(c_K(t - \delta)\) and \(c_K(t)\), depicted blue in the figure. Since the lines \(c_K(t)\) and \(d_K(t)\) are parallel of distance one, and so are \(c_K(t - \delta)\) and \(d_K(t - \delta)\), and the two pairs of lines make an angle of \(\delta\), follow the dashed lines in the figure and we get \(r = q + v_t + \tan(\delta / 2) u_t\). As \(K\) have angle set \(\Theta\) and the angles \(t\) and \(t + \delta\) are adjacent, we have \(r = C_K^-(t)\). By \Cref{pro:cap-tangent-arm-length} we have \(r + g_K^-(t) u_t = \mathbf{y}_K(t)\). Summing up, we now have
\begin{align*}
\beta & = (\mathbf{x}_K(t) - q) \cdot u_t = (\mathbf{x}_K(t) + v_t + \tan(\delta/2) u_t - r) \cdot u_t \\
& = (\mathbf{x}_K(t) + v_t + \tan(\delta / 2) u_t - \mathbf{y}_K(t) + g_K^-(t)u_t) \cdot u_t \\
& = g_K^-(t) + \tan(\delta/2) - 1
\end{align*}
and the result follows from \(\mathcal{H}^1\left( \vec{b}_K(t) \cap H_K^\mathrm{d}(t - \delta) \right) = \max(\alpha, 0) = \tan \delta \cdot \max(0, \beta)\).

(2) The proof is analogous to that of (1). See the right side of \Cref{fig:leg-computation}. Let \(p\) be the intersection of \(d_K(t + \delta)\) and \(b_K(t)\). Since \(p \in b_K(t)\) there is a unique real value \(\alpha\) such that \(p + \alpha v_t = \mathbf{x}_K(t)\). As \(\vec{b}_K(t)\) is the half-line from \(\mathbf{x}_K(t)\) extending in the direction of \(-v_t\), we have \(\mathcal{H}^1 \left( \vec{b}_K(t) \cap H_K^\mathrm{d}(t + \delta) \right) = \max(\alpha, 0)\) and it remains to evaluate \(\alpha\). Define \(q\) as the intersection of \(d_K(t + \delta)\) and \(d_K(t)\), and \(\beta\) as the unique real value such that \(\mathbf{x}_K(t) + \beta u_t = q\). Since the points \(p\), \(q\), \(\mathbf{x}_K(t)\) form a right-angled triangle of angle \(\delta\) at \(q\), we have \(\alpha = \tan \delta \cdot \beta\). It remains to compute \(\beta\).

Let \(r\) be the intersection of \(c_K(t + \delta)\) and \(c_K(t)\). Since the lines \(c_K(t)\) and \(d_K(t)\) are parallel of distance one, and so are \(c_K(t + \delta)\) and \(d_K(t + \delta)\), and the two pairs of lines make an angle of \(\delta\), we have \(r = q + v_t - \tan(\delta / 2) u_t\). As \(K\) have angle set \(\Theta\) and the angles \(t\) and \(t + \delta\) are adjacent, we have \(r = C_K^+(t)\). By \Cref{pro:cap-tangent-arm-length} we have \(r + g_K^+(t) u_t = \mathbf{y}_K(t)\). Summing up, we now have
\begin{align*}
\beta & = (q - \mathbf{x}_K(t)) \cdot u_t = (- \mathbf{x}_K(t) - v_t + \tan(\delta/2) u_t + r) \cdot u_t \\
& = (-\mathbf{x}_K(t) - v_t + \tan(\delta / 2) u_t + \mathbf{y}_K(t) - g_K^+(t)u_t) \cdot u_t \\
& = 1 - g_K^+(t) + \tan(\delta/2)
\end{align*}
and the result follows from \(\mathcal{H}^1 \left( \vec{b}_K(t) \cap H_K^\mathrm{d}(t + \delta) \right) = \max(\alpha, 0) = \tan \delta \cdot \max(0, \beta)\).
\end{proof}

We introduce the following auxiliary real functions. Note that \(m_0\) is monotonically increasing, and is piecewise linear on the intervals \([0, 1]\), \([1, 2]\), and \([2, \infty)\) with values \(m_0(0) = -1\), \(m_0(1) = 1/2\), \(m_0(x) = 1\) for \(x \geq 2\).

\begin{definition}

Define \(k_0, m_0 : \mathbb{R}_{\geq0} \to \mathbb{R}\) as \(k_0(x) := \max\left( |x - 1|, (|x - 1| + 1) / 2 \right)\) and \(m_0(x) = x - k_0(x)\).

\label{def:magic-function}
\end{definition}

We now bound the side lengths of a maximum polygon cap using balancedness. This is the most important part of the analysis.

\begin{theorem}

Let \(K \in \mathcal{K}_\Theta\) be any maximum polygon cap of \(n\) steps with step size \(\delta = (\pi/2) / n\) and uniform angle set \(\Theta := \Theta_n\). Take any \(t \in \left\{ 0 \right\} \cup \Theta\). We have the following upper bound of side length \(\sigma_K(t)\).
\[
\sigma_K(t) \leq k_0(g_K^+(t)) \cdot \delta + O(\delta^2)
\]

\label{thm:balanced-discrete-ineq}
\end{theorem}

\begin{proof}
If \(t = 0\), then as \(K\) have angle set \(\Theta\) and rotation angle \(\omega = \pi/2\), it is an intersection of a finite number of half-planes with normal angles \(\neq 0 \in S^1\) (see \Cref{def:angled-cap-space}). So we have \(\sigma_K(t) = 0\) and the bound holds trivially. Now assume \(t \in \Theta\).

Write \(s := \sigma_K(t)\). Because \(K\) is balanced by \Cref{thm:balanced-polygon-sofa}, we have \(s = \tau_K(t)\). By (1) of \Cref{lem:polyline-length}, the value \(\tau_K(t) = \mathcal{H}^1(X)\) is the length of the set \(X := \vec{b}_K(t) \cap \partial \mathcal{N}_\Theta(K)\), the side(s) of the polygon niche \(\mathcal{N}_\Theta(K)\) contributed by the half-line \(\vec{b}_K(t)\) form \(\mathbf{x}_K(t)\). We will bound \(\mathcal{H}^1(X)\) from above, thus bounding \(s\).

Define the set \(U := \left\{ t - \delta, t, t + \delta \right\}\) and \(R := \bigcup_{u \in U} H_K^\mathrm{d} (u)\) as the union of three half-planes. Divide \(X\) into \(X \cap R\) and \(X \setminus R\). We bound \(\mathcal{H}^1(X \cap R)\) and \(\mathcal{H}^1(X \setminus R)\) from above separately.

We first bound \(\mathcal{H}^1(X \cap R)\) from above. As \(X \subseteq \vec{b}_K(t)\), we have
\[
\mathcal{H}^1(X \cap R) \leq \mathcal{H}^1\left(\vec{b}_K(t) \cap R\right) = \max\left(\vec{b}_K(t) \cap H_K^\mathrm{d}(t - \delta), \vec{b}_K(t) \cap H_K^\mathrm{d}(t + \delta)\right)
\]
and by computations in \Cref{lem:leg-computation} and \Cref{lem:leg-bounded}, we have
\begin{align*}
\mathcal{H}^1(X \cap R) & \leq \delta \cdot \max(0, g_K^-(t) - 1, 1 - g_K^+(t)) + O(\delta^2).
\end{align*}
Let \(M := \max(0, g_K^-(t) - 1, 1 - g_K^+(t))\). We have \(g_K^-(t) \leq g_K^+(t)\) by definition. So if \(g_K^-(t) \leq 1\) then \(M = \max(0, 1 - g_K^+(t)) \leq |1 - g_K^+(t)|\). If \(g_K^-(t) > 1\) then we have \(M = \max(0, g_K^-(t) - 1) \leq |1 - g_K^+(t)|\). Either way, we have \(M \leq |1 - g_K^+(t)|\) and
\begin{equation}
\label{eqn:bound-x-r}
\mathcal{H}^1(X \cap R) \leq \left| g_K^+(t) - 1 \right| \cdot \delta + O(\delta^2).
\end{equation}

Now we bound \(\mathcal{H}^1(X \setminus R)\) from above. Define the closed region \(S := \bigcap_{u \in U} H_K^\mathrm{b} (u)\). Our intermediate goal is to show \(\mathcal{H}^1(X \setminus R) \leq \mathcal{H}^1(b_K(t) \cap S)\). Define \(H_0 := H_+(\pi/2, 0)\) as the closed half-plane bounded from below by the line \(y=0\). Let \(p_0\) be the intersection of \(b_K(t)\) and the line \(y = 0\). Because \(\mathcal{N}_\Theta(K)\) is open in the subspace topology of \(H_0\), we have \(X \setminus \left\{ p_0 \right\} \subseteq H_0 \setminus \mathcal{N}_\Theta(K)\). For each \(u \in U\), we have \(Q_K^-(u) \cup H_K^\mathrm{b}(u) \cup H_K^\mathrm{d}(u) = \mathbb{R}^2\) so
\begin{align*}
X \setminus \left\{ p_0 \right\} & \subseteq H_0 \setminus \mathcal{N}_\Theta(K) \subseteq H_0 \setminus \bigcup_{u \in U} Q_K^-(u) \\
& = H_0 \cap \bigcap_{u \in U} (H_K^\mathrm{b}(u) \cup H_K^\mathrm{d}(u)) \\
& \subseteq R \cup S. 
\end{align*}
So \(X \setminus \left\{ p_0 \right\} \setminus R \subseteq S\) and we have \(\mathcal{H}^1(X \setminus R) \leq \mathcal{H}^1(b_K(t) \cap S)\).

Define the intersections \(B_+ := b_K(t) \cap b_K(t + \delta)\) and \(B_- := b_K(t) \cap b_K(t - \delta)\). There is a unique real value \(\beta\) such that \(B_+ = B_- - \beta v_t\). Elementary geometry shows that the side \(b_K(t) \cap S\) of \(S\) is empty if \(\beta < 0\), or is a finite segment of length \(\beta\) if \(\beta \geq 0\). Another elementary geometry shows that \(\beta = 2 \tan (\delta / 2) - \sigma_K(t)\). So
\begin{equation}
\label{eqn:bound-x-not-r}
\mathcal{H}^1(X \setminus R) \leq \max(0, 2 \tan (\delta / 2) - \sigma_K(t)) = \max(0, \delta - s) + O(\delta^2).
\end{equation}

Recall that \(s = \nu_K(t) = \mathcal{H}^1(X)\). We divide the proof into cases on whether \(s \leq \delta\) or not. If \(s \leq \delta\), then adding \Cref{eqn:bound-x-r} and \Cref{eqn:bound-x-not-r} gives
\[
s \leq \left| g_K^+(t) - 1 \right| \cdot \delta + \delta - s + O(\delta^2)
\]
so rearranging gives
\[
s \leq (\left| g_K^+(t) - 1 \right| + 1)/2 \cdot \delta + O(\delta^2) \leq k_0(g_K^+(t)) \cdot \delta + O(\delta^2).
\]
On the other hand, if \(s > \delta\), then adding \Cref{eqn:bound-x-r} and \Cref{eqn:bound-x-not-r} gives
\[
s \leq \left| g_K^+(t) - 1 \right| \cdot \delta + O(\delta^2) \leq k_0(g_K^+(t)) \cdot \delta + O(\delta^2).
\]
Either way, the claimed inequality holds.
\end{proof}

\section{Inequality on Balanced Maximum Caps}
\label{sec:inequality-on-balanced-maximum-caps}
In this \Cref{sec:inequality-on-balanced-maximum-caps}, we prove the inequality \(\sigma_K \leq k_0(g_K(t)) \, \mathrm{d} t\) on the interval \(t \in [0, \pi/2)\) for balanced maximum caps \(K\) in \Cref{thm:balanced-ineq-limit}. This is done by taking the limit of \Cref{thm:balanced-discrete-ineq} on maximum polygon caps. Using this, we prove the inequality \(f'_K(t) \geq m_0(g_K(t))\) on the arm lengths of \(K\) in \Cref{thm:leg-length-bounds}.

\begin{lemma}

Let \(K\) be any maximum polygon cap with \(n\) steps of step size \(\delta := (\pi/2) / n\) and angle set \(\Theta := \Theta_n\). For any angle \(t \in \left\{ 0 \right\} \cup \Theta\) and open interval \(I := (t,t + \delta)\), the followings are true.

\begin{enumerate}
\def\labelenumi{\arabic{enumi}.}
\tightlist
\item
  For any \(t' \in I\), we have \(g_K^+(t) \geq g_K^{+}(t') = g_K^-(t') \geq g_K^-(t + \delta)\).
\item
  Moreover, the gap \(g_K^+(t) - g_K^-(t + \delta)\) between the upper and the lower bound is at most \(5 \delta\).
\end{enumerate}

\label{lem:arm-length-discrete-bound}
\end{lemma}

\begin{proof}
Since the angles \(t\) and \(t + \delta\) are two adjacent angles of the finite set \(\left\{ 0, \pi/2 \right\} \cup \Theta\), we have
\[
A_K^+(t) = A_K^-(t + \delta) = A_K^+(t') = A_K^-(t')
\]
for all \(t' \in I\). Call the common point \(A\). Similarly, we have
\[
C_K^+(t) = C_K^-(t + \delta) = C_K^{+}(t') = C_K^{-}(t')
\]
for all \(t' \in I\). Call the common point \(C\). Take any \(t' \in \overline{I} = [t, t + \delta]\). The lines \(a_K(t')\) and \(c_K(t')\) meet orthogonally at the point \(\mathbf{y}_K(t')\). The line \(a_K(t')\) passes through \(A\), and the line \(c_K(t')\) passes through \(C\). So by elementary geometry, the trajectory of the point \(\mathbf{y}_K(t')\) over all \(t' \in \overline{I}\) forms an arc \(\Delta\) of the circle \(\Gamma\) with the diameter of length \(\leq 5\) (\Cref{lem:leg-bounded}) connecting \(A\) and \(C\). As the lines \(a_K(t)\) and \(a_K(t + \delta)\) makes an angle of \(\delta\) at point \(A\), the arc \(\Delta\) have central angle \(2\delta\) in \(\Gamma\). The value \(g_K^+(t)\) is the distance from \(C\) to \(\mathbf{y}_K(t)\), and \(g_K^-(t + \delta)\) is the distance from \(C\) to \(\mathbf{y}_K(t + \delta)\). So (1) holds. The arc \(\Delta\) connecting \(\mathbf{y}_K(t)\) to \(\mathbf{y}_K(t + \delta)\) have length at most \(5\delta\). So (2) holds by triangle inequality.
\end{proof}

\begin{lemma}

Assume that a sequence of polygon caps \(K_n\) converge to a cap \(K \in \mathcal{K}^\mathrm{c}\) in Hausdorff distance as \(n \to \infty\). Then we have
\[
\lim_{ n \to \infty } \int_0^{\pi/2} \left| g_{K_n}^+(t) - g_K^+(t) \right| \, dt = 0. 
\]

\label{lem:leg-convergence}
\end{lemma}

\begin{proof}
The function \(g_{K_n}^+\) is nonnegative and bounded from above by the diameter of \(K_n\). The diameter of \(K_n\) converges to the diameter of \(K\). So by dominated convergence theorem, it suffices to show that for almost every \(t \in (0, \pi/2)\), we have \(g_{K_n}^+(t) \to g_K^+(t)\) as \(n \to \infty\).

Note that the set of all \(t \in (0, \pi/2)\) such that either \(\sigma_{K}\left( \left\{ t \right\} \right) > 0\) or \(\sigma_{K_n}\left( \left\{ t \right\} \right) > 0\) for some \(n\) is countable; otherwise, it contradicts that the measures \(\sigma_{K_n}\) and \(\sigma_K\) are finite. So we can exclude such measure zero case and assume that \(\sigma_{K}\left( \left\{ t \right\} \right) = 0\) and \(\sigma_{K_n}\left( \left\{ t \right\} \right) = 0\) for all \(n \geq 1\). By our choice of \(t\), we have \(g_{K_n}^-(t) = g_{K_n}^+(t)\) and \(g_K^-(t) = g_K^+(t)\). We now show \(g_{K_n}^+(t) \to g_K^+(t)\) as \(n \to \infty\) for such \(t\).

The function \(s : S^1 \to \mathbb{R}\), defined as \(s(u) := \sin(u - t)\) for \(u \in (t, t + \pi/2]\) and zero otherwise, is upper semicontinuous. By \Cref{lem:arm-length-convolution}, the value \(g_{K_n}^+(t)\) (resp. \(g_K^+(t)\)) is the integral of \(s\) over \(\sigma_{K_n}\) (resp. \(\sigma_K\)). Because the measure \(\sigma_{K_n}\) converges weakly to \(\sigma_K\) as \(n \to \infty\) (\Cref{thm:surface-area-weak-convergence}), by the Portmanteau theorem on finite measures we have \(\limsup_{ n } g_{K_n}^+(t) \leq g_K^+(t)\).

Similarly, define \(s^- : S^1 \to \mathbb{R}\) as \(s(u) := \sin(u - t)\) for \(u \in (t, t + \pi/2)\) and zero otherwise. Then \(s^-\) is lower semicontinuous. By \Cref{lem:arm-length-convolution} and \Cref{pro:surface-area-measure-side-length}, the value \(g_{K_n}^-(t)\) (resp. \(g_K^-(t)\)) is the integral of \(s^-\) over \(\sigma_{K_n}\) (resp. \(\sigma_K\)). Now by the Portmanteau theorem on finite measures we have \(\liminf_{ n } g_{K_n}^-(t) \geq g_K^-(t)\). Because \(g_{K_n}^-(t) = g_{K_n}^+(t)\) and \(g_K^-(t) = g_K^+(t)\), we get \(g_{K_n}^+(t) \to g_K^+(t)\) as \(n \to \infty\), completing the proof.
\end{proof}

We now take the limit of \Cref{thm:balanced-discrete-ineq} to show a corresponding inequality for the balanced maximum cup \(K\).

\begin{theorem}

Let \(K\) be any balanced maximum cap. Then on the interval \(t \in [0, \pi/2)\), we have \(\sigma_K \leq k_0(g_K^+(t)) \, \mathrm{d} t\).

\label{thm:balanced-ineq-limit}
\end{theorem}

\begin{proof}
Let \(\mu_K\) be the measure \(k_0(g_K^+(t)) \, \mathrm{d} t\) on \(t \in [0, \pi/2)\). It suffices to show
\begin{equation}
\label{eqn:to-show-interval}
\sigma_K(I) \leq \mu_K(I) = \int_{t \in I} k_0(g_K^+(t)) \, d t
\end{equation}
for any open interval \(I = (a, b) \subseteq [0, \pi/2)\) and half-open interval \(I = [0, b) \subseteq [0, \pi/2)\). Then for any set \(U\) open in the subspace topology of \([0, \pi/2)\) is a countable union of such intervals, so we have \(\sigma_K(U) \leq \int_{t \in U} k_0(g_K^+(t))\,dt\). Since \(k_0(g_K^+(t)) \mathrm{d} t\) is outer regular, by letting \(U\) to converge to arbitrary Borel subset of \([0, \pi/2)\) from above, we have \(\sigma_K \leq k_0(g_K^+(t)) \, \mathrm{d} t\) on \([0,\pi/2)\).

By \Cref{def:balanced-maximum-cap}, \(K\) is the limit of maximum polygon caps \(K_n\) of \(n\) steps with step size \(\delta = (\pi/2) / n\) as \(n \to \infty\), where \(n = n_1, n_2, \dots\) takes values in a strictly increasing powers of two. Let \(\Theta := \bigcup_n \Theta_n\) be the union of angle sets \(\Theta_n\) for powers of two \(n = n_1, n_2, \dots\), so that \(\Theta\) is the set of dyadic angles in \((0, \pi/2)\). As \(\Theta\) is dense in \([0, \pi/2)\), it suffices to show \Cref{eqn:to-show-interval} for \(I = (a, b)\) and \(I = [0, b)\) for \(a, b \in \Theta\) with \(a < b\) by taking limits.

Take sufficiently large \(n\) so that we can assume \(a, b \in \Theta_n\). For every \(t \in \left\{ 0 \right\} \cup \Theta_n\), we have
\begin{equation}
\label{eqn:single-step-integral}
\sigma_{K_n}\left( [t, t + \delta) \right) = \sigma_{K_n} (t) \leq k_0(g_{K_n}^+(t)) \cdot \delta + O(\delta^2) =  \int_t^{t + \delta} k_0(g_{K_n}^+(u)) \, du + O(\delta^2)
\end{equation}
by \Cref{thm:balanced-discrete-ineq} and \Cref{lem:arm-length-discrete-bound}. For the open interval \(I := (a, b)\), sum up \Cref{eqn:single-step-integral} for all \(t \in [a, b) \cap \Theta_{n}\) to get
\begin{equation}
\label{eqn:polygon-integral}
\begin{split}
\sigma_{K_n}(I) & \leq \sum_{t \in [a, b) \cap \Theta_n} \sigma_{K_n}\left( [t, t+ \delta) \right) = \sum_{t \in [a, b) \cap \Theta_n} \sigma_{K_n}\left( \left\{ t \right\}  \right) \\
& \leq \left( \sum_{t \in [a, b) \cap \Theta_n} \int_t^{t + \delta} k_0(g_{K_n}^+(u)) \, du \right)  + O_{\delta}(\delta) \\
& = \int_a^b k_0(g_{K_n}^+(u))\,du + O_{\delta}(\delta) = \mu_{K_n}(I) + O(\delta).
\end{split}
\end{equation}
Here we use that the size of \([a, b) \cap \Theta_n\) is \(\leq n = O(1/\delta)\).

We will now take \(n \to \infty\) in \Cref{eqn:polygon-integral}, first at the left-hand side. As \(K_n \to K\) in Hausdorff distance, \(\sigma_{K_n} \to \sigma_K\) in the weak convergence of measures (\Cref{thm:surface-area-weak-convergence}). Since \(I\) is open in \(S^1\), we have \(\sigma_K(I) \leq \liminf_{ n } \sigma_{K_n}(I)\). Now we take \(n \to \infty\) in the right-hand side. By \Cref{lem:leg-convergence} and that \(k_0\) is 1-Lipschitz, we have
\begin{equation}
\label{eqn:mu-limit}
\begin{split}
\lim_{ n } \left| \mu_{K_n}(I) - \mu_K(I) \right| & \leq \lim_{ n }  \int_{t \in I} \left| k_0(g_{K_n}^+(t)) - k_0(g_{K_n}^+(t)) \right|  \, d t \\
& \leq \lim_{ n } \int_{t \in I} \left|g_{K_n}^+(t) - g_{K_n}^+(t) \right| = 0.
\end{split}
\end{equation}
So by taking \(n \to \infty\) in \Cref{eqn:polygon-integral}, we get \(\sigma_K(I) \leq \mu_K(I)\) for \(I = (a, b)\) as desired.

Similarly, for the interval \(I := [0, b)\), set \(a = 0\) and observe that \Cref{eqn:polygon-integral} still holds. Because \(K_n\) and \(K\) are caps, if we set \(U = (-\pi/2, b)\) then we have \(\sigma_{K_n}(I) = \sigma_{K_n}(U)\) and \(\sigma_K(I) = \sigma_K(U)\). So we still have \(\sigma_K(I) \leq \liminf_{ n } \sigma_{K_n}(I)\). We also have \(\lim_{ n } \mu_{K_n}(I) = \mu_K(I)\) by following \Cref{eqn:mu-limit}. So by taking \(n \to \infty\) in \Cref{eqn:polygon-integral}, we have \(\sigma_K(I) \leq \mu_K(I)\) for \(I = [0, b)\) as desired.
\end{proof}

\begin{corollary}

Any balanced maximum cap \(K\) satisfies the condition (1) of \Cref{def:injectivity-condition}.

\label{cor:cap-nondegenerate}
\end{corollary}

\begin{proof}
For \(t \in [0, \pi/2)\), use \Cref{thm:balanced-ineq-limit} and the Radon-Nikodym theorem to prove the condition (4) and thus (1). For \(t \in (\pi/2, \pi]\), use a mirror-symmetric argument (\Cref{pro:balanced-maximum-cap-mirror}).
\end{proof}

Note that the condition (1) of \Cref{def:injectivity-condition} implies that \(A_K^+(t) = A_K^-(t)\) and \(f_K^+(t) = f_K^-(t)\) for all \(t \in [0, \pi/2)\). In the rest of this paper and the upcoming work, we will denote the common value as \(A_K(t)\) and \(f_K(t)\) respectively.

\begin{proposition}

Assume that a cap \(K \in \mathcal{K}^\mathrm{c}\) satisfies the condition (1) of \Cref{def:injectivity-condition}. Then the followings are true.

\begin{enumerate}
\def\labelenumi{\arabic{enumi}.}
\tightlist
\item
  For any \(t \in [0, \pi/2)\), we have \(A_K^+(t) = A_K^-(t)\) and \(f_K^+(t) = f_K^-(t)\).
\item
  For any \(t \in (0, \pi/2]\), we have \(C_K^+(t) = C_K^-(t)\) and \(g_K^+(t) = g_K^-(t)\).
\end{enumerate}

\label{pro:cap-nondegenerate}
\end{proposition}

\begin{proof}
We have \(\sigma_K\left( \left\{ t \right\} \right) = 0\) for any \(t \in [0, \pi] \setminus \left\{ \pi/2 \right\}\). Now use \Cref{pro:surface-area-measure-side-length}.
\end{proof}

\begin{definition}

For any cap \(K \in \mathcal{K}^\mathrm{c}\) satisfying (1) of \Cref{def:injectivity-condition}, define \(A_K, C_K : [0, \pi/2] \to \mathbb{R}^2\) and \(f_K, g_K : [0, \pi/2] \to \mathbb{R}^2\) as below.

\begin{itemize}
\tightlist
\item
  For any \(t \in [0, \pi/2)\), denote the common values in (1) of \Cref{pro:cap-nondegenerate} as \(A_K(t) := A_K^{\pm}(t)\) and \(f_K(t) = f_K^{\pm}(t)\).
\item
  Also, define \(A_K(\pi/2) := A_K^-(\pi/2)\) and \(f_K(\pi/2) := f_K^-(\pi/2)\).
\item
  For any \(t \in (0, \pi/2]\), denote the common values in (2) of \Cref{pro:cap-nondegenerate} as \(C_K(t) := C_K^{\pm}(t)\) and \(g_K(t) = g_K^{\pm}(t)\).
\item
  Also, define \(C_K(0) := C_K^+(0)\) and \(g_K(0) := g_K^+(0)\).
\end{itemize}

\label{def:cap-nondegenerate}
\end{definition}

\begin{proposition}

For any cap \(K \in \mathcal{K}^\mathrm{c}\) satisfying the condition (1) of \Cref{def:injectivity-condition}, the followings are true.

\begin{enumerate}
\def\labelenumi{\arabic{enumi}.}
\tightlist
\item
  The functions \(A_K, C_K, f_K, g_K\) in \Cref{def:cap-nondegenerate} are continuous on \([0, \pi/2]\).
\item
  The corners \(\mathbf{x}_K, \mathbf{y}_K : [0, \pi/2] \to \mathbb{R}^2\) are continously differentiable, and we have \(\mathbf{x}_K'(t) = -(f_K(t) - 1) u_t + (g_K(t) - 1) v_t\) and \(\mathbf{y}_K'(t) = -f_K(t) u_t + g_K(t) v_t\).
\end{enumerate}

\label{pro:cap-nondegenerate-continuity}
\end{proposition}

\begin{proof}
(1) is by \Cref{thm:limits-converging-to-vertex}. (2) is by \Cref{thm:inner-corner-deriv}.
\end{proof}

Now it makes sense to use the notations \(f_K(t) = f_K^{\pm}(t)\), \(g_K(t) = g_K^{\pm}(t)\), \(r_K\) and \(s_K\) for balanced maximum caps \(K\) by \Cref{cor:cap-nondegenerate} and \Cref{def:cap-nondegenerate}.

\section{Bounding Arm Lengths}
\label{sec:bounding-arm-lengths}
We state \Cref{thm:balanced-ineq-limit} purely in terms of arm lengths of \(K\), following \Cref{sec:solving-the-differential-inequality}.

\begin{theorem}

Let \(K\) be a balanced maximum cap. Then \(f_K : [0, \pi/2] \to \mathbb{R}\) is absolutely continuous, and we have \(f'_K(t) \geq m_0(g_K(t))\) for almost every \(t \in [0, \pi/2]\) except measure zero set.

\label{thm:leg-length-bounds}
\end{theorem}

\begin{proof}
By \Cref{cor:cap-nondegenerate}, we have \(\sigma_K = r_K(t) \mathrm{d} t\) on the domain \(t \in [0, \pi/2)\) for some measurable function \(r_K : [0, \pi/2) \to \mathbb{R}_{\geq 0}\) that is unique up to null set. In particular, by \Cref{thm:balanced-ineq-limit} we can assume \(r_K(t) \leq k_0(g_K(t))\) on \(t \in [0, \pi/2)\). Now by \Cref{thm:arm-length-differentiation} we have \(\mathrm{d} f_K(t) = \left( g_K(t) - r_K(t) \right) \mathrm{d} t\) on \(t \in [0, \pi/2]\) (check the case \(t=0\) separately). So by \Cref{pro:lebesgue-stieltjes-abs-cont}, the function \(f_K(t)\) is absolutely continuous and its derivative is \(f_K'(t) = g_K(t) - r_K(t)\) for almost every \(t \in [0, \pi/2]\). Now check
\[
g_K(t) - r_K(t) \geq g_K(t) - k_0(g_K(t)) = m_0(g_K(t)).
\]
to conclude \(f'_K(t) \geq m_0(g_K(t))\).
\end{proof}

Using \Cref{thm:leg-length-bounds}, we will iteratively obtain better lower bounds of \(f_K\) and \(g_K\) for balanced maximum caps \(K\), and show the lower bound \(f_K(t), g_K(t) > 1\) in \Cref{thm:lower-bound-one}. Then we use the equation \(\mathbf{x}_K'(t) = -(f_K(t) - 1) u_t + (g_K(t) - 1) v_t\) of \Cref{pro:cap-nondegenerate-continuity} to prove the main \Cref{thm:injectivity}.

\begin{definition}

Definte the operator \(\mathcal{F}\) mapping any continuous function \(f : [0, \pi/2] \to \mathbb{R}\), to another continuous function \(\mathcal{F} f : [0, \pi/2] \to \mathbb{R}\) as
\[
\mathcal{F} f(x) = 1 + \int_0^x m_0(f(\pi/2 - u))\,du.
\]

\label{def:integral-operator}
\end{definition}

\begin{definition}

Define the continuous functions \(f_n : [0, \pi/2] \to \mathbb{R}\) for all integers \(n \geq 0\) as the following.

\begin{enumerate}
\def\labelenumi{\arabic{enumi}.}
\tightlist
\item
  \(f_0(x) := 0\) for all \(x \in [0, \pi/2]\).
\item
  \(f_{n+1}(x) := \max\left( f_n(x), \mathcal{F} f_n(x) \right)\) for all \(x \in [0, \pi/2]\).
\end{enumerate}

\label{def:lower-bound-sequence}
\end{definition}

\begin{lemma}

The following holds for any \(n \geq 0\). Let \(K \in \mathcal{K}\) be any balanced maximum cap. Then we have \(f_K(t) \geq f_n(t)\) for all \(t \in [0, \pi/2)\) and \(g_K(t) \geq f_n(\pi/2 - t)\) for all \(t \in (0, \pi/2]\).

\label{lem:lower-bound-sequence}
\end{lemma}

\begin{proof}
Induct on \(n\). The base case \(n = 0\) holds trivially. Now assume the inductive hypothesis that for any balanced maximum cap \(K'\), we have \(f_{K'}(t) \geq f_n(t)\) for all \(t \in [0, \pi/2)\) and \(g_{K'}(t) \geq f_n(\pi/2 - t)\) for all \(t \in (0, \pi/2]\).

Fix an arbitrary balanced maximum cap \(K\). For any \(t \in [0, \pi/2)\), we have
\begin{align*}
f_K(t) & = f_K(0) + \int_0^t f_K'(u) \,du \\
& \geq f_K(0) + \int_0^t m_0(g_K(u)) \,du \\
& \geq f_K(0) + \int_0^t m_0(f_K(\pi/2 - u)) \,du 
\end{align*}
by the inductive hypothesis on \(K' := K\) and \Cref{thm:leg-length-bounds}. This combined with the inductive hypothesis \(f_{K}(t) \geq f_n(t)\) on \(K' := K\) implies that \(f_K(t) \geq f_{n+1}(t)\).

Next, apply the inductive hypothesis on the mirror image \(K' := K^{\mathrm{m}}\). Then we also have \(f_{K^\mathrm{m}}(t) \geq f_{n+1}(t)\) for all \(t \in [0, \pi/2)\). This, by mirror symmetry, is equivalent to the other inequality \(g_K(t) \geq f_{n+1}(\pi/2 - t)\) for all \(t \in (0, \pi/2]\).
\end{proof}

\begin{definition}

For any constant \(c \in [0, 1]\), define the function \(j_c : [0, \pi/2] \to \mathbb{R}\) as \(j_c(x) := \max(1 - x, c)\).

\label{def:lower-bound-j}
\end{definition}

\begin{lemma}

Define the constant \(d_0 := 1/12\). Then for any \(c \in [0, 2/3]\), we have \(\mathcal{F} j_c(x) \geq j_{c + d_0}(x)\) on \(x \in [0, \pi/2]\).

\label{lem:lower-bound-j-iter}
\end{lemma}

\begin{proof}
Because \(m_0\) is bounded from below by \(-1\), it is always guaranteed that \(\mathcal{F} j_c (x) \geq 1 - x\). So to show that \(\mathcal{F} j_c(x) \geq j_{c + d_0}(x)\) on \(x \in [0, 1]\), it suffices to show \(\mathcal{F} j_c (x) \geq c + d_0\).

We show this by simply computing \(\mathcal{F} j_c(x)\). Since the range of \(j_c\) is in \([0, 1]\), we have
\begin{equation}
\label{eqn:jc-principal}
\begin{split}
\mathcal{F} j_c(x) & = 1 + \int_0^x m_0(j_c(\pi/2 - u)) \, du \\
& = 1 - x + \int_0^x \frac{3}{2} j_c (\pi/2 - u) \,du
\end{split}
\end{equation}
and it suffices to show that the value is \(\geq c + d_0\).

First assume the case \(x \leq \pi/2 - 1 + c\). Then \(\pi/2 - u \geq \pi/2 - x \geq 1 - c\) in the integral of \Cref{eqn:jc-principal} so
\[
\begin{split}
\mathcal{F} j_c(x) & = 1 - x + \frac{3}{2} c x = 1 - (1 - 3/2 \cdot c)x \\
& \geq 1 - (1 - 3/2 \cdot c)(\pi/2 - 1 + c) \\
& > c + d_0.
\end{split}
\]
In the last inequality, the difference is minimized at \(c = 7/6 - \pi/4\) with the value \(-1/8 + 3\pi/8 - 3\pi^2/32 > 0\).

It remains to check the case \(x > \pi/2 - 1 + c\). Then \Cref{eqn:jc-principal} evaluates to
\begin{align*}
\mathcal{F} j_c(x) & = 1 - x + \frac{3}{2} c (\pi/2 - 1 + c) + \int_{\pi/2 - 1 + c}^x \frac{3}{2}(1 - \pi/2 + u) \, du \\
& = 1 - x + \frac{3}{2} c (\pi/2 - 1 + c) + \frac{3}{4} \left( (x - \pi/2 + 1)^2 - c^2 \right).
\end{align*}
Minimize the quadratic polynomial over \(x \in \mathbb{R}\), then we have
\[
\mathcal{F} j_c (x) \geq \frac{5}{3} - \frac{\pi}{2} + \frac{3}{4} c (c + \pi - 2)
\]
where the equality holds at \(x = \pi/2 - 1/3\). We also have
\[
\frac{5}{3} - \frac{\pi}{2} + \frac{3}{4} c (c + \pi - 2) >  c + \frac{1}{12}
\]
because the difference of both sides is minimized at \(c = 5/3-\pi/2\) with the value \(-1/2 + 3\pi/4 - 3\pi^2/16 > 0\). This completes the proof.
\end{proof}

\begin{lemma}

For any functions \(f, g : [0, \pi/2] \to \mathbb{R}_{\geq 0}\) such that \(f \leq g\) on \([0, \pi/2]\), we also have \(\mathcal{F} f \leq \mathcal{F} g\).

\label{lem:operator-monotonicity}
\end{lemma}

\begin{proof}
The function \(m_0\) is monotonically increasing; now check the \Cref{def:integral-operator} of \(\mathcal{F}\).
\end{proof}

\begin{lemma}

We have \(f_{11}(x) > 1\) on \(x \in (0, 1]\).

\label{lem:lower-bound-threshold}
\end{lemma}

\begin{proof}
For any integer \(1 \leq m \leq 10\), we will show that \(f_m \geq j_{(m - 1)/12}\). The base case \(f_1 = j_0\) is done by simple computation. Assuming the inductive hypothesis \(f_m \geq j_{(m-1)/12}\) for \(m \leq 9\), the inductive step
\[
f_{m+1} \geq \mathcal{F} f_m \geq \mathcal{F} j_{(m-1)/12} \geq j_{m/12}
\]
can be done using \Cref{lem:operator-monotonicity} and \Cref{lem:lower-bound-j-iter}. So \(f_{10} \geq j_{9/12} > 2/3\) in particular.

Because \(m_0(y) > 0\) for all \(y > 2/3\), we now have
\[
f_{11}(x) \geq \mathcal{F} f_{10} (x) = 1 + \int_0^x m_0(f_{10}(\pi/2 - u)) \, du > 1
\]
for all \(x \in (0, \pi/2]\).
\end{proof}

\begin{theorem}

For any balanced maximum cap \(K\), we have \(f_K(t) > 1\) on \(t \in (0, \pi/2]\) and \(g_K(t) > 1\) on \(t \in [0, \pi/2)\).

\label{thm:lower-bound-one}
\end{theorem}

\begin{proof}
By \Cref{lem:lower-bound-sequence} and \Cref{lem:lower-bound-threshold}.
\end{proof}

Now we prove the injectivity condition for any balanced maximum cap.

\begin{proof}[Proof of \Cref{thm:injectivity}]
We check conditions (1)-(3) of \Cref{def:injectivity-condition} for balanced maximum caps \(K\). Condition (1) holds by \Cref{cor:cap-nondegenerate}. Condition (2) holds by (2) of \Cref{pro:cap-nondegenerate-continuity}. Finally, condition (3) holds because
\[
\mathbf{x}_K'(t) = -(f_K(t) - 1) u_t + (g_K(t) - 1) v_t
\]
by (2) of \Cref{pro:cap-nondegenerate-continuity} and \Cref{thm:lower-bound-one}.
\end{proof}

\chapter{Convex Domain and Convex Curves}
\label{sec:convex-domain-and-convex-curves}
This section prepares a minimal amount of technology needed for the next \Cref{sec:optimality-of-gerver's-sofa}.

\begin{itemize}
\tightlist
\item
  \Cref{sec:convex-domain} defines the general notion of \emph{convex domain} \(\mathcal{V}\) and linear/quadratic functionals on \(\mathcal{V}\). We show that if the directional derivative of a concave quadratic functional \(f\) is non-negative at \(V \in f\), then \(f\) attains the global maximum value at \(V\).
\item
  \Cref{sec:curve-area-functional} reviews the notion of a \emph{Jordan curve} \(\mathbf{x}\) that encloses a region \(R\) of the plane, then defines the \emph{curve area functional} \(\mathcal{J}(\mathbf{x})\) that measures the area of \(R\) using Green’s theorem. This will work for any continuous \(\mathbf{x} : [a, b] \to \mathbb{R}\) of bounded variation.
\item
  \Cref{sec:convex-curve} defines a \emph{convex curve} segment \(\mathbf{u}_K^{a, b}\) of the boundary of a planar convex body \(K\), and calculate its curve area functional.
\item
  \Cref{sec:mamikon's-theorem} establishes the \emph{Mamikon’s theorem} on a general convex body \(K\) that may have a non-differentiable boundary.
\end{itemize}
\section{Convex Domain}
\label{sec:convex-domain}
\begin{definition}

A \emph{convex domain} \(\mathcal{V}\) is a space with the barycentric operation \(c_\lambda : \mathcal{V} \times \mathcal{V} \to \mathcal{V}\) for all \(\lambda \in [0, 1]\), such that there is an embedding \(e : \mathcal{V} \to V\) to a convex subspace of a vector space \(V\) preserving \(c_\lambda\). That is, \(e\) is injective and \(e(c_{\lambda}(v_1, v_2)) = (1 - \lambda) e(v_1) + \lambda e(v_2)\).

\label{def:convex-spaces}
\end{definition}

\begin{remark}

Although we will not use this notion in our work, \Cref{def:convex-spaces} is equivalent to saying that \((\mathcal{V}, c_\lambda)\) is a \emph{cancellative convex space} (Theorem 2 of \autocite{stonePostulatesBarycentricCalculus1949}).

\label{rem:convex-spaces}
\end{remark}

\begin{definition}

Call a function \(f : \mathcal{V}_1 \to \mathcal{V}_2\) between convex domains \emph{convex-linear} if it preserves the barycentric operation \(c_\lambda\).

\label{def:convex-linear}
\end{definition}

The composition \(f \circ g\) of two convex-linear maps \(f, g\) would also be convex-linear, as it preserves \(c_\lambda\).

\begin{definition}

Call a function \(g : \mathcal{V}_1 \times \mathcal{V}_2 \to \mathcal{V}_3\) \emph{convex-bilinear} if the maps \(v \mapsto g(v_1, v)\) and \(v \mapsto g(v, v_2)\) are convex-linear for any fixed \(v_1 \in \mathcal{V}_1\) and \(v_2 \in \mathcal{V}_2\).

\label{def:convex-bilinear}
\end{definition}

\begin{definition}

Call \(h : \mathcal{V} \to \mathbb{R}\) a \emph{quadratic functional} on a convex domain \(\mathcal{V}\) if \(h(K) = g(K, K)\) for some convex-bilinear \(g : \mathcal{V} \times \mathcal{V} \to \mathbb{R}\).

\label{def:convex-space-quadratic}
\end{definition}

\begin{theorem}

The space \(\mathcal{K}\) of all planar convex bodies is a convex domain with the barycentric operation
\begin{align*}
c_\lambda(K_1, K_2) & := (1 - \lambda)K_1 + \lambda K_2 \\
& = \left\{ (1 - \lambda) p_1 + \lambda p_2 : p_1 \in K_1, p_2 \in K_2 \right\} 
\end{align*}
given by the Minkowski sum of convex bodies.

\label{thm:convex-body-space}
\end{theorem}

\begin{proof}
By Remark 1.7.7 of \autocite{schneider_2013}, the map \(K \mapsto h_K\) embeds the space \(\mathcal{K}\) to a convex cone of the space of all continuous functions from \(S^1\) to \(\mathbb{R}\), preserving the barycentric operations of each space.
\end{proof}

The values on convex body \(K \in \mathcal{K}\) appearing in \Cref{sec:planar-convex-body} are convex-linear in \(K\).

\begin{theorem}

The following values are convex-linear in \(K \in \mathcal{K}\).

\begin{enumerate}
\def\labelenumi{\arabic{enumi}.}
\tightlist
\item
  Support function \(h_K\)
\item
  For fixed constants \(a, b \in \mathbb{R}\) so that \(a < b < a + \pi\), the vertices \(v_K^\pm(a)\) and \(v_K(a, b)\).
\item
  Surface area measure \(\sigma_K\)
\end{enumerate}

\label{thm:convex-body-linear}
\end{theorem}

\begin{proof}
(1) follows from Theorem 1.7.5 (a) of \autocite{schneider_2013} that \(h_{K_1 + K_2} = h_{K_1} + h_{K_2}\), and that \(h_{a K } = a h_K\) for any \(a \geq 0\) which is easy to check. For (2), note that \(p := v_K(a, b)\) is the unique solution satisfying the equations \(p \cdot u_a = h_K(a)\) and \(p \cdot u_b = h_K(b)\). Let \(U\) be the constant \(2 \times 2\) matrix with column vectors \(u_a\) and \(u_b\). Then we have \(p = U^{-1} [h_K(a), h_K(b)]^T\) convex-linear in \(K\) by (1). Use \Cref{thm:limits-converging-to-vertex} to see that \(v_K^{\pm}(a)\), the limit of \(v_K(a, b)\) as \(b \to a\) in either direction, is linear too. (3) comes from (2) and \(\sigma_K = v_t \cdot \mathrm{d} v_K^+(t)\) which follows from \Cref{thm:boundary-measure}.
\end{proof}

The area \(|K|\) of \(K \in \mathcal{K}\) is a quadratic functional.

\begin{theorem}

(Remark 5.1.2, page 276 of \autocite{schneider_2013}) For any \(K \in \mathcal{K}\), we have
\[
|K| = \frac{1}{2} \int_{t \in S^1} h_K(t)\, \sigma_K(dt)
\]
which is quadratic in \(K\) by \Cref{thm:convex-body-linear}.

\label{thm:area-quadratic-expression}
\end{theorem}

\begin{definition}

For any quadratic functional \(f : \mathcal{V} \to \mathbb{R}\) on a convex space \(\mathcal{V}\) and \(K, K' \in \mathcal{V}\), define the \emph{directional derivative}
\[
Df(K; K') := \left. \frac{d}{d \lambda} \right|_{\lambda = 0} f(c_\lambda(K, K'))
\]
of \(f\) at \(K\) in the direction towards \(K'\).

\label{def:convex-space-directional-derivative}
\end{definition}

For any quadratic functional \(f\) and a fixed \(K \in \mathcal{V}\), the value \(Df(K; K')\) is well-defined and always a linear functional of \(K'\).

\begin{lemma}

Let \(f\) be a quadratic functional on a convex domain \(\mathcal{V}\), so that \(f(K) = h(K, K)\) for a convex-bilinear map \(h : \mathcal{V} \times \mathcal{V} \to \mathbb{R}\). Then we have the following for any \(K, K' \in \mathcal{V}\).
\[
Df(K; K') = h(K, K') + h(K', K) - 2 h (K, K)
\]
So the map \(Df(K; -) : \mathcal{V} \to \mathbb{R}\) is always well-defined and a linear functional.

\label{lem:derivative-calculation}
\end{lemma}

\begin{proof}
We have
\begin{equation}
\label{eqn:quadratic-functional}
\begin{split}
f(c_\lambda(K, K')) & = h(c_\lambda(K, K'), c_\lambda(K, K')) \\
& = (1 - \lambda)^2 h(K, K) + \lambda (1 - \lambda) \left( h(K, K') + h (K', K) \right) + \lambda^2 h(K', K')
\end{split}
\end{equation}
by bilinearity of \(h\). Take the derivative at \(\lambda = 0\).
\end{proof}

\begin{definition}

A functional \(f : \mathcal{V} \to \mathbb{R}\) on a convex domain \(\mathcal{V}\) is \emph{concave} (resp. \emph{convex}) if \(f(c_\lambda(K_1, K_2)) \geq (1 - \lambda) f(K_1) + \lambda f(K_2)\) (resp. \(f(c_\lambda(K_1, K_2)) \leq (1 - \lambda) f(K_1) + \lambda f(K_2)\)) for all \(K_1, K_2 \in \mathcal{V}\) and \(\lambda \in [0, 1]\).

\label{def:convex-space-concavity}
\end{definition}

To prove that \(K\) maximizes a concave quadratic functional \(f(K)\) on \(\mathcal{V}\), we only need to prove that \(Df(K; -)\) is a nonpositive linear functional on \(\mathcal{V}\).

\begin{theorem}

For any concave quadratic functional \(f\) on a convex domain \(\mathcal{V}\), the value \(K \in \mathcal{V}\) maximizes \(f(K)\) if and only if the convex-linear functional \(Df(K; -)\) is nonpositive.

\label{thm:quadratic-variation}
\end{theorem}

\begin{proof}
Assume that \(K\) is the maximizer of \(f(K)\). Then for any \(K' \in \mathcal{V}\), the value \(f(c_\lambda(K, K'))\) over all \(\lambda \in [0, 1]\) is maximized at \(\lambda = 0\). So taking the derivative at \(\lambda = 0\), we should have \(Df(K; K') \leq 0\).

Now assume on the other hand that \(K \in \mathcal{V}\) is chosen such that \(Df(K; -)\) is always nonpositive. Fix an arbitrary \(K' \in \mathcal{V}\). Observe that \(f(c_\lambda(K, K'))\) is a polynomial \(p(\lambda)\) of \(\lambda \in [0, 1]\) by \Cref{eqn:quadratic-functional}. Because \(f\) is concave, the polynomial \(p(\lambda)\) is also concave with respect to \(\lambda\) and the quadratic coefficient of \(p(\lambda)\) is nonpositive. The linear coefficient of \(p(\lambda)\) is \(Df(K; K')\) and this is nonpositive as well. So \(p(\lambda)\) is monotonically decreasing with respect to \(\lambda\) and we have \(f(K) \geq f(K')\) as desired.
\end{proof}

\begin{definition}

For any real-valued functionals \(f(V)\) and \(g(V)\) on the convex domain \(V \in \mathcal{V}\), write \(f(V) \equiv_V g(V)\) if and only if \(f(V) - g(V)\) is convex-linear in \(V \in \mathcal{V}\).

\label{def:modulo-linear}
\end{definition}

It is easy to check that the relation \(\equiv_V\) in \Cref{def:modulo-linear} is an equivalence relation on the real-valued functionals on convex domain \(\mathcal{V}\).

\begin{lemma}

Let \(h : \mathcal{V} \times \mathcal{V} \to \mathbb{R}\) be a convex-bilinear map on a convex domain \(\mathcal{V}\). Fix constants \(c_1, c_2 \in \mathcal{V}\). Define quadratic functionals \(f(K) = h(K, K)\) and \(g(K) = h(K + c_1, K + c_2)\), then we have \(f(K) \equiv_K g(K)\).

\label{lem:modulo-linear-const}
\end{lemma}

\begin{proof}
We have \(g(K) - f(K) = h(c_1, K) + h(K, c_2) + h(c_1, c_2)\) which is linear in \(K\).
\end{proof}

\section{Curve Area Functional}
\label{sec:curve-area-functional}
We first recall the notion of Jordan arc and curve following Chapter 8 of \autocite{apostolMathematicalAnalysisModern}. See the reference for details.

\begin{definition}

A \emph{Jordan arc} \(\Gamma\) is the image of a continuous and injective function \(\mathbf{x} : [a, b] \to \mathbb{R}^2\). We call \(\mathbf{x}\) a \emph{parametrization} of \(\Gamma\). A \emph{Jordan curve} \(\Gamma\) is the image of a continuous and injective function \(\mathbf{x} : S^1 \to \mathbb{R}^2\).

\label{def:jordan-curve}
\end{definition}

Equivalently, a Jordan curve \(\Gamma\) is the image of a continuous function \(\mathbf{x} : [a, b] \to \mathbb{R}^2\) with \(a < b\) which is injective on \([a, b)\) and \(\mathbf{x}(a) = \mathbf{x}(b)\). We allow \(a=b\) for Jordan arcs, but we require \(a \neq b\) for Jordan curves in order for the following famous theorem to hold.

\begin{theorem}

(Jordan curve theorem) The complement \(\mathbb{R}^2 \setminus \Gamma\) of a Jordan curve \(\Gamma\) is a disjoint union of two connected components \(U\) and \(V\) where \(U\) is bounded and \(V\) is unbounded. Moreover, \(U\) and \(V\) are open and \(\partial U = \partial V = X\). Say that \(U\) is the \emph{region enclosed by \(\Gamma\)}.

\label{thm:jordan-curve}
\end{theorem}

We give orientations to Jordan arcs and curves.

\begin{definition}

Let \(\Gamma\) be any Jordan arc. Take any parametrization \(\mathbf{x} : [a, b] \to \mathbb{R}^2\) of \(\Gamma\) and call the points \(\mathbf{x}(a)\) and \(\mathbf{x}(b)\) the \emph{endpoints} of \(\Gamma\). Note that the endpoints of \(\Gamma\) are independent of the choice of \(\mathbf{x}\). An \emph{oriented Jordan arc} \(\overrightarrow{\Gamma}\) is a Jordan arc \(\Gamma\) with one endpoint \(p\) marked as the \emph{starting point} and the other endpoint \(q\) marked as the \emph{ending point}. A \emph{parametrization} \(\mathbf{x} : [a, b]\to\mathbb{R}^2\) of an oriented Jordan arc \(\overrightarrow{\Gamma}\) is a parametrization of \(\Gamma\) that respects the order \(\mathbf{x}(a) = p\) and \(\mathbf{x}(b) = q\) of the endpoints of \(\overrightarrow{\Gamma}\).

\label{def:jordan-arc-orientation}
\end{definition}

\begin{definition}

Let \(\Gamma\) be any Jordan curve. An \emph{oriented Jordan curve} \(\overrightarrow{\Gamma}\) is the unoriented curve \(\Gamma\) with a \emph{clockwise} or \emph{counterclockwise} direction assigned. A \emph{parametrization} \(\mathbf{x} : S^1 \to \mathbb{R}^2\) of an oriented Jordan curve \(\overrightarrow{\Gamma}\) is a parametrization of \(\Gamma\) such that, for any point \(p\) inside the region enclosed by \(\Gamma\), the parametrization \(\mathbf{x}\) rotates around \(p\) in the specified direction.

\label{def:jordan-curve-orientation}
\end{definition}

We now define the \emph{curve area functional} \(\mathcal{J}(\mathbf{x})\) of \(\mathbf{x} : [a, b] \to \mathbb{R}^2\), so that if \(\mathbf{x}\) parametrizes a counterclockwise Jordan curve, then \(\mathcal{J}(\mathbf{x})\) is the area of the region enclosed by \(\mathbf{x}\).

\begin{definition}

For two vectors \(p = (a, b)\) and \(q = (c, d)\) in \(\mathbb{R}^2\), define their cross product \(p \times q := ad - b c \in \mathbb{R}\). For a pair \(p = (p_1, p_2)\) of bounded measurable functions and a pair \(\mu = (\mu_1, \mu_2)\) of finite signed measures on a set \(X\), define the finite signed measure \(p \times \mu := p_1 \, \mu_2 - p_2 \, \mu_1\) on \(X\).

\label{def:plane-cross-product}
\end{definition}

\begin{definition}

Let \(C^\mathrm{BV}[a, b]\) be the real vector space of all continuous maps of bounded variation from \([a, b]\) to \(\mathbb{R}^2\).

\label{def:bounded-variation-space}
\end{definition}

\begin{definition}

For any \(\mathbf{x} \in C^{\mathrm{BV}}[a, b]\), define its \emph{curve area functional} \(\mathcal{J}(\mathbf{x})\) as
\[
\mathcal{J}(\mathbf{x}) := \frac{1}{2} \int_a^b \mathbf{x}(t) \times d\mathbf{x}(t).
\]

\label{def:curve-area-functional}
\end{definition}

\begin{proposition}

\(\mathcal{J}(\mathbf{x})\) is quadratic on \(\mathbf{x} \in C^{\mathrm{BV}}[a, b]\).

\label{pro:curve-area-functional-quadratic}
\end{proposition}

\begin{proof}
\(\mathcal{J}(\mathbf{x}) = \mathcal{B}(\mathbf{x}, \mathbf{x})\) where \(\mathcal{B}(\mathbf{x}_1, \mathbf{x}_2) := \frac{1}{2} \int_a^b \mathbf{x}_1(t) \times d \mathbf{x}_2(t)\) is bilinear in \(\mathbf{x}_1, \mathbf{x}_2 \in C^{\mathrm{BV}}[a, b]\).
\end{proof}

Note that the integral in \Cref{def:curve-area-functional} is on the Lebesgue–Stieltjes measure \(\mathrm{d} \mathbf{x}(t)\) of \(\mathbf{x}\) taken the cross product with \(\mathbf{x}(t)\) as in \Cref{def:plane-cross-product}. Writing the coordinates of \(\mathbf{x}(t) = (x(t), y(t))\), we can write \(\mathrm{d} \mathbf{x} = (\mathrm{d} x, \mathrm{d} y)\) and \(\mathcal{J}(\mathbf{x})\) more explicitly as
\begin{equation}
\label{eqn:line-integral}
\mathcal{J}(\mathbf{x}) = \frac{1}{2} \int_a^b R_{\pi/2}(\mathbf{x}(t)) \cdot d\mathbf{x}(t) = \frac{1}{2} \int_a^b x(t)\, dy(t) - y(t)\, dx(t)
\end{equation}
instead, where \(R_{\pi/2}(x, y) = (-y, x)\) is the rotation of \(\mathbb{R}^2\) along the origin by \(\pi/2\).

\begin{theorem}

If \(\mathbf{x}\) is a rectifiable parametrization of a Jordan curve oriented counterclockwise, then \(\mathcal{J}(\mathbf{x})\) is the area of the region enclosed by \(\mathbf{x}\).

\label{thm:curve-area-functional-area}
\end{theorem}

\begin{proof}
Apply Green’s theorem (Theorem 10.43, page 289 of \autocite{apostolMathematicalAnalysisModern}) on the curve \(\mathbf{x}\) and vector field \((P, Q) = (-y, x)\).
\end{proof}

\begin{remark}

If \(\mathbf{x}\) is not closed (that is, \(\mathbf{x}(a) \neq \mathbf{x}(b)\)), the sofa area functional \(\mathcal{J}(\mathbf{x})\) measures the signed area of the region bounded by the curve \(\mathbf{x}\), and two line segments connecting the origin to \(\mathbf{x}(a)\) and \(\mathbf{x}(b)\) respectively.

\label{rem:curve-area-functional-segment}
\end{remark}

Fix an oriented Jordan arc \(\Gamma\) and take any parametrization \(\mathbf{x}\) of \(\Gamma\). Then the value \(\mathcal{J}(\mathbf{x})\) is a line integral on \(\Gamma\) (\Cref{eqn:line-integral}), so the value of \(\mathcal{J}(\mathbf{x})\) is independent of the choice of \(\mathbf{x}\). Similarly, any parametrization of an oriented Jordan curve \(\Gamma\) are circular shifts of each other, so the value \(\mathcal{J}(\mathbf{x})\) is fixed.

\begin{definition}

For any oriented Jordan arc or curve \(\Gamma\), the value of \(\mathcal{J}(\mathbf{x})\) is independent of the choice of parametrization \(\mathbf{x}\) of \(\Gamma\); call this value \(\mathcal{J}(\Gamma)\).

\label{def:curve-area-jordan-arc}
\end{definition}

In particular, the curve area functional of the oriented line segment from \(p\) to \(q\) is the following.

\begin{definition}

For any two points \(p, q \in \mathbb{R}^2\), define \(\mathcal{J}(p, q) := (p \times q) / 2\).

\label{def:curve-area-line-segment}
\end{definition}

\begin{proposition}

The curve area functional of the oriented line segment from point \(p\) to \(q\) is \(\mathcal{J}(p, q)\). Moreover, if there is some \(t \in S^1\) and \(h \in \mathbb{R}\) so that \(p, q \in l(t, h)\) and \(q - p = d v_t\) for some \(d \in \mathbb{R}\), then \(\mathcal{J}(p, q) = hd/2\).

\label{pro:curve-area-line-segment}
\end{proposition}

\begin{proof}
Parametrize the line segment from \(p\) to \(q\) as \(\mathbf{x}(t) = p + (q - p) t\) for \(t \in [0, 1]\). Then \(\mathcal{J}(\mathbf{x})\) evaluates to \((p \times q) / 2\) which is \(\mathcal{J}(p, q)\). Moreover,
\[
p \times q = p \times (q - p) = p \times (d v_t) = s (p \cdot u_t) = hd
\]
so \(\mathcal{J}(p, q) = hd/2\).
\end{proof}

\begin{proposition}

If two points \(p, q \in \mathbb{R}^2\) and the origin \(O = (0, 0)\) are on a common line, then \(\mathcal{J}(p, q) = 0\).

\label{pro:curve-area-line-segment-colinear}
\end{proposition}

\begin{proof}
Set \(h=0\) in \Cref{pro:curve-area-line-segment}.
\end{proof}

We will often form a Jordan curve by concatenating multiple Jordan arcs.

\begin{definition}

Say that an oriented Jordan arc or curve \(\Gamma\) is the \emph{concatenation} of the Jordan arcs \(\Gamma_1, \Gamma_2, \dots, \Gamma_n\) in order, if the ending point of \(\Gamma_{i}\) matches with the starting point of \(\Gamma_{i+1}\) for all \(1 \leq i \leq n - 1\), and \(\Gamma\) is obtained by following \(\Gamma_1, \Gamma_2, \dots\) in order.

\label{def:concatenation}
\end{definition}

\begin{proposition}

If \(\Gamma\) is the concatenation of the Jordan arcs \(\Gamma_1, \Gamma_2, \dots, \Gamma_n\) in order, then \(\mathcal{J}(\Gamma) = \sum_{i=1}^n \mathcal{J}(\Gamma_i)\).

\label{pro:curve-area-functional-additive}
\end{proposition}

\begin{proof}
Take the parametrization \(\mathbf{x}\) of \(\Gamma\), and integrate \Cref{def:curve-area-functional} into
\end{proof}

\begin{proposition}

Let \(\Gamma\) be a Jordan curve which is the concatenation of the Jordan arcs \(\Gamma_1, \Gamma_2, \dots, \Gamma_n\) in order. Assume that there is a half-plane \(H'\) containing \(\Gamma\) with the boundary \(l'\) and normal angle \(t \in S^1\) (and thus the normal vector \(u_t\)). Assume that some arc \(\Gamma_i\) of \(\Gamma\) is an oriented line segment \(s\) of length \(> 0\) on \(l'\) in the positive direction of \(v_t\). Then \(\Gamma\) is oriented counterclockwise.

\label{pro:jordan-curve-orientation}
\end{proposition}

\begin{proof}
Say that \(\Gamma_i\) is the line segment from \(p\) to \(q\) in \(l'\). Fix the endponits \(p\) and \(q\), and deform \(\Gamma_i\) slightly towards outside \(H'\) so that \(\Gamma_i \setminus \left\{ p, q \right\}\) is strictly outside \(H'\). Now take the point \(r := (p + q) / 2\). The segment \(\Gamma \setminus \Gamma_i \cup \left\{ p, q \right\}\) from \(q\) to \(p\) is inside \(H'\), so it rotates around the point \(r\) in the counterclockwise angle of \(\pi\). Similarly, the deformed curve \(\Gamma_i\) also rotates around the point \(r\) in the counterclockwise angle of \(\pi\). So the total angle is \(2\pi\) clockwise, and \(\Gamma\) is oriented counterclockwise.
\end{proof}

\section{Convex Curve}
\label{sec:convex-curve}
Define the following convex curve segment of the boundary of a convex body \(K\).

\begin{definition}

For any planar convex body \(K\) and \(a, b \in \mathbb{R}\) be arbitrary such that \(a < b < a + \pi\), define the segment
\[
\mathbf{u}_K^{a, b} := \left\{ v_K^+(a) \right\} \cup \bigcup_{t \in (a, b)} e_K(t) \cup \left\{ v_K^-(b) \right\}.
\]
of the boundary of \(K\).

\label{def:convex-curve}
\end{definition}

The goal of this \Cref{sec:convex-curve} is to show that \(\mathbf{u}_K^{a, b}\) is a rectifiable curve and evaluate its curve area functional.

\begin{lemma}

Assume arbitrary \(K \in \mathcal{K}\) and \(a, b \in \mathbb{R}\) such that \(a < b < a + \pi\). If \(v_K^+(a) = v_K^-(b)\), then \(v_K(a, b) = v_K^+(a) = v_K^-(b)\) and \(\mathbf{u}_K^{a, b}\) is the single point \(\left\{ v_K^+(a) \right\}\). If \(v_K^+(a) \neq v_K^-(b)\), then the followings are true.

\begin{enumerate}
\def\labelenumi{\arabic{enumi}.}
\tightlist
\item
  The point \(v_K(a, b)\) is not on the line \(l'\) connecting \(v_K^+(a)\) and \(v_K^-(b)\).
\item
  The closed half-plane \(H'\) with the boundary \(l'\) containing \(v_K(a, b)\) have normal angle \(t' + \pi\) for some \(t' \in (a, b)\).
\item
  The intersection \(K' := K \cap H'\) is a planar convex body satisfying the followings.

  \begin{enumerate}
  \def\labelenumii{\roman{enumii}.}
  \tightlist
  \item
    For any \(t \in (t' - \pi, a]\), we have \(e_{K'}(t) = \left\{ v_K^+(a) \right\}\).
  \item
    For any \(t \in (a, b)\), we have \(e_{K'}(t) = e_K(t)\).
  \item
    For any \(t \in [b, t' + \pi)\), we have \(e_{K'}(t) = \left\{ v_K^-(b) \right\}\).
  \item
    The edge \(e_{K'}(t' + \pi)\) is the line segment from \(v_K^-(b)\) to \(v_K^+(a)\).
  \end{enumerate}
\end{enumerate}

\label{lem:convex-curve-cut}
\end{lemma}

\begin{proof}
Define the closed cone \(X := H_K(a) \cap H_K(b)\) with vertex \(v_K(a, b)\) containing \(K\). The boundary of \(X\) is the union of two half-lines \(X \cap l_K(a)\) and \(X \cap l_K(b)\), each containing \(v_K^+(a)\) and \(v_K^-(b)\) respectively, and both meeting at \(v_K(a, b)\) with an angle of \(b-a \in (0, \pi)\). So we have \(v_K(a, b) = v_K^+(a) + \alpha v_a\) for some \(\alpha \geq 0\) and \(v_K^-(b) = v_K(a, b) + \beta v_b\) for some \(\beta \geq 0\).

First assume \(v_K(a, b) \in K\). Then since \(v_K(a, b) \in K \subseteq X\), we have \(v_K^+(a) = v_K^-(b) = v_K(a, b)\) and \(e_K(t) = \left\{ v_K(a, b) \right\}\) for all \(t \in (a, b)\). So \(v_K^+(a) = v_K^-(b)\) and \(\mathbf{u}_K^{a, b}\) degenerates to the single point \(v_K^+(a)\) as claimed.

Now assume \(v_K(a, b) \not\in K\). Then since \(v_K^+(a), v_K^-(b) \in K\) but \(v_K(a, b) \not\in K\), we have \(\alpha, \beta > 0\) and the points \(v_K^+(a), v_K^-(b), v_K(a, b)\) are not on the same line, showing (1). Also, the vector \(v_K^+(a) - v_K^-(b) = \alpha v_a + \beta v_b\) is equal to \(\tau v_{t'}\) for some \(\tau > 0\) and \(t' \in (a, b)\). So (2) holds. Define \(T := X \cap H'\), which is the triangle with vertices \(v_K^+(a), v_K^-(b), v_K(a, b)\).

Since \(e_T(t' + \pi)\) is the segment connecting \(v_K^+(a)\) and \(v_K^-(b)\), we have \(e_T(t' + \pi) \subseteq K' \subseteq T\). This, with that \(T\) have normal angles \(a, b, t' + \pi\), implies (i), (iii), (iv) of (3) except \(t = a, b\). For the case \(t=a\) in (i), use the definition of \(v_K^+(a)\) to see that the edge \(e_K(a)\) only intersects \(K\) at the single point \(v_K^+(a)\). Handle the case \(t=b\) in (iii) similarly.

It remains to show (ii) of (3). It suffices to show that for all \(t \in (a, b)\), we have \(e_K(t) \subseteq T\) as this implies \(e_K(t) \subseteq K' = K \cap T\). To prove \(e_K(t) \subseteq T\), we will show that for any \(p \in K \setminus T\) we have \(p \not\in e_K(t)\). If \(t \leq t'\), then the point \(p\) is in the convex cone \(H_K(a) \setminus H'\) with normal angles \(a\) and \(t\) that does not contain the vertex \(v_K^+(a)\). So we have \(p \cdot u_t < v_K^+(a) \cdot u_t \leq h_K(t)\) and \(p \not\in e_K(t)\) as desired. For the case \(t \geq t'\), we can do a similar argument using the cone \(H_K(b) \setminus H'\) with normal angles \(b\) and \(t\).
\end{proof}

\begin{theorem}

Let \(K \in \mathcal{K}\) and \(a, b \in \mathbb{R}\) be arbitrary such that \(a < b < a + \pi\). Then the set \(\mathbf{u}_K^{a, b}\) in \Cref{def:convex-curve} is a rectifiable oriented Jordan arc from \(v_K^+(a)\) to \(v_K^-(b)\), and its curve area functional is
\[
\mathcal{J}\left( \mathbf{u}_K^{a, b} \right) = \frac{1}{2} \int_{t \in (a, b)}h_K(t)\, \sigma_K(dt)
\]
which is quadratic in \(K \in \mathcal{K}\).

\label{thm:convex-curve-area-functional}
\end{theorem}

\begin{proof}
If \(v_K^+(a) = v_K^-(b)\), then by \Cref{lem:convex-curve-cut} the set \(\mathbf{u}_K^{a, b}\) degenerates to the single point \(p= v_K^+(a) = v_K^-(b)\). So for all \(t \in (a, b)\), the set \(e_K(t)\) is equal to \(p\), and by \Cref{thm:surface-area-measure} the measure \(\sigma_K\) is zero on the interval \((a, b)\). So the claimed equality holds.

Now assume the case \(v_K^+(a) \neq v_K^-(b)\). Define the convex body \(K'\) containing \(v_K^+(a)\) and \(v_K^-(b)\) as in \Cref{lem:convex-curve-cut}. Then by (3) of \Cref{lem:convex-curve-cut}, the supporting functions \(h_K(t)\) and \(h_{K'}(t)\) agree on \(t \in [a, b]\), and the surface area measures \(\sigma_K\) and \(\sigma_{K'}\) agree on \((a, b)\) by \Cref{thm:surface-area-measure}.

If \(K'\) has empty interior, then by (iv) of \Cref{lem:convex-curve-cut} \(K'\) should be the line segment connecting \(v_K^+(a)\) and \(v_K^-(b)\). So \(\mathbf{u}_K^{a, b}\) is the line segment from \(v_K^+(a)\) to \(v_K^-(b)\), and by \Cref{pro:curve-area-line-segment} we have \(\mathcal{J}(v_K^+(a), v_K^-(b)) = h_K(t') \sigma_K(t') / 2\), verifying the equality.

Now assume that \(K'\) have nonempty interior. It is a known fact that the boundary \(\partial K'\) of \(K'\) is a rectifiable Jordan curve.\footnote{We outline a proof from \autocite{hagenvoneitzenAnswerConnectednessBoundary2015}. Translate \(K'\) so that it has \((0, 0)\) in the interior of \(K'\). Identifying \(S^1\) with the unit vectors of \(\mathbb{R}^2\), define \(f : \partial K' \to S^1\) as the map \(\mathbf{v} \mapsto \mathbf{v} / \left| \mathbf{v} \right|\). Then the map is continuous. It is bijective by the convexity of \(K\). As a continuous bijection between compact sets, \(f\) is a homeomorphism and the inverse \(f^{-1} : S^1 \to \partial K'\) is continuous. Rectifiability of \(f^{-1}\) follows from that the \(x\)-coordinate (resp. the \(y\)-coordinate) of the point \(f^{-1}(t)\) is monotonically decreasing in one interval \(I\) of \(S^1\) and monotonically increasing in the complement \(S^1 \setminus I\).} By (3) of \Cref{lem:convex-curve-cut}, the boundary \(\partial K'\) of \(K'\) is the disjoint union of \(\mathbf{u}_K^{a, b}\) and the open line segment \(s\) from \(v_K^-(b)\) to \(v_K^+(a)\) excluding endpoints. Since \(s\) is a bounded open interval of \(\partial K' \simeq S^1\), the curve \(\mathbf{u}_K^{a, b}\) is an oriented Jordan arc from \(v_K^+(a)\) to \(v_K^-(b)\) inheriting the parametrization from \(\partial K'\). By \Cref{thm:curve-area-functional-area} we have
\begin{equation}
\label{eqn:uab-one}
|K'| = \mathcal{J}\left( \mathbf{u}_K^{a, b} \right) + \mathcal{J}\left( v_K^-(b), v_K^+(a) \right) .
\end{equation}
On the other hand, by \Cref{thm:area-quadratic-expression} we have
\[
|K'| = \frac{1}{2} \int_{t \in S^1} h_{K'}(t)\,\sigma_{K'}(dt).
\]
By (3) of \Cref{lem:convex-curve-cut} and \Cref{pro:curve-area-line-segment}, this evaluates to
\begin{equation}
\label{eqn:uab-two}
|K'| = \frac{1}{2} \int_{t \in (a, b)} h_{K}(t)\,\sigma_{K}(dt) + \mathcal{J}\left( v_K^-(b), v_K^+(a) \right).
\end{equation}
By comparing \Cref{eqn:uab-one} to \Cref{eqn:uab-two}, we get the desired equality. Quadraticity of \(\mathcal{J}\left( \mathbf{u}_K^{a, b} \right)\) comes from \Cref{thm:convex-body-linear}.
\end{proof}

\begin{lemma}

Fix \(a, b \in \mathbb{R}\) such that \(a < b < a + \pi\). The bilinear form \(\mathcal{B} : \mathcal{K} \times \mathcal{K} \to \mathbb{R}\) on \(\mathcal{K}\) defined as
\[
\mathcal{B}(K_1, K_2) := \frac{1}{2} \int_{t \in (a, b)} h_{K_1}(t) \, \sigma_{K_2}(dt)
\]
can also be expressed as
\[
\mathcal{B}(K_1, K_2) = \frac{1}{2} \int_{t \in (a, b)} v_{K_1}^+(t) \times d v_{K_2}^+(t) = \frac{1}{2} \int_{t \in (a, b)} v_{K_1}^-(t) \times d v_{K_2}^+(t).
\]
In particular, for any \(K \in \mathcal{K}\) we have
\[
\mathcal{J}\left( \mathbf{u}_K^{a, b} \right) = \frac{1}{2} \int_{t \in (a, b)} v_K^+(t) \times d v_K^+(t).
\]

\label{lem:convex-curve-bilinear-computation}
\end{lemma}

\begin{proof}
Using \Cref{thm:boundary-measure}, check
\[
v_{K_1}^{\pm}(t) \times \mathrm{d} v_{K_2}(t) = v_{K_1}^{\pm}(t) \times (v_t \sigma_{K_2}) = (v_{K_1}^{\pm}(t) \times v_t) \sigma_{K_2} = h_{K_1}(t) \sigma_{K_2}
\]
as pairs of measures on \(t \in (a, b)\). Integrate this on \(t \in (a, b)\) to check the equalities for \(\mathcal{B}(K_1, K_2)\). Now use \(\mathcal{J}\left( \mathbf{u}_K^{a, b} \right) = \mathcal{B}(K, K)\).
\end{proof}

\begin{lemma}

Let \(K \in \mathcal{K}\) and \(a, b, c \in \mathbb{R}\) be arbitrary such that \(a < b < c < a + \pi\). Then the oriented Jordan curve \(\mathbf{u}_K^{a, c}\) is the concatenation of \(\mathbf{u}_K^{a, b}\), \(e_K(b)\), \(\mathbf{u}_K^{b, c}\) in order.

\label{lem:convex-curve-concat}
\end{lemma}

\begin{proof}
Use that the union of Jordan curves \(\mathbf{u}_K^{a, b}\), \(e_K(b)\), \(\mathbf{u}_K^{b, c}\) is \(\mathbf{u}_K^{a, c}\), and the three curves only meet at respective endpoints \(v_K^-(b)\) and \(v_K^+(b)\).
\end{proof}

\begin{lemma}

Take any \(K \in \mathcal{K}\) and \(a, b \in \mathbb{R}\) such that \(a < b < a + \pi\). Assume \(v_K^+(a) \neq v_K^-(b)\). Define \(\Gamma\) as the closed curve formed by following the segment from \(v_K^+(a)\) to \(v_K(a, b)\), then the segment from \(v_K(a, b)\) to \(v_K^-(b)\), then the curve \(\mathbf{u}_K^{a, b}\) in reverse direction. Then the followings are true.

\begin{enumerate}
\def\labelenumi{\arabic{enumi}.}
\tightlist
\item
  \(\Gamma\) is a counterclockwise Jordan curve.
\item
  The region \(R\) enclosed by \(\Gamma\) is contained in the interior of \(H_K(a) \cap H_K(b)\).
\item
  The region \(R\) is disjoint from the set \(\cap_{t \in [a, b]} H_K(t)\).
\end{enumerate}

\label{lem:convex-curve-jordan-curve}
\end{lemma}

\begin{proof}
Let \(s_a\) be the line segment from \(v_K^+(a)\) to \(v_K(a, b)\). Let \(s_b\) be the line segment from \(v_K(a, b)\) to \(v_K^-(b)\). By 3-(i) of \Cref{lem:convex-curve-cut} at \(t = a\), the segment \(s_a\) intersects with \(\mathbf{u}_K^{a, b}\) at the single point \(v_K^+(a)\). Likewise, by 3-(iii) of \Cref{lem:convex-curve-cut} at \(t=b\), the segment \(s_b\) intersects with \(\mathbf{u}_K^{a, b}\) at the single point \(v_K^-(b)\). By (1) of \Cref{lem:convex-curve-cut}, the segments \(s_a\) and \(s_b\) overlap at the single point \(v_K(a, b)\). So the closed curve \(\Gamma\) is indeed a Jordan curve. Use \Cref{pro:jordan-curve-orientation} to decide the orientation of \(\Gamma\), proving (1).

Because the curves \(\mathbf{u}_K^{a, b}\), \(s_a\), \(s_b\) are contained in \(H_K(a)\) and \(H_K(b)\), the region \(R\) is contained in \(H_K(a) \cap H_K(b)\). As \(R\) is open, this proves (2).

Let \(X := \cap_{t \in [a, b]} H_K(t)\). Since \(\mathbf{u}_K^{a, b}\), \(s_a\), \(s_b\) are disjoint from the interior of \(X\), the curve \(\Gamma\) is contained in the region \(\mathbb{R}^2 \setminus X^\circ\). Because \(\mathbb{R}^2 \setminus X^\circ\) is simply connected, the region \(R\) enclosed by \(\Gamma\) is contained in \(\mathbb{R}^2 \setminus X^\circ\). Now the open set \(R\) is disjoint from \(X^\circ\), so is disjoint from \(X\), showing (3).
\end{proof}

\section{Mamikon's Theorem}
\label{sec:mamikon's-theorem}
We prove a generalized version of \emph{Mamikon’s theorem} \autocite{mnatsakanianAnnularRingsEqual1997} that works for general convex bodies with non-differentiable boundaries.

\begin{definition}

Fix an arbitrary \(K \in \mathcal{K}\) and \(a, b \in \mathbb{R}\) such that \(a < b < a + \pi\). Fix some measurable \(\alpha : [a, b] \to \mathbb{R}\). Define the \emph{Mamikon region} of such parameters as the region swept out by the tangent segment from \(v_K^+(t)\) to \(\mathbf{z}(t) := v_K^+(t) + \alpha(t) v_t\) over all \(t \in [a, b]\).

\label{def:mamikon-region}
\end{definition}

Follow the boundaries of the Mamikon region, and we can define its area as follows.

\begin{definition}

Fix an arbitrary \(K \in \mathcal{K}\) and \(a, b \in \mathbb{R}\) such that \(a < b < a + \pi\). For any continuous function \(\mathbf{z} : [a, b] \to \mathbb{R}^2\) of bounded variation such that \(\mathbf{z}(t) \in l_K(t)\) for all \(t \in [a, b]\), define the expression
\[
\mathcal{M}_K(a, b; \mathbf{z}) = \mathcal{J}\left( v_K^+(a), \mathbf{z}(a) \right) + \mathcal{J}\left( \mathbf{z}|_{[a, b]} \right)  + \mathcal{J}\left( \mathbf{z}(b), v_K^-(b) \right) - \mathcal{J}\left( \mathbf{u}_K^{a, b} \right).
\]

\label{def:mamikon}
\end{definition}

The expression \(\mathcal{M}_K(a, b; \mathbf{z})\) in \Cref{def:mamikon} is the area of the Mamikon region bounded by the two curves \(\mathbf{u}_K^{a, b}\) and \(\mathbf{z}\). Mamikon’s theorem states that the area is equal to \(\frac{1}{2} \int_a^b \,\alpha(t)^2 dt\) (see \Cref{fig:mamikon}).

\begin{figure}
\centering
\includegraphics{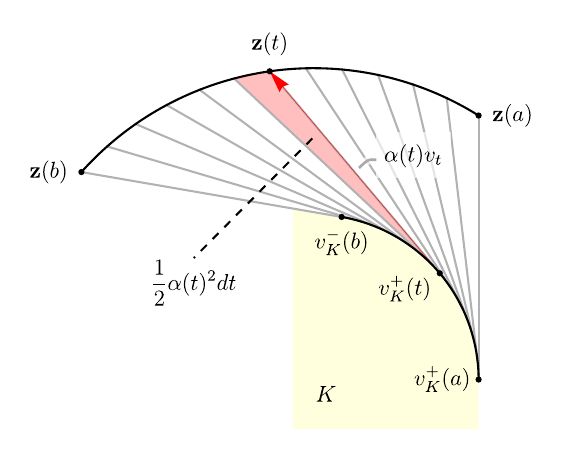}
\caption{A Mamikon region divided into differential fans of angle \(dt\) and radius \(\alpha(t)\).}
\label{fig:mamikon}
\end{figure}

\begin{theorem}

(Mamikon’s theorem, generalized) Assume the notations in \Cref{def:mamikon}. Then the function \(\alpha : [a, b] \to \mathbb{R}\) defined as \(\alpha(t) := \left( \mathbf{z}(t) - v_K^+(t) \right) \cdot v_t\) is bounded, measurable, and satisfies \(\mathbf{z}(t) = v_K^+(t) + \alpha(t) v_t\). Also, we have
\[
\mathcal{M}_K(a, b; \mathbf{z}) =  \frac{1}{2}\int_{a}^{b} \alpha(t) ^2 \, dt.
\]

\label{thm:mamikon}
\end{theorem}

\begin{proof}
It is easy to check the claimed conditions on \(\alpha\). We prove the equality on \(\mathcal{M}_K(a, b; \mathbf{z})\). Write \(\mathbf{v}(t) := v_K^+(t)\) for all \(t \in (a, b]\). Note that \(\mathbf{v}\) is right-continuous by \Cref{thm:limits-converging-to-vertex} and of bounded variation by \Cref{lem:vertex-bounded-variation}. So the Lebesgue–Stieltjes measure \(\mathrm{d}(\mathbf{z} \times \mathbf{v})\) on \([a, b]\) is well-defined, and we have the chain of equalities
\begin{align*}
& \phantom{{} = .} \mathbf{z} \times \mathrm{d} \mathbf{z} - \mathbf{v} \times \mathrm{d} \mathbf{v} + \mathrm{d} \left( \mathbf{z} \times \mathbf{v} \right)  \\
& = \mathbf{z} \times \mathrm{d} \mathbf{z} - \mathbf{v} \times \mathrm{d} \mathbf{v} + \left( \mathrm{d} \mathbf{z} \times \mathbf{v} + \mathbf{z} \times \mathrm{d} \mathbf{v}  \right)  \\
& = \left( \mathbf{z} - \mathbf{v} \right) \times \mathrm{d}\left( \mathbf{z} + \mathbf{v} \right)  \\
& = \left( \mathbf{z} - \mathbf{v} \right) \times \mathrm{d}\left( \mathbf{z} - \mathbf{v} \right)  \\
& = \alpha u_t \times \mathrm{d}\left( \alpha u_t \right) = \alpha u_t \times ( u_t \mathrm{d}\alpha + \alpha v_t \mathrm{d}t) = \alpha^2 \mathrm{d} t
\end{align*}
of measures and functions on \(t \in (a, b]\). The first equality uses \Cref{lem:lebesgue-stieltjes-product} and continuity of \(\mathbf{z}\). The second equality is bilinearity and antisymmetry of \(\times\). As we have \(\mathrm{d} \mathbf{v}(t) = v_t \sigma\) by \Cref{thm:boundary-measure} and \(\mathbf{z}(t) - \mathbf{v}(t) = \alpha_K(t)v_t\), they are parallel and we get \((\mathbf{z}(t) - \mathbf{v}(t)) \times d \mathbf{v}(t) = 0\) which is used in the third equality. The last chain of equalities are basic calculations.

If we integrate the formula above on the whole interval \((a, b)\), the terms \(\mathbf{z}(t) \times d \mathbf{z}(t)\) and \(\mathbf{v}(t) \times d \mathbf{v}(t)\) becomes \(2 \mathcal{J}(\mathbf{z})\) and \(2 \mathcal{J}(\mathbf{v})\) by \Cref{def:curve-area-functional} and the last equality of \Cref{lem:convex-curve-bilinear-computation} respectively. The measure \(d(\mathbf{z}(t) \times \mathbf{v}(t))\) integrates to the difference \(2 \mathcal{J} \left( \mathbf{z}(b), v_K^-(b) \right) - 2 \mathcal{J} \left( \mathbf{z}(a), v_K^+(a) \right)\). So the integral matches twice the \Cref{def:mamikon}, completing the proof.
\end{proof}

\begin{theorem}

Fix \(a, b \in \mathbb{R}\) such that \(a < b < a + \pi\). Assume that the function \(\mathbf{z}_K \in C^{\mathrm{BV}}[a, b]\) is determined by \(K \in \mathcal{K}\) so that (i) \(\mathbf{z}_K(t)\) is on the line \(l_K(t)\) for all \(t \in [a, b]\), and (ii) the map \(K \mapsto \mathbf{z}_K\) is convex-linear in \(K\). Then the expression \(\mathcal{M}_K(a, b; \mathbf{z}_K)\) is quadratic and convex as a functional on \(K \in \mathcal{K}\).

\label{thm:mamikon-convex}
\end{theorem}

\begin{proof}
Let \(\alpha_K(t) := \left( \mathbf{z}_K(t) - v_K^+(t) \right) \cdot v_t\), then \(\alpha_K\) is convex-linear in \(K\) by convex-linearity of \(\mathbf{z}_K\) and \(v_K^+(t)\) (\Cref{thm:convex-body-linear}). So \(\mathcal{M}_K(a, b; \mathbf{z}_K)\) which is \(\frac{1}{2} \int_a^b \alpha_K(t)^2 \, dt\) by \Cref{thm:mamikon} is quadratic and convex in \(K\).
\end{proof}

\chapter{Optimality of Gerver's Sofa}
\label{sec:optimality-of-gerver's-sofa}
This chapter proves the optimality of Gerver’s sofa (\Cref{thm:main}). Each section of this chapter establishes the following portion of the overview in \Cref{sec:_optimality-of-gerver's-sofa}. Recall that the previous \Cref{sec:convex-domain-and-convex-curves} prepares necessary technical lemmas for this chapter.

\begin{itemize}
\tightlist
\item
  \Cref{sec:domain-of-q} defines the domain \(\mathcal{L}\) of \(\mathcal{Q}\), which is the tuple \((K, B, D)\) of three convex bodies as described in \Cref{sec:quadraticity-of-q}, and build the embedding \(K \mapsto (K, B_K, D_K)\) from caps \(K \in \mathcal{K}^\mathrm{i}\) with injectivity condition to a subset of \(\mathcal{L}\).
\item
  \Cref{sec:definition-of-q} formally defines \(\mathcal{Q}\) using quadratic expressions on \(K, B\), and \(D\). The formula of \(\mathcal{Q}\) traces the boundary of the overestimated region \(R\) described in \Cref{sec:_definition-of-q} and \Cref{sec:quadraticity-of-q}. We use Jordan curve theorem to rigorously show that \(\mathcal{Q}(K, B_K, D_K)\) is an upper bound of sofa area functional \(\mathcal{A}(K)\) on \(K \in \mathcal{K}^\mathrm{i}\).
\item
  \Cref{sec:concavity-of-q} establishes the concavity of \(\mathcal{Q}\) on \(\mathcal{L}\) using Mamikon’s theorem.
\item
  \Cref{sec:gerver's-sofa} translates the local optimality conditions of Gerver’s sofa \(G\) derived by Romik in \autocite{romikDifferentialEquationsExact2018} to equalities of surface area measures of \(K, B_K\), and \(D_K\) corresponding to \(G\).
\item
  \Cref{sec:directional-derivative-of-q} calculates the directional derivative of \(\mathcal{Q}\) at \((K, B_K, D_K) \in \mathcal{L}\) arising from Gerver’s sofa \(G\). The conditions on \(G\) in \Cref{sec:gerver's-sofa} show that
\end{itemize}
\section{Domain of $\texorpdfstring{\mathcal{Q}}{Q}$}
\label{sec:domain-of-q}
In this \Cref{sec:domain-of-q}, we define the domain \(\mathcal{L}\) of the soon-to-be-established upper bound \(\mathcal{Q}\) which extends the collection \(\mathcal{K}^\mathrm{i}\) of caps satisfying the injectivity condition.

\begin{definition}

Define \(\mathcal{K}^\mathrm{i}\) as the subset of caps \(K \in \mathcal{K}^\mathrm{c}\) that (i) satisfies the injectivity condition (\Cref{def:injectivity-condition}), and (ii) have area \(|K| \geq 2.2\).

\label{def:cap-space-special}
\end{definition}

Remind that for any \(K \in \mathcal{K}^\mathrm{i}\), we can talk about the density function \(r_K\) and \(s_K\) of the surface area measure \(\sigma_K\) on \([0, \pi/2)\) and \((\pi/2, \pi]\) respectively by (1) of \Cref{def:injectivity-condition}. Likewise, the walls \(a_K(t)\) and \(c_K(t)\) makes contact with \(K\) at unique points \(A_K(t)\) and \(C_K(t)\) respectively (see \Cref{def:cap-nondegenerate}), and we can talk about the arm lengths \(f_K(t) = f_K^{\pm}(t)\) and \(g_K(t) = g_K^{\pm}(t)\).

\begin{theorem}

The space \(\mathcal{K}^\mathrm{i}\) is a convex subspace of \(\mathcal{K}^\mathrm{c}\) containing all balanced maximum caps \(K \in \mathcal{K}^\mathrm{c}\) and the cap \(\mathcal{C}(G)\) of Gerver’s sofa.

\label{thm:cap-space-special}
\end{theorem}

\begin{proof}
We first show that \(\mathcal{K}^\mathrm{i} \subset \mathcal{K}^\mathrm{c}\) is a convex subset. The \Cref{def:injectivity-condition} of injectivity condition on \(K\) is made of linear constraints in \(K\), and so is preserved under the barycentric operation of \(\mathcal{K}^\mathrm{c}\). For any \(\lambda \in [0, 1]\), by the Brunn-Minkowski theorem
\[
|(1 - \lambda)K_1 + \lambda K_2| \geq |K_1|^{1 - \lambda} |K_2|^{\lambda}
\]
we have \(|c_\lambda(K_1, K_2)| \geq 2.2\) if \(|K_1|, |K_2| \geq 2.2\). This shows that the condition \(|K| \geq 2.2\) is also closed under \(c_\lambda\). So \(\mathcal{K}^\mathrm{i}\) defines a convex subset of \(\mathcal{K}^\mathrm{c}\).

By \Cref{thm:injectivity} and \Cref{thm:injectivity-gerver}, both the balanced maximum caps \(K\) and the cap \(\mathcal{C}(G)\) of Gerver’s sofa satisfy the injectivity condition. It is known that \(G\) have area \(2.2195\dots \geq 2.2\), and any balanced maximum cap \(K\) attains the maximum value of sofa area functional \(\mathcal{A}(K)\), so it should have area \(|K| \geq \mathcal{A}(K) \geq \mathcal{A}(G) = 2.2195\dots \geq 2.2\) again.
\end{proof}

Define the angle constants determining Gerver’s sofa \(G\).

\begin{definition}

Define \(\varphi\) and \(\theta\) as the angles satisfying \(0 < \varphi < \theta < \pi/4\) and the Equations 27 to 44 in \autocite{romikDifferentialEquationsExact2018}. Note that \(\varphi \in [0.039, 0.040]\) in particular (see Table 1 of \autocite{romikDifferentialEquationsExact2018}). Define the constants \(\varphi^{\mathrm{R}} := \varphi\) and \(\varphi^{\text{L}} := \pi/2 - \varphi\).

\label{def:gerver-constants}
\end{definition}

\begin{remark}

The angles \(\varphi\) and \(\theta\) are used in the definition of Gerver’s sofa \(G\); see \Cref{sec:gerver's-sofa} for the details. In particular, the blue core of \(G\) in \Cref{fig:gerver} is the trajectory of the inner corner \(\mathbf{x}_G(t)\) on the interval \(t \in [\varphi^\mathrm{R}, \varphi^\mathrm{L}]\).

\label{rem:gerver-constants}
\end{remark}

\begin{definition}

Define \(\mathcal{L}\) as the space of all tuples \((K, B, D)\) of convex bodies such that the followings are true.

\begin{enumerate}
\def\labelenumi{\arabic{enumi}.}
\tightlist
\item
  \(K \in \mathcal{K}^\mathrm{i}\), \(B \subseteq K\), and \(D \subseteq K\).
\item
  For every \(t \in [\varphi^\mathrm{R}, \pi/2]\), we have \(h_K(t) + h_B(\pi + t) \leq 1\).
\item
  Equality holds in (2) at \(t = \varphi^\mathrm{R}, \pi/2\).
\item
  For every \(t \in [0, \varphi^\mathrm{L}]\), we have \(h_K(\pi/2 + t) + h_D(3\pi/2 + t) \leq 1\).
\item
  Equality holds in (4) at \(t = 0, \varphi^\mathrm{L}\).
\end{enumerate}

\label{def:cap-tail-space}
\end{definition}

\begin{proposition}

The space \(\mathcal{L}\) is a convex domain with the barycentric operation
\[
c_\lambda((K_1, B_1, D_1), (K_2, B_2, D_2)) := (c_\lambda(K_1, K_2), c_\lambda(B_1, B_2), c_\lambda(D_1, D_2)). 
\]

\label{pro:cap-tail-space}
\end{proposition}

\begin{proof}
The product \(\mathcal{K} \times \mathcal{K} \times \mathcal{K}\) of convex domains \(\mathcal{K}\) is naturally a convex domain. Recall that \(\mathcal{K}^\mathrm{i}\) is a convex subspace of \(\mathcal{K}\). That \(B \subseteq K\) (resp. \(D \subseteq K\)) can be written as linear constraints \(h_B(t) \leq h_K(t)\) (resp. \(h_D(t) \leq h_K(t)\)) over all \(t \in S^1\). So the condition (1) of \Cref{def:cap-tail-space} is preserved under barycentric operations. By (1) of \Cref{thm:convex-body-linear}, the conditions (2) to (5) of \Cref{def:cap-tail-space} are linear constraints.
\end{proof}

We now build the injection \(\mathcal{K}^\mathrm{i} \to \mathcal{L}\) mapping \(K \in \mathcal{K}^\mathrm{i}\) to \((K, B_K, D_K) \in \mathcal{L}\). That \((K, B_K, D_K) \in \mathcal{L}\) is not too difficult to show using injectivity condition, but requires bookkeeping of many definitions.

\begin{definition}

Take arbitrary cap \(K \in \mathcal{K}^\mathrm{i}\). Define the convex body
\[
B_K := K \cap \bigcap_{t \in [\varphi_R, \pi/2]} H_K^\mathrm{b}(t).
\]
Similarly, define the convex body
\[
D_K := K \cap \bigcap_{t \in [0, \varphi_L]} H_K^\mathrm{d}(t).
\]

\label{def:right-left-body}
\end{definition}

\begin{figure}
\centering
\includegraphics{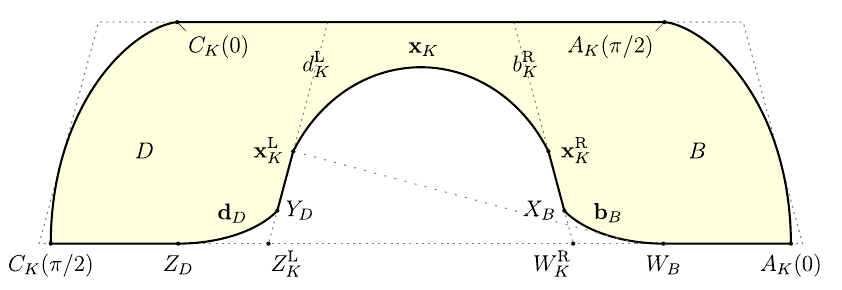}
\caption{Diagram for the upper bound \(\mathcal{Q}(K, B, D)\) in \Cref{def:upper-bound-q}.}
\label{fig:upper-bound-full}
\end{figure}

\begin{definition}

(See \Cref{fig:upper-bound-full}) For any \(K \in \mathcal{K}^\mathrm{c}\), define the line \(b_K^\mathrm{R} := b_K(\varphi^\mathrm{R})\) and the half-plane \(\breve{H}_K^\mathrm{R} := H_K^\mathrm{b}(\varphi^\mathrm{R})\) bounded from below by the line \(b_K^\mathrm{R}\). Define the points \(W_K^\mathrm{R} := W_K(\varphi^\mathrm{R})\) and \(\mathbf{x}_K^\mathrm{R} := \mathbf{x}_K(\varphi^\mathrm{R})\) on the line \(b_K^\mathrm{R}\).

Similarly, define the line \(d_K^\mathrm{L} := d_K(\varphi^\mathrm{L})\) and the half-plane \(\breve{H}_K^\mathrm{L} := H_K^\mathrm{d}(\varphi^\mathrm{L})\) bounded from below by the line \(d_K^\mathrm{L}\). Define the point \(Z_K^\mathrm{L} := Z_K(\varphi^\mathrm{L})\) and \(\mathbf{x}_K^\mathrm{L} := \mathbf{x}_K(\varphi^\mathrm{L})\) on the line \(d_K^\mathrm{L}\).

\label{def:cap-left-right}
\end{definition}

In \Cref{def:cap-left-right}, the superscripts \(\mathbf{L}\) and \(\mathbf{R}\) denote that we are essentially plugging in the values \(\varphi^\mathrm{L}\) and \(\varphi^\mathrm{R}\) respectively. The only exceptions are half-planes \(\breve{H}_K^\mathrm{R}\) and \(\breve{H}_K^\mathrm{L}\), where the accent \(\breve{H}\) means that we are using the half-planes \(H^\mathrm{b}_K\) and \(H^\mathrm{d}_K\) bounding the sofa from below, not the supporting planes \(H_K\) bounding \(K\) from above.

\begin{definition}

For any \(K \in \mathcal{K}^\mathrm{i}\) with \(|K| \geq 2.2\), define the following parallelograms.
\[
P_K^\mathrm{R} := H \cap H_K(\varphi^\mathrm{R}) \cap \breve{H}_K^\mathrm{R} \qquad P_K^\mathrm{L} := H \cap H_K(\pi/2 + \varphi^\mathrm{L}) \cap \breve{H}_K^\mathrm{L}
\]

\label{def:cap-right-left-parallelogram}
\end{definition}

\begin{lemma}

The parallelogram \(P_K^\mathrm{R}\) is bounded by the lines \(l(\pi/2, 0)\), \(l(\pi/2, 1)\), \(a_K(\varphi^\mathrm{R})\) and \(b_K(\varphi^\mathrm{R})\). It has base \(\sec \varphi\), height 1, and angle of \(\pi/2 \pm \varphi\) at each vertices. The lower-left corner of \(P_K^\mathrm{R}\) is \(W_K^\mathrm{R}\).

\label{lem:cap-right-left-parallelogram}
\end{lemma}

\begin{proof}
Recall that the point \(W_K^\mathrm{R} = W_K(\varphi^\mathrm{R})\) is the intersection \(b_K^\mathrm{R} \cap l(\pi/2, 0)\) by definition. So \(W_K(\varphi^\mathrm{R})\) is the lower-left corner of the parallelogram \(P_K^\mathrm{R}\).
\end{proof}

\begin{lemma}

For any \(K \in \mathcal{K}^\mathrm{i}\), the sets \(K \cap \breve{H}_K^\mathrm{R}\) and \(K \cap \breve{H}_K^\mathrm{L}\) are disjoint.

\label{lem:cap-left-right-disjoint}
\end{lemma}

\begin{proof}
Fix an \(K \in \mathcal{K}^\mathrm{i}\). Because \(K\) is in the horizontal strip \(H\) and the supporting half-plane \(H_K(\varphi^\mathrm{R})\), the set \(K \cap \breve{H}_K^\mathrm{R}\) is in the paralleogram \(P_K^\mathrm{R}\) in \Cref{def:cap-right-left-parallelogram}. Likewise, the set \(K \cap \breve{H}_K^\mathrm{L}\) is contained in the parallelogram \(P_K^\mathrm{L}\). Now assume by contradictory that \(K \cap \breve{H}_K^\mathrm{R}\) and \(K \cap \breve{H}_K^\mathrm{L}\) intersect. Then \(P_K^\mathrm{R}\) and \(P_K^\mathrm{L}\) should intersect as well. The bases of \(P_K^\mathrm{R}\) and \(P_K^\mathrm{L}\) on the line \(l(\pi/2, 1)\) have length \(\sec \varphi\). Note that \(K\) is contained in the trapezoid \(R := H \cap H_K(\varphi^\mathrm{R}) \cap H_K(\pi/2 + \varphi^\mathrm{L})\) containing \(P_K^\mathrm{R}\) and \(P_K^\mathrm{L}\). Since \(P_K^\mathrm{R}\) and \(P_K^\mathrm{L}\) overlaps, the side length of \(R\) on the line \(l(\pi/2, 1)\) is at most \(2 \sec \varphi\), and so the base of \(R\) on the line \(l(\pi/2, 0)\) is at most \(2 \sec \varphi + 2 \tan \varphi = 2.08\cdots < 2.2\) which can be checked by computation. Now we have \(|K| \leq |R| < 2.2\) and get contradiction.
\end{proof}

\begin{lemma}

For any \(K \in \mathcal{K}^\mathrm{i}\), the points \(W_K^\mathrm{R}, Z_K^\mathrm{L}\) are in the edge \(e_K(3\pi/2)\) excluding the endpoints \(A_K(0)\) and \(C_K(\pi/2)\).

\label{lem:cap-wz-in-edge}
\end{lemma}

\begin{proof}
By \Cref{thm:wedge-ends-in-cap}, the point \(W_K^\mathrm{R}\) is strictly on the left side of the right endpoint \(A_K(0)\) of \(e_K(3\pi/2)\). So the base of the parallelogram \(P_K^\mathrm{R}\) contains \(A_K(0)\). Because \(|K| \geq 2.2\), the edge \(e_K(3\pi/2)\) have length \(\geq 2.2\), which is strictly larget than the base \(\sec \varphi < 1.1\) of \(P_K^\mathrm{R}\). So the point \(W_K^\mathrm{R}\) is also strictly on the right side of the left endpoint \(C_K(\pi/2)\) as well, completing the proof for \(W_K^\mathrm{R}\). Use a mirror-symmetric argument for \(Z_K^\mathrm{L}\).
\end{proof}

The following lemma shows that the region \(\breve{H}_K^\mathrm{R}\) cannot distinguish the wedge \(T_K(t)\) from the complement of \(H_K^\mathrm{b}(t)\).

\begin{lemma}

Fix an arbitrary \(K \in \mathcal{K}^\mathrm{i}\).

\begin{enumerate}
\def\labelenumi{\arabic{enumi}.}
\tightlist
\item
  Take any \(t \in (\varphi^\mathrm{R}, \pi/2]\). Then the point \(\mathbf{x}_K(t)\) is outside the half-plane \(\breve{H}_K^\mathrm{R}\), and we have \(\breve{H}_K^\mathrm{R} \cap Q_K^-(t) = \breve{H}_K^\mathrm{R} \setminus H_K^\mathrm{b}(t)\). So we also have \(\breve{H}_K^\mathrm{R} \cap T_K(t) = \breve{H}_K^\mathrm{R} \cap H_+(\pi/2, 0) \setminus H_K^\mathrm{b}(t)\).
\item
  Take any \(t \in [0, \varphi^\mathrm{L})\). The point \(\mathbf{x}_K(t)\) is outside the half-plane \(\breve{H}_K^\mathrm{L}\), and we have \(\breve{H}_K^\mathrm{L} \cap Q_K^-(t) = \breve{H}_K^\mathrm{L} \setminus H_K^\mathrm{d}(t)\). So we also have \(\breve{H}_K^\mathrm{L} \cap T_K(t) = \breve{H}_K^\mathrm{L} \cap H_+(\pi/2, 0) \setminus H_K^\mathrm{d}(t)\).
\end{enumerate}

\label{lem:monotonicity-intervals}
\end{lemma}

\begin{proof}
We prove (1). The point \(\mathbf{x}_K(\varphi^\mathrm{R})\) is on the boundary \(b_K^\mathrm{R}\) of \(\breve{H}_K^\mathrm{R}\) by definition. As \(K \in \mathcal{K}^\mathrm{i}\) satisfies the injectivity condition by definition, it satisfies (3) of \Cref{def:injectivity-condition} and we have \(\mathbf{x}'_K(t) \cdot u_{\varphi^\mathrm{R}} < 0\) for all \(t \in (\varphi^\mathrm{R}, \pi/2)\). By integrating, we have \(\mathbf{x}_K(t) \cdot u_{\varphi^\mathrm{R}} < \mathbf{x}_K(\varphi^\mathrm{R}) \cdot u_{\varphi^\mathrm{R}}\) for all \(t \in (\varphi^\mathrm{R}, \pi/2]\), so we have \(\mathbf{x}_K(t) \not\in \breve{H}_K^\mathrm{R}\).

We now show \(\breve{H}_K^\mathrm{R} \cap Q_K^-(t) = \breve{H}_K^\mathrm{R} \setminus H_K^\mathrm{b}(t)\). Recall that the quadrant \(Q_K^-(t)\) have half-lines \(\vec{b}_K(t)\) and \(\vec{d}_K(t)\) as boundary, so we have \(Q_K^-(t) = \mathbb{R}^2 \setminus H_K^\mathrm{b}(t) \setminus H_K^\mathrm{d}(t)\). Take the cone \(C := H_K^\mathrm{d}(t) \setminus H_K^\mathrm{b}(t)\). Then the vertex \(\mathbf{x}_K(t)\) of \(C\) is outside the half-plane \(\breve{H}_K^\mathrm{R}\) of normal angle \(\varphi^\mathrm{R}\). Also, as \(t \in (\varphi^{\mathrm{R}}, \pi/2]\)the spanning directions \(u_t\) and \(-v_t\) of \(C\) are in the positive direction of the normal angle \(-u_{\varphi^\mathrm{R}}\) of \(\breve{H}_K^\mathrm{R}\). So \(C\) and \(\breve{H}_K^\mathrm{R}\) are disjoint, and we have
\[
\breve{H}_K^\mathrm{R} \cap Q_K^-(t) = \breve{H}_K^\mathrm{R} \cap \left( Q_K^-(t) \cup C  \right) = \breve{H}_K^\mathrm{R} \cap (\mathbb{R}^2 \setminus H_K^\mathrm{b}(t)).
\]
Intersect above with \(H_+(\pi/2, 0)\) to get the second equation of (1). The proof of (2) is a mirror-symmetric argument.
\end{proof}

\begin{lemma}

For any cap \(K \in \mathcal{K}^\mathrm{i}\), let \(B := B_K\) and \(D := D_K\). Then the followings are true.

\begin{enumerate}
\def\labelenumi{\arabic{enumi}.}
\tightlist
\item
  For every \(t \in [\varphi^\mathrm{R}, \pi/2]\), we have \(h_K(t) + h_B(\pi + t) \leq 1\).
\item
  Equality holds in (1) at \(t = \varphi^\mathrm{R}, \pi/2\). So \(l_B(3\pi/2) = l(\pi/2, 0)\) and \(l_B(\pi + \varphi^\mathrm{R}) = b_K^\mathrm{R}\).
\item
  For every \(t \in [0, \varphi^\mathrm{L}]\), we have \(h_K(\pi/2 + t) + h_D(3\pi/2 + t) \leq 1\).
\item
  Equality holds in (3) at \(t = 0, \varphi^\mathrm{R}\). So \(l_D(3\pi/2) = l(\pi/2, 0)\) and \(l_D(3\pi/2 + \varphi^\mathrm{L}) = d_K^\mathrm{L}\).
\end{enumerate}

\label{lem:right-left-body}
\end{lemma}

\begin{proof}
By \Cref{def:right-left-body}, we have \(B \subseteq H_K^\mathrm{b}(t) = H_-(\pi + t, 1 - h_K(t))\) for all \(t \in [\varphi_R, \pi/2]\), and this completes the proof of (1). We now prove (2).

By \Cref{thm:wedge-ends-in-cap}, we have \(A_K(0) \in H_K^\mathrm{b}(t)\) for all \(t \in [\varphi^\mathrm{R}, \pi/2]\) and so \(A_K(0) \in B_K\). This with \(B_K \subseteq K\) implies that \(l_B(3\pi/2) = l(\pi/2, 0)\). It remains to prove \(l_B(\pi + \varphi^\mathrm{R}) = b_K^\mathrm{R}\). Since \(B \subseteq H_K^\mathrm{b} (\varphi^\mathrm{R})\), and \(b_K^\mathrm{R}\) is the boundary of \(\breve{H}_K^\mathrm{R} = H_K^\mathrm{b}(\varphi^\mathrm{R})\), it suffices to show that there exists some point \(p \in B \cap b_K^\mathrm{R}\).

We will show that the upper boundary \(\delta K\) of \(K\) and the line \(b_K^\mathrm{R}\) intersect at a single point \(p\). By \Cref{lem:cap-wz-in-edge}, the line \(b_K^\mathrm{R}\) passes through \(W_K^\mathrm{R}\) which is in the edge \(e_K(3\pi/2) \setminus \left\{ A_K(0), C_K(\pi/2) \right\}\) of \(K\). So it should pass through exactly one another point \(p\) of the boundary \(\partial K\) of \(K\). As \(\delta K = \partial K \setminus e_K(3\pi/2) \cup \left\{ A_K(0), C_K(\pi/2) \right\}\), the point \(p\) is in \(\delta K\) and is the unique intersection of \(\delta K\) and \(b_K^\mathrm{R}\).

Now take \(p\) as the unique intersection of \(\delta K\) and \(b_K^\mathrm{R}\). Recall that our goal is to show \(p \in B \cap b_K^\mathrm{R}\). Since \(p \in K \cap b_K^\mathrm{R}\), it suffices to show that \(p \in H_K^\mathrm{b}(t)\) for all \(t \in [\varphi^\mathrm{R}, \pi/2]\). By (2) of \Cref{thm:monotonization-connected-iff} and that \(p \in \delta K\), the point \(p\) is outside the niche \(\mathcal{N}(K)\). So \(p\) is outside all \(Q_K^-(t)\) over \(t \in [0, \pi/2]\). By (1) of \Cref{lem:monotonicity-intervals}, the point \(p\) should be inside \(H_K^\mathrm{b}(t)\) for all \(t \in [\varphi^\mathrm{R}, \pi/2]\). This completes the goal.

The proofs of (3) and (4) can be done using mirror-symmetric arguments.
\end{proof}

We now make the extension \(\mathcal{K}^\mathrm{i} \to \mathcal{L}\).

\begin{theorem}

For any \(K \in \mathcal{K}^\mathrm{i}\), the triple \((K, B_K, D_K)\) is in \(\mathcal{L}\).

\label{thm:cap-tail-extension}
\end{theorem}

\begin{proof}
We need to check the conditions of \Cref{def:cap-tail-space} for \(B := B_K\) and \(D := D_K\). Condition (1) is evident from the definitions. Conditions (2) to (5) are exactly \Cref{lem:right-left-body}.
\end{proof}

\begin{remark}

Even though both \(\mathcal{K}^\mathrm{i}\) and \(\mathcal{L}\) are convex domains, the injection \(\mathcal{K}^\mathrm{i} \to \mathcal{L}\) defined as map \(K \mapsto (K, B_K, D_K)\) in \Cref{thm:cap-tail-extension} is \emph{not} convex-linear. The lower left portion of the convex body \(B_K\) is the \emph{Aleksandrov body} (or the \emph{Wulff shape}) \(\bigcap_{t \in [\varphi^\mathrm{R}, \pi/2]} H_+(t, f(t))\) of the function \(f(t) := 1 - h_K(\pi + t)\) over \(t \in [\varphi^\mathrm{R}, \pi/2]\), which is hard to understand in terms of \(f\) directly. This makes the injection \(\mathcal{K}^\mathrm{i} \to \mathcal{L}\) very important; it ‘irons out’ the sofa area functional \(\mathcal{A} : \mathcal{K}^\mathrm{i} \to \mathbb{R}\), which is not quadratic, to a quadratic functional \(\mathcal{Q} : \mathcal{L} \to \mathbb{R}\) (\Cref{thm:upper-bound-q}).

\label{rem:cap-tail-extension}
\end{remark}

\section{Definition of $\texorpdfstring{\mathcal{Q}}{Q}$}
\label{sec:definition-of-q}
In this \Cref{sec:definition-of-q}, we give the full \Cref{def:upper-bound-q} of the upper bound \(\mathcal{Q} : \mathcal{L} \to \mathbb{R}\) of area (see \Cref{fig:upper-bound-full}). For the cap \(K := \mathcal{C}(G)\) of Gerver’s sofa \(G\), the inner corner \(\mathbf{x}_K(t)\) draws a main ‘core’ part of the niche \(\mathcal{N}(K)\) at the interval \(t \in [\varphi^\mathrm{R}, \varphi^\mathrm{L}]\). Following this, we break the niche \(\mathcal{N}(K)\) into three parts for the construction of the upper bound \(\mathcal{Q}\).

\begin{definition}

(See \Cref{fig:upper-bound-full}) For any convex body \(B\), define the convex curve
\[
\mathbf{b}_B := \mathbf{u}_{B}^{\pi + \varphi^\mathrm{R}, 3\pi/2}
\]
which is from the vertex \(X_B := v_B^+(\pi + \varphi^\mathrm{R})\) to the vertex \(W_B := v_B^-(3\pi/2)\). Likewise, for any convex body \(D\), define the convex curve
\[
\textbf{d}_D := \mathbf{u}_D^{3\pi/2, 3\pi/2 + \varphi^\mathrm{L}}
\]
which is from the vertex \(Z_D := v_D^+(3\pi/2)\) to the vertex \(Y_D := v_D^-(3\pi/2 + \varphi^\mathrm{L})\).

\label{def:right-left-tails}
\end{definition}

\begin{definition}

Define the function \(\mathcal{Q} : \mathcal{L} \to \mathbb{R}\) as \(\mathcal{Q}(K, B, D) :=\)
\[
|K| + \mathcal{J}\left( \mathbf{d}_D \right) + \mathcal{J} \left( Y_D,   \mathbf{x}_K^\mathrm{L} \right) - \mathcal{J}\left( \mathbf{x}_K|_{[\varphi^\mathrm{R}, \varphi^\mathrm{L}]} \right) + \mathcal{J} \left( \mathbf{x}_K^\mathrm{R}, X_B \right)  + \mathcal{J}\left( \mathbf{b}_B \right) .
\]

\label{def:upper-bound-q}
\end{definition}

\begin{proposition}

\(\mathcal{Q}\) is a quadratic functional on \(\mathcal{L}\).

\label{pro:upper-bound-q-quadratic}
\end{proposition}

\begin{proof}
The term \(\mathcal{J}\left( \mathbf{x}_K|_{[\varphi^\mathrm{R}, \varphi^\mathrm{L}]} \right)\) is quadratic in \(K\) by \Cref{pro:curve-area-functional-quadratic} and that \(\mathbf{x}_K(t) = (h_K(t) - 1) u_t + (h_K(t + \pi/2) - 1) v_t\) is convex-linear in \(K\) (\Cref{pro:rotating-hallway-parts} and (1) of \Cref{thm:convex-body-linear}). The terms \(\mathcal{J}\left( \mathbf{d}_D \right)\) and \(\mathcal{J}\left( \mathbf{b}_B \right)\) are quadratic in \(D\) and \(B\) by \Cref{thm:convex-curve-area-functional}. The terms \(\mathcal{J} \left( Y_D, \mathbf{x}_K^\mathrm{L} \right)\) and \(\mathcal{J} \left( \mathbf{x}_K^\mathrm{R}, X_B \right)\) are quadratic because \(X_B\), \(\mathbf{x}_K^\mathrm{R}\) and \(\mathbf{x}_K^\mathrm{L}\), \(Y_D\) are linear in \(B\), \(K\), \(D\) respectively by (2) of \Cref{thm:convex-body-linear}.
\end{proof}

Now we bound the right and left part of \(\mathcal{N}(K)\).

\begin{lemma}

Fix an arbitrary \(K \in \mathcal{K}^\mathrm{i}\) and take \(B := B_K\). We have
\[
\left| \mathcal{N}(K) \cap \breve{H}_K^\mathrm{R} \right| \geq \mathcal{J}\left( X_B, W_K^\mathrm{R} \right) -  \mathcal{J}\left( \mathbf{b}_B \right) 
\]
and similarly for \(D := D_K\), we have
\[
\left| \mathcal{N}(K) \cap \breve{H}_K^\mathrm{L} \right| \geq \mathcal{J}\left( Z_K^\mathrm{L}, Y_D \right)  -  \mathcal{J}\left( \mathbf{d}_D \right).
\]

\label{lem:cap-left-right-tail}
\end{lemma}

\begin{proof}
We show the first inequality. The second inequality can be proven by a mirror-symmetric argument.

Recall the \Cref{def:right-left-tails} that the curve \(\mathbf{b}_B := \mathbf{u}_{B}^{\pi + \varphi^\mathrm{R}, 3\pi/2}\) is the segment of \(\partial B\) from the vertex \(X_B := v_B^+(\varphi^\mathrm{R} + \pi)\) of \(B\) to the vertex \(W_B := v_B^-(3\pi/2)\) of \(B\). Also, we have \(W_K^\mathrm{R} = v_B\left( \pi + \varphi^\mathrm{R}, 3\pi/2 \right)\) by (2) of \Cref{lem:right-left-body}. Because \(W_K^\mathrm{R}\) and \(W_B\) are on the line \(l(\pi/2, 0)\), we have \(\mathcal{J}\left( W_K^\mathrm{R}, W_B \right) = 0\), which we will use implicitly.

If \(X_B = W_B\), then by \Cref{lem:convex-curve-cut} the curve \(\mathbf{b}_B\) and the points \(X_B, W_B, W_K^\mathrm{R}\) degenerate to a single point. So the lower bound becomes zero if \(X_B = W_B\). Now assume otherwise that \(X_B \neq W_B\).

Define the closed curve \(\Gamma\) obtained by following \(\mathbf{b}_B\) in the reverse direction, and then the line segments from \(X_B\) to \(W_K^\mathrm{R}\) and from \(W_K^\mathrm{R}\) to \(W_B\) respectively. Then by (1) of \Cref{lem:convex-curve-jordan-curve}, the curve \(\Gamma\) is a counterclockwise Jordan curve enclosing a region \(R\). By \Cref{thm:curve-area-functional-area} and \Cref{pro:curve-area-functional-additive}, the area of \(R\) is equal to \(\mathcal{J}\left( X_B, W_K^\mathrm{R} \right) - \mathcal{J}\left( \mathbf{b}_B \right)\). So it remains to show that the region \(R\) is contained in the set \(\mathcal{N}(K) \cap \breve{H}_K^\mathrm{R}\).

Take any point \(p \in R\). Our goal is to show that \(p \in \mathcal{N}(K) \cap \breve{H}_K^\mathrm{R}\). By (2) of \Cref{lem:convex-curve-jordan-curve} and (2) of \Cref{lem:right-left-body}, we have
\[
p \in H_B(\pi + \varphi^\mathrm{R}) \cap H_B(3\pi/2) = \breve{H}_K^\mathrm{R} \cap H_+(\pi/2, 0).
\]
Because \(B \subseteq K\) by definition, we have \(\mathbf{b}_B \subseteq K\). Also by \Cref{lem:cap-wz-in-edge}, we have \(W_K^\mathrm{R} \subseteq K\) so the closed curve \(\Gamma\) is contained in \(K\). Now \(R \subseteq K\) as \(K\) is contractible. We also have \(R\) disjoint from \(B\) by (3) of \Cref{lem:convex-curve-jordan-curve}. By \(p \in R \subseteq K\) and the \Cref{def:right-left-body} of \(B\), we have \(p \not\in H_K^\mathrm{b}(t)\) for some \(t \in [\varphi^\mathrm{R}, \pi/2]\). By (2) of \Cref{lem:convex-curve-jordan-curve}, we have \(t \neq \varphi^\mathrm{R}, \pi/2\). By (1) of \Cref{lem:monotonicity-intervals}, we now have \(p \in \breve{H}_K^\mathrm{R} \cap T_K(t) \subseteq \breve{H}_K^\mathrm{R} \cap \mathcal{N}(K)\).
\end{proof}

We now approximate the middle part of the niche traced out by the core \(\mathbf{x}_K\) restricted to the interval \([\varphi^\mathrm{R}, \varphi^\mathrm{L}]\).

\begin{lemma}

For any \(K \in \mathcal{K}^\mathrm{i}\), we have
\[
\left| \mathcal{N}(K) \setminus \breve{H}_K^\mathrm{R} \setminus \breve{H}_K^\mathrm{L} \right|  \geq \mathcal{J}(W_K^\mathrm{R}, \mathbf{x}_K^\mathrm{R}) + \mathcal{J}\left( \mathbf{x}_K|_{[\varphi^\mathrm{R}, \varphi^\mathrm{L}]} \right)  + \mathcal{J}(\mathbf{x}_K^\mathrm{L}, Z_K^\mathrm{L}).
\]

\label{lem:cap-middle-lower-estimate}
\end{lemma}

\begin{proof}
Denote \(\mathbf{x}_K|_{[\varphi^\mathrm{R}, \varphi^\mathrm{L}]}\), \(b_K^\mathrm{R}\), \(d_K^\mathrm{L}\) simply as \(\mathbf{x}\), \(b\), \(d\). Define \(Y := \mathbb{R}^2 \setminus \breve{H}_K^\mathrm{R} \setminus \breve{H}_K^\mathrm{L}\) which is the open cone bounded by the lines \(b\) and \(d\). The curve \(\mathbf{x}(t)\) on \(t \in [\varphi^\mathrm{R}, \varphi^\mathrm{L}]\) starts at \(b\) and ends at \(d\) by definition. Also, by \Cref{lem:monotonicity-intervals}, the middle parts \(\mathbf{x}(t)\) for \(t \in (\varphi^\mathrm{R}, \varphi^\mathrm{L})\) are inside the open cone \(Y\).

Take the horizontal line \(l_h\) described by the equation \(y = -h\) for sufficiently large \(h > 0\), so that the trajectory of \(\mathbf{x}\) is strictly above \(l_h\). We construct a closed curve \(\Gamma\) which is the concatenation of the following four curves in order.

\begin{enumerate}
\def\labelenumi{\arabic{enumi}.}
\tightlist
\item
  \(\mathbf{x} : [\varphi^\mathrm{R}, \varphi^\mathrm{L}] \to \mathbb{R}^2\).
\item
  The line segment \(s_L\) from \(\mathbf{x}(\varphi^\mathrm{L})\) to \(d \cap l_h\).
\item
  The line segment \(s_M\) from \(d \cap l_h\) to \(b \cap l_h\).
\item
  The line segment \(s_R\) from \(b \cap l_h\) to \(\mathbf{x}(\varphi^\mathrm{R})\).
\end{enumerate}

Then \(\Gamma\) is a Jordan curve since the interior of \(\mathbf{x}\) is inside the open cone \(Y\), and \(h\) is taken so that \(s_M\) is disjoint from \(\mathbf{x}\). Let \(G\) be the open region enclosed by the Jordan curve \(\Gamma\). Similarly, let \(R\) be the closed trapezoid right below the line \(l(\pi/2, 0)\) with the vertices \(Z_K^\mathrm{L}, d \cap l_h, b \cap l_h, W_K^\mathrm{R}\) in counterclockwise order. Then the lower bound of the stated inequality is \(|G| - |R|\) by following the boundaries of \(G\) and \(R\) in order and using \Cref{thm:curve-area-functional-area} twice.

So it suffices to show \(G \setminus R \subseteq \mathcal{N}(K) \cap Y\), as this will imply
\[
\left| \mathcal{N}(K) \setminus \breve{H}_K^\mathrm{R} \setminus \breve{H}_K^\mathrm{L} \right| = |\mathcal{N}(K) \cap Y| \geq |G| - |R|
\]
and proving the theorem. Let \(H_h\) be the closed half-plane with boundary \(l_h\) above the line \(l_h\). As \(G\) is contained in the cone \(Y\) and strictly above the line \(l_h\), we have \(G \subseteq Y \cap H_h\). On the other hand, by the definition of \(R\) we have \(Y \cap H_h \setminus H_+(\pi/2, 0) \subseteq R\). Thus we have \(G \setminus R \subseteq H_+(\pi/2, 0) \cap Y\). So it remains to show \(G \cap H_+(\pi/2, 0) \subseteq \mathcal{N}(K)\).

Take any point \(p \in G \cap H_+(\pi/2, 0)\). Let \(r\) be the half-line from \(p\) in the direction of \(v_0\), then since \(p \in G\) the half-line \(r\) intersects \(\Gamma\) at a point \(q \neq p\). As \(p \in H_+(\pi/2, 0)\), the point \(q\) is not in \(s_M\). Thus \(q\) is in one of the curves \(\mathbf{x}\), \(s_L\), or \(s_R\). If \(q = \mathbf{x}(t)\) for some \(t \in [\varphi^\mathrm{R}, \varphi^\mathrm{L}]\), then \(p \in Q_K^-(t)\). If \(q \in s_R\), then \(p \in Q_K^-(\varphi^\mathrm{R})\). If \(q \in s_L\), then \(p \in Q_K^-(\varphi^\mathrm{L})\). In any case, we have \(p \in \mathcal{N}(K)\), completing the proof.
\end{proof}

\begin{theorem}

For any \(K \in \mathcal{K}^\mathrm{i}\), we have \(\mathcal{A}(K) \leq \mathcal{Q}(K, B_K, D_K)\).

\label{thm:upper-bound-q}
\end{theorem}

\begin{proof}
Set \(B := B_K\) and \(D := D_K\). Add all the inequalities in \Cref{lem:cap-left-right-tail} and \Cref{lem:cap-middle-lower-estimate}. We get
\begin{align*}
\mathcal{A}(K) & = |K| - |\mathcal{N}(K)| \\
& = (|K| - |\mathcal{N}(K) \setminus \breve{H}_K^\mathrm{R} \setminus \breve{H}_K^\mathrm{L}|) - |\mathcal{N}(K) \cap \breve{H}_K^\mathrm{R}| - |\mathcal{N}(K) \cap \breve{H}_K^\mathrm{L}| \\
& \leq |K| - \mathcal{J}(W_K^\mathrm{R}, \mathbf{x}_K^\mathrm{R}) - \mathcal{J}\left( \mathbf{x}_K|_{[\varphi^\mathrm{R}, \varphi^\mathrm{L}]} \right) - \mathcal{J}(\mathbf{x}_K^\mathrm{L}, Z_K^\mathrm{L}) \\
& \phantom{{}={}} - \mathcal{J}\left( X_B, W_K^\mathrm{R} \right) + \mathcal{J}\left( \mathbf{b}_B \right) - \mathcal{J}\left( Z_K^\mathrm{L}, Y_D \right) +  \mathcal{J}\left( \mathbf{d}_D \right) \\ 
& = \mathcal{Q}(K, B, D)
\end{align*}
proving the theorem.
\end{proof}

\section{Concavity of $\texorpdfstring{\mathcal{Q}}{Q}$}
\label{sec:concavity-of-q}
We prove the concavity of \(\mathcal{Q}\) (\Cref{thm:upper-bound-concave}). The main idea is depicted in \Cref{fig:mamikon-sofa}. Each grey Mamikon region in the figure have an area that is \emph{convex} in the domain \(\mathcal{K}^\mathrm{i}\) or \(\mathcal{L}\) by \Cref{thm:mamikon-convex}. The function \(\mathcal{Q}\) plus the area of such Mamikon regions turns out to be linear in \(K \in \mathcal{K}^\mathrm{i}\), so the function \(\mathcal{Q}\) should be concave.

\begin{figure}
\centering
\includegraphics{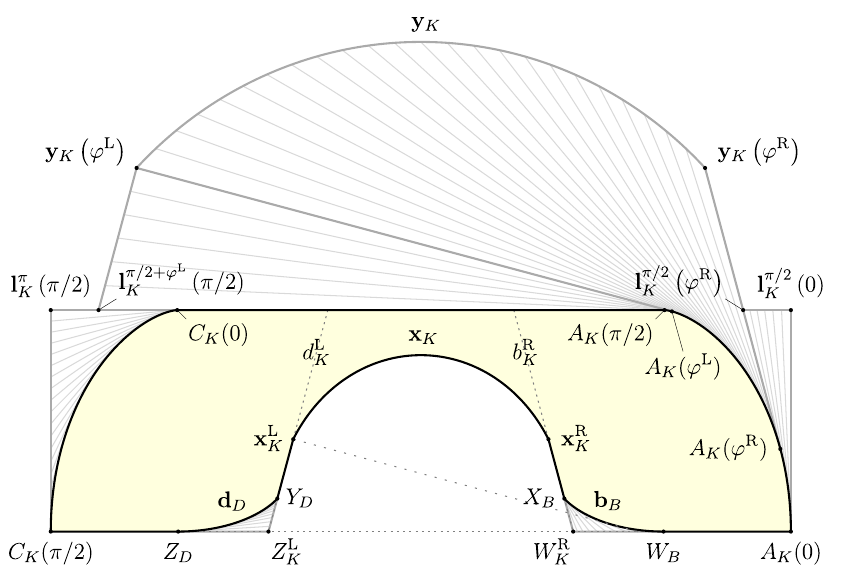}
\caption{The Mamikon regions used in the proof of \Cref{thm:upper-bound-concave} (compare with \Cref{fig:upper-bound-full}).}
\label{fig:mamikon-sofa}
\end{figure}

We first give a particular parametrization of the supporting line \(l_K(t)\) of a convex body \(K\).

\begin{definition}

Let \(K \in \mathcal{K}\) and \(t \in \mathbb{R}\) be arbitrary. Define \(\mathbf{l}^t_K : (t - \pi, t] \to \mathbb{R}^2\) as \(\mathbf{l}^t_K(s) = v_K(s, t)\) for \(s < t\) and \(\mathbf{l}^t_K(t) = v_K^-(t)\).

\label{def:tangent-line-parametrization}
\end{definition}

\begin{theorem}

Fix arbitrary \(K \in \mathcal{K}\) and \(t \in \mathbb{R}\). For any \(a, b \in (t - \pi, t]\) such that \(a \leq b\), the function \(\mathbf{l}_K^t\) restricted to \([a, b]\) is a continuous parametrization of the closed segment in \(l_K(t)\) from \(\mathbf{l}^t_K(a)\) to \(\mathbf{l}^t_K(b)\). Consequently, we have \(\mathcal{J}\left( \mathbf{l}_K^t|_{[a, b]} \right) = \mathcal{J}(\mathbf{l}_K^t(a), \mathbf{l}_K^t(b))\).

\label{thm:tangent-line-parametrization}
\end{theorem}

\begin{proof}
If \(b < t\), then the result follows from the geometric fact that for all \(s \in [a, b]\), the intersection \(v_K(s, t) = l_K(s) \cap l_K(t)\) is continuous in \(s\) and contained in the line segment connecting \(v_K(a, t)\) and \(v_K(b, t)\). Use the limit \(\lim_{ b \to t^- } \mathbf{l}_K^t(b) = \mathbf{l}_K^t(t)\) in \Cref{thm:limits-converging-to-vertex} to extend this to \(b = t\).
\end{proof}

\begin{theorem}

For fixed \(t \in \mathbb{R}\) and \(a, b \in (t - \pi, t]\), the value \(\mathbf{l}_K^t|_{[a, b]} \in C^{\mathrm{BV}}[a, b]\) is linear in \(K \in \mathcal{K}\).

\label{thm:tangent-line-param-linear}
\end{theorem}

\begin{proof}
By (2) of \Cref{thm:convex-body-linear}.
\end{proof}

We give names to the Mamikon regions in \Cref{fig:mamikon-sofa}.

\begin{definition}

For any \(K \in \mathcal{K}^\mathrm{i}\), define \(\mathcal{S}_K :=\)
\[
\mathcal{M}_K\left( 0, \varphi^\mathrm{R}; \mathbf{l}_K^{\pi/2} \right) + 
\mathcal{M}_K\left( \varphi^\mathrm{R}, \varphi^\mathrm{L}; \mathbf{y}_K \right) + 
\mathcal{M}_K\left( \varphi^\mathrm{L}, \pi/2; \mathbf{l}_K^{\pi/2 + \varphi^\mathrm{L}} \right) + 
\mathcal{M}_K\left( \pi/2, \pi; \mathbf{l}_K^{\pi} \right).
\]

\label{def:mamikon-middle}
\end{definition}

\begin{definition}

For convex bodies \(B\) and \(D\), define
\[
\mathcal{R}_B := \mathcal{M}_B\left( \pi/2 + \varphi^\mathrm{R}, 3\pi/2; \mathbf{l}_B^{3\pi/2} \right) \quad \text{ and } \quad \mathcal{L}_D := \mathcal{M}_D\left( 3\pi/2, 3\pi/2 + \varphi^\mathrm{L}; \mathbf{l}_D^{3\pi/2 + \varphi^\mathrm{L}} \right).
\]

\label{def:mamikon-right-left}
\end{definition}

Note that \Cref{def:mamikon-middle} and \Cref{def:mamikon-right-left} only uses Mamikon areas in \Cref{def:mamikon}. So by \Cref{thm:mamikon-convex}, these are convex and quadratic.

\begin{lemma}

The values \(\mathcal{S}_K\), \(\mathcal{L}_D\), \(\mathcal{R}_B\) are convex and quadratic as functionals on the convex bodies \(K, B, D\) respectively.

\label{lem:mamikon-sofa-convex}
\end{lemma}

\begin{proof}
By \Cref{thm:mamikon-convex}, and the linearity of \(\mathbf{y}_K(t) = h_K(t) u_t + h_K(\pi/2 + t) v_t\) and \(\mathbf{l}_K(t)\) in \(K\) (\Cref{thm:convex-body-linear} and \Cref{thm:tangent-line-param-linear}).
\end{proof}

\begin{definition}

For any \(K \in \mathcal{K}^\mathrm{i}\), define the functional
\[
\mathcal{P}_K := |K| + \mathcal{J}\left( Z_K^\mathrm{L}, \mathbf{x}_K^\mathrm{L} \right) -  \mathcal{J}\left( \mathbf{x}_K|_{[\varphi^\mathrm{R}, \varphi^\mathrm{L}]} \right) + \mathcal{J}\left( \mathbf{x}_K^\mathrm{R}, W_K^\mathrm{R} \right).
\]

\label{def:upper-bound-middle}
\end{definition}

\begin{lemma}

For any \((K, B, D) \in \mathcal{L}\), we have
\[
\mathcal{Q}(K, B, D) = \mathcal{P}_K - \mathcal{R}_B - \mathcal{L}_D.
\]

\label{lem:upper-bound-decomposition}
\end{lemma}

\begin{proof}
Unfold \Cref{def:upper-bound-q} of \(\mathcal{Q}\), \Cref{def:upper-bound-middle} of \(\mathcal{S}_K\), and \Cref{def:mamikon-right-left} of \(\mathcal{R}_B\) and \(\mathcal{L}_D\). It remains to verify
\begin{align*}
& |K| + \mathcal{J}\left( \mathbf{d}_D \right) + \mathcal{J} \left( Y_D,   \mathbf{x}_K^\mathrm{L} \right) - \mathcal{J}\left( \mathbf{x}_K|_{[\varphi^\mathrm{R}, \varphi^\mathrm{L}]} \right) + \mathcal{J} \left( \mathbf{x}_K^\mathrm{R}, X_B \right)  + \mathcal{J}\left( \mathbf{b}_B \right) \\
= \phantom{.} & |K| + \mathcal{J}\left( Z_K^\mathrm{L}, \mathbf{x}_K^\mathrm{L} \right) -  \mathcal{J}\left( \mathbf{x}_K|_{[\varphi^\mathrm{R}, \varphi^\mathrm{L}]} \right) + \mathcal{J}\left( \mathbf{x}_K^\mathrm{R}, W_K^\mathrm{R} \right) \\
\phantom{{}={}} & - \mathcal{J}\left( Z_D, Z_K^\mathrm{L} \right) - \mathcal{J}\left( Z_K^\mathrm{L}, Y_D \right) + \mathcal{J}\left( \mathbf{d}_D \right)  \\
\phantom{{}={}} & - \mathcal{J}\left( X_B, W_K^\mathrm{R}) - \mathcal{J}(W_K^\mathrm{R},  W_B \right) + \mathcal{J}\left( \mathbf{b}_B \right) 
\end{align*}
which holds because the four points \(Z_D, Z_K^\mathrm{L}, W_K^\mathrm{R}, W_B\) are colinear with \(O\), the three points \(Z_K^\mathrm{L}, \mathbf{x}_K^\mathrm{L}, Y_D\) are colinear, and the three points \(W_K^\mathrm{R}, \mathbf{x}_K^\mathrm{R}, X_B\) are colinear.
\end{proof}

Recall the \Cref{def:modulo-linear} that we write \(f(K) \equiv_K g(K)\) for functionals \(f, g\) on cap \(K\) if the difference \(f(K) - g(K)\) is linear in \(K\).

\begin{lemma}

For any \(K \in \mathcal{K}^\mathrm{i}\), we have
\[
|K| \equiv_K \mathcal{J}\left( \mathbf{u}_K^{0, \varphi^\mathrm{R}} \right)  + \mathcal{J}\left( \mathbf{u}_K^{\varphi^\mathrm{R}, \varphi^\mathrm{L}} \right) +
\mathcal{J}\left( \mathbf{u}_K^{\varphi^\mathrm{L}, \pi/2} \right) +
\mathcal{J}\left( \mathbf{u}_K^{\pi/2, \pi} \right) .
\]

\label{lem:upper-boundary-tracing}
\end{lemma}

\begin{proof}
We first break the value \(|K|\) into a sum of convex curve area functionals. By \Cref{thm:area-quadratic-expression}, we have \(|K| = \frac{1}{2} \int_{t \in S^1} h_K(t) \, \sigma_K(dt)\). As \(K\) is a cap with rotation angle \(\pi/2\), the measure \(\sigma_K\) is zero on the set \((\pi, 2\pi) \setminus \left\{ 3\pi/2 \right\}\) and \(h_K(3\pi/2) = 0\). As \(K \in \mathcal{K}^\mathrm{i}\), we have \(\sigma_K(0) = \sigma_K(\pi) = 0\) too by (1) of \Cref{def:injectivity-condition}. So
\[
|K| = \frac{1}{2} \int_{t \in (0, \pi)} h_K(t) \, \sigma_K(dt).
\]
This with \Cref{thm:convex-curve-area-functional} implies \(|K| = \mathcal{J}\left( \mathbf{u}_K^{0, \pi} \right)\). Now use \Cref{lem:convex-curve-concat} multiple times to obtain
\begin{align*}
|K| = \mathcal{J}\left( \mathbf{u}_K^{0, \varphi^\mathrm{R}} \right)  + \mathcal{J}\left( \mathbf{u}_K^{\varphi^\mathrm{R}, \varphi^\mathrm{L}} \right) +
\mathcal{J}\left( \mathbf{u}_K^{\varphi^\mathrm{L}, \pi/2} \right) +
\mathcal{J}(A_K(\pi/2), C_K(0)) + 
\mathcal{J}\left( \mathbf{u}_K^{\pi/2, \pi} \right) .
\end{align*}
Here we use that each of \(e_K(\varphi^\mathrm{R})\) and \(e_K(\varphi^\mathrm{L})\) is a single point, because \(K \in \mathcal{K}^\mathrm{i}\) and so (1) of \Cref{def:injectivity-condition} holds. The expression \(\mathcal{J}(A_K(\pi/2), C_K(0))\) is equal to \(\sigma_K \left( \left\{ \pi/2 \right\} \right) / 2\) by \Cref{pro:surface-area-measure-side-length} and \Cref{pro:curve-area-line-segment}, so is linear in \(K\). This completes the proof.
\end{proof}

\begin{lemma}

For any \(K \in \mathcal{K}^\mathrm{i}\), we have the followings.

\begin{enumerate}
\def\labelenumi{\arabic{enumi}.}
\tightlist
\item
  \(\mathcal{J}\left( \mathbf{y}_K|_{[\varphi^\mathrm{R}, \varphi^\mathrm{L}]} \right) \equiv_K \mathcal{J}\left( \mathbf{x}_K|_{[\varphi^\mathrm{R}, \varphi^\mathrm{L}]} \right)\)
\item
  \(\mathcal{J}\left( \mathbf{l}_K^{\pi/2}(\varphi^\mathrm{R}), \mathbf{y}_K(\varphi^\mathrm{R}) \right) \equiv_K \mathcal{J}\left( W_K^\mathrm{R}, \mathbf{x}_K^\mathrm{R} \right)\)
\item
  \(\mathcal{J}\left( \mathbf{l}_K^{\pi/2 + \varphi^\mathrm{L}}(\varphi^\mathrm{L}), \mathbf{y}_K(\varphi^\mathrm{L}) \right) \equiv_K \mathcal{J}\left( Z_K^\mathrm{L}, \mathbf{x}_K^\mathrm{L} \right)\)
\end{enumerate}

\label{lem:linvals}
\end{lemma}

\begin{proof}
For (1), we have \(\mathbf{y}_K(t) = \mathbf{x}_K(t) + u_t + v_t\). So with \(I := \left[ \varphi^\mathrm{R}, \varphi^\mathrm{L} \right]\) and \(c_t := u_t + v_t\), we have
\[
\begin{split}
\mathcal{J}(\mathbf{y}_K|_I) & = \frac{1}{2} \int_{\varphi^\mathrm{R}}^{\varphi^\mathrm{L}} \mathbf{y}_K(t) \times d \mathbf{y}_K(t) \\
& = \frac{1}{2} \int_{\varphi^\mathrm{R}}^{\varphi^\mathrm{L}} (\mathbf{x}_K(t) + c_t) \times d (\mathbf{x}_K(t) + c_t)  \\
& = \mathcal{J}(\mathbf{x}_K) + \frac{1}{2} \left( \int_{0}^\omega c_t \times d \mathbf{x}_K(t) 
+ \int_{0}^\omega \mathbf{x}_K(t) \times d c_t + \int_{0}^\omega c_t \times d c_t \right) 
\end{split}
\]
and the term in large bracket is convex-linear in \(K\).

For (2), observe that \(\mathbf{y}_K(\varphi^\mathrm{R}) - \mathbf{x}_K^\mathrm{R} = u_{\varphi^\mathrm{R}} + v_{\varphi^\mathrm{R}}\) is a constant \(c_1\) independent of \(K\). Likewise, the points \(\mathbf{l}_K^{\pi/2}(\varphi^\mathrm{R})\) and \(W_K^\mathrm{R}\) are the vertices of the parallelogram \(P_K^\mathrm{R}\), so their difference is a constant \(c_2 := \left(\sec(\varphi^\mathrm{R})(1 - \sin (\varphi^\mathrm{R})), 1\right)\) independent of \(K\). Now \(\mathcal{J}\left( W_K^\mathrm{R}, \mathbf{x}_K^\mathrm{R} \right) \equiv_K \mathcal{J}\left( W_K^\mathrm{R} + c_2, \mathbf{x}_K^\mathrm{R} + c_1 \right)\) by bilinearity of \(\mathcal{J}\). Proof of (3) is similar as (2).
\end{proof}

\begin{lemma}

For any \(K \in \mathcal{K}^\mathrm{i}\), we have \(\mathcal{S}_K \equiv_K - \mathcal{P}_K\).

\label{lem:mamikon-middle-eq}
\end{lemma}

\begin{proof}
Expand each term \(\mathcal{M}_K(-)\) in the \Cref{def:mamikon-middle} of \(\mathcal{S}_K\) using \Cref{def:mamikon}. Then \(\mathcal{S}_K\) is equal to the sum of all terms in the matrix below; each row sums up to a single expression of form \(\mathcal{M}_K(-)\). It is easiest to verify this by following the bold boundaries of four Mamikon regions (colored grey) in the upper part of \Cref{fig:mamikon-sofa} from right to left.
\begin{equation*}
\begin{alignedat}{8}
&
\mathcal{J}\left( A_K(0), \mathbf{l}_K^{\pi/2}(0) \right) & & \;
\mathcal{J}\left( \mathbf{l}_K^{\pi/2}(0), \mathbf{l}_K^{\pi/2}(\varphi^\mathrm{R}) \right) & & \;
\mathcal{J}\left( \mathbf{l}_K^{\pi/2}(\varphi^\mathrm{R}), A_K(\varphi^\mathrm{R}) \right) & & \; -
\mathcal{J}\left( \mathbf{u}_K^{0, \varphi^\mathrm{R}} \right)
\\
&
\mathcal{J}\left( A_K(\varphi^\mathrm{R}), \mathbf{y}_K(\varphi^\mathrm{R}) \right) & & \; 
\mathcal{J}\left( \mathbf{y}_K|_{[\varphi^\mathrm{R}, \varphi^\mathrm{L}]} \right) & & \;
\mathcal{J}\left( \mathbf{y}_K(\varphi^\mathrm{L}), A_K(\varphi^\mathrm{L}) \right) & & \; -
\mathcal{J}\left( \mathbf{u}_K^{\varphi^\mathrm{R}, \varphi^\mathrm{L}} \right)
\\
&
\mathcal{J}\left( A_K(\varphi^\mathrm{L}), \mathbf{y}_K(\varphi^\mathrm{L}) \right) & & \;
\mathcal{J}\left( \mathbf{y}_K(\varphi^\mathrm{L} ), \mathbf{l}_K^{\pi/2 + \varphi^\mathrm{L}}(\pi/2) \right) & & \;
\mathcal{J}\left( \mathbf{l}_K^{\pi/2 + \varphi^\mathrm{L}}(\pi/2), A_K(\pi/2) \right) & & \; - 
\mathcal{J}\left( \mathbf{u}_K^{\varphi^\mathrm{L}, \pi/2} \right)
\\
&
\mathcal{J}\left( C_K(0), \mathbf{l}_K^{\pi}(\pi/2) \right) & & \;
\mathcal{J}( \mathbf{l}_K^{\pi}(\pi/2), C_K(\pi/2)) & & & & \; - 
\mathcal{J}\left( \mathbf{u}_K^{\pi/2, \pi} \right)
\end{alignedat}
\end{equation*}
Call the term in \(i\)’th row and \(j\)’th column, including the signs, as simply \(J_{ij}\). Now check the following calculations.

\begin{itemize}
\tightlist
\item
  \(\sum_{i=1}^4 J_{i4} \equiv_K -|K|\) by \Cref{lem:upper-boundary-tracing}.
\item
  \(J_{11} = h_K(0) / 2 \equiv_K 0\) by (1) of \Cref{thm:convex-body-linear}.
\item
  \(J_{12} = \left( \mathbf{l}_K^{\pi/2}(0) - \mathbf{l}_K^{\pi/2}(\varphi^\mathrm{R}) \right) \cdot u_0 \equiv_K 0\) by (2) of \Cref{thm:convex-body-linear}.
\item
  \(J_{13} + J_{21} = \mathcal{J}\left( \mathbf{l}_K^{\pi/2}(\varphi^\mathrm{R}), \mathbf{y}_K(\varphi^\mathrm{R}) \right) \equiv_K \mathcal{J}\left( W_K^\mathrm{R}, \mathbf{x}_K^\mathrm{R} \right)\) by (2) of \Cref{lem:linvals}.
\item
  \(J_{22} = \mathcal{J}\left( \mathbf{y}_K|_{[\varphi^\mathrm{R}, \varphi^\mathrm{L}]} \right) \equiv_K \mathcal{J}\left( \mathbf{x}_K|_{[\varphi^\mathrm{R}, \varphi^\mathrm{L}]} \right)\) by (1) of \Cref{lem:linvals}.
\item
  \(J_{23} + J_{31} = 0\).
\item
  \(J_{32} \equiv_K \mathcal{J}(\mathbf{x}_K^\mathrm{L}, Z_K^\mathrm{L})\) by (3) of \Cref{lem:linvals}.
\item
  \(J_{33} \equiv_K J_{41} \equiv_K 0\) by (2) of \Cref{thm:convex-body-linear} and that the points are on \(l_K(\pi/2) = l(\pi/2, 1)\).
\item
  \(J_{42} = h_K(\pi/2)/2 \equiv_K 0\) by (1) of \Cref{thm:convex-body-linear}.
\end{itemize}

Add all the calculations in the list above to conclude \(\mathcal{S}_K \equiv_K - \mathcal{P}_K\).
\end{proof}

We finally assemble all the lemmas to prove the concavity of \(\mathcal{Q}\).

\begin{theorem}

The functional \(\mathcal{Q} : \mathcal{L} \to \mathbb{R}\) is concave.

\label{thm:upper-bound-concave}
\end{theorem}

\begin{proof}
We need to show that the value \(\mathcal{Q}(K, B, D)\) is quadratic and concave on \((K, B, D) \in \mathcal{L}\). By \Cref{lem:upper-bound-decomposition} and \Cref{lem:mamikon-middle-eq}, we have
\[
\mathcal{Q}(K, B, D) = \mathcal{P}_K - \mathcal{R}_B - \mathcal{L}_D \equiv_K - \mathcal{S}_K - \mathcal{R}_B - \mathcal{L}_D.
\]
By \Cref{lem:mamikon-sofa-convex}, the right-hand side is quadratic and concave on \((K, B, D) \in \mathcal{L}\).
\end{proof}

\section{Gerver's Sofa}
\label{sec:gerver's-sofa}
In this \Cref{sec:gerver's-sofa}, we extract the properties of Gerver’s sofa \(G\) we need for proving that \(G\) is a global optimum. We follow the derivation of \(G\) by Romik in Section 4 of \autocite{romikDifferentialEquationsExact2018}.

\subsubsection{\texorpdfstring{Boundary Curves of \(G\)}{Boundary Curves of G}}

In Romik’s derivation, Gerver’s sofa \(G\) is a monotone sofa with the boundary parametrized by five oriented curves \(\mathbf{A}, \mathbf{B}, \mathbf{C}, \mathbf{D}\), and \(\mathbf{x}\) (see \Cref{fig:gerver-curves}).

\begin{figure}
\centering
\includegraphics{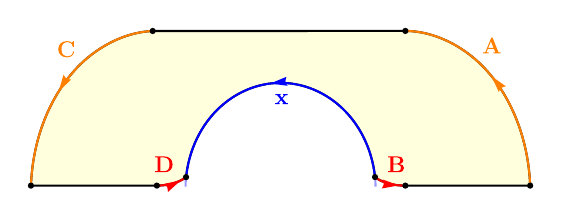}
\caption{Gerver’s sofa \(G\) is determined by five oriented curves \(\mathbf{A}, \mathbf{B}, \mathbf{C}, \mathbf{D}\), and \(\mathbf{x}\) solved in Section 4 of \autocite{romikDifferentialEquationsExact2018}.}
\label{fig:gerver-curves}
\end{figure}

The five curves are parametrized differently in each of the five intervals \(I_1, \dots, I_5\) in the following \Cref{def:gerver-intervals}.

\begin{definition}

Define
\[
(t_0, t_1, t_2, t_3, t_4, t_5) := (0, \varphi, \theta, \pi/2 - \theta, \pi/2 - \varphi, \pi/2)
\]
so that \(0 = t_0 < t_1 < \dots < t_5 = \pi/2\) forms a partition of the interval \([0, \pi/2]\). For each \(1 \leq i \leq 5\), define the interval \(I_i := [t_{i-1}, t_i]\).

\label{def:gerver-intervals}
\end{definition}

\begin{definition}

Define \(\mathbf{x} : [0, \pi/2] \to \mathbb{R}^2\) as the continuously differentiable function satisfying Equations 25 to 44 of \autocite{romikDifferentialEquationsExact2018} as solved in the Section 4 of \autocite{romikDifferentialEquationsExact2018}. In particular, by Equation 25 of \autocite{romikDifferentialEquationsExact2018}, for every \(1 \leq i \leq 5\) the restriction \(\mathbf{x}_i := \mathbf{x}|_{I_i}\) at the \(i\)’th interval \(I_i\) is smooth and satisfies the \(i\)’th ODE in Theorem 2 in \autocite{romikDifferentialEquationsExact2018}.

\label{def:gerver-x}
\end{definition}

\begin{definition}

Define the continuous, piecewise smooth curves \(\mathbf{A} : [0, \pi/2] \to \mathbb{R}^2\), \(\mathbf{B} : [t_3, t_5] \to \mathbb{R}^2\), \(\mathbf{C} : [0, \pi/2] \to \mathbb{R}^2\), and \(\mathbf{D} : [t_0, t_2] \to \mathbb{R}^2\) determined by \(\mathbf{x}\) in \Cref{def:gerver-x} according to Equations 9 to 12 of \autocite{romikDifferentialEquationsExact2018} respectively.

\label{def:gerver-abcd}
\end{definition}

The boundary of \(G\) is traced out by the four curves \(\mathbf{A}, \mathbf{B}, \mathbf{C}, \mathbf{D}\) in their full domains and the curve \(\mathbf{x}\) restricted to the interval \([t_1, t_4] = [\varphi^\mathrm{R}, \varphi^\mathrm{L}]\) (see \Cref{fig:gerver-curves} and the table below \Cref{thm:gerver-odes}).

\begin{theorem}

Gerver’s sofa \(G\) is a monotone sofa. Let \(K := \mathcal{C}(G)\) be the cap of \(G\). Then the followings are true.

\begin{enumerate}
\def\labelenumi{\arabic{enumi}.}
\tightlist
\item
  The cap \(K\) have vertices \(A_K(t) = \mathbf{A}(t)\) and \(C_K(t) = \mathbf{C}(t)\) and the inner corner \(\mathbf{x}_K(t) = \mathbf{x}(t)\) over \(t \in [0, \pi/2]\).
\item
  The niche \(\mathcal{N}(K)\) is the region enclosed counterclockwise by the following curves concatenated in order.

  \begin{enumerate}
  \def\labelenumii{\arabic{enumii}.}
  \tightlist
  \item
    The curve \(\mathbf{B} : [t_3, t_5] \to \mathbb{R}^2\) reversed in direction.
  \item
    The curve \(\mathbf{x}\) restricted to \([t_1, t_4] \to \mathbb{R}^2\).
  \item
    The curve \(\mathbf{D} : [t_0, t_2] \to \mathbb{R}^2\) reversed in direction.
  \item
    The horizontal line segment on \(l(\pi/2, 0)\) from \(\mathbf{D}(0)\) to \(\mathbf{B}(\pi/2)\).
  \end{enumerate}
\item
  The half-line inner wall \(\vec{b}_K(t)\) (resp. \(\vec{d}_K(t)\)) of the supporting hallway \(L_K(t)\) passes through the point \(\mathbf{B}(t)\) in the domain \(t \in [t_3, t_5]\) of \(\mathbf{B}\) (resp. \(\mathbf{D}(t)\) in the domain \(t \in [t_0, t_2]\) of \(\mathbf{D}\)).
\item
  The tangent direction \(\mathbf{B}'(t)\) of \(\mathbf{B}\) is parallel to \(v_t\), and \(\mathbf{B}'(t) \cdot v_t < 0\) on the domain \(t \in [t_3, t_5]\) of \(\mathbf{B}\). The tangent direction \(\mathbf{D}'(t)\) of \(\mathbf{D}\) is parallel to \(u_t\), and \(\mathbf{D}'(t) \cdot u_t > 0\) on the domain \(t \in [t_0, t_2]\) of \(\mathbf{D}\).
\end{enumerate}

\label{thm:gerver-monotone}
\end{theorem}

\begin{remark}

\Cref{thm:gerver-monotone} summarizes the properties of Gerver’s sofa \(G\) as described in the remark at the end of \autocite{gerverMovingSofaCorner1992} and recasts them our terminology. Rigorously speaking, these properties must be established logically \emph{before} asserting that \(G\) constitutes a valid moving sofa with five specific stages of movement as claimed in existing works on \(G\) \autocite{gerverMovingSofaCorner1992,romikDifferentialEquationsExact2018}.

The properties are easy to verify numerically and implicitly assumed in \autocite{gerverMovingSofaCorner1992,romikDifferentialEquationsExact2018}. Their truth is not entirely self-evident however. Indeed, Gerver observes a subtlety of the construction of \(G\) in his remark, that the endpoints \(\mathbf{x}(\varphi^R)\) and \(\mathbf{x}(\varphi^L)\) of the ‘core’ of the niche of \(G\) come within a distance of only \(0.0012\) from the inner walls \(b_K(t)\) and \(d_K(t)\) of the hallway \(L\) at \(t=\pi/4\). In other words, the construction of \(G\) comes within a distance of \(0.0012\) of ‘breaking’ near the hallway \(L_t\) with \(t=\pi/4\). With this, a rigorous symbolic verification of \Cref{thm:gerver-monotone} would still be worthy.

\label{rem:gerver-monotone}
\end{remark}

\subsubsection{\texorpdfstring{ODEs of Boundary Curves of \(G\)}{ODEs of Boundary Curves of G}}

We recall the concept of \emph{contact points} from \autocite{romikDifferentialEquationsExact2018}. Fix an interval \(I_i = [t_{i-1}, t_i]\) and let \(t \in I_i\) parametrize \(I_i\). Then Gerver’s sofa \(G\) with the cap \(K := \mathcal{C}(G)\) makes contact with the supporting hallway \(L_K(t)\) at the point \(\mathbf{A}(t)\) (resp. \(\mathbf{B}(t)\), \(\mathbf{C}(t)\), \(\mathbf{D}(t)\), and \(\mathbf{x}(t)\)) if and only if the interval \(I_i\) is contained in the domain of \(\mathbf{A}\) (resp. the domain of \(\mathbf{B}, \mathbf{C}, \mathbf{D}\), or the restricted domain \([t_1, t_4]\) of \(\mathbf{x}\)).

As this set of \emph{contact points} \(\Gamma(t)\), a subset of \(\left\{ \mathbf{A}, \mathbf{B}, \mathbf{C}, \mathbf{D}, \mathbf{x} \right\}\), that \(G\) makes with \(L_K(t)\) is determined, a corresponding ODE that balances the differential side lengths on those contact points can be derived. For example, the \Cref{eqn:ode-example} is derived in \Cref{sec:a-differential-inequality} by assuming \(\Gamma(t) = \left\{ \mathbf{A}, \mathbf{B}, \mathbf{C}, \mathbf{x} \right\}\). Each interval \(t \in I_i\) have a corresponding contact points \(\Gamma(t)\) fixed in \(t \in I_i\) and their respective ODE as follows.

\begin{definition}

Denote the dot product \(a \cdot b\) of \(a, b \in \mathbb{R}^2\) as also \(\left< a, b \right>\).

\label{def:bracket-dot-product}
\end{definition}

\begin{theorem}

For each \(t \in I_i\), the ODEs involving \(\mathbf{A}\), \(\mathbf{B}\), \(\mathbf{C}\), \(\mathbf{D}\), and \(\mathbf{x}\) in the last column of the following table holds.

\label{thm:gerver-odes}
\end{theorem}

\begin{longtable}[]{@{}
  >{\raggedright\arraybackslash}m{(\columnwidth - 6\tabcolsep) * \real{0.0685}}
  >{\raggedright\arraybackslash}m{(\columnwidth - 6\tabcolsep) * \real{0.1233}}
  >{\raggedright\arraybackslash}m{(\columnwidth - 6\tabcolsep) * \real{0.3425}}
  >{\raggedright\arraybackslash}m{(\columnwidth - 6\tabcolsep) * \real{0.4658}}@{}}
\toprule\noalign{}
\begin{minipage}[b]{\linewidth}\raggedright
\end{minipage} & \begin{minipage}[b]{\linewidth}\raggedright
Interval Contacts
\end{minipage} & \begin{minipage}[b]{\linewidth}\raggedright
Figure
\end{minipage} & \begin{minipage}[b]{\linewidth}\raggedright
Equation (Numbering in \autocite{romikDifferentialEquationsExact2018})
\end{minipage} \\
\midrule\noalign{}
\endhead
\bottomrule\noalign{}
\endlastfoot
(1)-v & \(t \in I_1\) \(\mathbf{A}, \mathbf{C}, \mathbf{D}\) & \includegraphics{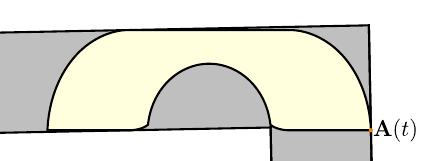} & \(\left< \mathbf{A}'(t), v_t \right> = 0\) (21) \\
(1)-u & \(t \in I_1\) \(\mathbf{A}, \mathbf{C}, \mathbf{D}\) & \includegraphics{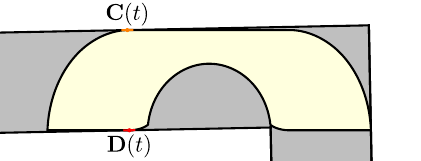} & \(\left< -\mathbf{C}'(t), u_t \right> = \left< \mathbf{D}'(t), u_t \right>\) (22) \\
(2)-v & \(t \in I_2\) \(\mathbf{A}, \mathbf{C}, \mathbf{D}, \mathbf{x}\) & \includegraphics{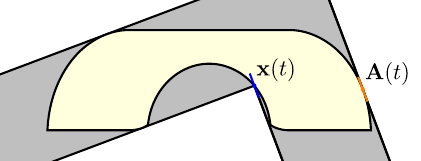} & \(\left< \mathbf{A}'(t), v_t \right> = \left< \mathbf{x}'(t), v_t \right>\) (17) \\
(2)-u & \(t \in I_2\) \(\mathbf{A}, \mathbf{C}, \mathbf{D}, \mathbf{x}\) & \includegraphics{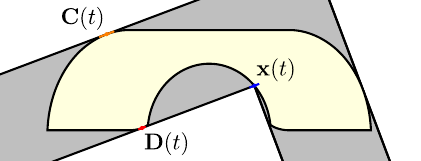} & \(\left< -\mathbf{C}'(t), u_t \right> = \left< \mathbf{D}'(t) - \mathbf{x}'(t), u_t \right>\) (19) \\
(3)-v & \(t \in I_3\) \(\mathbf{A}, \mathbf{C}, \mathbf{x}\) & \includegraphics{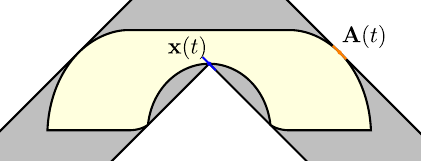} & \(\left< \mathbf{A}'(t), v_t \right> = \left< \mathbf{x}'(t), v_t \right>\) (17) \\
(3)-u & \(t \in I_3\) \(\mathbf{A}, \mathbf{C}, \mathbf{x}\) & \includegraphics{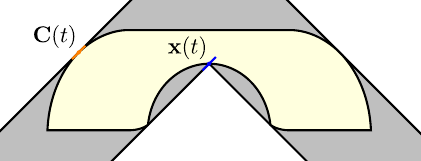} & \(\left< -\mathbf{C}'(t), u_t \right> = \left< -\mathbf{x}'(t), u_t \right>\) (18) \\
(4)-v & \(t \in I_4\) \(\mathbf{A}, \mathbf{B}, \mathbf{C}, \mathbf{x}\) & \includegraphics{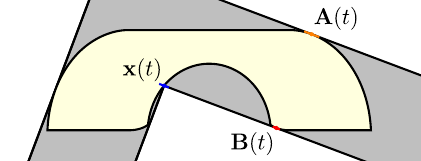} & \(\left< \mathbf{A}'(t), v_t \right> = \left< - \mathbf{B}'(t) + \mathbf{x}'(t) , v_t \right>\) (20) \\
(4)-u & \(t \in I_4\) \(\mathbf{A}, \mathbf{B}, \mathbf{C}, \mathbf{x}\) & \includegraphics{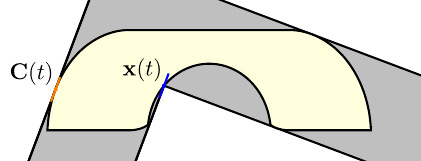} & \(\left< -\mathbf{C}'(t), u_t \right> = \left< -\mathbf{x}'(t), u_t \right>\) (18) \\
(5)-v & \(t \in I_5\) \(\mathbf{A}, \mathbf{B}, \mathbf{C}\) & \includegraphics{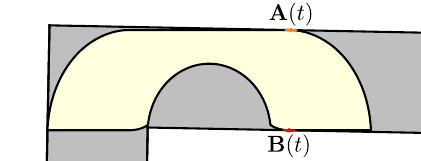} & \(\left< \mathbf{A}'(t), v_t \right> = \left<- \mathbf{B}'(t), v_t \right>\) (22)\footnote{\autocite{romikDifferentialEquationsExact2018} omits exact equations for the interval \(I_5\) as the cases \(I_1\) and \(I_5\) are symmetric to each other.} \\
(5)-u & \(t \in I_5\) \(\mathbf{A}, \mathbf{B}, \mathbf{C}\) & \includegraphics{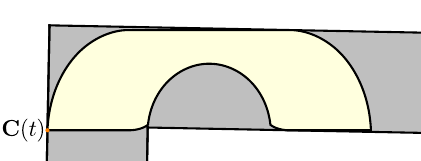} & \(\left< -\mathbf{C}'(t), u_t \right> = 0\) (21)\footnote{See above.} \\
\end{longtable}

\begin{proof}
They are exactly the equations with respective numberings in the proof of Theorem 2 of \autocite{romikDifferentialEquationsExact2018}.
\end{proof}

\subsubsection{\texorpdfstring{Left, Middle, and Right Parts of \(G\)}{Left, Middle, and Right Parts of G}}

\begin{theorem}

Let \(K := \mathcal{C}(G)\) be the cap of Gerver’s sofa. Then the followings are true.

\begin{enumerate}
\def\labelenumi{\arabic{enumi}.}
\tightlist
\item
  We have \(\mathbf{D}(t) = v_{D_K}^{\pm}(3\pi/2 + t)\) on \(t \in (t_0, t_2)\). Likewise, we have \(\mathbf{B}(t) = v_{B_K}^{\pm}(\pi + t)\) on \(t \in (t_3, t_5)\).
\item
  \(\mathbf{x}_K^\mathrm{L} = Y_{D_K} = \mathbf{D}(t_2) = v_{D_K}^{\pm}(3\pi/2 + t_2)\) and \(\mathbf{d}_{D_K} = \mathbf{D}\) as oriented curves. Likewise, \(\mathbf{x}_K^\mathrm{R} = X_{B_K} = \mathbf{D}(t_3) = v_{B_K}^{\pm}(\pi + t_3)\) and \(\mathbf{b}_{B_K} = \mathbf{B}\) as oriented curves.
\item
  We have \(h_K(\pi/2 + t) + h_{D_K}(3\pi/2 + t) = 1\) on \(t \in [t_0, t_2]\), and \(h_K(t) + h_{B_K}(\pi + t) = 1\) on \(t \in [t_3, t_5]\).
\end{enumerate}

\label{thm:gerver-left-right}
\end{theorem}

\begin{proof}
Recall that \(K \in \mathcal{K}^\mathrm{i}\) by \Cref{thm:injectivity-gerver}.

We first show (1). First, observe that \(\mathbf{D}\) is in the closure of \(\mathcal{N}(K)\) by (2) of \Cref{thm:gerver-monotone}, so we have \(\mathbf{D} \subseteq K\) by \Cref{thm:niche-in-cap}. By (2) of \Cref{thm:gerver-monotone}, we have the matching endpoints \(\mathbf{D}(t_2) = \mathbf{x}_K^\mathrm{L}\) of \(\mathbf{D}\) and \(\mathbf{x}\), so \(\mathbf{D}(t_2) \in d_K^\mathrm{L}\). Then by (4) of \Cref{thm:gerver-monotone}, we have \(\mathbf{D} \subseteq \breve{H}_K^\mathrm{L}\). Now \(\mathbf{D}\) is a subset of \(K \cap \breve{H}_K^\mathrm{L}\). Second, the curve \(\mathbf{D}\) itself is disjoint from \(\mathcal{N}(K)\) by (2) of \Cref{thm:gerver-monotone}, so by (2) of \Cref{lem:monotonicity-intervals}, \(\mathbf{D}\) is contained in all \(H_K^\mathrm{d}(t)\) over all \(t \in [0, \varphi^\mathrm{L})\). So we have \(\mathbf{D} \subseteq D_K\) by the definition of \(D_K\). Finally, as \(\mathbf{D}(t) \in D_K \cap d_K(t)\) for all \(t \in [t_0, t_2]\) by (3) of \Cref{thm:gerver-monotone}, and \(D_K \subseteq H_K^\mathrm{d}(t)\) by definition, the line \(d_K(t)\) should be a supporting line of \(D_K\) and \(\mathbf{D}(t) \in e_{D_K}(3\pi/2 + t)\) over all \(t \in [t_0, t_2]\). For each \(t \in (t_0, t_2)\), use the continuity of \(\mathbf{D}(t) \in e_{D_K}(3\pi/2 + t)\) at \(t\), and take the left and right limit on \(t\), then use \Cref{thm:limits-converging-to-vertex} to conclude \(\mathbf{D}(t) = v_{D_K}^{\pm}(3\pi/2 + t)\). Use a mirror-symmetric argument to prove the cooresponding statement of (1) on \(\mathbf{B}\).

We now show (2) using (1). Recall the definitions \(\mathbf{d}_{D_K} = \mathbf{u}_{D_K}^{3\pi/2, 3\pi/2 + \varphi^\mathrm{L}}\) and \(Y_{D_K} = v_{D_K}^-(3\pi/2 + \varphi^\mathrm{L})\). By taking the limits \(t \to t_0^+\) and \(t \to t_2^-\) to \(\mathbf{D}(t) = v_{D_K}^{\pm}(3\pi/2 + t)\) of (1) and using \Cref{thm:limits-converging-to-vertex}, we have \(\mathbf{D} = \mathbf{u}_{D_K}^{3\pi/2, 3\pi/2 + t_2}\) as oriented curves and \(\mathbf{D}(t_2) = v_{D_K}^-(3\pi/2 + t_2)\). We have \(\mathbf{D}(t_2) = \mathbf{x}_K^\mathrm{L}\) by (2) of \Cref{thm:gerver-monotone}, and the point is on the supporting line \(d_K^\mathrm{L} = l_{D_K}(3\pi/2 + \varphi^\mathrm{L})\) of \(D_K\) by (4) of \Cref{lem:right-left-body}. So we have
\[
v_{D_K}^-(3\pi/2 + t_2) \in l_{D_K}(3\pi/2 + \varphi^\mathrm{L})
\]
and this implies the degeneracy
\[
v_{D_K}^-(3\pi/2 + \varphi^\mathrm{L}) = v_{D_K}^\pm(3\pi/2 + t_2)
\]
and that \(\mathbf{u}_{D_K}^{3\pi/2 + t_2, 3\pi/2 + \varphi^\mathrm{L}}\) is a single point. Now the statement of (2) for the curve \(\mathbf{D}\) and convex body \(D_K\) holds. Use a mirror-symmetric argument to prove the corresponding statement of (2) on \(\mathbf{B}\) and \(B_K\).

We show (3) using (1) and (2). Let \(t \in (t_0, t_2]\) be arbitrary. We have \(\mathbf{C}(t) \in c_K(t)\) and \(\mathbf{D}(t) \in d_K(t)\) by (3) of \Cref{thm:gerver-monotone}, and the lines are parallel and of distance one. We also have \(\mathbf{C}(t) = C_K(t)\) by (1) of \Cref{thm:gerver-monotone} and \(\mathbf{D}(t) = v_{D_K}^{\pm}(3\pi/2 + t)\) by (1) and (2). So the first equality of (3) holds. Take limits to show the equality for \(t = t_0\) too. Again, use a mirror-symmetric argument to prove the second equality of (3) using (1).
\end{proof}

\subsubsection{\texorpdfstring{Surface Area Measures of \(G\)}{Surface Area Measures of G}}

Our next goal is to translate the ODEs of Romik in \Cref{thm:gerver-odes} in terms of surface area measure as differential side lengths. We will use the following \Cref{def:opposite-surface-area} for \(C = B, D\) to denote the surface area measure and support function of the bottom sides \(\mathbf{d}_D\) and \(\mathbf{b}_B\).

\begin{definition}

For any convex body \(C\), let \(\breve{\sigma}_C\) denote the measure on \(S^1\) such that \(\breve{\sigma}_C(X) = \sigma_C(X + \pi)\) for any Borel subset \(X \subseteq S^1\). Likewise, let \(\breve{h}_C(t) = h_C(t + \pi)\).

\label{def:opposite-surface-area}
\end{definition}

\begin{proposition}

Let \(K := \mathcal{C}(G)\) be the cap of Gerver’s sofa. Let \(B := B_K\) and \(D := D_K\). Then the followings are true.

\begin{enumerate}
\def\labelenumi{\arabic{enumi}.}
\tightlist
\item
  \(\left< \mathbf{A}'(t), v_t \right> \mathrm{d} t = \sigma_K\) as measures on \(t \in [0, \pi/2)\).
\item
  \(\left< -\mathbf{B}'(t), v_t \right> \mathrm{d} t = \breve{\sigma}_B\) as measures on \(t \in [t_3, t_5)\).
\item
  \(\left<-\mathbf{C}'(t - \pi/2), u_{t - \pi/2} \right> \mathrm{d} t = \sigma_K\) as measures on \(t \in (\pi/2, \pi]\).
\item
  \(\left< \mathbf{D}'(t), u_t \right> \mathrm{d} t = \breve{\sigma}_D\) as measures on \(t \in (t_0, t_2]\).
\end{enumerate}

\label{pro:measure-translation}
\end{proposition}

\begin{proof}
(1) and (3) for \(t \in (0, \pi/2)\) and \(t \in (\pi/2, \pi)\) respectively follow from (1) of \Cref{thm:gerver-monotone} and \Cref{thm:boundary-measure}. (1) and (3) on the singletons \(\left\{ 0 \right\}\) and \(\left\{ \pi \right\}\) follow from \Cref{thm:injectivity-gerver} and (1) of \Cref{def:injectivity-condition}. (2) and (4) follow from (1) and (2) of \Cref{thm:gerver-left-right} and \Cref{thm:boundary-measure}.
\end{proof}

We also introduce a measure \(\iota_K\) that captures the differential side lengths contributed by the inner corner \(\mathbf{x}_K\) of cap \(K\).

\begin{definition}

For any cap \(K \in \mathcal{K}^{\mathrm{i}}\), define the function \(i_K : (0, \pi] \to \mathbb{R}\) as \(i_K(t) := \left< \mathbf{x}_K'(t), v_t \right>\) and \(i_K(t + \pi / 2) := \left< -\mathbf{x}_K'(t) , u_t \right>\) for every \(t \in (0, \pi/2]\). Define \(\iota_K\) as the measure on \([0, \pi]\) derived from the density function \(i_K\). That is, \(\iota_K(dt) = i_K(t) dt\) and \(\iota_K\left( \left\{ 0 \right\} \right) = 0\) in particular.

\label{def:i-cap}
\end{definition}

Now we translate the equations in \Cref{thm:gerver-odes} to the equations of aforementioned measures.

\begin{definition}

For \(1 \leq i \leq 5\), define \(J_i := [t_{i-1}, t_i)\) where the values \(t_i\) are as in \Cref{def:gerver-intervals}. For \(6 \leq i \leq 10\), define \(J_i := \pi - J_{11 - i}\).

\label{def:interval-j}
\end{definition}

The intervals \(J_1, \dots, J_{10}\) and the singleton \(\left\{ \pi/2 \right\}\) partition the whole interval \([0, \pi]\).

\begin{theorem}

Let \(K := \mathcal{C}(G)\) be the cap of Gerver’s sofa, and let \(B := B_K\) and \(D := D_K\).

\begin{enumerate}
\def\labelenumi{\arabic{enumi}.}
\tightlist
\item
  \(\sigma_K = 0\) on \(J_1\).
\item
  \(\sigma_K = \iota_K\) on \(J_2 \cup J_3\).
\item
  \(\sigma_K = \breve{\sigma}_B + \iota_K\) on \(J_4\).
\item
  \(\sigma_K = \breve{\sigma}_B\) on \(J_5\).
\item
  \(\sigma_K = \breve{\sigma}_D\) on \(J_6\).
\item
  \(\sigma_K = \breve{\sigma}_D + \iota_K\) on \(J_7\).
\item
  \(\sigma_K = \iota_K\) on \(J_8 \cup J_9\).
\item
  \(\sigma_K = 0\) on \(J_{10}\).
\end{enumerate}

\label{thm:upper-bound-q-gerver}
\end{theorem}

\begin{proof}
Let \(1 \leq i \leq 5\) be arbitrary. For the interval \(J_i\) (resp. \(J_{i+5}\)), convert the Equation (\(i\))-v (resp. Equation (\(i\))-u) of \Cref{thm:gerver-odes} on the interval \(t \in I_i\) and direction \(v_t\) (resp. \(u_t\)) using (1) of \Cref{thm:gerver-monotone}, \Cref{def:i-cap}, and \Cref{pro:measure-translation}. Check the endpoints carefully.
\end{proof}

\begin{theorem}

Let \(K := \mathcal{C}(G)\) be the cap of Gerver’s sofa. Then \(\mathcal{A}(K) = \mathcal{Q}(K, B_K, D_K)\).

\label{thm:upper-bound-q-gerver-match}
\end{theorem}

\begin{proof}
Write \(B := B_K\) and \(D := D_K\). Start from the expression of \(\mathcal{Q}(K, B, D)\) in \Cref{def:upper-bound-q}. By (2) of \Cref{thm:gerver-left-right}, we have \(\mathcal{J}(\mathbf{d}_{D}) = \mathcal{J}(\mathbf{D})\) and \(\mathcal{J}(\mathbf{b}_{B}) = \mathcal{J}(\mathbf{B})\). By (2) of \Cref{thm:gerver-left-right}, we have \(\mathcal{J} \left( Y_D, \mathbf{x}_K^\mathrm{L} \right) = \mathcal{J} \left( \mathbf{x}_K^\mathrm{R}, X_B \right) = 0\). By integrating the Jordan curve boundary of \(\mathcal{N}(K)\) as in (2) of \Cref{thm:gerver-monotone}, we have \(|\mathcal{N}(K)| = \mathcal{J}\left( \mathbf{x}|_{[\varphi^\mathrm{R}, \varphi^\mathrm{L}]} \right) - \mathcal{J}(\mathbf{B}) - \mathcal{J}(\mathbf{D})\). So the value of \(\mathcal{Q}(K, B, D)\) in \Cref{def:upper-bound-q} reduces to \(|K| - |\mathcal{N}(K)| = \mathcal{A}(K)\).
\end{proof}

\section{Directional Derivative of $\texorpdfstring{\mathcal{Q}}{Q}$}
\label{sec:directional-derivative-of-q}
We calculate the directional derivative \(D \mathcal{Q}\) of \(\mathcal{Q}\) at the point \((K, B_K, D_K) \in \mathcal{L}\) corresponding to the cap \(K := \mathcal{C}(G)\) of Gerver’s sofa \(G\). The value turns out to be non-positive (\Cref{thm:variation-a2-gerver}). Since \(\mathcal{Q}\) is concave and quadratic (\Cref{thm:upper-bound-concave}), the value \((K, B_K, D_K) \in \mathcal{L}\) is the maximizer of \(\mathcal{Q}\) (\Cref{thm:quadratic-variation}).

We first prepare the theorems that calculate the directional derivative on each curve of \(\mathcal{Q}\).

\begin{theorem}

The area \(|K|\) of a cap \(K \in \mathcal{K}^\mathrm{i}\) is quadratic as a functional on \(\mathcal{K}^\mathrm{i}\). Moreover, the directional derivative of \(f(K) := |K|\) from \(K\) to \(K^*\) evaluates to
\[
D f(K; K^*) := \int_{t \in [0, \pi]} \left( h_{K^*}(t) - h_K(t) \right) \, \sigma_K (dt).
\]

\label{thm:variation-convex-body}
\end{theorem}

\begin{proof}
Recall that the \emph{mixed volume} \(V(K_1, K_2)\) on convex bodies \(K_1, K_2 \in \mathcal{K}\) is a symmetric bilinear functional on \(\mathcal{K}\) such that \(|K| = V(K, K)\) (e.g.~Chapter 5 of \autocite{schneider_2013}). So by \Cref{lem:derivative-calculation}, we have
\[
Df(K; K^*) = 2V(K^*, K) - 2V(K, K).
\]
We use the formula \(V(K_1, K_2) = \frac{1}{2} \int_{t \in S^1} h_{K_1}(t)\, \sigma_{K_2}(dt)\) of mixed volume (Equation (5.19) of \autocite{schneider_2013}). Substitue this in the equation above, and use that for caps \(K, K^* \in \mathcal{K}^\mathrm{i}\) we have \(\sigma_K = \sigma_{K^*} = 0\) on \(t \in (\pi, 2\pi) \setminus \left\{ 3\pi/2 \right\}\) and \(h_K(3\pi/2) = h_{K^*}(3\pi/2) = 0\).
\end{proof}

\begin{theorem}

Fix \(a, b \in \mathbb{R}\) such that \(a < b < a + \pi\). Define \(f(K) := \mathcal{J}\left( \mathbf{u}_K^{a, b} \right)\) as a functional on \(K \in \mathcal{K}\). Then \(f\) is quadratic on \(\mathcal{K}\) and its directional derivative from \(K\) to \(K^*\) evaluates to
\[
D f(K; K^*) = \int_{t \in (a, b)} \left( h_{K^*}(t) - h_K(t) \right) \, \sigma_{K}(dt) + \left[ \mathcal{J}\left( v_K^-(b), v_{K^*}^- (b) \right) - \mathcal{J}\left( v_K^+(a), v_{K^*}^+ (a) \right) \right] .
\]

\label{thm:convex-curve-area-variation}
\end{theorem}

\begin{proof}
Write \(f(K) = \mathcal{B}(K, K)\) where \(\mathcal{B}(K_1, K_2) := \frac{1}{2} \int_{t \in (a, b)} h_{K_1}(t) \, \sigma_{K_2}(dt)\) is defined as in \Cref{lem:convex-curve-bilinear-computation}. By integration by parts (\Cref{lem:integration-by-parts}) and that the left limit of \(v_{K_1}^+\) at \(t\) is \(v_{K_1}^-(t)\) (\Cref{thm:limits-converging-to-vertex}), we get
\[
\int_{t \in (a, b)} d v_{K_1}^+(t) \times v_{K_2}^+(t) + \int_{t \in (a, b)} v_{K_1}^-(t) \times d v_{K_2}^+(t) = v_{K_1}^-(b) \times v_{K_2}^-(b) - v_{K_1}^+(a) \times v_{K_2}^+(a).
\]
With \Cref{lem:convex-curve-bilinear-computation} and the anticommutativity \(a \times b = -b \times a\) of cross product, the equality becomes
\[
\mathcal{B}(K_1, K_2) - \mathcal{B}(K_2, K_1) = \mathcal{J}(v_{K_1}^-(b), v_{K_2}^-(b)) - \mathcal{J} (v_{K_1}^+(a), v_{K_2}^+(a))
\]
after dividing by two. Now let \(K_1 := K\) and \(K_2 := K^*\) then use \Cref{lem:derivative-calculation} on \(f(K) = \mathcal{B}(K, K)\) to get
\begin{align*}
D f(K; K^*) & = \mathcal{B}(K, K^*) + \mathcal{B}(K^*, K) - 2 \mathcal{B}(K, K) \\
& = 2\mathcal{B}(K^*, K) - 2 \mathcal{B}(K, K) + \mathcal{J}\left( v_K^-(b), v_{K^*}^- (b) \right) - \mathcal{J}\left( v_K^+(a), v_{K^*}^+ (a) \right)
\end{align*}
and use the definition of \(\mathcal{B}\).
\end{proof}

\begin{theorem}

Treat \(\mathbb{R}^2 \times \mathbb{R}^2\) as a convex domain. Then the curve area functional \(\mathcal{J}(p, q)\) on \((p, q) \in \mathbb{R}^2 \times \mathbb{R}^2\) is quadratic. For any \((p, q), (p^*, q^*) \in \mathbb{R}^2 \times \mathbb{R}^2\), the directional derivative of \(\mathcal{J}(-, -)\) evaluates to
\[
D\mathcal{J}(p, q ; p^*, q^*) = \frac{1}{2} \left( (p^* + q^*) \times (q - p) - 2(p \times q) \right) + \left[ \mathcal{J}(q, q^*) - \mathcal{J}(p, p^*) \right] 
\]
where the first term in the round bracket \(\left( - \right)\) equates to zero when \(p = q\).

\label{thm:variation-segment}
\end{theorem}

\begin{proof}
Define the bilinear form \(\mathcal{B}((p_1, q_1), (p_2, q_2)) = p_1 \times q_2 + p_2 \times q_1\) on \((p_1, q_1), (p_2, q_2) \in \mathbb{R}^2 \times \mathbb{R}^2\). Apply \Cref{lem:derivative-calculation} to \(4\mathcal{J}(p, q) = \mathcal{B}((p, q), (p, q))\) and use that \(\mathcal{B}\) is symmetric to obtain
\begin{align*}
2 \, D \mathcal{J}(p, q; p^*, q^*) &= \mathcal{B}((p^*, q^*), (p, q)) - \mathcal{B}((p, q), (p, q)) \\
& = p^* \times q + p \times q^* - 2(p \times q) \\
& = \left[ (p^* + q^*) \times (q - p) - 2(p \times q) \right] + q \times q^* - p \times p^*
\end{align*}
and divide by two to conclude the proof.
\end{proof}

\begin{theorem}

The curve area functional \(\mathcal{J}(\mathbf{x})\) on \(\mathbf{x} \in C^\mathrm{BV}[a, b]\) is quadratic. Moreover, the directional derivative of \(\mathcal{J}\) from \(\mathbf{x}\) to \(\mathbf{x}^*\) evaluates to
\[
D \mathcal{J}(\mathbf{x} ; \mathbf{x}^*) = \int_a^b (\mathbf{x}^*(t) - \mathbf{x}(t))  \times d\mathbf{x} (t) + \left[ \mathcal{J}(\mathbf{x}(b), \mathbf{x}^*(b)) - \mathcal{J}(\mathbf{x}(a), \mathbf{x}^*(a)) \right] .
\]

\label{thm:variation-curve}
\end{theorem}

\begin{proof}
Consider the bilinear form \(\mathcal{B}(\mathbf{x}_1, \mathbf{x}_2) = \int_a ^b \mathbf{x}_1(t) \times d \mathbf{x}_2(t)\) on \(\mathbf{x}_1, \mathbf{x}_2 \in C^\mathrm{BV}[a, b]\). Apply \Cref{lem:derivative-calculation} to \(2\mathcal{J}(\mathbf{x}) = \mathcal{B}(\mathbf{x}, \mathbf{x})\) to get
\begin{equation}
\label{eqn:variation-curve}
2 D \mathcal{J}(\mathbf{x} ; \mathbf{x}^*) = \mathcal{B}(\mathbf{x}, \mathbf{x}^*) + \mathcal{B}(\mathbf{x}^*, \mathbf{x}) - 2 \mathcal{B}(\mathbf{x}, \mathbf{x}).
\end{equation}
With integration by parts (\Cref{lem:integration-by-parts}), we get
\[
\int_a^b \mathbf{x}(t) \times d \mathbf{x}^*(t) = \mathbf{x} (b) \times \mathbf{x}^*(b) - \mathbf{x}(a) \times \mathbf{x}^*(a) + \int_a^b \mathbf{x}^*(t) \times d\mathbf{x} (t)
\]
or
\[
\mathcal{B}(\mathbf{x}, \mathbf{x}^*) = 2\mathcal{J}(\mathbf{x}(b), \mathbf{x}^*(b)) - 2\mathcal{J}(\mathbf{x}(a), \mathbf{x}^*(a)) + \mathcal{B}(\mathbf{x}^*, \mathbf{x}).
\]
Plug this back in \Cref{eqn:variation-curve} and rearrange the terms to get the desired equality.
\end{proof}

\begin{theorem}

Let \(I := [\varphi^\mathrm{R}, \varphi^\mathrm{L}]\). The value \(f(K) := \mathcal{J}\left( \mathbf{x}_K|_{I} \right)\) on caps \(K \in \mathcal{K}^\mathrm{i}\) is quadratic. Moreover, the directional derivative of \(f(K)\) from \(K\) to \(K^*\) evaluates to
\[
D f(K; K^*) := \left< h_{K^*} - h_K, \iota_K \right>_{I \cup (I + \pi/2)} + \left[ \mathcal{J}(\mathbf{x}_K^\mathrm{L}, \mathbf{x}_{K^*}^\mathrm{L}) - \mathcal{J}(\mathbf{x}_K^\mathrm{R}, \mathbf{x}_{K^*}^\mathrm{R}) \right].
\]

\label{thm:variation-inner-corner}
\end{theorem}

\begin{proof}
Apply \Cref{thm:variation-curve} to \(\mathbf{x} := \mathbf{x}_K|_I\) and \(\mathbf{x}^* := \mathbf{x}_{K^*}|_I\). Then compute
\begin{align*}
& \phantom{{}={}} \int_{t \in I} (\mathbf{x}^*(t) - \mathbf{x}(t))  \times d\mathbf{x} (t) \\
& = \int_{t \in I} (\mathbf{x}^*(t) - \mathbf{x}(t))  \times \left( \left<\mathbf{x}'(t), u_t\right> u_t + \left<\mathbf{x}'(t), v_t\right> v_t \right) dt \\
& = \int_{t \in I} \left<\mathbf{x}^*(t) - \mathbf{x}(t), v_t\right> i_K(\pi/2 + t) \, dt + \int_{t \in I} \left<\mathbf{x}^*(t) - \mathbf{x}(t), u_t\right> i_K(t) \, dt \\
& = \left< h_{K^*} - h_K, \iota_K \right>_{I \cup (I + \pi/2)}
\end{align*}
using the \Cref{def:i-cap} of \(i_K\) and \(\iota_K\) to finish the proof.
\end{proof}

\begin{theorem}

Let \(I := [\varphi^\mathrm{R}, \varphi^\mathrm{L}]\). Assume that \((K, B, D) \in \mathcal{L}\) satisfies the constraints \(X_B = \mathbf{x}_K^\mathrm{R}\) and \(Y_D = \mathbf{x}_K^\mathrm{L}\). Then the directional derivative of \(\mathcal{Q}\) at \((K, B, D)\) in the direction towards arbitrary \((K^*, B^*, D^*) \in \mathcal{L}\) is
\begin{align*}
& D\mathcal{Q}(K, B, D; K^*, B^*, D^*) = \\
& \phantom{{}={}} \left< h_{K^*} - h_K, \sigma_K \right>_{[0, \pi]} -
\left< h_{K^*} - h_K, \iota_K \right>_{I \cup (I + \pi/2)} \\
& \phantom{{}={}} + 
\left< \breve{h}_{B^*} - \breve{h}_B, \breve{\sigma}_B \right>_{(\varphi^\mathrm{R}, \pi/2)} + 
\left< \breve{h}_{D^*} - \breve{h}_D, \breve{\sigma}_D \right>_{(\pi/2, \pi/2 + \varphi^{\mathrm{L}})}.
\end{align*}

\label{thm:variation-a2}
\end{theorem}

\begin{proof}
Apply the following theorems to each term of \Cref{def:upper-bound-q}.

\begin{enumerate}
\def\labelenumi{\arabic{enumi}.}
\tightlist
\item
  \Cref{thm:variation-convex-body} on \(|K|\).
\item
  \Cref{thm:convex-curve-area-variation} on \(\mathcal{J}\left( \mathbf{d}_D \right) = \mathcal{J}\left( \mathbf{u}_D^{3\pi/2, 3\pi/2 + \varphi^\mathrm{L}} \right)\).
\item
  \Cref{thm:variation-segment} on \(\mathcal{J} \left( Y_D, \mathbf{x}_K^\mathrm{L} \right)\).
\item
  \Cref{thm:variation-inner-corner} on \(- \mathcal{J}\left( \mathbf{x}_K|_{[\varphi^\mathrm{R}, \varphi^\mathrm{L}]} \right)\).
\item
  \Cref{thm:variation-segment} on \(\mathcal{J} \left( \mathbf{x}_K^\mathrm{R}, X_B \right)\).
\item
  \Cref{thm:convex-curve-area-variation} on \(\mathcal{J}\left( \mathbf{b}_B \right) = \mathcal{J} \left( \mathbf{u}_B^{\pi + \varphi^\mathrm{R}, 3\pi/2} \right)\).
\end{enumerate}

Sum up all results from the theorems. The differences appearing in square brackets come from the matching endpoints, so they telescope and zeros in total.\footnote{Except for the endpoints \(W_B, W_{B^*}, Z_D\), and \(Z_{D^*}\) on the \(x\)-axis, but their cross product also zeroes out.} The remaining terms from (3) and (5) are zero by the assumptions \(X_B = \mathbf{x}_K^\mathrm{R}\) and \(Y_D = \mathbf{x}_K^\mathrm{L}\). The remaining terms from (1), (2), (4), and (6) correspond to the right-hand side in the claimed equality.
\end{proof}

We finally calculate the directional derivative at Gerver’s sofa.

\begin{theorem}

Let \(K := \mathcal{C}(G)\) be the cap of Gerver’s sofa \(G\). Let \(B := B_K\) and \(D := D_K\). Take arbitrary \((K^*, B^*, D^*) \in \mathcal{L}\). Then we have
\[
D \mathcal{Q}(K, B, D; K^*, B^*, D^*) \leq 0.
\]

\label{thm:variation-a2-gerver}
\end{theorem}

\begin{proof}
By (2) of \Cref{thm:gerver-left-right}, we can use \Cref{thm:variation-a2} to calculate the directional derivative. Split each term of \Cref{thm:variation-a2} into integrals over each interval \(J_i\) in \Cref{def:interval-j} as the following table. Note that the integrals on singleton \(\left\{ \pi/2 \right\}\) is zero because the differences of support functions are zero at \(t = \pi/2, 3\pi/2\) ((1) of \Cref{def:cap} and (2) and (4) of \Cref{lem:right-left-body}).

\begin{longtable}[]{@{}
  >{\raggedright\arraybackslash}m{(\columnwidth - 16\tabcolsep) * \real{0.2500}}
  >{\raggedright\arraybackslash}m{(\columnwidth - 16\tabcolsep) * \real{0.0938}}
  >{\raggedright\arraybackslash}m{(\columnwidth - 16\tabcolsep) * \real{0.0938}}
  >{\raggedright\arraybackslash}m{(\columnwidth - 16\tabcolsep) * \real{0.0938}}
  >{\raggedright\arraybackslash}m{(\columnwidth - 16\tabcolsep) * \real{0.0938}}
  >{\raggedright\arraybackslash}m{(\columnwidth - 16\tabcolsep) * \real{0.0938}}
  >{\raggedright\arraybackslash}m{(\columnwidth - 16\tabcolsep) * \real{0.0938}}
  >{\raggedright\arraybackslash}m{(\columnwidth - 16\tabcolsep) * \real{0.0938}}
  >{\raggedright\arraybackslash}m{(\columnwidth - 16\tabcolsep) * \real{0.0938}}@{}}
\toprule\noalign{}
\begin{minipage}[b]{\linewidth}\raggedright
\end{minipage} & \begin{minipage}[b]{\linewidth}\raggedright
\(J_1\)
\end{minipage} & \begin{minipage}[b]{\linewidth}\raggedright
\(J_2 \cup J_3\)
\end{minipage} & \begin{minipage}[b]{\linewidth}\raggedright
\(J_4\)
\end{minipage} & \begin{minipage}[b]{\linewidth}\raggedright
\(J_5\)
\end{minipage} & \begin{minipage}[b]{\linewidth}\raggedright
\(J_6\)
\end{minipage} & \begin{minipage}[b]{\linewidth}\raggedright
\(J_7\)
\end{minipage} & \begin{minipage}[b]{\linewidth}\raggedright
\(J_8 \cup J_9\)
\end{minipage} & \begin{minipage}[b]{\linewidth}\raggedright
\(J_{10}\)
\end{minipage} \\
\midrule\noalign{}
\endhead
\bottomrule\noalign{}
\endlastfoot
\(\left< h_{K^*} - h_K, \sigma_K \right>\) & O & O & O & O & O & O & O & O \\
\(- \left< h_{K^*} - h_K,\iota_K \right>\) & & O & O & & & O & O & \\
\(\left< \breve{h}_{B^*} - \breve{h}_B, \breve{\sigma}_B \right>\) & & & O & O & & & & \\
\(\left< \breve{h}_{D^*} - \breve{h}_D, \breve{\sigma}_D \right>\) & & & & & O & O & & \\
\end{longtable}

For each interval, the sum of integrals in \Cref{thm:variation-a2} on the interval \(J_i\) is non-positive as follows. For the \(i\)’th interval, we start with the \(i\)’th equality on measures in \Cref{thm:upper-bound-q-gerver}.

\begin{enumerate}
\def\labelenumi{\arabic{enumi}.}
\tightlist
\item
  As \(\sigma_K = 0\) on \(J_1\), the sum is zero.
\item
  As \(\sigma_K = \iota_K\) on \(J_2 \cup J_3\), the sum is zero.
\item
  As \(\sigma_K = \breve{\sigma}_B + \iota_K\) on \(J_4\), the sum equates to \(\left< h_{K^*} - h_K + h_{B^*} - h_B, \breve{\sigma}_B \right>\). We have \(h_K + h_{B} = 1\) on \(J_4 \cup J_5\) by (3) of \Cref{thm:gerver-left-right}, and \(h_{K^*} + h_{B^*} \leq 1\) on \(J_4 \cup J_5\) by (1) of \Cref{lem:right-left-body}. This with \(\breve{\sigma}_B \geq 0\) shows that the sum should be nonnegative.
\item
  As \(\sigma_K = \breve{\sigma}_B\) on \(J_5\), the sum equates to \(\left< h_{K^*} - h_K + h_{B^*} - h_B, \breve{\sigma}_B \right>\). Proceed as (3) above.
\item
  As \(\sigma_K = \breve{\sigma}_D\) on \(J_6\), the sum equates to \(\left< h_{K^*} - h_K + h_{D^*} - h_D, \breve{\sigma}_D \right>\). Proceed as (6) below.
\item
  As \(\sigma_K = \breve{\sigma}_D + \iota_K\) on \(J_7\), the sum equates to \(\left< h_{K^*} - h_K + h_{D^*} - h_D, \breve{\sigma}_D \right>\). We have \(h_K + h_{D} = 1\) on \(J_6 \cup J_7\) by (3) of \Cref{thm:gerver-left-right}, and \(h_{K^*} + h_{D^*} \leq 1\) on \(J_6 \cup J_7\) by (3) of \Cref{lem:right-left-body}. This with \(\breve{\sigma}_D \geq 0\) shows that the sum should be nonnegative.
\item
  As \(\sigma_K = \iota_K\) on \(J_8 \cup J_9\), the sum is zero.
\item
  As \(\sigma_K = 0\) on \(J_1\), the sum is zero.
\end{enumerate}

Summing up (1)-(8) above, the directional derivative is non-positive as desired.
\end{proof}

\begin{corollary}

Let \(K := \mathcal{C}(G)\) be the cap of Gerver’s sofa \(G\). Then the triple \((K, B_K, D_K) \in \mathcal{L}\) attains the maximum area of \(\mathcal{Q} : \mathcal{L} \to \mathbb{R}\).

\label{cor:gerver-max-cap}
\end{corollary}

\begin{proof}
By \Cref{thm:variation-a2-gerver} and \Cref{thm:quadratic-variation}.
\end{proof}

We finally prove the optimality of Gerver’s sofa \(G\) by assembling the pieces.

\begin{proof}[Proof of \Cref{thm:main}]
By \Cref{thm:rotation-angle-simple-bound}, \Cref{thm:limiting-maximum-sofa}, and \Cref{thm:angle}, some balanced maximum sofa \(S^*\) with rotation angle \(\pi/2\) attains the maximum area of a moving sofa. Let \(K^*\) be the cap of \(S^*\), then we have \(K^* \in \mathcal{K}^\mathrm{i}\) by \Cref{thm:cap-space-special}, and \((K^*, B_{K^*}, D_{K^*}) \in \mathcal{L}\) by \Cref{thm:cap-tail-extension}. We also have \(\left| S^* \right| = \mathcal{A}(K^*) \leq \mathcal{Q}(K^*, B_{K^*}, D_{K^*})\) by \Cref{thm:sofa-area-functional} and \Cref{thm:upper-bound-q}.

Let \(K := \mathcal{C}(G)\) be the cap of Gerver’s sofa. By \Cref{cor:gerver-max-cap} and \Cref{thm:upper-bound-q-gerver-match}, we have \(\mathcal{Q}(K^*, B_{K^*}, D_{K^*}) \leq \mathcal{Q}(K, B_K, D_K) = |G|\). Summarizing, we have \(|S^*| \leq |G|\) and Gerver’s sofa \(G\) attains the maximum area.
\end{proof}

\begin{remark}

The choice of constants \(\varphi^\mathrm{R}\) and \(\varphi^\mathrm{L}\) different from \((\varphi^\mathrm{R}, \varphi^\mathrm{L}) := (\varphi, \pi/2 - \varphi)\) of Gerver’s sofa will also give a valid upper bound \(\mathcal{Q} : \mathcal{L} \to \mathbb{R}\) of the sofa area. However, for such a \(\mathcal{Q}\), the maximizer \((K, B, D)\) of \(\mathcal{Q}\) does not give a valid moving sofa (that is, not in the embedding \(\mathcal{K}^\mathrm{i} \to \mathcal{L}\) of \Cref{thm:cap-tail-extension}) because the ends of the tails \(\mathbf{d}_D\), \(\mathbf{b}_B\) and core \(\mathbf{x}_K|_{[\varphi^\mathrm{R}, \varphi^\mathrm{L}]}\) do not match. The choice \((\varphi^\mathrm{R}, \varphi^\mathrm{L}) := (\varphi, \pi/2 - \varphi)\) ensures that the maximizer \((K, B, D)\) of \(\mathcal{Q}\) is in the image of \(\mathcal{K}^\mathrm{i} \to \mathcal{L}\) and gives back the moving sofa \(G\).

\label{rem:quadratic-upper-bound}
\end{remark}

\appendix
\chapter{Table of Symbols}
\label{sec:table-of-symbols}
\begin{longtable}[]{@{}
  >{\raggedright\arraybackslash}m{(\columnwidth - 4\tabcolsep) * \real{0.1522}}
  >{\raggedright\arraybackslash}m{(\columnwidth - 4\tabcolsep) * \real{0.6087}}
  >{\raggedright\arraybackslash}m{(\columnwidth - 4\tabcolsep) * \real{0.2391}}@{}}
\toprule\noalign{}
\begin{minipage}[b]{\linewidth}\raggedright
Symbol
\end{minipage} & \begin{minipage}[b]{\linewidth}\raggedright
Meaning
\end{minipage} & \begin{minipage}[b]{\linewidth}\raggedright
Location
\end{minipage} \\
\midrule\noalign{}
\endhead
\bottomrule\noalign{}
\endlastfoot
\(\lvert X \rvert\) & Area of \(X \subseteq \mathbb{R}^2\) & \Cref{def:area} \\
\(\overline{X}, X^\circ, \partial X\) & Topological closure, interior, and boundary of \(X \subseteq \mathbb{R}^2\) & \Cref{def:topological-operations} \\
\(L, H_L, V_L\) & The right-angled hallway of unit width and its horizontal/vertical sides & \Cref{def:hallway} \\
\(a_L, b_L, c_L, d_L\), \(\vec{b}_L, \vec{d}_L\) & Walls of the hallway \(L\) & \Cref{def:hallway-parts} \\
\(\mathbf{x}_L, \mathbf{y}_L\) & Inner and outer corner of the hallway \(L\) & \Cref{def:hallway-parts} \\
\(S\) & Moving sofa & \Cref{def:moving-sofa} \\
\(G\) & Gerver’s Sofa & \Cref{sec:introduction} \\
\(\alpha_{\max}\) & Maximum area of a moving sofa & \Cref{sec:introduction} \\
\(\omega\) & Rotation angle of a sofa & \Cref{def:rotation-angle} \\
\(H\) & Horizontal strip \([0, 1] \times \mathbb{R}\) & \Cref{def:strips} \\
\(V, V_\omega\) & Vertical strip \(\mathbb{R} \times [0, 1]\), and its counterclockwise rotation by \(\omega\) at the origin & \Cref{def:strips} \\
\(P_\omega\) & Parallelogram of rotation angle \(\omega\) & \Cref{def:parallelogram} \\
\(O\) & Origin \((0, 0)\) & \Cref{def:parallelogram} \\
\(o_\omega\) & Upper-right vertex of \(P_\omega\) & \Cref{def:parallelogram} \\
\(R_t\) & Rotation of \(\mathbb{R}^2\) at the origin by a counterclockwise angle \(t\) & \Cref{def:rotation-map} \\
\(u_t, v_t\) & Orthogonal unit vectors \((\cos t, \sin t)\) and \((-\sin t, \cos t)\) of angle \(t\) & \Cref{def:frame} \\
\(l(t, h)\) & Line with normal angle \(t\) and distance \(h\) from the origin & \Cref{def:line} \\
\(H_{\pm}(t, h)\), \(H_{\pm}^\circ(t, h)\) & Half-plane with normal angle \(t\) and distance \(h\) from the origin & \Cref{def:half-plane} \\
\(K\) & A cap or a planar convex body & \Cref{def:convex-body} \\
\(\mathcal{K}\) & Space of all planar convex bodies & \Cref{def:convex-body} \\
\(h_K(t)\) & Supporting function of a planar convex body \(K\) & \Cref{def:support-function} \\
\(l_K(t)\) & Supporting line of a planar convex body \(K\) with normal angle \(t\) & \Cref{def:supporting-line-half-plane} \\
\(H_K(t)\) & Supporting half-plane of a planar convex body \(K\) with normal angle \(t\) & \Cref{def:supporting-line-half-plane} \\
\(v_K^{\pm}(t)\) & Vertices of a planar convex body \(K\) & \Cref{def:convex-body-vertex} \\
\(v_K(a, b)\) & Intersection \(l_K(a) \cap l_K(b)\) & \Cref{def:convex-body-tangent-lines-intersection} \\
\(e_K(t)\) & Edge (face) of a planar convex body \(K\) & \Cref{def:convex-body-vertex} \\
\(\mathcal{H}^1\) & Hausdorff measure of dimension one on \(\mathbb{R}^2\) & \Cref{def:hausdorff-measure} \\
\(\sigma_K\) & Surface area measure of a planar convex body \(K\) & \Cref{def:surface-area-measure} \\
\(d_{\text{H}}\) & Hausdorff distance between two convex bodies & \Cref{def:hausdorff-distance} \\
\(L_S(t), L_K(t)\) & Supporting hallway of set \(S\) or cap \(K\) & \Cref{def:tangent-hallway} \\
\(a_K(t), b_K(t)\), \(c_K(t), d_K(t)\), \(\vec{b}_K(t), \vec{d}_K(t)\) & Walls of supporting hallway \(L_K(t)\) & \Cref{def:rotating-hallway-parts} \\
\(\mathbf{x}_K(t), \mathbf{y}_K(t)\) & Inner and outer corner of supporting hallway \(L_K(t)\) & \Cref{def:rotating-hallway-parts} \\
\(\mathcal{K}^\mathrm{c}_\omega\) & Space of all caps with rotation angle \(\omega\) & \Cref{def:cap-space} \\
\(\mathcal{K}^\mathrm{c}\) & Space of all caps with rotation angle \(\omega = \pi/2\) & \Cref{def:cap-space-right-angle} \\
\(\mathcal{K}^\mathrm{i}\) & Space of all caps with rotation angle \(\pi/2\) satisfying injectivity condition & \Cref{def:cap-space-special} \\
\(\mathcal{K}^\mathrm{c}_\Theta\) & Space of polygon caps with angle set \(\Theta\) & \Cref{def:angled-cap-space} \\
\(\mathcal{K}^\mathrm{t}_\Theta\) & Space of all translations of polygon caps with angle set \(\Theta\) & \Cref{def:cap-trans} \\
\(\mathcal{H}_\Theta\) & Space of generalized support functions with angle set \(\Theta\) & \Cref{def:height-space} \\
\(\mathcal{I}(S)\) & The intersection/monotone sofa arising from a moving sofa \(S\) in standard position & \Cref{def:monotonization} \\
\(\mathcal{C}(S)\) & The cap of a moving sofa & \Cref{def:cap-sofa} \\
\(\delta K\) & Upper boundary of a cap \(K\) & \Cref{def:upper-boundary-of-cap} \\
\(\mathcal{N}(K)\) & The niche of a cap & \Cref{def:niche} \\
\(T_K(t)\) & Wedge of angle \(t\) from cap \(K\) & \Cref{def:wedge} \\
\(W_K(t), Z_K(t)\) & Left and right endpoints of the wedge \(T_K(t)\) & \Cref{def:wedge-endpoints} \\
\(w_K(t), z_K(t)\) & Wedge gap of a cap \(K\) & \Cref{def:wedge-side-lengths} \\
\(A_K^\pm(t), C_K^\pm(t)\) & Vertices of a cap \(K\) & \Cref{def:cap-vertices} \\
\(\Theta\) & Angle set & \Cref{def:angle-set} \\
\(\Theta_{\omega, n}\) & Uniform angle set of \(n\) intervals and rotation angle \(\omega\) & \Cref{def:uniform-angle-set} \\
\(\Theta_n\) & Uniform angle set of \(n\) intervals and rotation angle \(\omega = \pi/2\) & \Cref{def:right-angle-set} \\
\(\tau_K\) & Edge length of the niche \(\mathcal{N}_\Theta(K)\) of polygon cap \(K\) & \Cref{def:polyline-length} \\
\(k_0, m_0\) & Real-valued functions used in the proof of injectivity condtion. & \Cref{def:magic-function} \\
\(f^{\pm}_K(t), g^{\pm}_K(t)\) & Arm lengths; distance from \(\mathbf{y}_K(t)\) to \(A^\pm_K(t)\) and \(C^\pm_K(t)\) respectively & \Cref{def:cap-tangent-arm-length} \\
\(\mathrm{d} f\) & Lebesgue–Stieltjes measure of \(f\) & \Cref{def:lebesgue-stieltjes} \\
\(\mathcal{V}, c_\lambda\) & Convex domain and its barycentric operation & \Cref{def:convex-spaces} \\
\(D f\) & Directional derivative of quadratic functional \(f\) on a convex domain & \Cref{def:convex-space-directional-derivative} \\
\(\theta, \varphi\) & Angle constants determining Gerver’s sofa & \Cref{def:gerver-constants} \\
\(\varphi^\mathrm{R}, \varphi^\mathrm{L}\) & Constants \(\varphi^\mathrm{R} := \varphi\) and \(\varphi^\mathrm{L} := \pi/2 - \varphi\) & \Cref{def:gerver-constants} \\
\(\mathcal{L}\) & Space of triples \((K, B, D)\) of convex bodies extending \(\mathcal{K}^\mathrm{i}\) & \Cref{def:cap-tail-space} \\
\(\mathcal{Q}\) & Upper bound of sofa area on the space \(\mathcal{L}\) & \Cref{def:upper-bound-q} \\
\(\mathcal{J}(\mathbf{x})\) & Curve area functional & \Cref{def:curve-area-functional} \\
\(\mathcal{J}(p, q)\) & Curve area functional from point \(p\) to \(q\) & \Cref{def:curve-area-line-segment} \\
\(\mathbf{u}_K^{a, b}\) & Convex curve segment of the boundary of \(K\) & \Cref{def:convex-curve} \\
\(\mathcal{M}_K(a, b; \mathbf{z})\) & The area of a Mamikon region bounded by \(\mathbf{u}_K^{a, b}\) and \(\mathbf{z}:[a, b] \to \mathbb{R}\). & \Cref{def:mamikon} \\
\(t_0, \dots, t_5\) & The angles \(0 = t_0 < \dots < t_5 = \pi/2\) used in Gerver’s sofa \(G\) & \Cref{def:gerver-intervals} \\
\(I_1, \dots, I_5\) & The intervals \(I_i = [t_{i-1}, t_i]\) & \Cref{def:gerver-intervals} \\
\(\mathbf{A}, \mathbf{B}, \mathbf{C}, \mathbf{D}, \mathbf{x}\) & The boundaries of Gerver’s sofa \(G\) & \Cref{def:gerver-abcd} \\
\end{longtable}

\backmatter
\printbibliography

\end{document}